\documentclass[11pt]{article}

\usepackage{enumerate}

\usepackage[margin=0.8in]{geometry}
\usepackage{amsmath,amsfonts,amssymb,graphics,amsthm, comment}
\usepackage{hyperref}
\hypersetup{
    colorlinks=true,
    linkcolor=blue,
    citecolor=red,
    urlcolor=blue,
    pdfborder={0 0 0}
}

\usepackage{graphicx}
\usepackage{xcolor}
\usepackage{bbm}
\usepackage{wrapfig} % for figures inserted on the side of the text

\usepackage[font=sf, labelfont={sf,bf}, margin=1cm]{caption}

\usepackage{verbatim}
\usepackage[title,titletoc]{appendix}
\usepackage{mathrsfs}

\usepackage[T1]{fontenc}
\usepackage[utf8]{inputenc}

\usepackage{import}
\usepackage{xifthen}
\usepackage{pdfpages}
\usepackage{transparent}
\usepackage{chngcntr}
\counterwithin{figure}{section}
\newtheorem*{theorem*}{Theorem}
\newtheorem{theorem}{Theorem}[section]
\newtheorem{proposition}[theorem]{Proposition}
\newtheorem{lemma}[theorem]{Lemma}
\newtheorem{corollary}[theorem]{Corollary}
\newtheorem{notation}[theorem]{Notation}
\newtheorem{definition}[theorem]{Definition}

\theoremstyle{remark}
\newtheorem{remark}[theorem]{Remark}

\newcommand{\red}[1]{\textcolor{red}{#1}}

\newcommand{\R}{\mathbb{R}}
\newcommand{\N}{\mathbb{N}}
\newcommand{\Z}{\mathbb{Z}}

\newcommand{\C}{\mathbb{C}}

\newcommand{\D}{\mathbb{D}}
\renewcommand{\H}{\mathbb{H}}

\newcommand{\cA}{\mathcal{A}}

\newcommand{\cB}{\mathcal{B}}

\newcommand{\Cc}{\mathcal{C}}

\newcommand{\Fc}{\mathcal{F}}
\newcommand{\Gc}{\mathcal{G}}
\newcommand{\Hc}{\mathcal{H}}

\newcommand{\Lc}{\mathcal{L}}
\newcommand{\cL}{\mathcal{L}}
\newcommand{\Mc}{\mathcal{M}}
\newcommand{\Nc}{\mathcal{N}}

\newcommand{\Tc}{\mathcal{T}}

\newcommand{\Fs}{\mathsf{F}}
\newcommand{\Bs}{\mathsf{B}}
\newcommand{\Hs}{\mathsf{H}}

\newcommand{\Expect}[1]{\mathbb{E} \left[ #1 \right] }
\newcommand{\EXPECT}[2]{\mathbb{E}_{#1} \left[ #2 \right] }
\newcommand{\Prob}[1]{\mathbb{P} \left( #1 \right) }
\newcommand{\PROB}[2]{\mathbb{P}_{#1} \left( #2 \right) }

\renewcommand{\P}{\mathbb{P}}
\newcommand{\E}{\mathbb{E}}

\newcommand{\abs}[1]{\left\vert #1 \right\vert}
\newcommand{\norme}[1]{\left\| #1 \right\| }
\newcommand{\scalar}[1]{\left\langle #1 \right\rangle }
\newcommand{\floor}[1]{\left\lfloor #1 \right\rfloor}
\newcommand{\ceil}[1]{\left\lceil #1 \right\rceil}
\newcommand{\indic}[1]{ \mathbf{1}_{ \left\{ #1 \right\} } }
\newcommand{\eps}{\varepsilon}

\DeclareMathOperator{\CR}{CR}
\DeclareMathOperator{\diam}{diam}

\renewcommand{\d}{\mathrm{d}}
\def \loopmeasure{\mu^{\rm loop}}

\renewcommand{\Im}{\mathrm{Im}}

\DeclareMathOperator{\PD}{PD}

\newcommand{\aj}[1]{{\noindent {\textcolor{blue}{AJ: #1}}}}

\DeclareMathOperator{\thk}{\rho}
\DeclareMathOperator{\mi}{\mathrm{mi}}
\DeclareMathOperator{\Mi}{\mathrm{Mi}}
\DeclareMathOperator{\g}{\frac{1}{2\pi}}

\usepackage{accents}
\newlength{\dhatheight}
\newcommand{\dhat}[1]{%
    \settoheight{\dhatheight}{\ensuremath{\hat{#1}}}%
    \addtolength{\dhatheight}{-0.35ex}%
    \hat{\vphantom{\rule{1pt}{\dhatheight}}%
    \smash{\hat{#1}}}}

\numberwithin{equation}{section}

\title{Multiplicative chaos of the Brownian loop soup}
\author{\'{E}lie Aïd\'{e}kon \and Nathanaël Berestycki \and Antoine Jego \and Titus Lupu}
\date {\today}

\begin{document}

\maketitle

\abstract{We construct a measure on the thick points of a Brownian loop soup in a bounded domain $D$ of the plane with given intensity $\theta>0$, which is formally obtained by exponentiating the square root of its occupation field. The measure is constructed via a regularisation procedure, in which loops are killed at a fix rate, allowing us to make use of the Brownian multiplicative chaos measures previously considered in \cite{bass1994, AidekonHuShi2018, jegoBMC}, or via a discrete loop soup approximation. At the critical intensity $\theta = 1/2$, it is shown that this measure coincides with the hyperbolic cosine of the Gaussian free field, which is closely related to Liouville measure. This allows us to draw several conclusions which elucidate connections between Brownian multiplicative chaos, Gaussian free field and Liouville measure. For instance, it is shown that Liouville-typical points are of infinite loop multiplicity, with the relative contribution of each loop to the overall thickness of the point being described by the Poisson--Dirichlet distribution with parameter $\theta = 1/2$. Conversely, the Brownian chaos associated to each loop describes its microscopic contribution to Liouville measure. Along the way, our proof reveals a surprising exact integrability of the multiplicative chaos associated to a killed Brownian loop soup. We also obtain some estimates on the discrete and continuous loop soups which may be of independent interest.}

\tableofcontents

\section{Introduction and main results}

The two-dimensional Gaussian free field (GFF) and its associated Gaussian multiplicative chaos (sometimes called Liouville measure) have been in recent years at the heart of some extraordinary developments, in particular in connection with the study of Liouville quantum gravity. Formally, the multiplicative chaos associated to a field $h$ in a domain $D \subset \R^2$ is a measure of the form
\begin{equation}\label{E:liouville}
\mu_\gamma(dz) = \lim_{\eps \to 0} \eps^{\gamma^2/2}{e^{\gamma h_\eps(z)}} dz
\end{equation}
where $\gamma \in \R$ is a parameter, $h$ is typically a logarithmically correlated field, and $h_\eps$ denotes some regularisation of $h$ at scale $\eps$.
The convergence of this procedure (as the regularisation scale $\eps$ converges to 0) is by no means obvious; in the case where $h$ is in addition assumed to be Gaussian, this is precisely the purpose of Gaussian multiplicative chaos theory, initially introduced by Kahane \cite{kahane} in the 1980s to model turbulence (following ideas of Kolmogorov and Mandelbrot) and further considerably developed in the last decade \cite{RobertVargas2010, DuplantierSheffield, RhodesVargasGMC, ShamovGMC, BerestyckiGMC}. Gaussian multiplicative chaos is a powerful tool to study properties of the underlying field $h$, particulary in connection with its extreme values.
By now, Gaussian multiplicative chaos is a fundamental object in its own right which describes scaling limits arising naturally in many different contexts, including random matrices \cite{FyodorovKeating2014, Webb2015, NikulaSaksmanWebb2018, Lambert2018, BerestyckiWebWong2018}, the Riemann zeta function \cite{SaksmanWebb2016}, and stochastic volatility models in finance \cite{BacryDelourMuzy} (see also \cite{MR2881881}); see the surveys \cite{MR3274356}, \cite{powell_review} and the book in preparation \cite{BerestyckiPowell} for more context and references.

\medskip More recently, it has been shown that an analogous theory can be developed in the case where $h$ describes (at least formally) the square root of the local time (i.e., occupation field) of a Brownian trajectory; see \cite{bass1994, AidekonHuShi2018, jegoBMC, jegoCritical, jegoRW}. The construction of the associated multiplicative chaos, a measure which we will denote in the following by $\Mc^{\wp}$ and which is now termed \textbf{Brownian multiplicative chaos} (following the terminology of \cite{jegoBMC}), is one of the first examples (together with \cite{Junnila18} which studies random Fourier series with i.i.d. coefficients) of a multiplicative chaos in which the field $h$ is not Gaussian or approximately Gaussian. It is, however, logarithmically correlated as will be clear from the discussion below. More generally, as shown in \cite{jegoRW}, given a finite number of independent Brownian trajectories $\wp_1, \ldots, \wp_n$, it is possible to define a multiplicative chaos associated to the square root of the \emph{combined} occupation field of $\wp_1, \ldots, \wp_n$; the corresponding measure (let us denote it by $\Mc^{\wp_1, \dots, \wp_n}$ in this introduction) can be thought of as a uniform measure on points that are thick for the combined local times of all paths.
A nontrivial fact proved in \cite{jegoRW} is that, sampling from this measure yields a point of multiplicity $k$ (i.e., is visited by exactly $k$ paths) with positive probability for each $1 \le k \le n$. More precisely, one can make sense of a measure $\Mc^{\wp_1 \cap \ldots \cap \wp_n}$ which is the restriction of $\Mc^{\wp_1, \dots, \wp_n}$ to points on the intersection of \emph{all} trajectories; those two types of measures are related by the a.s. identity
\begin{equation}
\label{eq:intro_brownian_measure}
\Mc^{\wp_1, \dots, \wp_n} = \sum_{k=1}^n \sum_{ \{ \tilde{\wp}_1, \dots, \tilde{\wp}_k \}  } \Mc^{ \tilde{\wp}_1 \cap \dots \cap \tilde{\wp}_k}
\end{equation}
where the second sum runs over all the possible choices of collections $\{ \tilde{\wp}_1, \dots, \tilde{\wp}_k \} \subset \{ \wp_1, \dots, \wp_n \}$ of pairwise distinct trajectories. This identity corresponds to choosing the trajectories which actually contribute to the overall thickness at a given point $z$. (However we caution the reader that the identity above is not entirely trivial because the measures $\Mc^{ \tilde{\wp}_1 \cap \dots \cap \tilde{\wp}_k}$ do not require the remaining paths in $\{ \wp_1, \dots, \wp_n \} \setminus \{ \tilde{\wp}_1, \dots, \tilde{\wp}_k \} $ to avoid this point).

\medskip Another, very different approach to the Gaussian free field is provided by the \textbf{Brownian loop soup}, first introduced by Lawler and Werner \cite{Lawler04}. This consists in a Poisson point process $\Lc_D^\theta$ of Brownian loops remaining in a domain $D$, where the intensity measure is of the form $\theta \mu_D^{\text{loop}}$. Here $\mu_D^{\text{loop}}$ is a certain infinite measure on unrooted loops (see \eqref{Eq loop meas} for a definition), and the intensity $\theta>0$ describes roughly speaking the local density of loops. The Brownian loop soup is a fundamental object closely connected to other conformally invariant random processes such as SLE, the conformal loop ensemble CLE, and the Gaussian free field. In particular, the Gaussian free field and the Brownian loop soup with critical intensity parameter $\theta = 1/2$ can be coupled in such a way that they are related via Le Jan's isomorphism (\cite{LeJan2010LoopsRenorm,LeJan2011Loops}), i.e., the occupation field of the loop soup (suitably recentered) is given by half of the square of the Gaussian free field (also suitably recentered). See Section \ref{sec:background} for more references on Brownian loop soup and in particular Theorem \ref{Thm Iso Le Jan} for Le Jan's isomorphism.

\medskip The main purpose of this paper is to show how these two \emph{a priori} orthogonal points of view on the Gaussian free field are in fact deeply interwoven.
To do so we first extend the construction of \cite{AidekonHuShi2018, jegoBMC} to a finite number of loops, or in fact even to an infinite number of loops but with finite ``density'', such as the loops of a Brownian loop soup of fixed intensity $\theta>0$ that are killed, if each loop is killed independently at constant rate $K>0$. This yields a measure $\Mc_a^K$ which, informally speaking, can be thought of as the uniform measure on the thick points of the occupation field of this ``killed'' loop soup. Viewing this killing as an ultraviolet regularisation of the loop soup which converges to the entire loop soup as $K \to \infty$, we show that, after suitable normalisation, the measures $\Mc_a^K$ converge to a limit $\Mc_a$ which may be thought of as the \textbf{multiplicative chaos associated to the loop soup} of intensity $\theta>0$ and is the main object of interest in this article.

\medskip We then specify this construction to the critical intensity $\theta = 1/2$, and show that this measure coincides with the hyperbolic cosine of the GFF, which is closely related to Liouville measure (essentially, it is an unsigned version of it). This identification may be considered the second main contribution of this paper. Together, these two results allow us to elucidate multiple connections between Gaussian free field, Brownian loop soup and Liouville measure. For instance, we are able to describe precisely the structure of Brownian loops in the vicinity of a Liouville typical point. Conversely, this result allows us to view the Brownian multiplicative chaos of \cite{bass1994, AidekonHuShi2018, jegoBMC} as describing the microscopic contribution of each loop to Liouville measure (or, more precisely, the hyperbolic cosine of the GFF).

\subsection{Construction of Brownian loop soup multiplicative chaos }

Let $a \in (0,2)$ and $\theta>0$ be respectively a thickness parameter and an intensity parameter. Let $D \subset \C$ be an
open bounded simply connected domain and let $\Lc_D^\theta$ be a Brownian loop soup in $D$ with intensity $\theta \loopmeasure_D$.
As mentioned above, the first aim of this article is to build the ``uniform measure'' $\Mc_a$ on $a$-thick points of $\Lc_D^\theta$. We need to start by recalling that  for any Brownian-like trajectory $\wp$, there exists a random Borel measure $\Mc_a^\wp$ supported on $a$-thick points of $\wp$ \cite{bass1994, AidekonHuShi2018, jegoBMC}. This measure is now known as Brownian multiplicative chaos and can be constructed, for instance, by exponentiating the square root of the local times of $\wp$. Recall also (see Section \ref{sec:preliminaries_BMC} for precise definitions) that for any finite number of independent Brownian-like trajectories $\wp_1, \dots, \wp_n$, there exists a measure $\Mc_a^{\wp_1 \cap \dots \cap \wp_n}$ supported on $a$-thick points that have been generated by the interaction of the $n$ trajectories \cite{jegoRW}.

To build the ``uniform measure'' on $a$-thick points of the loop soup, we start by thinning the set of loops that we consider by killing each loop independently of each other at some rate $K>0$, i.e. each given loop $\wp \in \Lc_D^\theta$ is killed with probability $1-e^{-KT(\wp)}$ where $T(\wp)$ denotes the duration of the loop $\wp$.
We denote by $\Lc_D^\theta(K)$ the set of loops that have been killed ({note that this differs from the perhaps more standard massive loop soup, which is actually the complementary}). Obviously, $\Lc_D^\theta(K) \to \Lc_D^\theta$ as $K \to \infty$ in the sense that $\Lc_D^\theta(K)$ is an increasing collection in $K>0$ and $\bigcup_{K>0} \Lc_D^\theta(K)  = \Lc_D^\theta  $. Consider
\begin{equation}
\label{eq:def_measure_killed_loops}
\Mc_a^K :=
\sum_{n \geq 1} \frac{1}{n!} \sum_{\substack{\wp_1, \dots, \wp_n \in \Lc_D^\theta(K)\\ \forall i \neq j, \wp_i \neq \wp_j}} \Mc_a^{\wp_1 \cap \dots \cap \wp_n}
\end{equation}
the measure on $a$-thick points that have been entirely created by loops in $\Lc_D^\theta(K)$. This definition is justified by \eqref{eq:intro_brownian_measure}. Note that the factor $\tfrac1{n!}$ ensures that we count each subset $\{ \wp_1, \dots, \wp_n \}$ of loops only once.
In reality, the construction of $\Mc_a^{\wp_1 \cap \dots \cap \wp_n}$ provided by \cite{jegoRW} does not apply directly to Brownian loops but rather to Brownian excursions from interior to boundary points; we explain in Section \ref{sec:def} (see in particular Definition \ref{def:measure_mass}) how to extend this definition to Brownian loops.
Furthermore, it is not \emph{a priori} obvious that the left hand side of \eqref{eq:def_measure_killed_loops} is a finite measure; roughly speaking this comes from the fact that the collection of loops $\Lc_D^\theta(K)$ has ``finite density'' for each $K< \infty$ (the number of loops in $\Lc_D^\theta(K)$ of diameter roughly $2^{-j}$ in the whole domain $D$ does not depend on $j$, which translates into a finite expected total occupation time for $\Lc_D^\theta(K)$; it is therefore not surprising that the corresponding thick point measure $\Mc_a^K$ is finite, see e.g. \eqref{eq:expectation_thick_points_killed_loops} for a computation of the expectation which implies a.s. finiteness).

The first result is the construction of the measure $\Mc_a$, the multiplicative chaos defined by the Brownian loop soup, and which is the main object of this paper.

\begin{theorem}\label{th:convergence_continuum}
Let $\theta>0$ and $a \in (0,2)$. Then as $K \to \infty$, the convergence
$$
 (\log K)^{-\theta} \Mc_a^K \to \Mc_a
  $$
takes place in probability for the topology of weak convergence, where the right hand side is defined by this convergence. Moreover, the limit $\Mc_a$ satisfies the following properties.
\begin{enumerate}
\item
\label{it:nondegenerate}
$\Mc_a$ is non-degenerate: for all open set $A \subset D$, $\Mc_a(A) \in (0,\infty)$ a.s. Furthermore, denoting by $\CR(z,D)$ the conformal radius of $D$ seen from a point $z \in D$, we have
%\item
%\label{it:first_moment}
%The first moment of $\Mc_a$ is given by
\begin{equation}
\Expect{\Mc_a(\d z)} = \frac{1}{2^{\theta} a^{1-\theta} \Gamma(\theta)} \CR(z,D)^a \d z.
\end{equation}
\item
\label{it:measurability}
Measurability:
$\Mc_a$ is independent of the labels underlying the definition of the killed loops and is therefore measurable with respect to the loop soup. More precisely, $\Mc_a$ is $\sigma( < \Lc_D^\theta > )$-measurable (see \eqref{eq:sigma_algebra_admissible}).
\item
\label{it:conformal_covariance}
Conformal covariance: if $\psi : D \to \widetilde{D}$ is a conformal map between two bounded simply connected domains, then
\[
\left( \Mc_{a,D} \circ \psi^{-1} \right)(d \tilde{z})
\overset{(d)}{=} \abs{ (\psi^{-1})'(\tilde{z}) }^{2+a}
\Mc_{a,\widetilde{D}}(d\tilde{z}).
\]
\item
\label{it:dim}
The carrying dimension of $\Mc_a$\footnote{Recall that the carrying dimension of a measure $\mu$ is given by the infimum of $d>0$ such that there exists a Borel set $A$ with Hausdorff dimension $d$ and such that $\mu(A)>0$.}
is almost surely equal to $2-a$.
\end{enumerate}
\end{theorem}

\begin{comment}
\[
\dim(\mu) = \inf \left\{ d>0: \exists \mathrm{~Borel~set~} A \mathrm{~such~that~} \mu(A)>0 \mathrm{~and~} \dim_\Hc(A) =d \right\}
\]
where $\dim_\Hc(A)$ denotes the Hausdorff dimension of $A$.
\end{comment}

\begin{remark}
  We will define in \eqref{eq:def_measure_discrete} below another, simpler approximation to $\Mc_a$ (essentially just a uniform measure on the thick points of a discrete loop soup, instead of $\Mc_a^K$). The corresponding convergence result is stated in Theorem \ref{th:convergence_discrete}.
\end{remark}

\begin{remark}
We also show that for all Borel sets $A, B \subset \C$,
$\lim_{K \to \infty} (\log K)^{-2 \theta} \Expect{ \Mc_a^K(A) \Mc_a^K(B) }$ is given by
\begin{equation}
\label{eq:rmk_second_moment}
\frac{1}{4^{\theta} a^{1-\theta} \Gamma(\theta)} \int_A \d z \int_B \d z' \CR(z,D)^a \CR(z',D)^a (2\pi G_D(z,z'))^{1-\theta} I_{\theta -1} \left( 4\pi a G_D(z,z') \right),
\end{equation}
where $I_{\theta-1}$ is a modified Bessel function of the first kind whose definition is recalled in \eqref{eq:modified_Bessel_function} and $G_D$ is the Green function in $D$ \eqref{eq:Green}. See Corollary \ref{cor:second_moment_simplified}.
In particular, for all open set $A \subset D$,  $\lim_{K \to \infty} (\log K)^{-2 \theta} \Expect{ \Mc_a^K(A)^2 } < \infty$ if, and only if, $a < 1$. It should be possible to show that one can exchange the expectation and the limit (in the $L^2$-phase $\{ a \in (0,1) \}$, this exchange is straightforward), and this would show that $\Expect{ \Mc_a(A) \Mc_a(B) }$ is given by \eqref{eq:rmk_second_moment}.  Because of the length of the paper, we preferred to not include a proof of this statement.
\end{remark}

\begin{remark}
In Theorem \ref{th:conformal_covariance_couple}, we give a stronger form of conformal covariance which concerns not only the measure $\Mc_a$ but the couple $(\Lc_D^\theta, \Mc_a)$.
\end{remark}

\subsection{Multiplicative chaos and hyperbolic cosine of the Gaussian free field}

We now turn to the connections between the multiplicative chaos measure $\Mc_a$ associated to the Brownian loop soup and Liouville measure. This will require choosing the intensity of the loop soup to be the critical value $\theta=1/2$. This value is already known to be special for two distinct (but related) reasons. On the one hand, this is the value such that the (renormalised) occupation field of the loop soup corresponds to the (Wick) square of the Gaussian free field (i.e., Le Jan's isomorphism holds, see Theorem \ref{Thm Iso Le Jan} in the discrete and Remark \ref{R:iso_cont} for the continuum case of interest here). On the other hand, this is also the critical value for the percolation of connected components of the loop soup clusters, as follows from the celebrated work of Sheffield and Werner \cite{SheffieldWernerCLE}. We show here that in addition, still at $\theta = 1/2$, the associated multiplicative chaos corresponds to the hyperbolic cosine of the Gaussian free field. Formally, this is the measure of the form
\begin{equation}\label{E:cosh}
2\cosh (\gamma h) dz =( e^{\gamma h} + e^{- \gamma h}) dz,
\end{equation}
where $h = \sqrt{2\pi} \varphi$ is a Gaussian free field, and where $a$ and $\gamma$ are related by the correspondence:
$$
\gamma = \sqrt{2a} ; \quad a = \frac{\gamma^2}{2}.
$$
In other words, the hyperbolic cosine of $h$ is defined in \eqref{E:cosh} as the sum (up to an appropriate multiplicative factor specified below) of the Liouville measures \eqref{E:liouville} with parameters $\gamma$ and $- \gamma$ respectively (as constructed e.g. in \cite{DuplantierSheffield}, \cite{BerestyckiGMC}). Note that formally, our multiplicative chaos measure $\Mc_a$ is the exponential of the square root of the (renormalised) occupation field
$:\!\ell(\Lc^{1/2}_D)\!:$ of the loop soup $\Lc^{\theta = 1/2}_D$, so it is natural to expect in view of Le Jan's isomorphism, that $\Mc_a (dz) = e^{\gamma |h|} dz$, which on first inspection does not immediately coincide with the hyperbolic cosine of $h$. However, since $h$ is not a continuous function, only points where $h$ is either \emph{very negative} or \emph{very positive} contribute to $e^{\gamma |h|} dz$, and it follows that for such points we may indeed write $e^{\gamma |h|}  = e^{\gamma h} + e^{- \gamma h}$. The theorem below makes this connection precise.

Let $h = \sqrt{2\pi} \varphi$ (where as before $\varphi$ is the Gaussian free field in $D$ with zero-boundary conditions whose covariance function is given by
$\Expect{\varphi(z)\varphi(w)} = G_D(z,w)$).
Thus with these notations, $\E( h(z) h(w)) = {2\pi} G_D(z,w)
\sim -\log |w-z|$ as $w-z \to 0$, which is consistent with the choice of normalisation in Liouville quantum gravity literature (see e.g. \cite{WernerPowell} and \cite{BerestyckiPowell} for an introduction to the Gaussian free field and to Liouville quantum gravity).

\begin{theorem}\label{th:identification1/2}
Let $\theta = 1/2$, $a \in (0,2)$ and $\gamma = \sqrt{2a}$. Then $\Mc_a$ has the same law as
\[
\frac{1}{\sqrt{2\pi a}} \cosh \left( \gamma h \right)
= \frac{1}{2\sqrt{2\pi a}} \left( e^{ \gamma  h} + e^{ -\gamma h} \right),
\]
where $e^{\pm \gamma h}$ is the Liouville measure with parameter $\pm \gamma$ associated with $h$. More precisely, there is a coupling $(\varphi, \Lc_D^{1/2}, \Mc_a)$ between a Gaussian free field $\varphi$, a Brownian loop soup with critical intensity $\theta = 1/2$, and a measure $\Mc_a$ in which the three components are pairwise related as follows:
\begin{itemize}

  \item $\Mc_a $ is the multiplicative chaos measure associated to $\Lc_D^{1/2}$ as in Theorem \ref{th:convergence_continuum};

  \item $\Mc_a$ is the hyperbolic cosine of $h = \sqrt{2\pi} \varphi$, i.e., $\Mc_a = \tfrac{1}{\sqrt{2\pi a}} \cosh \left( \gamma h \right)$;

  \item $\varphi$ and $\Lc_D^{1/2}$ satisfy Le Jan's isomorphism, in which the (renormalised) occupation field $:\!\ell(\Lc_D^{1/2})\!:$ of the loop soup $\Lc^{1/2}_D$ is equal to the (Wick) square of the Gaussian free field $\varphi$. That is,
      $  \tfrac12\! :\!\varphi^2 \!: \ =\ : \!\ell(\Lc_D^{1/2}) \!:$ (see Remark \ref{R:iso_cont}).
\end{itemize}
\end{theorem}

Theorem \ref{th:identification1/2} gives a new perspective on Liouville measures by embedding them, or more precisely the hyperbolic cosine of the GFF, in a two-dimensional family of measures indexed by $\theta >0$ and $\gamma \in (0,2)$.

\begin{remark}
  One informal consequence of Theorem \ref{th:identification1/2} is that it allows us to describe the contribution of each loop to Liouville measure (or more precisely to the hyperbolic cosine of the GFF): namely, each loop contributes a macroscopic amount (as we will see in Theorem \ref{th:PD}), given by its Brownian multiplicative chaos, as defined in \cite{jegoBMC} and \cite{AidekonHuShi2018} (see Section \ref{sec:def} for the extension to Brownian loops).
\end{remark}

\begin{remark}
  We caution the reader that the relation between the GFF $\varphi$ and the loop soup $\Lc_D^{\theta = 1/2}$ as stated here (namely, Le Jan's isomorphism) is not sufficient to determine uniquely the joint law of $(\varphi, \Lc_D^{\theta = 1/2})$.
\end{remark}

\begin{comment}
\aj{Continue to write the remark below}

\begin{remark}
The theorem above also allows us to identify in a more concrete way the measures $\Mc_a$ for other values of the intensity $\theta$. Indeed, let $\theta_1, \theta_2 >0$ and $\theta = \theta_1 + \theta_2$. Assume that the loop soups $\Lc_D^{\theta_1}$ and $\Lc_D^{\theta_2}$ are independent and that $\Lc_D^\theta$ comes from the superposition of $\Lc_D^{\theta_1}$ and $\Lc_D^{\theta_2}$.  We claim that we can decompose
\[
\Mc_{a,\theta} = \int_0^a \d \alpha ~\Mc_{\alpha, \theta_1} \cap \Mc_{a - \alpha, \theta_2} \quad \quad \mathrm{a.s.}
\]
Intuitively, this relation expresses the fact that an $a$-thick point of $\Lc_D^\theta = \Lc_D^{\theta_1} \cup \Lc_D^{\theta_2}$ comes from the interaction of the loops in $\Lc_D^{\theta_1}$ and the ones $\Lc_D^{\theta_2}$. The point is a.s. thick for both collections of loops and the joint thickness is distributed according to $(\alpha, a-\alpha)$ where $\alpha$ is uniform over the interval $(0,a)$. This is analogous to ***

More precisely, $\Mc_{\alpha, \theta_1} \cap \Mc_{a - \alpha, \theta_2}$ is a natural random measure supported by the intersection of the support of the two measures. Such a measure has been rigorously defined in the context of finitely many Brownian trajectories and one would need to show that

In particular, this identifies the law of $\Mc_{a,\theta}$ for all $\theta \in (1/2) \N$.
\end{remark}
\end{comment}

\subsection{Brownian loops at a typical thick point}

Theorem \ref{th:identification1/2} raises a number of questions concerning the relations between Brownian loop soup and multiplicative chaos (i.e., hyperbolic cosine of the Gaussian free field or ultimately Liouville measure). Chief among those are questions of the following nature: sample a point $z$ according to the multiplicative chaos measure $\Mc_a$. What does the loop soup look like in the neighbourhood of such points? In other words (for the value $\theta=1/2$) what does the Brownian loop soup look like in the vicinity of a Liouville-typical point? Obviously we know that the point $z$ is almost surely $\gamma$-thick from the point of view of Liouville measure (see e.g. Theorem 2.4 in \cite{BerestyckiPowell}) so we expect the point $z$ to also have an atypically high local time, and so to be also ``thick'' for the loop soup (this will be formulated precisely below in Theorem \ref{th:thick_points_continuum}). How do loops combine to create such a thick local time? Does the thickness come from a single loop which visits $z$ very often, or from an infinite number of loops that touch $z$, with each loop having a typical occupation field (so $z$ is not ``thick'' with respect to any single loop)? As we see, the answer turns out to be an intermediate scenario. More precisely, we show below that Liouville-typical points are of infinite loop multiplicity, with the relative contribution of each loop to the overall thickness of the point being described
by the Poisson–-Dirichlet distribution with parameter $\theta = 1/2$ (see e.g. \cite{ABT} for a definition and some properties of Poisson--Dirichlet distributions).

In fact, the theorem below will hold without restriction over $\theta>0$, and the parameter of the corresponding Poisson--Dirichlet distribution will precisely be the intensity $\theta$ of the loop soup. The behaviour above is encapsulated by the following theorem, which gives a precise description of the so-called ``rooted measure''.
To formulate the result, we will need to decompose the loops touching a point $z$ into excursions (analogous to It\^o excursions in one dimension).
Let us say that a function of $\Lc_D^\theta$ is \textbf{admissible} if it is invariant under reordering these excursions (see Definition \ref{def:admissible} for a more precise definition; see also Section \ref{sec:preliminaries_BLS} for details concerning the topology on the set of collections of loops).

Let $\{a_1, a_2, \dots\}$ be a random partition of $[0,a]$ distributed according to a Poisson--Dirichlet distribution with parameter $\theta$. Conditionally on this partition, let $\Xi_{a_i}^z$, $i \geq 1$, be independent loops with the following distribution: for all $i \geq 1$, $\Xi_{a_i}^z$ is the concatenation of the loops in a Poisson point process with intensity $2\pi a_i \mu_D^{z,z}$. Here, $\mu_D^{z,z}$ is an infinite measure on loops that go through $z$ (see \eqref{Eq mu D z w}).

\begin{theorem}\label{th:PD}
Let $\theta >0$ and $a \in (0,2)$. For any nonnegative measurable admissible function $F$,
\begin{equation}
\label{eq:th_PD}
\Expect{ \int_D F(z, \Lc_D^\theta) \Mc_a(dz) }
= \frac{1}{2^{\theta} a^{1-\theta} \Gamma(\theta)} \int_D \Expect{ F(z, \Lc_D^\theta \cup \{ \Xi_{a_i}^z, i \geq 1 \}  ) } \CR(z,D)^a \d z
\end{equation}
where the two collections of loops $\Lc_D^\theta$ and $\{ \Xi_{a_i}^z, i \geq 1 \}$ appearing on the right hand side term are independent.

Moreover, the joint law of the couple $(\Lc_D^\theta, \Mc_a)$ is characterised by the following:
\begin{itemize}
\item
$\Lc_D^\theta$ has the law of a Brownian loop soup in $D$ with intensity $\theta$;
\item
$\Mc_a$ is measurable w.r.t. the equivalence class $< \Lc_D^\theta >$ (see \eqref{eq:sigma_algebra_admissible});
\item
\eqref{eq:th_PD} is satisfied for any nonnegative measurable admissible function $F$.
\end{itemize}
\end{theorem}

Note that, before the current work, it was not even a priori immediately clear that points of infinite loop multiplicity exist with probability one.

\begin{remark}
Recall that, by Girsanov's theorem, shifting the probability measure by the hyperbolic cosine of the GFF amounts to adding a logarithmic singularity with strength $\gamma$ to the GFF. More precisely, and using the notations of Theorem \ref{th:identification1/2}, one has for any bounded measurable function $F$,
\[
\Expect{ \int_D F(z, h) \cosh(\gamma h(z))\d z} = \int_D \CR(z,D)^{\gamma^2/2} \Expect{ F(z, h + 2 \pi \gamma \sigma G_D(z, \cdot) )} \d z,
\]
where $\sigma$ is a sign independent of $h$ taking values $+1$ or $-1$ with equal probability $1/2$. Theorem \ref{th:PD} above can be seen as explaining the way the Brownian loop soup creates this logarithmic singularity at $z$. Since here it is easy to check that $\cosh (\gamma h (z))dz$ is measurable with respect to $h$, the above identity in fact characterises the joint law of $(h, \cosh (\gamma h))$ (see \cite{ShamovGMC} or (3.30) in \cite{BerestyckiPowell}).
\end{remark}

The above result, in conjunction with Theorem \ref{th:identification1/2}, immediately implies (in the case $\theta = 1/2$) some notable consequences in connection with Le Jan's isomorphism. We state below a simple instance of such a statement. The isomorphism below is closely related to (and in fact could also be deduced from) the isomorphism in \cite[Proposition 3.9]{ALS2} where the occupation field of a Poisson point process of boundary-to-boundary excursions is added.

\begin{corollary}\label{C:iso}
Let $z \in D$ and let $\Xi_a^z$ be a loop as in Theorem \ref{th:PD} independent of the Brownian loop soup $\Lc_D^{1/2}$ with critical intensity $\theta=1/2$.
Let $:\!\ell(\Lc_D^{1/2})\!:$ denote the (renormalised) occupation field of
$\Lc_D^{1/2}$, and let $\ell (\Xi_a^z)$ denote the occupation field of $\Xi_a^z$ (which is well defined as a Radon measure on $D $, without any centering).  Then
\[
\!:\!\ell(\Lc_D^{1/2})\!:\! + \ell( \Xi_a^z) \quad \overset{(d)}{=}  \quad \frac{1}{2} \!:\!\varphi^2\!:\! + \gamma \sqrt{2\pi} G_D(z,\cdot) \varphi + \frac{\gamma^2}{2} 2\pi G_D(z,\cdot)^2
\]
where, as before, $\gamma = \sqrt{2a}$. In particular, the expectation of $\ell( \Xi_a^z)$ is given by $a 2\pi G_D(z,\cdot)^2$.
\end{corollary}

\subsection{Dimension of the set of thick points}

The study of the multifractal behaviour of thick points of logarithmically correlated fields has attracted a lot of attention in the past two decades. In particular, the Hausdorff dimension of the set of thick points was established both in the case of planar Brownian motion \cite{DPRZ2001} and in the case of the 2D Gaussian free field \cite{HuMillerPeres2010}. Related results were also obtained in the discrete; see \cite{DPRZ2001, rosen2005, BassRosen2007, jego2020} for the random walk and \cite{daviaud2006} for the discrete GFF. Many more articles studied related questions concerning other log-correlated fields; see \cite{Shi15,arguin17} for more references.

We now define precisely a notion of thick points for the loop soup described informally earlier, and state some results concerning these points. We show
that with this definition, $\Mc_a$ is almost surely supported on ``$a$-thick points'' of the loop soup.
We also compute its Hausdorff dimension (a statement which does not involve the multiplicative chaos).
Our definition of thick points is in terms of crossings of annuli. For $z \in D$, $r>0$ and $\wp \in \Lc_D^\theta$ a loop, we denote by $N_{z,r}^\wp$ the number of upcrossings from $\partial D(z,r)$ to $\partial D(z,er)$ in $\wp$ (since $\wp$ is a loop, this is also equal to the number of downcrossings). Denote also $N_{z,r}^{\Lc_D^\theta} := \sum_{\wp \in \Lc_D^\theta} N_{z,r}^\wp$.

\begin{theorem}\label{th:thick_points_continuum}
Let $\theta >0$ and $a \in (0,2)$.
$\Mc_a$ is almost surely supported by the set
\begin{equation}
\label{eq:def_thick_crossings}
\Tc(a) := \left\{ z \in D: \lim_{n \to \infty} \frac{1}{n^2} N_{z,e^{-n}}^{\Lc_D^\theta} = a \right\},
\end{equation}
that is, $\Mc_a(D \setminus \Tc(a) ) = 0$ a.s. Moreover, the Hausdorff dimension of $\Tc(a)$ equals $2-a$ a.s.
\end{theorem}

We mention that it would have been possible to quantify the thickness of a point $z$ via the normalised occupation measure of small discs, or circles,  centred at $z$. This would have been closer to the notion of thick points in \cite{DPRZ2001} and \cite{jegoBMC}. To keep the paper of a reasonable size, we do not attempt to prove a result for these notions of thick points.

In the next section, we establish the scaling limit of the set of thick points of random walk loop soup. In particular, we will obtain in Corollary \ref{cor:thick_points_discrete} the convergence of the number of discrete thick points when appropriately normalised; as we will see this identifies a nontrivial subpolynomial term which goes beyond the calculation of the exponent $2-a$ corresponding to the above dimension; interestingly this subpolynomial term depends on the intensity $\theta$ itself.

\subsection{Random walk loop soup approximation}

As mentioned before, Theorem \ref{th:identification1/2} is natural from the point of Le Jan's isomorphism in the continuum. However this relation is
far too weak to obtain a proof of this theorem. Instead, we rely on a discrete approach where the relations hold pointwise, and with no renormalisations, so that this type of difficulties does not arise. This approach also provides a very natural approximation of the multiplicative chaos measure $\Mc_a$ from a discrete random walk loop soup (\cite{LawlerFerreras07RWLoopSoup}): namely, $\Mc_a$ is the limit of the uniform measure on thick points of the discrete loop soup.
Let us now detail this result.

Without loss of generality, assume that the domain $D$ contains the origin. For all $N \geq 1$, we consider a discrete approximation $D_N \subset D \cap \frac{1}{N} \Z^2$ of $D$ by a portion of the square lattice with mesh size $1/N$. Specifically,
\begin{align}
\label{eq:DN}
D_N := \Big\{ z \in D \cap \frac{1}{N} \Z^2 :
\begin{array}{c}
\mathrm{there~exists~a~path~in~} \tfrac{1}{N} \Z^2 \mathrm{~from~} z \mathrm{~to~the~origin} \\
\mathrm{whose~distance~to~the~boundary~of~} D \mathrm{~is~at~least~} \tfrac{1}{N}
\end{array}
\Big\}.
\end{align}
Let $\Lc_{D_N}^\theta$ be a random walk loop soup with intensity
$\theta$.
See Section \ref{sec:preliminaries_RWLS} for a precise definition.
For any vertex $z \in D_N$ and any discrete path $(\wp(t))_{0 \leq t \leq T(\wp)}$
parametrised by continuous time, we denote by $\ell_z(\wp)$ the local time of $\wp$ at $z$, i.e.
\[
\ell_z(\wp) := N^2 \int_0^{T(\wp)} \indic{\wp(t) = z} \d t.
\]
With our normalisation,
\[
\E \bigg[
\sum_{\wp \in \Lc_{D_N}^\theta} \ell_z(\wp)
\bigg]
\sim \frac{{\theta}}{2\pi} \log N
\quad \text{as~} N \to \infty.
\]
We define the set of $a$-thick points by
\begin{equation}
\label{eq:def_thick_points_discrete}
\Tc_N(a) := \bigg\{ z \in D_N :
\sum_{\wp \in \Lc_{D_N}^\theta} \ell_z(\wp) \geq \g a (\log N)^2 \bigg\}.
\end{equation}
We encode this set in the following point measure: for all Borel set $A \subset \C$, define
\begin{equation}
\label{eq:def_measure_discrete}
\Mc_a^N(A) := \frac{(\log N)^{1-\theta}}{N^{2-a}} \sum_{z \in \Tc_N(a)} \indic{z \in A}.
\end{equation}
In the next result and in the rest of the paper, we will denote
\begin{equation}\label{eq:c0}
c_0 := 2 \sqrt{2} e^{\gamma_{\mathrm{EM}}},
\end{equation}
where $\gamma_{\mathrm{EM}}$ is the Euler--Mascheroni constant \eqref{eq:Euler-Mascheroni}. The constant $c_0$ arises from the asymptotic behaviour of the discrete Green function on the diagonal; see Lemma \ref{lem:Green_discrete}.

\begin{theorem}\label{th:convergence_discrete}
Let $\theta >0$ and $a \in (0,2)$.
The couple $(\Lc_{D_N}^\theta, \Mc_a^N)$ converges in distribution towards $(\Lc_D^\theta, 2^\theta c_0^a \Mc_a)$, relatively to the topology induced by $d_{\mathfrak{L}}$ \eqref{eq:d_frac_L} for $\Lc_{D_N}^\theta$ and the weak topology on $\C$ for $\Mc_a^N$.
\end{theorem}

In particular,

\begin{corollary}\label{cor:thick_points_discrete}
The convergence
\[
\frac{(\log N)^{1-\theta}}{N^{2-a}} \# \Tc_N(a) \to 2^\theta c_0^a \Mc_a(D)
\]
holds in distribution.
\end{corollary}

Theorem \ref{th:convergence_discrete} can be seen as an interpolation and extrapolation of the scaling limit results of \cite{jegoRW} and \cite{BiskupLouidor} concerning, respectively, thick points of finitely many random walk trajectories (informally, $\theta \to 0^+$) and thick points of the discrete GFF ($\theta =1/2$).

The proof of Theorem \ref{th:convergence_discrete} ends up taking a large part of this article (essentially, all of Part Two). At a high level, the difficulties stem from the fact that (unlike in the continuum) it is very difficult to compare directly two random walk loop soups with different lattice mesh sizes, thereby ruling out the possibility to apply an $L^1$ convergence argument as in Gaussian multiplicative chaos \cite{BerestyckiGMC}. Instead, we rely on results of \cite{jegoRW} in which analogous difficulties were resolved in the case of a finite number of random walk trajectories,
together with a new discrete description (see Proposition \ref{Prop 1st mom Mc N}) of the rooted discrete measure (i.e., a discrete loop soup version of the Girsanov transform) which must be proved by hand. These computations reveal a surprising amount of integrability, which we think is interesting in its own right. Another technical ingredient which we obtain along the way is a strengthening of a KMT-type coupling between the discrete loop soup and the continuum loop soup proved by Lawler and Trujillo-Ferreras \cite{LawlerFerreras07RWLoopSoup}. This coupling allows us to show that discrete and continuous loops of \emph{all} mesoscopic scales are close to one another (in contrast with \cite{LawlerFerreras07RWLoopSoup}, where the comparison holds for sufficiently large mesoscopic scales), provided we are only interested in loops that are localised close to a given point $z \in D$. This coupling is useful to obtain rough estimates on the discrete loop soup such as large deviations for the number of crossings of annuli of a given scale. See Lemma \ref{L:coupling} for details.

\subsection{Martingale and exact solvability}

Before starting the proofs it is useful to highlight a few nontrivial aspects of the proofs. A crucial idea is the identification of a certain measure-valued martingale $m^K_a(dz)$ with respect to the filtration $\Fc_K$ generated by $\Lc_D^\theta(K)$. The definition of this martingale is in itself highly nontrivial and is described in Proposition \ref{prop:martingale}. As follows \emph{a posteriori} from our analysis, this martingale corresponds to the conditional expectation of $\Mc_a$ given $\Fc_K$. Although it is \emph{a priori} far from clear that this conditional expectation should take the given form, it is nevertheless possible to guess a rough form for this conditional expectation. For the purpose of the following discussion, let us assume that the intensity $\theta$ is critical so that we may use Le Jan's isomorphisms. Consider the decomposition of the entire loop soup $\Lc_D^\theta$ into the killed part $(\Lc_D^\theta(K))$ and its complement. These two parts are independent. Furthermore, by the isomorphism theorem (see Theorem \ref{Thm Iso Le Jan}), the occupation field of the complement is given by one half of the square of a \textbf{massive Gaussian free field}. This suggests that $\Mc_a$ can be described by the sum of two terms. The first term comes just from the hyperbolic cosine of this massive free field (since it is possible that a point is thick without being visited at all by $\Lc_D^\theta(K)$). The second term on the other hand describes the possible interactions between these two parts: it measures the contribution of points whose thickness comes in part from the massive free field and in another part from the killed loop soup.
%(note that no contribution to $\Mc_a$ can come from points that are ``just'' thick for the killed loop soup -- this is roughly speaking because the massive free field exists everywhere).
This interaction term is thus described by an integral in which the integrand describes the respective thickness of each part; however the precise law of this mixture cannot be easily inferred from combinatorial arguments and was instead obtained by trial and error. We stress however that the appearance of the massive free field (and its hyperbolic cosine) is what makes the ultraviolet regularisation by killing particularly attractive from our point of view.

While these arguments are useful to guess the general rough form of the martingale, they cannot be used to give a proof of the martingale property: rather, the martingale property is the engine that drives the proof and the above explanation may only be seen as a justification after the facts. The proof of the martingale property relies instead on a central observation (stated in Proposition \ref{prop:first_moment_killed_loops} and proved in Section \ref{Sec 1st moment}), which allows us to compute \emph{exactly} the expectation of the approximate measure $\Mc_a^K$ with finite $K<\infty$. This expectation is computed in terms of the hypergeometric function ${}_1F_1$ and the conformal radius of a point. This computation is the result of the triple differentiation of a certain infinite series whose $n$th term involves an $n$-dimensional integral, see Lemma \ref{lem:Fs}. The fact that such a computation is at all possible is another stroke of luck which suggests that the choice of ultraviolet regularisation (by killing as opposed to, say, by diameter) is particularly well suited to this problem. The exact solvability which seems to underly this calculation is in fact a constant feature of the paper; as shown in Part Two, analogous remarkable identities hold even at the discrete level. The existence of such exact formulae for the ultraviolet regularisation of the Brownian loop soup by killing seems to not have been noticed before; we hope it may prove useful in other contexts as well.

%The proof of convergence in Theorem \ref{th:convergence_continuum} then follows from the martingale convergence theorem and the observation that when $K \to \infty$, the martingale $m^K_a$ and the suitably normalised chaos of the killed loop soup $\Mc_a^K$ are very close to one another in an $L^1$ sense.

\medskip

We end this introduction by pointing out that the results of this paper open the door to a generalisation, in particular to non-half integer values of $\theta$, of constructions from the Euclidean Quantum Field Theory that relate the Wick powers of the GFF, the Gaussian multiplicative chaos and the intersection and self-intersection local times of Brownian paths (see e.g., \cite{Symanzik65Scalar, Symanzik66Scalar, Symanzik1969QFT, Varadhan, Dynkin1984IsomorphismPresentation, Dynkin84, Simon, Wolpert1, Wolpert2, LeGall_SILT, LeJan2011Loops}). We plan to develop this in future works.

\paragraph*{Organisation of the paper}
In the next section, we will give some background on loop soups and measures on paths both in the continuum and in the discrete. We will also recall the definitions of Brownian multiplicative chaos measures.
The rest of the paper is then be divided into two main parts dealing with the continuum and the discrete settings respectively. Each of these parts starts with a preliminary section (Sections \ref{sec:high_level} and \ref{sec:discrete_high_level} respectively) outlining the proofs of the main theorems at a high level. The structure of each part is then described more thoroughly in these preliminary sections.

\paragraph*{Acknowledgements}

We thank Marcin Lis for raising with two of us (AJ and NB) questions which ultimately triggered our interest in this problem. Part of this work was carried out when AJ visited EA and TL at NYU-ECNU Institute of Mathematical Sciences at NYU Shanghai. The hospitality of that department is gratefully acknowledged.
We also would like to thank anonymous referees for their careful reading and comments that helped improve the paper.

EA and TL acknowledge the support of the French National Research Agency (ANR) within the project MALIN (ANR-16-CE93-0003).
AJ is recipient of a DOC Fellowship of the Austrian Academy of Sciences at the Faculty of Mathematics of the University of Vienna. AJ's research is partly supported by the EPSRC grant EP/L016516/1 for the University of Cambridge Centre for Doctoral Training, the Cambridge Centre for Analysis. NB’s work was supported by: University of
Vienna start-up grant, and FWF grant P33083 on “Scaling limits in random conformal geometry”. 

\section{Background}
\label{sec:background}

\subsection{Measures on Brownian paths and Brownian loop soup}
\label{sec:preliminaries_BLS}

We start first by recalling some basic properties of the Brownian loop soup, mostly to introduce our notations and choice of normalisations. 

By Brownian motion we will denote the 2D Brownian motion
with infinitesimal generator $\Delta$ rather than the standard Brownian motion,
which has generator $\frac{1}{2}\Delta$ (this is to have more tractable constants in isomorphism relations).
Let $D$ be an open domain which we may assume to be bounded without loss of generality. 
Let $p_D(t, z, z)$ denote the transition probability of Brownian motion killed upon leaving the domain $D$.
If $$
p_{\C}(t,z,w) = \frac1{4\pi t} \exp \Big( - \frac{|w-z|^2}{4\pi t }\Big)
$$ 
denotes the transition probabilities of this Brownian motion in the full plane, and if $\pi_D(t, z, w)$ denotes the probability that a Brownian bridge of duration $t$ remains in the domain $D$ throughout, then 
$$
p_D(t, z, w) = p_{\C}(t,z,w) \pi_D(t, z, w).
$$
Let $G_D(z,w)$ denote \textbf{Green function} of $-\Delta$ on $D$ with
Dirichlet $0$ boundary conditions; that is,
\begin{equation}
\label{eq:Green}
G_D(z,w) = \int_0^\infty p_D(t, z, w) \d t.
\end{equation}
In our normalisation, 
\begin{equation}\label{eq:log}
G_D(z,w) \sim - \frac1{2\pi} \log (|w-z|)
\end{equation}
as $|w-z| \to 0$. 

Next we recall the definitions of some natural measures on Brownian paths and loops.
For details, we refer to 
\cite[Chapter 5]{LawlerConformallyInvariantProcesses}
and \cite{Lawler04}.
Given $z,w \in D$ and $t>0$,
let $\P_{D,t}^{z,w}$ denote the probability measure on
Brownian bridges from $z$ to $w$ of duration $t$,
conditioned on staying in $D$.
Let $\mu_{D}^{z,w}$ denote the following measure on continuous paths
from $z$ to $w$ in $D$:
\begin{equation}
\label{Eq mu D z w}
\mu_{D}^{z,w}(d\wp)
=
\int_{0}^{+\infty}\P_{D,t}^{z,w}(d\wp)p_D(t, z, w) \d t.
\end{equation}
The total mass of $\mu_{D}^{z,w}$ is $G_D(z,w)$.
In particular, it is infinite if $z=w$.
The image of $\mu_{D}^{z,w}$ by time reversal is $\mu_{D}^{w,z}$.
Given a subdomain $D'\subset D$ and $z,w\in D'$,
\begin{equation}\label{eq:measure_path_restriction}
\mu_{D'}^{z,w}(d\wp)
=
\indic{\wp \text{ stays in } D'}
\mu_{D}^{z,w}(d\wp).
\end{equation}

Further, if $z\in D$ and 
$x\in\partial D$, and $\partial D$ is smooth near $x$,
we will denote
\begin{equation}\label{Eq mu D z w boundary}
\mu_{D}^{z,x}(d\wp)
=\lim_{\varepsilon\to 0}
\varepsilon^{-1} \mu_{D}^{z,
x + \varepsilon \overrightarrow{n}_{x}}(d\wp),
\end{equation}
where $\overrightarrow{n}_{x}$ is the normal unit vector at $x$
pointing inwards.
In this way, $\mu_{D}^{z,x}$ is a measure on
interior-to-boundary
Brownian excursions from $z$ to $x$.
Its total mass is given by
\begin{equation}
\label{Eq Pk}
H_{D}(z,x) =
\lim_{\varepsilon\to 0}
\varepsilon^{-1}
G_{D}(z,x + \varepsilon \overrightarrow{n}_{x}).
\end{equation}
This $H_{D}(z,x)$ is the \textbf{Poisson kernel},
the density of the harmonic measure from $z$.
The probability measure
$\mu_{D}^{z,x}/H_{D}(z,x)$
is the law of the Brownian motion starting from $z$ up to the first hitting time of $\partial D$,
conditioned on hitting $\partial D$ in $x$.
Now, if $x,y\in \partial D$
and $\partial D$ is smooth near $x$ and near $y$,
we similarly define
\begin{equation}
\label{Eq boundary exc}
\mu_{D}^{x,y}(d\wp)
=\lim_{\varepsilon\to 0}
\varepsilon^{-2} 
\mu_{D}^{x + \varepsilon \overrightarrow{n}_{x},
y + \varepsilon \overrightarrow{n}_{y}}(d\wp).
\end{equation}
In this way, $\mu_{D}^{x,y}$ is a measure on
boundary-to-boundary
Brownian excursions from $x$ to $y$.
Its total mass is given by
\begin{equation}
\label{Eq bPk}
H_{D}(x,y)
=\lim_{\varepsilon\to 0}
\varepsilon^{-2} 
G_{D}(x + \varepsilon \overrightarrow{n}_{x},
y + \varepsilon \overrightarrow{n}_{y}).
\end{equation}
Here, $H_{D}(x,y)$ is the \textbf{boundary Poisson kernel}.
Note that $H_{D}(x,x)= +\infty$.

\begin{notation}\label{not:paths}
For any $z \in D$ and $w \in D$, respectively $w \in \partial D$, we will denote by $\wp_D^{z,w}$ a Brownian trajectory distributed according to 
\begin{equation}
\label{eq:proof_not}
\mu_D^{z,w} / G_D(z,w), \quad \text{respectively} \quad \mu_D^{z,w} / H_D(z,w).
\end{equation}
If $z \in \partial D$ and $w \in D$, we will denote by $\wp_D^{z,w}$ a trajectory which is the time reversal of a path distributed according to $\mu_D^{w,z} / H_D(w,z)$.
\end{notation}

The natural measure on Brownian loops in $D$ is
\begin{equation}
\label{Eq loop meas}
\loopmeasure_{D}(d\wp)=
\int_{D}\int_{0}^{+\infty}
\P_{D,t}^{z,z}(d\wp)
p_D(t, z, z)\dfrac{\d t}{t} \d z.
\end{equation}
The measure $\loopmeasure_{D}$ has an infinite total mass
because of the ultraviolet divergence.
The measure on loops is invariant under time reversal.
It also satisfies a restriction property: given a subdomain $D'\subset D$,
\begin{equation}
\label{Eq restriction loops}
\loopmeasure_{D'}(d\wp)
=\indic{\wp \text{ stays in } D'}
\loopmeasure_{D}(d\wp).
\end{equation}
The measure $\loopmeasure_{D}$ can be rewritten as
\begin{equation}
\label{Eq loop to excursion}
\loopmeasure_{D}(d\wp)
=\dfrac{1}{T(\wp)}\int_{D}\mu_{D}^{z,z}(d\wp) \d z,
\end{equation}
where $T(\wp)$ denotes the total duration of a generic path 
$\wp$.

We will also need in what follows the massive version of the measure on Brownian loops.
Let $K>0$ be a constant.
Let $G_{D,K}(z,w)$ denote the \textbf{massive Green function}
associated to $-\Delta + K$,
with Dirichlet $0$ boundary conditions.
We have that
\begin{equation}
  \label{eq:Green_mass}
  G_{D,K} (z,w) = \int_0^\infty e^{- K t} p_D(t,z,w) \d t.
\end{equation}
In Quantum Field Theory, $K$ corresponds to the square of a particle mass.
In terms of Brownian motion, $K$ is just a killing rate.
The massive measure on Brownian loops in $D$ is defined by
\begin{equation}
\label{Eq loop meas K}
\loopmeasure_{D,K}(d\wp) = 
e^{-K T(\wp)}
\loopmeasure_{D}(d\wp).
\end{equation}
Note that the massive measure on Brownian loops was introduced in early works on Euclidean QFT by Symanzik
\cite{Symanzik65Scalar,Symanzik66Scalar,Symanzik1969QFT}.

\medskip

The loops under the measures $\loopmeasure_{D}$ \eqref{Eq loop meas}
and $\loopmeasure_{D,K}$ \eqref{Eq loop meas K} are \textbf{rooted},
that is to say the loops $\wp$ have a well defined
starting time and end time.
However, one usually considers \textbf{unrooted loops} 
\cite{Lawler04,LawlerConformallyInvariantProcesses},
that is to say one identifies the loops under circular shifts
of the parametrisation.
Two rooted loops $\wp$ and $\tilde{\wp}$ correspond to the same unrooted loop if $T(\wp)=T(\tilde{\wp})$,
and there is $s\in [0, T(\wp)]$ such that
$\tilde{\wp}(t) = \wp(t+s)$ for
$t\in [0,T(\wp)-s]$, and
$\tilde{\wp}(t) = \wp(t+s - T(\wp))$ for
$t\in [T(\wp)-s, T(\wp)]$.
We will denote by 
$\mu^{\rm{loop} \ast}_{D}$,
respectively $\mu^{\rm{loop} \ast}_{D,K}$,
the measures on unrooted loops induced by
$\loopmeasure_{D}$,
respectively $\loopmeasure_{D,K}$.

By considering unrooted loops, one gains a covariance under 
conformal maps for $\mu^{\rm{loop} \ast}_{D}$.
Let $D$ and $\widetilde{D}$ be two conformally equivalent open domains
and $\psi : D\rightarrow \widetilde{D}$ a conformal map.
Let $\Tc_{\psi}$ be the following transformation of paths induced
by $\psi$.
Given $\wp$ a path in $D$, one applies to $\wp$ the map $\psi$
and performs a change of time
$\d s = \vert \psi'(\wp(t))\vert^{2} \d t$.
Then $\mu^{\rm{loop} \ast}_{\widetilde{D}}$ is the image measure of
$\mu^{\rm{loop} \ast}_{D}$ under $\Tc_{\psi}$;
see \cite[Proposition~6]{Lawler04}
and \cite[Proposition~5.27]{LawlerConformallyInvariantProcesses}.
Note that in general,
$\loopmeasure_{\widetilde{D}}$ is not the image of
$\loopmeasure_{D}$ under $\Tc_{\psi}$.

\medskip

Given $\theta>0$,
a \textbf{Brownian loop soup} $\Lc_{D}^{\theta}$,
as introduced in \cite{Lawler04},
is a Poisson point process of intensity $\theta\loopmeasure_{D}$.
We see it as a random infinite countable collection of Brownian
loops in $D$.
We will consider both rooted and unrooted loops, depending on the context,
and use the same notation $\Lc_{D}^{\theta}$ in both cases.
On simply connected domains,
the Brownian loop soups were used in the construction of
Conformal Loop Ensembles CLE$_{\kappa}$
\cite{SheffieldWernerCLE}.
At the particular value of the intensity parameter
$\theta=1/2$,
the loop soup $\Lc_{D}^{1/2}$ is related to the
continuum Gaussian free field (GFF) and to the
CLE$_{4}$
\cite{LeJan2010LoopsRenorm,LeJan2011Loops,SheffieldWernerCLE,
QianWerner19Clusters,ALS2}.
These relations are part of the random walk/Brownian motion
representations of the GFF, 
also known as \textbf{isomorphism theorems}
\cite{Symanzik65Scalar,Symanzik66Scalar,Symanzik1969QFT,
BFS82Loop,Dynkin1984Isomorphism,
Dynkin1984IsomorphismPresentation,
MarcusRosen2006MarkovGaussianLocTime,
Sznitman2012LectureIso}.

Now let us define the loops in $\Lc_{D}^{\theta}$
killed by a killing rate $K$.
Let $\{ U_\wp, \wp \in \Lc_D^\theta \}$ be a collection of i.i.d. uniform random variables on $[0,1]$.
Given $K>0$, set
\begin{equation}
\label{Eq L K}
\Lc_D^\theta(K) := 
\left\{ \wp \in \Lc_D^\theta : U_\wp < 
1-e^{- K T(\wp)} \right\}.
\end{equation}
The subset $\Lc_D^\theta(K)$ of $\Lc_D^\theta$
consists of loops killed by $K$.
The complementary $\Lc_D^\theta\setminus\Lc_D^\theta(K)$
is a Poisson point process of intensity
$\theta\loopmeasure_{D,K}$. 
In other words it is a massive Brownian loop soup.
The construction through the uniform r.v.s $U_\wp$-s
allows to couple $\Lc_{D}^{\theta}$ and the 
$\Lc_D^\theta(K)$ for all $K>0$ on the same probability space.
Moreover, this coupling is monotone:
if $K'\leq K$, 
then $\Lc_D^\theta(K')\subset\Lc_D^\theta(K)$ a.s.

It is easy to see that a.s., for every $K > 0$, 
$\Lc_D^\theta(K)$ is infinite.
However,
\begin{displaymath}
\E\big[\big\vert\{\wp\in \Lc_D^\theta(K)
\vert T(\wp)>\varepsilon\}\big\vert\big]
\asymp \log(\varepsilon^{-1}),
\qquad
\E\big[\big\vert\{\wp\in \Lc_D^\theta(K)
\vert \diam(\wp)>\varepsilon\}\big\vert\big]
\asymp \log(\varepsilon^{-1}),
\end{displaymath}
whereas for the whole loop soup $\Lc_D^\theta$,
\begin{displaymath}
\E\big[\big\vert\{\wp\in \Lc_D^\theta
\vert T(\wp)>\varepsilon\}\big\vert\big]
\asymp \varepsilon^{-1},
\qquad
\E\big[\big\vert\{\wp\in \Lc_D^\theta
\vert \diam(\wp)>\varepsilon\}\big\vert\big]
\asymp \varepsilon^{-2}.
\end{displaymath}

\medskip

For the sequel we will need to formalize a topology on
collections of unrooted loops.
First, let us defined a distance on the
continuous paths in $\C$ of finite duration.
Given $(\wp_{1}(t))_{0\leq t\leq T(\wp_{1})}$ and
$(\wp_{2}(t))_{0\leq t\leq T(\wp_{2})}$
such paths, let be the distance
\begin{equation}
\label{Eq dist paths}
d_{\rm paths}(\wp_{1},\wp_{2})
:=
\vert\log(T(\wp_{2})/T(\wp_{1}))\vert
+
\max_{0\leq s\leq 1}
\vert \wp_{2}(s T(\wp_{2})) 
- \wp_{1}(s T(\wp_{1}))\vert .
\end{equation}
If $\wp_{1}$ and $\wp_{2}$ are two rooted loops,
i.e. $\wp_{1}(T(\wp_{1})) = \wp_{1}(0)$ and
$\wp_{2}(T(\wp_{2})) = \wp_{2}(0)$,
and if $[\wp_{1}]$ and $[\wp_{2}]$
are the corresponding unrooted loops,
i.e. the equivalence classes under circular shifts of
parametrisation,
then let be the distance
\begin{displaymath}
d_{\rm unrooted}([\wp_{1}],[\wp_{2}])
:=\min_{\tilde{\wp}\in[\wp_{1}]}d_{\rm paths}(\tilde{\wp},\wp_{2})
=
\min_{\tilde{\wp}\in[\wp_{2}]}
d_{\rm paths}(\wp_{1},\tilde{\wp}).
\end{displaymath}

Now let us consider finite collections of unrooted loops.
Here and in the sequel by collection we mean a multiset.
The elements of a multiset are unordered, but may come each with a
finite multiplicity.
A collection can also be empty.
Given $\Lc_{1}$ and $\Lc_{2}$ two such finite collections of unrooted loops on $\C$,
we set the distance
\begin{displaymath}
d_{\rm fin. col.}(\Lc_{1},\Lc_{2})
:=
\min_{\sigma\in \operatorname{Bij}(\Lc_{1},\Lc_{2})}
\sum_{\wp\in \Lc_{1}}
d_{\rm unrooted}(\wp,\sigma(\wp))
\end{displaymath}
if $\Lc_{1}$ and $\Lc_{2}$ have same cardinal with multiplicities taken into account,
and $d_{\rm fin. col.}(\Lc_{1},\Lc_{2})=+\infty$ otherwise.
In particular, the distance of the empty collection to any non-empty
collection is $+\infty$.

Given $z\in\C$ and $r>0$, 
let $D(z,r)$ denote the open disc with center $z$ and radius $r$.
Given $\Lc$ a collection of unrooted loops,
not necessarily finite,
and $r>0$, denote
\begin{displaymath}
\Lc_{\vert r}
:=
\{
\wp\in \Lc :
\wp \text{ stays in } \overline{D(0,r)},
\diam(\wp)\geq r^{-1}
\} .
\end{displaymath}
Let $\mathfrak{L}$ be the following space:
\begin{displaymath}
\mathfrak{L}:=
\{
\Lc \text{ collection of unrooted loops on } \C
:
\forall r > 0,
\Lc_{\vert r} \text{ is finite}
\} .
\end{displaymath}
The empty collection also belongs to $\mathfrak{L}$.
All the collections belonging to $\mathfrak{L}$ are countable.
We endow $\mathfrak{L}$ with the following distance:
\begin{equation}
\label{eq:d_frac_L}
d_{\mathfrak{L}}(\Lc_{1},\Lc_{2})
:=
\int_{1}^{+\infty} e^{-r}
(d_{\rm fin. col.}((\Lc_{1})_{\vert r},(\Lc_{2})_{\vert r})\wedge 1 )
\d r.
\end{equation}
A sequence $(\Lc_{k})_{k\geq 0}$ converges to $\Lc$ for
$d_{\mathfrak{L}}$ if and only if
there is a positive increasing sequence $(r_{j})_{j\geq 0}$,
with $\lim_{j\to +\infty } r_{j} = +\infty$,
such that for every $j\geq 0$,
\begin{displaymath}
\lim_{k\to +\infty}
d_{\rm fin. col.}((\Lc_{k})_{\vert r_{j}},\Lc_{\vert r_{j}})
= 0 .
\end{displaymath}
It is easy to see that the induced metric space
$(\mathfrak{L},d_{\mathfrak{L}})$
is complete.
Moreover, the finite collections are dense in $\mathfrak{L}$.
Further, the finite collections can be approximated by a countable subset of finite collections. Consider for instance the trigonometric series.
Thus, the metric space $(\mathfrak{L},d_{\mathfrak{L}})$ is
separable.
So, $(\mathfrak{L},d_{\mathfrak{L}})$ is a Polish space.
We will often see the Brownian loop soups
$\Lc_D^\theta$ and $\Lc_D^\theta(K)$ as r.v.s with values in 
$\mathfrak{L}$.

\paragraph{Equivalence relation on $(\mathfrak{L},d_{\mathfrak{L}})$ and admissible functions}

We now formalise the notion of functions $F : \mathfrak{L} \to \R$ that are invariant by exchanging the order of the excursions in the loops at a given point $z \in \C$. We will call such functions $z$-\emph{admissible} functions.

Let $\wp : t \in [0,T(\wp)] \mapsto \wp_t \in \C$ be a continuous path in $\C$ with finite duration and such that $\wp_0 = \wp_{T(\wp)}$. Let $z \in \C$ be a point visited by $\wp$. To $\wp$ and $z$ we can uniquely associate an at most countable collection of excursions $\{ e_i^{\wp,z}, i \in I \}$, where by an excursion $e$ we mean a continuous path $(e_t, 0 \leq t \leq \zeta)$ such that $e_0 = e_\zeta = z$ and $e_t \neq z$ for all $t \in (0,\zeta)$, and such that the reunion of all $e_i^{\wp,z}$ coincides with the loop $\wp$. In fact, these excursions inherit from $\wp$ a chronological order but we will not need this.

For a fixed $z \in \C$, we define an equivalence relation $\sim_z$ on unrooted loops by saying that two loops $\wp$ and $\wp'$ are equivalent if, and only if, 
\begin{itemize}
\item
either $z$ is not visited by $\wp$, nor $\wp'$, and in that case the unrooted loops $[\wp]$ and $[\wp']$ agree;
\item
or $z$ is visited by both $\wp$ and $\wp'$ and the collections of \emph{unordered} excursions $\{e_i^{ \wp,z}, i \in I \}$ and $\{e_i^{ \wp',z}, i \in I' \}$ coincide. 
\end{itemize}
We will denote $< \wp >_z$ the equivalence class of a loop $\wp$ under the relation $\sim_z$. If $\Cc \in \mathfrak{L}$ is a collection of loops, we will denote $< \Cc >_z := \{ < \wp >_z, \wp \in \Cc \}$.

We can now give a precise definition of admissible functions.

\begin{definition}\label{def:admissible}
Let $z \in \C$. We will say that a function $F : \mathfrak{L} \to \R$ is $z$-admissible if $F(\cdot)$ is invariant under the relation $\sim_z$, i.e. if for all $\Cc, \Cc' \in \mathfrak{L}$, $F(\Cc) = F(\Cc')$ as soon as $< \Cc >_z = < \Cc' >_z$.

Functions $F : D \times \mathfrak{L} \to \R$ (resp. $F : D \times D \times \mathfrak{L} \to \R$) are called admissible if for all $z \in D$, $F(z,\cdot)$ is $z$-admissible (resp. if for all $z, z' \in D$, $F(z,z',\cdot)$ is $z$-admissible and $z'$-admissible).
\end{definition}

Examples of admissible functions include total time duration, number of crossings of an annulus, etc.

Finally, we introduce the $\sigma$-algebra
\begin{equation}
\label{eq:sigma_algebra_admissible}
\sigma ( \scalar{ \Lc_D^\theta } ) := \sigma \left( F( \Lc_D^\theta), F : \mathfrak{L} \to \R \text{~bounded~measurable~s.t.~} \forall z \in \C, F \text{~is~} z \text{-admissible} \right).
\end{equation}
It is the $\sigma$-algebra generated by the equivalence class of $\Lc_D^\theta$ where two loops $\wp$ and $\wp'$ are identified if and only if $\wp \sim_z \wp'$ for all $z \in \C$. Note that this $\sigma$-algebra is included in $\sigma(\scalar{\Lc_D^\theta}_z)$ for any $z \in D$.

\subsection{Measures on discrete paths and random walk loop soup}
\label{sec:preliminaries_RWLS}

Here we will recall some properties of the
continuous-time discrete-space random walk loop soups.

Let $N\geq 1$ be an integer.
We will denote
$\Z_{N}:=\frac{1}{N}\Z$,
and work on the rescaled square lattice
$\Z_{N}^{2}$.
Let $\Delta_{N}$ be the discrete Laplacian on $\Z_{N}^{2}$:
\begin{displaymath}
(\Delta_{N} f)(z) :=
N^{2}\sum_{\substack{w\in\Z_{N}^{2}\\
\vert w-z\vert = \frac{1}{N}}}
(f(w)-f(z)),
~~ z\in \Z_{N}^{2}.
\end{displaymath}
Note that with our normalisation,
$\Delta_{N}$ converges as $N\to +\infty$ to the continuum
Laplacian $\Delta$ on $\C$.
Let $(X^{(N)}_{t})_{t\geq 0}$ be the Markov jump process on
$\Z_{N}^{2}$ with infinitesimal generator $\Delta_{N}$.
In other words, this is the continuous-time
simple symmetric random walk,
with exponential holding times with mean $\frac{1}{4 N^{2}}$.
As $N\to +\infty$,
$(X^{(N)}_{t})_{t\geq 0}$ converges in law to
the Brownian motion on $\C$ with infinitesimal generator
$\Delta$.

Let $D_{N}$ be a non-empty subset of $\Z_{N}^{2}$.
Note that in the sequel we will typically consider
sequences $(D_{N})_{N\geq 1}$
converging to continuum domains $D\subset\C$ as in \eqref{eq:DN}.
Let $\tau_{\Z_{N}^{2}\setminus D_{N}}$
denote the first hitting time of
$\Z_{N}^{2}\setminus D_{N}$ by $X^{(N)}_{t}$.
Denote
\begin{displaymath}
p_{D_{N}}(t,z,w)
:=
N^{2}
\P^{z}
\big(X^{(N)}_{t}=w,~\tau_{\Z_{N}^{2}\setminus D_{N}}> t\big),
~~ z,w\in D_{N}.
\end{displaymath}
Note that
$p_{D_{N}}(t,z,w)=p_{D_{N}}(t,w,z)$.
Denote
\begin{displaymath}
G_{D_{N}}(z,w) =\int_{0}^{+\infty} p_{D_{N}}(t,z,w) \d t.
\end{displaymath}
If $z$ or $w$ is in $\Z_{N}^{2}\setminus D_{N}$,
we set $G_{D_{N}}(z,w)=0$.
Defined this way, $G_{D_{N}}$ is the discrete Green function.
It satisfies
\begin{displaymath}
- \Delta_{N,w}
G_{D_{N}}(z,w)
= N^{2}\indic{z=w},
~~ z,w\in D_{N},
\end{displaymath}
where the notation $\Delta_{N,w}$ indicates that the discrete
Laplacian $\Delta_{N}$ is taken with respect to the variable $w$.

Let $\P_{D_{N},t}^{z,w}$
denote the law of
$(X^{(N)}_{s})_{0\leq s\leq t}$,
with $X^{(N)}_{0}= z$,
conditionally on $X^{(N)}_{t} = w$
and $\tau_{\Z_{N}^{2}\setminus D_{N}} > t$.
Next we recall the discrete analogues of measures
\eqref{Eq mu D z w} and \eqref{Eq loop meas}.
For details, we refer to
\cite{LeJan2010LoopsRenorm,LeJan2011Loops}.
The measure $\mu_{D_{N}}^{z,w}$ will be a measure on
nearest-neighbour paths from $z$ to $w$ in $D_{N}$,
parametrised by continuous time, and of final
total duration:
\begin{equation}
\label{Eq discr paths}
\mu_{D_{N}}^{z,w}(d\wp)
=
\int_{0}^{+\infty}\P_{D_{N},t}^{z,w}(d\wp)p_{D_{N}}(t, z, w)
\d t.
\end{equation}
The total mass of $\mu_{D_{N}}^{z,w}$ is $G_{D_{N}}$.
The image of $\mu_{D_{N}}^{z,w}$ by time reversal
is $\mu_{D_{N}}^{w,z}$.

In the case when
$\Z_{N}^{2}\setminus D_{N}$ is also non-empty,
let $\partial D_{N}$ denote the subset of
$\Z_{N}^{2}\setminus D_{N}$ made of vertices at
graph distance $1$ from $D_{N}$,
i.e. at Euclidean distance $\frac{1}{N}$.
Given $z\in D_{N}$ and $x\in \partial D_{N}$,
denote
\begin{equation}
\label{Eq discr int to bound}
\mu_{D_{N}}^{z,x}
=
N\sum_{
\substack
{
w\in D_{N} \\
\vert w-x\vert = \frac{1}{N}
}
}
\mu_{D_{N}}^{z,w}.
\end{equation}
Let $H_{D_{N}}(z,x)$ denote the total mass of
the measure $\mu_{D_{N}}^{z,x}$.
We have that
\begin{equation}
\label{Eq discr Pk}
H_{D_{N}}(z,x) =
N\sum_{
\substack
{
w\in D_{N} \\
\vert w-x\vert = \frac{1}{N}
}
}
G_{D_{N}}(z,w).
\end{equation}
Usually, we will add to trajectories under
$\mu_{D_{N}}^{z,x}$ an additional instantaneous jump to
$x$ at the end, without local time spent at $x$.
In this way, the probability measure
$\mu_{D_{N}}^{z,x}/H_{D_{N}}(z,x)$
is actually the distribution of
$(X^{(N)}_{t})_{0\leq t\leq \tau_{\Z_{N}^{2}\setminus D_{N}}}$
given that $X^{(N)}_{0} = z$ and
conditionally on
$X^{(N)}_{\tau_{\Z_{N}^{2}\setminus D_{N}}} = x$.
Moreover,
\begin{displaymath}
H_{D_{N}}(z,x) =
N
\P^{z}\big(
X^{(N)}_{\tau_{\Z_{N}^{2}\setminus D_{N}}} = x
\big).
\end{displaymath}
So we see $H_{D_{N}}(z,x)$ as the
discrete Poisson kernel.
%Further, given $x, y\in \partial D_{N}$,
%denote
%\begin{displaymath}
%\mu_{D_{N}}^{x,y}
%=
%N^{2}
%\sum_{\substack{z,w\in D_{N}
%\\
%\vert z-x\vert = \frac{1}{N}
%\\
%\vert w-y\vert = \frac{1}{N}
%}}
%\mu_{D_{N}}^{z,w} .
%\end{displaymath}
%The total mass of $\mu_{D_{N}}^{x,y}$ is
%$H_{D_{N}}(x,y)$, where
%\begin{displaymath}
%H_{D_{N}}(x,y)  =
%N^{2}
%\sum_{\substack{z,w\in D_{N}
%\\
%\vert z-x\vert = \frac{1}{N}
%\\
%\vert w-y\vert = \frac{1}{N}
%}}
%G_{D_{N}}(z,w) .
%\end{displaymath}
%The function $H_{D_{N}}(x,y)$ is the discrete
%boundary Poisson kernel.
%\tl{Do we actually use it?}
%Again, usually we will add to trajectories under
%$\mu_{D_{N}}^{x,y}$ an initial jump from $x$ and a final jump %to $y$, without local time spent at $x$ or $y$.

The measure $\loopmeasure_{D_{N}}$
will be a measure on rooted nearest-neighbour loops
in $D_{N}$, parametrised by continuous time, and of final
total duration:
\begin{equation}
\label{Eq discr loops}
\loopmeasure_{D_{N}}(d\wp)=
\dfrac{1}{N^{2}}
\sum_{z\in D_{N}}
\int_{0}^{+\infty}
\P_{D_{N},t}^{z,z}(d\wp)
p_{D_{N}}(t, z, z)\dfrac{\d t}{t} .
\end{equation}
The measure $\loopmeasure_{D_{N}}$ is invariant by time reversal.
Note that the total mass of $\loopmeasure_{D_{N}}$
is always infinite because of the ultraviolet divergence.
The measure puts an infinite mass on trivial "loops" that
stay in one vertex, without performing jumps.
To the contrary, $\loopmeasure_{D_{N}}$ puts a finite mass on loops
that visit at least two vertices and stay inside a finite box.
More precisely, given $z_{1},z_{2},\dots,z_{2n}\in D_{N}$,
with $\vert z_{i}-z_{i-1}\vert=\frac{1}{N}$
and $\vert z_{2n}-z_{1}\vert=\frac{1}{N}$,
the weight given to the set of rooted loops starting from
$z_{1}$, then successively visiting $z_{2},\dots,z_{2n}$,
and then returning to $z_{1}$ is
$(2n)^{-1}4^{-2n}$.
Moreover, conditionally on this discrete skeleton,
the holding times are i.i.d.
exponential r.v.s with mean
$\frac{1}{4 N^{2}}$.
Given a subset
$D'_{N}\subset D_{N}$,
\begin{displaymath}
\mu_{D'_{N}}^{z,w}
=
\indic{\wp \text{ stays in } D'_{N}}
\mu_{D_{N}}^{z,w},
~~ z,w\in D'_{N},
\qquad
\loopmeasure_{D'_{N}}
=
\indic{\wp \text{ stays in } D'_{N}}
\loopmeasure_{D_{N}}.
\end{displaymath}

The measure on continuous-time discrete-space loops
\eqref{Eq discr loops} first appeared in \cite{LeJan2010LoopsRenorm,LeJan2011Loops}.
Related measures on discrete-time loops
appeared in
\cite{BFS82Loop,LawlerFerreras07RWLoopSoup,LawlerLimic2010RW}.

We will also need a measure
$\check{\mu}_{D_{N}}^{z,w}$
related but different from
$\mu_{D_{N}}^{z,w}$.
Given $z,w\in D_{N}$, denote
\begin{equation}
\label{Eq check mu}
\check{\mu}_{D_{N}}^{z,w}
=
\sum_{\substack{z',w'\in D_{N}\setminus \{z,w\}
\\
\vert z'-z\vert = \frac{1}{N}
\\
\vert w'-w\vert = \frac{1}{N}
}}
\mu_{D_{N}\setminus \{z,w\}}^{z',w'} .
%=
%\dfrac{1}{N^{2}}
%\mu_{D_{N}\setminus \{z,w\}}^{z,w} .
\end{equation}
This is a measure on continuous-time nearest-neighbour
paths from a neighbour of $z$ to a neighbour of $w$,
and staying in $D_{N}\setminus \{z,w\}$.
Actually, to a path under $\check{\mu}_{D_{N}}^{z,w}$
we will add an initial jump from $z$ to the corresponding neighbour
$z'$, and a final jump to $w$ from the corresponding neighbour $w'$.
In this way we get a path from $z$ to $w$,
but with zero holding time in $z$ and $w$.

We will also consider the massive case.
Let $K>0$ be a constant.
Denote $G_{D_{N},K}(z,w)$ the massive Green function
\begin{equation}
\label{eq:def_massive_green_discrete}
G_{D_{N},K}(z,w) =
\int_{0}^{+\infty} e^{-K t} p_{D_{N}}(t,z,w) \d t.
\end{equation}
The massive version of the measure on loops \eqref{Eq discr loops}
is
\begin{displaymath}
\loopmeasure_{D_{N},K}(d\wp)
=
e^{-K T(\wp)}\loopmeasure_{D_{N}}(d\wp),
\end{displaymath}
where $T(\wp)$ is the total duration of a loop.

\medskip

Again, given $\theta>0$,
we will consider Poisson point processes of intensity
$\theta\loopmeasure_{D_{N}}$,
denoted $\Lc^{\theta}_{D_{N}}$.
We will consider both rooted and unrooted loops.
These are random countable collections of loops in $D_{N}$,
known as continuous time \textbf{random walk loop soups}.
Note that, if $D_{N}$ is finite, then
$\Lc^{\theta}_{D_{N}}$ contains a.s. only finitely many
non-trivial loops that visit at least two vertices.
However, a.s., for every $z\in D_{N}$,
$\Lc^{\theta}_{D_{N}}$ contains infinitely many trivial
"loops" that only stay in $z$.

Now, consider a constant $K>0$.
Let $U_\wp, \wp \in \Lc^\theta_{D_{N}}$,
be a collection of i.i.d. uniform random variables on $[0,1]$.
Define
\begin{displaymath}
\Lc^\theta_{D_{N}}(K) :=
\left\{ \wp \in \Lc_{D_{N}}^\theta : U_\wp < 1-e^{-KT(\wp)} \right\}.
\end{displaymath}
The subset $\Lc^\theta_{D_{N}}(K)$ corresponds to loops killed
by the killing rate $K$.
The complementary
$\Lc^\theta_{D_{N}}\setminus\Lc^\theta_{D_{N}}(K)$
is a Poisson point process with intensity measure
$\theta\loopmeasure_{D_{N},K}$.
Unlike in the continuum case,
$\Lc^\theta_{D_{N}}(K)$ is a.s. finite
if $D_{N}$ is finite.
This is because
\begin{displaymath}
\int_{0}^{\varepsilon}
(1 - e^{-K t}) \dfrac{\d t}{t}
< +\infty.
\end{displaymath}

\medskip

For a vertex $z \in \Z_{N}^{2}$ and a path on $\Z_{N}^{2}$
parametrised by continuous time
$(\wp(t))_{0 \leq t \leq T(\wp)}$,
we denote by $\ell_z(\wp)$ the local time accumulated by $\wp$ at $z$, i.e.
\begin{displaymath}
\ell_z(\wp) := N^2 \int_0^{T(\wp)} \indic{\wp(t) = z} \d t.
\end{displaymath}
Given $\Lc$ a collection of path on $\Z_{N}^{2}$,
we denote
\begin{equation}
\label{Eq Occup field}
\ell_z(\Lc) := \sum_{\wp\in\Lc} \ell_z(\wp) .
\end{equation}

First we state some Markovian decomposition properties for the measures
$\mu_{D_{N}}^{z,w}$ and $\check{\mu}_{D_{N}}^{z,w}$.
These are elementary, so we do not provide proofs.

\begin{lemma}
\label{Lem Markov discrete}
Let $D_{N}\subset\Z_{N}^{2}$
such that both $D_{N}$ and
$\Z_{N}^{2}\setminus D_{N}$
are non-empty.
\begin{enumerate}
\item Given $z\in D_{N}$,
under the probability measure
$G_{D_{N}}(z,z)^{-1} \mu_{D_{N}}^{z,z}(d \wp)$,
the local time $\ell_z(\wp)$
is an exponential r.v. with mean $G_{D_{N}}(z,z)$.
Conditionally on $\ell_z(\wp)$,
the behaviour of $\wp$ outside $z$ is given by a Poisson point process
of excursions from $z$ to $z$ with intensity measure
$\ell_z(\wp)\check{\mu}_{D_{N}}^{z,z}$.
\item Let $z,w,z'\in D_{N}$ such that $z'$ is at a graph distance at least $2$ from both $z$ and $w$,
i.e. $\vert z' - z\vert >\frac{1}{N}$ and
$\vert z' - w\vert >\frac{1}{N}$.
Then for any bounded measurable function $F$,
\begin{displaymath}
\int
\indic{\wp \text{ visits } z'}
F(\wp)
\check{\mu}_{D_{N}}^{z,w}(d\wp)
=
\int \check{\mu}_{D_{N}}^{z,z'}(d\wp_{1})
\int \mu_{D_{N}\setminus\{z,w\}}^{z',z'}(d\wp)
\int \check{\mu}_{D_{N}}^{z',w}(d\wp_{2})
F(\wp_{1}\wedge\wp\wedge\wp_{2}),
\end{displaymath}
where $\wedge$ denotes the concatenation of paths.
\item Let $z,w\in D_{N}$ such that $z$ and $w$ are at a graph distance at least $2$,
i.e. $\vert w - z\vert >\frac{1}{N}$.
Then for any bounded measurable function $F$,
\begin{displaymath}
\int
F(\wp)
\mu_{D_{N}}^{z,w}(d\wp)
=
\int \mu_{D_{N}}^{z,z}(d\wp_{1})
\int \check{\mu}_{D_{N}}^{z,w}(d\wp)
\int \mu_{D_{N}\setminus\{z\}}^{w,w}(d\wp_{2})
F(\wp_{1}\wedge\wp\wedge\wp_{2}).
\end{displaymath}
\end{enumerate}
\end{lemma}

Next we describe the law of the local times of loops in a
random walk loop soup. For details, we refer to
\cite{LeJan2010LoopsRenorm,LeJan2011Loops}.

\begin{proposition}[Le Jan \cite{LeJan2010LoopsRenorm,LeJan2011Loops}]
\label{Prop Le Jan subordinator}
Let $D_{N}\subset\Z_{N}^{2}$
such that both $D_{N}$ and
$\Z_{N}^{2}\setminus D_{N}$
are non-empty.
Fix $\theta>0$ and consider the random walk loop soup
$\Lc^\theta_{D_{N}}$.
Given $z\in D_{N}$,
the collection of random times
$(\ell_z(\wp))_{\wp\in \Lc^\theta_{D_{N}},
\wp \text{ visits } z}$
is a Poisson point process of $(0,+\infty)$ with intensity measure
\begin{equation}
\label{Eq Gamma subord}
\indic{t>0}\theta
e^{-t/G_{D_{N}}(z,z)}
\dfrac{\d t}{t},
\end{equation}
that is to say these are the jumps of a Gamma subordinator.
In particular, $\ell_z(\Lc^\theta_{D_{N}})$
follows a Gamma$(\theta)$ distribution with density
\begin{displaymath}
\indic{t>0}
\dfrac{1}{\Gamma(\theta) G_{D_{N}}(z,z)^{\theta}}
t^{\theta - 1}
e^{-t/G_{D_{N}}(z,z)}
.
\end{displaymath}
Conditionally on the family of local times
$(\ell_z(\wp))_{\wp\in \Lc^\theta_{D_{N}},
\wp \text{ visits } z}$,
the loops $\wp$ visiting $z$ are obtained,
up to rerooting, by concatenating independent Poisson point processes
of excursions from $z$ to $z$ with respective intensities
$\ell_z(\wp)\check{\mu}_{D_{N}}^{z,z}$.
The collections of loops not visiting $z$ is independent from
the loops visiting $z$,
and distributed as
$\Lc^\theta_{D_{N}\setminus\{ z\}}$.

Furthermore, given $K>0$ and $z\in D_{N}$,
the collection of random times
$(\ell_z(\wp))_{\wp\in \Lc^\theta_{D_{N}}\setminus \Lc^\theta_{D_{N}}(K),
\wp \text{ visits } z}$
is a Poisson point process of $(0,+\infty)$ with intensity measure
\begin{equation}
\label{Eq Gamma subord K}
\indic{t>0}\theta
e^{-t/G_{D_{N},K}(z,z)}
\dfrac{\d t}{t} .
\end{equation}
\end{proposition}

For the particular value of the intensity parameter
$\theta = 1/2$, the random walk loop soup
$\Lc^{1/2}_{D_{N}}$ is related to the discrete Gaussian free
field (GFF)
through
the \textbf{Le Jan's isomorphism theorem}
\cite{LeJan2010LoopsRenorm,LeJan2011Loops}.
Let $\varphi_{N}$
denote the discrete (massless) GFF on
$D_{N}$ with condition $0$ on
$\Z_{N}^{2}\setminus D_{N}$.
It is a random centred Gaussian field with
covariance kernel given by the Green function
$G_{D_{N}}$.
Given a constant $K>0$,
there is also the massive discrete GFF
$\varphi_{N,K}$, with covariance kernel $G_{D_{N},K}$.

\begin{theorem}[Le Jan \cite{LeJan2010LoopsRenorm,LeJan2011Loops}]
\label{Thm Iso Le Jan}
Let $D_{N}\subset\Z_{N}^{2}$
such that both $D_{N}$ and
$\Z_{N}^{2}\setminus D_{N}$
are non-empty.
Consider the random walk loop soup
$\Lc^{1/2}_{D_{N}}$.
Then, the occupation field
$(\ell_z(\Lc^{1/2}_{D_{N}}))_{z\in D_{N}}$
is distributed as
$\frac{1}{2}\varphi_{N}^{2}$.
Further,
given a constant $K>0$,
the occupation field
$(\ell_z(\Lc^{1/2}_{D_{N}}\setminus\Lc^{1/2}_{D_{N}}(K)))_{z\in D_{N}}$
is distributed as
$\frac{1}{2}\varphi_{N,K}^{2}$.
\end{theorem}

\begin{remark}\label{R:iso_cont}
Note that in dimension 2, Le Jan's isomorphism has
a renormalised version in continuum space
involving the Wick's square of the continuum
GFF \cite{LeJan2010LoopsRenorm,LeJan2011Loops}.
We recall in Appendix \ref{app_Wick} a construction of this normalised square based on a discrete approximation; see Lemma \ref{L:Wick_discrete_continuous}. This result will be needed in the proof of Theorem \ref{th:identification1/2}. We also include a proof of this folklore result since we could not find any in the literature.
\end{remark}

\subsection{Brownian multiplicative chaos}
\label{sec:preliminaries_BMC}

This section recalls some facts about Brownian multiplicative chaos measures. These measures were introduced in \cite{bass1994, AidekonHuShi2018, jegoBMC} in the case of one given Brownian trajectory and can be formally defined as the exponential of the square root of the local time of the trajectory (see \cite[Theorems 1.1 and 1.2]{jegoBMC} for a construction that uses an exponential approximation). In the current article, we will need to consider ``multipoint'' versions of these measures for finitely many independent trajectories. This generalisation has been studied in \cite{jegoRW} and was key in order to characterise the law of Brownian multiplicative chaos. The current article focuses on the subcritical regime, but let us mention that Brownian chaos measures have also been constructed at criticality, i.e. when $a = 2$ (equivalently, $\gamma = 2$); see \cite{jegoCritical}.

For all $i \geq 1$, let $D_i \subset \C$ be a bounded simply connected domain and let $z_i \in D_i$ be a starting point. Let us consider independent random processes $\wp_i = (\wp_i(t))_{0 \leq t \leq \tau_i}, i \geq 1$, in the plane such that for each $i \geq 1$, the law of $\wp_i$ is locally mutually absolutely continuous with respect to the law of Brownian motion starting at $z_i$ and killed upon exiting for the first time $D_i$
(we will later explain in Section \ref{sec:def} how to treat the case of Brownian loops which is more relevant to this article).
In order to recall a rigorous definition of the Brownian chaos measures that we will consider in this article, we first introduce local times of circles: for all $i \geq 1$, $z \in D_i$ and $\eps>0$ be such that $D(z,\eps) \subset D_i$, let
\[
L_{z,\eps}^i := \lim_{r \to 0^+} \frac{1}{2r} \int_0^{\tau_i} \indic{ \eps - r \leq |\wp_i(t) - z| \leq \eps+r} \d t.
\]
As shown in \cite[Proposition 1.1]{jegoBMC}, these local times are well-defined simultaneously for all $z$ and $\eps$.  Recall that, in the current article, we consider Brownian motion with infinitesimal generator $\Delta$ instead of the standard Brownian motion considered in \cite{jegoBMC, jegoRW} which has generator $\tfrac12 \Delta$. Because of this difference of normalisation, the local times defined above are 2 times smaller than the local times used in \cite{jegoBMC, jegoRW}.

This article will consider the following measures (recall that the case of Brownian loops will be treated in Section \ref{sec:def}):
\begin{itemize}
\item
$\Mc_a^{\wp_1 \cap \dots \cap \wp_n}$, $a \in (0,2)$: measure on $a$-thick points coming from the interaction of the $n$ trajectories. Each trajectory is required to visit the thick point, but the way the thickness is distributed among the $n$ trajectories is not specified. This measure is defined as the limit in probability, relatively to the topology of weak convergence, of
\[
\Mc_a^{\wp_1 \cap \dots \cap \wp_n}(A) := \lim_{\eps \to 0} |\log \eps| \eps^{-a} \int_A \indic{ \frac{1}{\eps} \sum_{i=1}^n L_{x,\eps}^i \geq a |\log \eps|^2} \indic{ \forall i=1 \dots n, L_{x,\eps}^i>0} \d x,
\quad A \subset \C \text{~Borel}.
\]
See \cite[Proposition 1.1]{jegoRW}.
\item
$\bigcap_{i=1}^n \Mc_{a_i}^{\wp_i}$, $\sum a_i < 2$: measure supported on the intersection of the support of each measure, the $i$-th trajectory is required to contribute exactly $a_i$ to the overall thickness. It is defined by:
\begin{align*}
\bigcap_{i=1}^n \Mc_{a_i}^{\wp_i}(A)
:=
\lim_{\eps \to 0} |\log \eps|^n \eps^{-\sum a_i} \int_A \prod_{i=1}^n \indic{ \frac{1}{\eps} L_{x,\eps}^i \geq a_i |\log \eps|^2} \d x,
\quad A \subset \C \text{~Borel},
\end{align*}
where the convergence holds in probability relatively to the topology of weak convergence. See \cite[Section 1.4]{jegoRW}.
\end{itemize}

These two types of measures are closely related. Indeed, on the one hand, $\bigcap_{i=1}^n \Mc_{a_i}^{\wp_i}$ is the Brownian chaos measure $\Mc_{a_n}^{\wp_n}$ with reference measure $\bigcap_{i=1}^{n-1} \Mc_{a_i}^{\wp_i}$, i.e. $\bigcap_{i=1}^n \Mc_{a_i}^{\wp_i}$ is also equal to
\begin{equation}
\label{eq:intersection_measure}
\lim_{\eps \to 0} |\log \eps| \eps^{-a_n} \indic{ \frac{1}{\eps} L_{x,\eps}^n \geq a_n |\log \eps|^2} \bigcap_{i=1}^{n-1} \Mc_{a_i}^{\wp_i}(dx).
\end{equation}
See \cite[Proposition 1.2 (ii)]{jegoRW}.
On the other hand, the following disintegration formula holds \cite[Proposition 1.3]{jegoRW}:
\begin{equation}
\label{eq:disintegration}
\Mc_a^{\wp_1 \cap \dots \cap \wp_n} = \int_{\mathsf{a} \in E(a,n)} \d \mathsf{a} \bigcap_{i=1}^n \Mc_{a_i}^{\wp_i}
\end{equation}
showing that the thickness is uniformly distributed among $\wp_1, \dots, \wp_n$. In this formula and in the remaining of the article, we denote by $E(a,n)$ the $(n-1)$-dimensional simplex: for all $n \geq 1, a>0$,
\begin{equation}
\label{eq:def_Ean}
E(a,n) := \{ \mathsf{a} = (a_1, \dots, a_n) \in (0,a]^n: a_1 + \dots + a_n = a \}.
\end{equation}
This disintegration formula allows us to naturally extend these definitions to ``mixed'' cases. For instance, for $a + a' < 2$, we define
\[
\Mc_a^{\wp_1 \cap \dots \cap \wp_n} \cap \Mc_{a'}^{\wp_{n+1} \cap \dots \cap \wp_{n+m}}
= \int_{\mathsf{a} \in E(a,n)} \d \mathsf{a} \int_{\mathsf{a}' \in E(a',m)} \d \mathsf{a}' \bigcap_{i=1}^{n+m} \Mc_{a_i}^{\wp_i}.
\]

\medskip

We finally explain a Girsanov-transform-type result associated to these measures, i.e. the way the laws of the paths $\wp_i$ change after shifting the probability measure by $\Mc_{a_i}^{\wp_i}(dz)$. For this purpose, we need to specify the laws of the trajectories $\wp_i$, $i \geq 1$. For all $i \geq 1$, let $D_i$ be a bounded simply connected domain, let $x_i \in D$ and let $z_i \in \partial D_i$ be a point where the boundary $D_i$ is locally analytic. The independent trajectories $\wp_i, i \geq 1$ are then assumed to be Brownian paths from $x_i$ to $z_i$ in $D_i$, i.e. $\wp_i \sim \mu_{D_i}^{x_i,z_i}/H_{D_i}(x_i,z_i)$ \eqref{Eq mu D z w boundary}. Let $n \geq 1$ and $a_i >0, i =1 \dots n$, be thickness parameters such that $\sum a_i <2$.

We will see that this shift amounts to adding infinitely many excursions from $z$ to $z$ that are sampled according to a Poisson point process. Such excursions will play a prominent role in this paper and we define them now.

\begin{notation}\label{not:xi_a^z}
We will denote by $\Xi_{a}^z$ (or by $\Xi_a^{z,D}$ when we want to emphasise the dependence in the domain $D$) the random loop rooted at $z$,
obtained by concatenating a Poisson point process of Brownian excursions
from $z$ to $z$ of intensity $2\pi a \mu_{D}^{z,z}$ \eqref{Eq mu D z w}.
Such a Poisson point process appears in the description of
a Brownian trajectory seen from a typical $a$-thick point
\cite{bass1994,AidekonHuShi2018,jegoBMC}.
We will denote by $\wedge$ the concatenation of paths.
\end{notation}

Recall also Notation \ref{not:paths}. \cite[Proposition 1.4]{jegoRW} states that for all bounded measurable function $F$,
\begin{align}\label{eq:Girsanov_intersection}
\Expect{ \int_\C F(z,\wp_1, \dots, \wp_n) \bigcap_{i=1}^n \Mc_{a_i}^{\wp_i}(dz) }
= (2\pi)^n & \int_{\cap D_i} \Big( \prod_{i=1}^n \frac{H_{D_i}(z,z_i)}{H_{D_i}(x_i,z_i)} G_{D_i}(x_i,z) \CR(z,D_i)^{a_i} \Big) \\
\nonumber
& \times \Expect{ F(z, \{ \wp_{D_i}^{x_i,z} \wedge \Xi_{a_i}^{z,D_i} \wedge \wp_{D_i}^{z,z_i} \}_{i=1 \dots n} ) } \d z
\end{align}
where all the paths above are independent. The factor $(2\pi)^n$ is due to the different normalisations of the Green function in \cite{jegoRW} and in the current paper.
In words, after the shift, the path $\wp_i$ is distributed as the concatenation of three independent paths: a trajectory $\wp_{D_i}^{x_i,z}$ from $x_i$ to $z$ in $D_i$; a loop $\Xi_{a_i}^{z,D_i}$ rooted at $z$ going infinitely many times through $z$; and a path $\wp_{D_i}^{z,z_i}$ from $z$ to $z_i$.
Such a description was already present in the paper \cite{bass1994} in the context of one trajectory.

Once again, these results concern Brownian multiplicative chaos associated to independent Brownian trajectories from internal points to boundary points in fixed domains, but they can be extended to the loops in the Brownian loop soup. This will be made clear in Section \ref{sec:def}.

\part*{Part One: Continuum}

\addcontentsline{toc}{part}{Part One: Continuum}

\section{High-level description of Proof of Theorem \ref{th:convergence_continuum}}\label{sec:high_level}

In this section, we give a high-level description of the proof of Theorem \ref{th:convergence_continuum}. We start with the first moment computations for  $\Mc_a^K$ \eqref{eq:def_measure_killed_loops}.
As mentioned in the introduction the first moment is surprisingly explicit, which suggests that there is a certain amount of exact solvability or integrability in this approximation of the loop soup.
Indeed we will see that the first moment is expressed in terms of Kummer's confluent hypergeometric function ${}_1F_1(\theta,1,\cdot)$ whose definition is recalled in \eqref{eq:Kummer_function} in Appendix \ref{app:special}. Recall also that $\CR(z,D)$ denotes the conformal radius of $D$ seen from a point $z \in D$.

\begin{proposition}\label{prop:first_moment_killed_loops}
Define for all $u \geq 0$,
\begin{equation}
\label{eq:def_Fs}
\mathsf{F}(u) := \theta \int_0^u e^{-t} {}_1F_1(\theta,1,t) \d t
\end{equation}
and for all $z \in D$,
\begin{equation}
\label{eq:def_Cpz}
C_K(z) := 2\pi (G_D - G_{D,K})(z,z) =2\pi \int_0^\infty p_D(t,z,z)(1-e^{-Kt})\d t.
\end{equation}
Then
\begin{equation}
\label{eq:expectation_thick_points_killed_loops}
\E[\Mc_a^K(dz)] = \frac{1}{a} \mathsf{F} \left( C_K(z) a \right) \CR(z,D)^a dz.
\end{equation}
\end{proposition}

The function $C_K(z)$ plays a prominent role in the following; except for the factor of $2\pi$ in front, $C_K(z)$ corresponds to the expectation of the occupation field of loops that are killed.

\begin{remark}\label{R:Girsanov}
In Lemma \ref{lem:Girsanov_K}, we will obtain a more precise version of Proposition \ref{prop:first_moment_killed_loops}: we will get analogous (but more complicated) expressions when the underlying probability measure has been tilted by $\Mc_a^K(\d z)$, thereby showing a version of Theorem \ref{th:PD} valid even when $K< \infty$. This will then play a crucial role in second moment computations.
\end{remark}

Proposition \ref{prop:first_moment_killed_loops} allows us to compute asymptotics of the first moment in a relatively straightforward manner.

\begin{lemma}\label{L:asymptoticsF_C_K}
We have the following asymptotics:
\begin{enumerate}
\item There exists $C>0$ such that for all $u>0$,
\begin{equation}
\label{eq:lem_upper_bound_Fs}
\Fs(u) \leq C \left\{
\begin{array}{cc}
u, & \text{if} \quad u \leq 1,\\
u^\theta, & \text{if} \quad u \geq 1,
\end{array}
\right.
\end{equation}
Moreover,
\begin{equation}
\label{eq:lem_f_asymptotic}
\lim_{u \to \infty} u^{-\theta} \Fs(u) = \frac{1}{\Gamma(\theta)}.
\end{equation}
\item
\begin{equation}
\label{eq:lem_asymptotic-C_K}
\lim_{K \to \infty} \frac{C_K(z)}{  \log K }= \frac{1}{2}.
\end{equation}
\end{enumerate}
\end{lemma}

We note that this justifies the normalisation $(\log K)^{-\theta}$ chosen in the statement of Theorem \ref{th:convergence_continuum}. Heuristically, \eqref{eq:lem_asymptotic-C_K} can be derived by noting that loops in \eqref{eq:def_Cpz} have a duration of order $1/K$ and hence a typical diameter of order $1/\sqrt{K}$, so that $C_K$ corresponds roughly to the Green function $G_D$ evaluated at points $z,w$ separated by $\eps = 1/\sqrt{K}$. Plugging this in \eqref{eq:log} yields \eqref{eq:lem_asymptotic-C_K}.

\begin{proof}[Proof of Lemma \ref{L:asymptoticsF_C_K}]
\eqref{eq:lem_upper_bound_Fs} and \eqref{eq:lem_f_asymptotic} are direct consequences of the asymptotic behaviour of ${}_1F_1(\theta,1,t)$ as $t \to \infty$; see \eqref{eq:Kummer_function_asymptotic}. \eqref{eq:lem_asymptotic-C_K} follows from the definition \eqref{eq:def_Cpz} of $C_K(z)$ and the asymptotic behaviour of the heat kernel of the diagonal which, in our normalisation, is
$
p_D(t,z,z) \sim \frac{1}{4\pi t}
$ as $t \to 0$.
\end{proof}

A crucial consequence of this explicit first moment is a positive martingale which plays a key role in our analysis. 
Recall that by \eqref{Eq L K}, the collections 
$\Lc_D^\theta(K)$ are coupled on the same probability space for
different values of $K$, and the set of $K$-killed loops increases with $K$.
We will denote by $\Fc_K$ the $\sigma$-algebra generated by the $K$-killed loops.

\begin{proposition}\label{prop:martingale}
Define a Borel measure $m_a^K(dz)$ as follows:
\begin{equation}
\label{eq:prop_martingale}
m_a^K(dz) := \frac{1}{a^{1-\theta}} \CR(z,D)^a e^{-aC_K(z)} dz +
\int_0^a \d {\thk} \frac{1}{(a-{\thk})^{1-\theta}} \CR(z,D)^{a-{\thk}} e^{-(a-{\thk}) C_K(z)} \Mc_{\thk}^K(dz).
\end{equation}
Then $(m_a^K(dz), K >0)$ is a $(\Fc_K, K >0)$-martingale (that is, $m_a^K(A) $ is a martingale in that filtration, for any Borel set $A \subset D$).
\end{proposition}

We mention that the measure in \eqref{eq:prop_martingale} is well-defined since we show that the process $a \in (0,2) \mapsto \Mc_a^K$ is measurable relatively to the topology of weak convergence; see Definition \ref{def:measure_mass} and the discussion below.

The proof of Proposition \ref{prop:martingale} will be given in Section \ref{S:martingale} (see also Section \ref{sec:PD} for an alternative proof). Intuitively (and as follows \emph{a posteriori} from our results and L\'evy's martingale convergence theorem), the measure on the left hand side corresponds to the conditional expectation of $\Mc_{a}$ given $\Fc_K$. To understand what the identity \eqref{eq:prop_martingale} expresses, or alternatively to motivate the definition of $m_a^K(dz)$, consider for simplicity of this discussion the special case $\theta = 1/2$ where we may use isomorphism theorems for clarity
(Theorem \ref{Thm Iso Le Jan}).
This conditional expectation should consist of two parts. The first part of the conditional expectation is given by thick points created only by the massive GFF with mass $\sqrt{K}$ (this is the first term in the right hand side). The second part is given by points whose thickness comes from a combination of the massive GFF \emph{and} killed loops. The respective contribution to the overall thickness $a$ of the point is arbitrary in the interval $[0, a]$, resulting in an integral. The variable $\rho \in [0,a]$ of integration corresponds to points which have a thickness of order $\rho$ in the soup of killed loops, and a thickness $a- \rho$ in the massive GFF. This identity is therefore an analogue of Proposition 1.3 in \cite{jegoRW} (see also \eqref{eq:disintegration}). The presence of the factor $1/ (a- \rho)^{1- \theta}$ in front is not straightforward. A posteriori, it may be viewed as describing the ``law'' of this mixture of thicknesses. See Remark \ref{R:notproved} for more discussion on this point.

\medskip We now assume the conclusion of Proposition \ref{prop:martingale} and see how the proof proceeds.
Since $m_a^K(A) \geq 0$ for all Borel set $A$, we deduce that $(m_a^K, K>0)$ converges almost surely for the topology of weak convergence towards a Borel measure $m_a$ (see e.g. Section 6 of \cite{BerestyckiGMC}). We will show that except for a normalising factor, this is the same as $\Mc_a$ in the statement of Theorem \ref{th:convergence_continuum}. To do this, the main step will be to show that when $K \to \infty$, the integral in the right hand side of
\eqref{eq:prop_martingale} concentrates around the value ${\thk} = a$, so  that $m_a^K$ is in fact very close to $\Mc_a^K$ (up to a certain multiplicative constant). This is the content of the following proposition:

\begin{proposition}\label{prop:m_vs_Mc}
For all Borel set $A \subset \C$,
\begin{equation}
\label{eq:prop_m_vs_Mc}
\lim_{K \to \infty} \Expect{\abs{m_a^K(A) - \frac{2^\theta \Gamma(\theta)}{(\log K)^\theta} \Mc_a^K(A)}} = 0.
\end{equation}
\end{proposition}

The convergence of $((\log K)^{-\theta} \Mc_a^K, K > 0)$ follows directly from Propositions \ref{prop:martingale} and \ref{prop:m_vs_Mc}.

\medskip

We now explain how Proposition \ref{prop:m_vs_Mc} is obtained.
The core of the proof, that we encapsulate in the following result, consists in controlling the oscillations of $\Mc_a^K$ with respect to the thickness parameter $a$.

\begin{proposition}\label{prop:oscillations_Mc}
Let $a \in (0,2)$ and $A \Subset D$. Then,
\[
\limsup_{{\thk} \to a} \limsup_{K \to \infty} \sup_f \frac{1}{\norme{f}_\infty (\log K)^\theta}
\Expect{ \abs{ \int_D f(z) \Mc_a^K(dz) - \int_D f(z) \Mc_{\thk}^K(dz) } } =0,
\]
where the supremum runs over all bounded, non-zero, non-negative measurable function $f:D \to [0,\infty)$ with compact support included in $A$.
\end{proposition}

The proof of Proposition \ref{prop:oscillations_Mc} will be given in Sections \ref{Sec 2nd mom} and \ref{sec:beyond_L2}. We now explain how to prove Proposition \ref{prop:m_vs_Mc} assuming Proposition \ref{prop:oscillations_Mc}.

\begin{proof}[Proof of Proposition \ref{prop:m_vs_Mc}, assuming Proposition \ref{prop:oscillations_Mc}.]
Let $A \subset \C$ be a Borel set and for $\delta >0$, define $A_\delta = A \cap \{ z \in D: \d(z,D^c) > \delta \}$.
Proposition \ref{prop:first_moment_killed_loops} shows that
\[
\lim_{\delta \to 0} \limsup_{K \to \infty} \Expect{\abs{m_a^K(A) - m_a^K(A_\delta)}} = \lim_{\delta \to 0} \limsup_{K \to \infty} \frac{1}{(\log K)^\theta} \Expect{\abs{\Mc_a^K(A) - \Mc_a^K(A_\delta)}} = 0.
\]
Therefore, it is sufficient to show that for all $\delta >0$,
\[
\lim_{K \to \infty} \Expect{\abs{m_a^K(A_\delta) - \frac{2^\theta \Gamma(\theta)}{(\log K)^\theta} \Mc_a^K(A_\delta)}} = 0.
\]
In other words, we can assume that $A$ is compactly included in $D$. It is then easy to see that one has the crude lower bound:
\begin{equation}
\label{eq:proof_prop_m_vs_Mc1}
\inf_{z \in A} C_K(z) \geq c \log K.
\end{equation}
(Indeed, if $z \in A_\delta$, then $C_K(z)$ is at least equal to the function $C_K(z)$ associated with a ball of radius $\delta$ around $z$, a quantity which in fact does not depend on $z$ and whose asymptotics is given by Lemma \ref{L:asymptoticsF_C_K}). Let $\eta > 0$ be small. Proposition \ref{prop:first_moment_killed_loops} implies that
\begin{align}
& %\limsup_{K \to \infty}
\Expect{\abs{ m_K^a(A) - \int_A \int_{a-\eta}^a \frac{\d {\thk}}{(a-{\thk})^{1-\theta}} \CR(z,D)^{a-{\thk}} e^{-(a-{\thk}) C_K(z)} \Mc_{\thk}^K(dz) }} \label{eq:comp_m_Mc}\\
& = % \limsup_{K \to \infty}
\frac{1}{a^{1-\theta}} \int_A \CR(z,D)^a e^{-a C_K(z)} \d z
%\nonumber\\
%&  %\limsup_{K \to \infty}
+ \int_A \int_0^{a-\eta} \frac{\d {\thk}}{(a-{\thk})^{1-\theta}} \CR(z,D)^{a-{\thk}} e^{-(a-{\thk}) C_K(z)} \Expect{ \Mc_{\thk}^K(dz) } \nonumber\\
& = o(1) + %\limsup_{K \to \infty}
\int_A \d z \CR(z,D)^a \int_0^{a-\eta} \frac{\d {\thk}}{{\thk} (a-{\thk})^{1-\theta}} \Fs \left( C_K(z) {\thk} \right) e^{-(a-{\thk})C_K(z)} \nonumber
\end{align}
as the first integral clearly converges to 0 when $K \to \infty$ using \eqref{eq:proof_prop_m_vs_Mc1}. Using \eqref{eq:lem_upper_bound_Fs} we can bound the second integral by
\begin{align*}
 & C \int_A \d z \int_0^{a-\eta} \frac{\d {\thk}}{{\thk} (a-{\thk})^{1-\theta}} \max \left( C_K(z) {\thk}, C_K(z)^\theta {\thk}^\theta \right) e^{-(a-{\thk})C_K(z)} \\
 & \le C \sup_{z \in A}\{ e^{- \eta C_K(z)} \max (C_K(z), C_K(z)^\theta) \} |A|\int_0^{a} \frac{\d {\thk}}{{\thk} (a-{\thk})^{1-\theta}} \max (\rho, \rho^\theta).
\end{align*}
The integral is in any case finite since $\theta >0$ and does not depend on $K$. Since $C_K(z) \to \infty$, we deduce that the right hand side above tends to zero. Overall, we see that \eqref{eq:comp_m_Mc} tends to 0 as $K \to \infty$.

Hence
\begin{align*}
& \Expect{\abs{m_a^K(A) - \frac{2^\theta \Gamma(\theta)}{(\log K)^\theta} \Mc_a^K(A)}} \\
& \leq o(1) + \Expect{\abs{ \int_A \int_{a-\eta}^a \frac{\d {\thk}}{(a-{\thk})^{1-\theta}} \CR(z,D)^{a-{\thk}} e^{-(a-{\thk}) C_K(z)} \Mc_{\thk}^K(dz) - \frac{2^\theta \Gamma(\theta)}{(\log K)^\theta} \Mc_a^K(A)}} \\
& \leq o(1) + \int_{a-\eta}^a \frac{\d {\thk}}{(a-{\thk})^{1-\theta}} \Expect{\abs{ \int_A \CR(z,D)^{a-{\thk}} e^{-(a-{\thk}) C_K(z)} ( \Mc_{\thk}^K(dz)  - \Mc_a^K(dz) ) }} \\
& + \int_A \Expect{\Mc_a^K(dz)} \abs{ \int_{a-\eta}^a \frac{\d {\thk}}{(a-{\thk})^{1-\theta}} \CR(z,D)^{a-{\thk}} e^{-(a-{\thk}) C_K(z)} - \frac{2^\theta \Gamma(\theta)}{(\log K)^\theta} }.
\end{align*}
To control the third term of the above sum, we recall that $\Expect{\Mc_a^K(dz)} \asymp (\log K)^\theta$ (by Proposition \ref{prop:first_moment_killed_loops} and Lemma \ref{L:asymptoticsF_C_K}), and we make a change of variable $ \CR(z,D)^{a-{\thk}} e^{-(a-{\thk}) C_K(z)} = e^{-t}$.  So the third term is bounded by
\begin{align*}
& C \abs{ (\log K)^\theta \int_{a-\eta}^a \frac{\d {\thk}}{(a-{\thk})^{1-\theta}} \CR(z,D)^{a-{\thk}} e^{-(a-{\thk}) C_K(z)} - 2^\theta \Gamma(\theta) } \\
& = \abs{ \left( \frac{\log K}{C_K(z) - \log \CR(z,D)} \right)^\theta \int_0^{\eta (C_K(z) - \log \CR(z,D))} \frac{\d t}{t^{1-\theta}} e^{-t} - 2^\theta \Gamma(\theta) }
\end{align*}
which goes to zero as $K \to \infty$, uniformly in $z \in A$ (see \eqref{eq:lem_asymptotic-C_K} and \eqref{eq:gamma_function}). Therefore, the third term of the sum vanishes.
To bound the second term we use Proposition \ref{prop:oscillations_Mc} where the function $f$ is taken to be $f(z) = \CR(z,D)^{a-{\thk}} e^{-(a-{\thk}) C_K(z)} \indic{z \in A}$ (this depends on $K$, but since the estimate in Proposition \ref{prop:oscillations_Mc} is uniform, this is not a problem). We obtain that it is bounded by:

\begin{align*}
o_{\eta}(1) (\log K)^\theta \int_{a-\eta}^a \frac{\d {\thk}}{(a-{\thk})^{1-\theta}} e^{-c(a-{\thk}) \log K}
\leq C o_{\eta}(1)
\end{align*}
where the term $o_{\eta }(1)$ can be made arbitrarily small by choosing $\eta$ sufficiently close to zero, uniformly in $K$.
To conclude, we have proven that
\[
\limsup_{K \to \infty} \Expect{\abs{m_a^K(A) - \frac{2^\theta \Gamma(\theta)}{(\log K)^\theta} \Mc_a^K(A)}}
\leq C o_{\eta }(1).
\]
Since the above left hand side term does not depend on $\eta$, by letting $\eta \to 0$, we deduce that it vanishes. This finishes the proof.
\end{proof}

The rest of Part One is organised as follows:

\begin{itemize}
\item
Section \ref{sec:def}: Brownian chaos measures were defined for Brownian trajectories killed upon exiting for the first time a given domain. This section explains how to transfer the definition to loops. This specific choice of definition is important for some proofs in subsequent sections.
\item
Section \ref{Sec 1st moment}: We study the first moment of $\Mc_a^K$ and provide a Girsanov-type transform associated to $\Mc_a^K$ (Lemma \ref{lem:Girsanov_K}). In particular, this gives an explicit expression for the first moment of $\Mc_a^K$. The formula obtained is expressed as a complicated sum of convoluted integrals, but we show in Lemma \ref{lem:Fs} that it reduces to a very simple form as stated in Proposition \ref{prop:first_moment_killed_loops} above. Finally, this first moment study culminates in Section \ref{S:martingale} with a proof of the fact that $(m_a^K, K>0)$ is a martingale.
\item
Section \ref{Sec 2nd mom}: We initiate the study of the second moment of $\Mc_a^K$ and give in Lemma \ref{lem:second_moment} an exact expression for the second moment of the (two-point) rooted measure. The exact formula we obtain is arguably lengthy and the goal of Lemma \ref{lem:Hs} is to analyse its asymptotic behaviour. This section concludes the proof of Proposition \ref{prop:oscillations_Mc} in the $L^2$-phase $\{a \in (0,1)\}$.
\item
Section \ref{sec:beyond_L2}: This section aims to go beyond the $L^2$-phase to cover the whole subcritical regime $\{a \in (0,2)\}$. To this end, we introduce a truncation requiring the number of crossings of dyadic annuli to remain below a certain curve. Adding this truncation does change the measure with high probability (Lemma \ref{lem:first_moment_good_event}) and turns the truncated measure bounded in $L^2$ (Lemma \ref{lem:second_moment_bdd}). The truncated measure is then shown to vary smoothly with respect to the thickness parameter (Lemma \ref{lem:second_moment_aa'}).
\item
Section \ref{sec:measurability_etc}: A proof of Theorem \ref{th:PD} is given. As a consequence of our approach, a new proof of Proposition \ref{prop:martingale} is given.
A proof that the limiting measure $\Mc_a$ is independent of the labels underlying the definition of the killing is given (Theorem \ref{th:convergence_continuum}, Point \ref{it:measurability}). We then show that the characterisation of the law of the couple $(\Lc_D^\theta, \Mc_a)$ given in Theorem \ref{th:PD} implies the conformal covariance of this couple (Theorem \ref{th:conformal_covariance_couple}). Finally, the conformal covariance of the measure is shown to imply its almost sure positivity (Theorem \ref{th:convergence_continuum}, Point \ref{it:nondegenerate}).
\end{itemize}

\section{Multiplicative chaos for finitely many loops}\label{sec:def}

\subsection{Definition of \texorpdfstring{$\Mc_a^K$}{M a K}}

Brownian multiplicative chaos measures have been defined for Brownian trajectories confined to a given domain (for instance, killed upon exiting for the first time the domain). The purpose of this section is to explain that we can also define these measures for the loops coming from the Brownian loop soup.
This is not a difficult task, but some proofs (not the results) in the subsequent sections depend on the precise definition that we will take.
There are various ways to proceed. The one below has for instance the advantage that it is closest to the setup of \cite{jegoRW} (Brownian excursions from interior to boundary points), and it immediately gives measurability with respect to the loop soup.

The rough strategy is to cut the loops into two pieces for which we can define a Brownian chaos. We decided to do this by rooting the loops at the point with minimal imaginary part.
We will restrict ourselves to loops with height larger than a given threshold $\eps$ and we first want to describe the law of this collection of loops.
We start by introducing a few notations.

\begin{notation}\label{not:mi}
For any $\wp \in \Lc_D^\theta$, we denote by
\begin{equation}
\label{Eq mi}
\mi(\wp) := \inf \{ \Im(\wp(t)): t \in [0,T(\wp)] \}, \quad
\Mi(\wp) := \sup \{ \Im(\wp(t)): t \in [0,T(\wp)] \} ,
\end{equation}
and
\begin{equation}
\label{Eq h}
h(\wp) := \Mi(\wp) - \mi(\wp)
\end{equation}
the height, or vertical displacement, of $\wp$. We also write
\[
\mi(D) := \inf \{ \Im(z): z \in D \}
\quad \text{and} \quad
\Mi(D) := \sup \{ \Im(z): z \in D \},
\]
and for any real numbers $y < y'$,
\begin{equation}
\label{eq:def_H_S}
\H_y := \{ z \in \C: \Im(z) > y \}
\quad \text{and} \quad
S_{y,y'} := \{ z \in \C: y < \Im(z) < y' \}.
\end{equation}
\end{notation}

Consider now the collection of loops with height larger than some given $\eps>0$:
\begin{equation}
\label{eq:BLS_height}
\Lc_{D,\eps}^\theta := \{ \wp \in \Lc_D^\theta: h(\wp) > \eps \}.
\end{equation}
In Lemma \ref{lem:loops_height} below, we describe the law of $\Lc_{D,\eps}^\theta$. To do that, for each $\wp \in \Lc_{D,\eps}^\theta$, we will root $\wp$ at the unique point $z_\bot$ where the imaginary part of $\wp$ is at its minimum.  We will then stop the loop when its height becomes for the first time larger than $\eps$:
\[
\tau_\eps(\wp) := \inf \{ t \in [0,T(\wp)]: \Im (\wp(t)) \geq \mi(\wp) + \eps \}.
\]
The loop will therefore be decomposed into two parts:
\begin{equation}
\label{eq:app_wp2}
\wp_{\eps,1} := (\wp(t))_{0 \leq t \leq \tau_\eps}
\quad \text{and} \quad
\wp_{\eps,2} := (\wp(t))_{\tau_\eps \leq t \leq T(\wp)}.
\end{equation}
By construction, $\wp_{\eps,1}$ is an excursion from $z_\bot$ to $z_\eps := \wp(\tau_\eps)$ in the domain $D \cap S_{\mi(\wp), \mi(\wp)+\eps}$ and $\wp_{\eps,2}$ is an excursion from the internal point $z_\eps$ to the boundary point $z_\bot$ in the domain $D \cap \H_{\mi(\wp)}$. See Figure \ref{fig1}.

\begin{figure}[ht]
   \centering
    \def\svgwidth{0.5\columnwidth}
   %% Creator: Inkscape 1.1 (c4e8f9e, 2021-05-24), www.inkscape.org
%% Accompanies image file 'drawing.pdf' (pdf, eps, ps)
\begingroup%
  \makeatletter%
  \providecommand\color[2][]{%
    \errmessage{(Inkscape) Color is used for the text in Inkscape, but the package 'color.sty' is not loaded}%
    \renewcommand\color[2][]{}%
  }%
  \providecommand\transparent[1]{%
    \errmessage{(Inkscape) Transparency is used (non-zero) for the text in Inkscape, but the package 'transparent.sty' is not loaded}%
    \renewcommand\transparent[1]{}%
  }%
  \providecommand\rotatebox[2]{#2}%
  \newcommand*\fsize{\dimexpr\f@size pt\relax}%
  \newcommand*\lineheight[1]{\fontsize{\fsize}{#1\fsize}\selectfont}%
  \ifx\svgwidth\undefined%
    \setlength{\unitlength}{361.26992796bp}%
    \ifx\svgscale\undefined%
      \relax%
    \else%
      \setlength{\unitlength}{\unitlength * \real{\svgscale}}%
    \fi%
  \else%
    \setlength{\unitlength}{\svgwidth}%
  \fi%
  \global\let\svgwidth\undefined%
  \global\let\svgscale\undefined%
  \makeatother%
  \begin{picture}(1,0.72756756)%
    \lineheight{1}%
    \setlength\tabcolsep{0pt}%
    \put(0,0){\includegraphics[width=\unitlength,page=1]{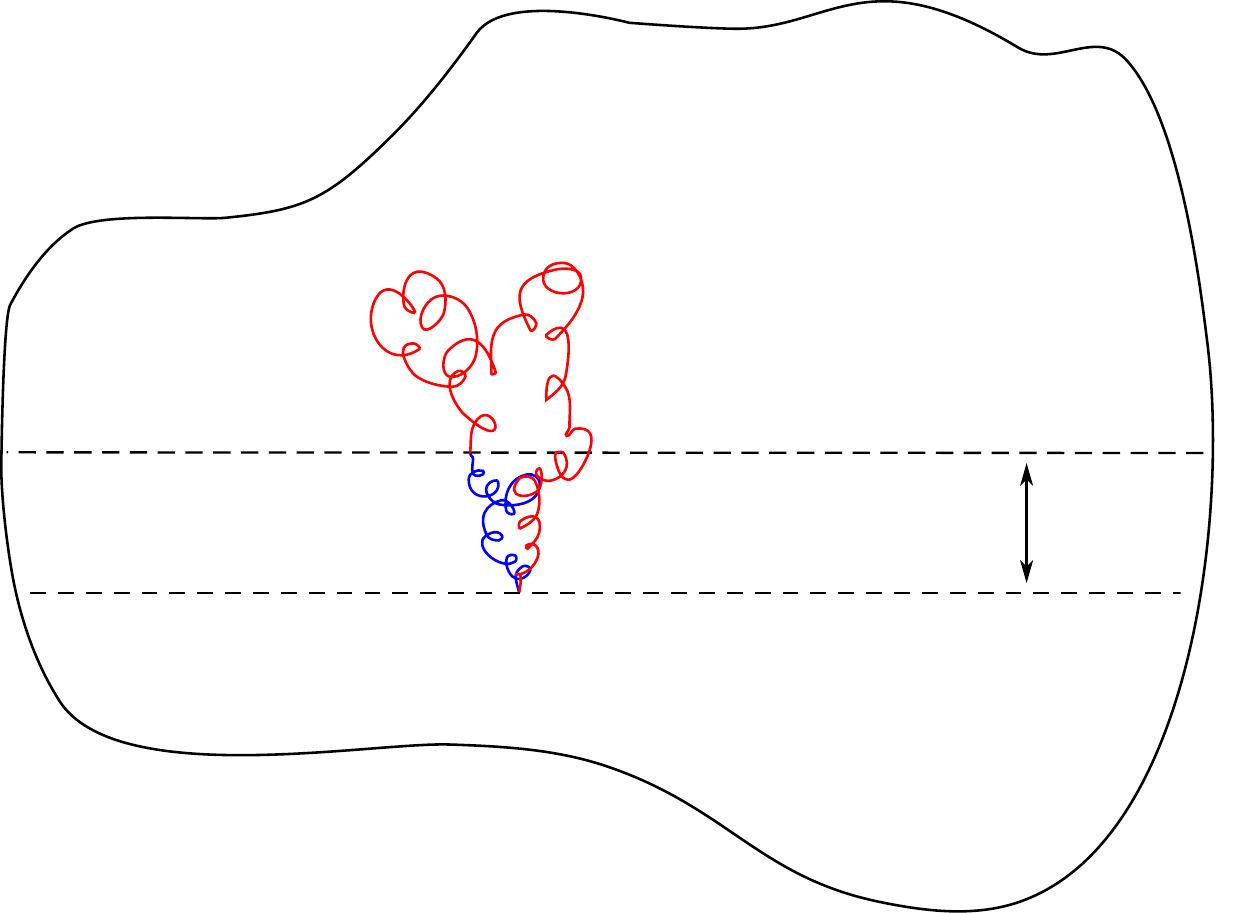}}%
    \put(0.83067543,0.29865239){\makebox(0,0)[lt]{\lineheight{1.25}\smash{\begin{tabular}[t]{l}$\eps$\end{tabular}}}}%
    \put(0.21030001,0.60589449){\makebox(0,0)[lt]{\lineheight{1.25}\smash{\begin{tabular}[t]{l}$D$\end{tabular}}}}%
    \put(0.47772485,0.44765128){\color[rgb]{1,0,0}\makebox(0,0)[lt]{\lineheight{1.25}\smash{\begin{tabular}[t]{l}$\wp_{\eps,2}$\end{tabular}}}}%
    \put(0.30063193,0.30825222){\color[rgb]{0,0,1}\makebox(0,0)[lt]{\lineheight{1.25}\smash{\begin{tabular}[t]{l}$\wp_{\eps,1}$\end{tabular}}}}%
    \put(0.39337815,0.21222293){\makebox(0,0)[lt]{\lineheight{1.25}\smash{\begin{tabular}[t]{l}$z_\bot$\end{tabular}}}}%
    \put(0.32493353,0.38054743){\makebox(0,0)[lt]{\lineheight{1.25}\smash{\begin{tabular}[t]{l}$z_\eps$\end{tabular}}}}%
    \put(0,0){\includegraphics[width=\unitlength,page=2]{drawing.pdf}}%
  \end{picture}%
\endgroup%

   \caption{Rooting a loop at the point with minimal imaginary part.}\label{fig1}
\end{figure}

We can now describe the law of $\Lc_{D,\eps}^\theta$.

\begin{lemma}\label{lem:loops_height}
$\# \Lc_{D,\eps}^\theta$ is a Poisson random variable with mean given by $\theta \loopmeasure_D(h(\wp)>\eps)$, with
\begin{displaymath}
\loopmeasure_D(h(\wp)>\eps)
=
\int_{\mi(D)}^{\Mi(D)-\eps} \d m
\int_{D \cap (\R + im)} \d z_1
\int_{D \cap (\R+i(m+\eps))} \d z_2
H_{D \cap S_{m,m+\eps}}(z_1,z_2) H_{D \cap \H_m}(z_2,z_1)
,
\end{displaymath}
where $H_{D \cap S_{m,m+\eps}}(z_1,z_2)$
is a boundary Poisson kernel \eqref{Eq bPk}
in $D \cap S_{m,m+\eps}(z_1,z_2)$
and $H_{D \cap \H_m}(z_2,z_1)$
is a Poisson kernel \eqref{Eq Pk}
in $D \cap \H_m$.
Conditionally on $\{ \# \Lc_{D,\eps}^\theta = n \} $,  $\Lc_{D,\eps}^\theta$ is composed of $n$ i.i.d. loops with common law given by
\begin{equation}
\label{eq:law_loop_height}
\indic{h(\cdot)>\eps} \loopmeasure_D(\cdot) / \loopmeasure_D( \{ \wp: h(\wp) > \eps \}).
\end{equation}
Moreover, if $\wp$ is distributed according to the law \eqref{eq:law_loop_height} above, then the law of $(z_\bot, z_\eps, \wp_{1,\eps}, \wp_{2,\eps})$ is described as follows:
\begin{enumerate}
\item
Conditionally on $(z_\bot,z_\eps)$ and denoting $m = \Im(z_\bot)$,  $\wp_{1,\eps}$ and $\wp_{2,\eps}$ are two independent Brownian trajectories distributed according to
\[
\mu_{D \cap S_{m,m+\eps}}^{z_\bot,z_\eps}
/
H_{D \cap S_{m,m+\eps}}(z_\bot,z_\eps)
\quad \text{and} \quad
\mu_{D \cap H_m}^{z_\eps,z_\bot}
/
H_{D \cap H_m}(z_\eps,z_\bot )
\]
respectively.
\item
The joint law of $(z_\bot,z_\eps)$ is given by: for all bounded measurable function $F:\C^2 \to \R$,
\begin{align}
\label{eq:continuum_law_zbot_zeps}
\Expect{F(z_\bot,z_\eps)} = \frac{1}{Z} &
\int_{\mi(D)}^{\Mi(D)-\eps} \d m \int_{D \cap (\R + im)} \d z_1 \int_{D \cap (\R+i(m+\eps))} \d z_2 \\
& H_{D \cap S_{m,m+\eps}}(z_1,z_2) H_{D \cap \H_m}(z_2,z_1) F(z_1,z_2)
\nonumber
\end{align}
\end{enumerate}
\end{lemma}

\begin{proof}
Since $D$ is bounded, we may assume without loss of generality that 
$D$ is contained in the upper half-plane
$\H = \H_{0}$. Next, we consider the measure on loops on 
$\H$, $\mu_{\H}^{\rm loop}$, and root the loops at their lowest imaginary part. According to \cite[Proposition 7]{Lawler04},
$\mu_{\H}^{\rm loop}$ then disintegrates as
\begin{displaymath}
\int_{0}^{+\infty} \d m
\int_{\R + im} \d z_{1}~
\mu_{\H_{m}}^{z_{1},z_{1}},
\end{displaymath}
where $\mu_{\H_{m}}^{z_{1},z_{1}}$ is given by
\eqref{Eq boundary exc}.
Further, a path $\gamma$ under a measure $\mu_{\H_{m}}^{z_{1},z_{1}}$
with $h(\wp)>\varepsilon$
can be decomposed as
\begin{displaymath}
\int \indic{h(\wp)>\varepsilon} F(\wp)
\mu_{\H_{m}}^{z_{1},z_{1}}(d\wp)
=
\int_{\R + i(m+\varepsilon)}
\d z_{2}
\iint
F(\wp_{1}\wedge \wp_{2})
\mu_{S_{m,m+\eps}}^{z_{1},z_{2}}(d\wp_{1})
\mu_{\H_{m}}^{z_{2},z_{1}}(d\wp_{2}).
\end{displaymath}
This is similar to decompositions appearing in
\cite[Section 5.2]{LawlerConformallyInvariantProcesses}.
So one gets the lemma in the case of the upper half-plane $\H$.
The case of a domain $D\subset \H$
can be obtained by using the restriction property
\eqref{Eq restriction loops}.
Indeed, given
$z_{1}\in D\cap (\R + im)$
and $z_{2}\in D\cap (\R + i(m+\varepsilon))$,
we have that
\begin{displaymath}
\mu_{D\cap S_{m,m+\eps}}^{z_{1},z_{2}}(d\wp_{1})
=
\indic{\wp_{1} \text{ stays in } D}
\mu_{S_{m,m+\eps}}^{z_{1},z_{2}}(d\wp_{1})
,\qquad
\mu_{D\cap\H_{m}}^{z_{2},z_{1}}(d\wp_{2})
=
\indic{\wp_{2} \text{ stays in } D}
\mu_{\H_{m}}^{z_{2},z_{1}}(d\wp_{2}).
\qedhere
\end{displaymath}
\end{proof}

From this lemma, it becomes clear that we can define Brownian chaos associated to the loops as soon as we are able to define it for independent Brownian trajectories with random domains, starting points and ending points. We explain this carefully in the next section (Section \ref{S:BMC}); see especially Lemma \ref{lem:measurability_measures}.
We can now give a precise definition of Brownian multiplicative chaos associated to the loops in $\Lc_D^\theta(K)$. We start by fixing $\eps>0$. For any $\wp \in \Lc_{D,\eps}^\theta$, we denote by $\wp_{2,\eps}$ the second part of the trajectory defined in \eqref{eq:app_wp2}. Thanks to Lemmas \ref{lem:loops_height} and \ref{lem:measurability_measures}, we can define
\begin{equation}\label{eq:measure_mass_height}
\Mc_a^{K,\eps} := \sum_{n \geq 1} \frac{1}{n!} \sum_{\substack{\wp^{(1)}, \dots, \wp^{(n)} \in \Lc_{D,\eps}^\theta \cap \Lc_D^\theta(K) \\ \forall i \neq j, \wp^{(i)} \neq \wp^{(j)}}} \Mc_a^{ \wp_{2,\eps}^{(1)} \cap \dots \cap \wp_{2,\eps}^{(n)} }.
\end{equation}

\begin{definition}\label{def:measure_mass}
$\Mc_a^K$ is defined as being the nondecreasing limit of $\Mc_a^{K,\eps}$ as $\eps \to 0$.
\end{definition}

This definition gives a precise meaning to the quantities on the right hand side of \eqref{eq:def_measure_killed_loops} that we only defined informally in the introduction.
Thanks to this definition, we do not only get $\Mc_a^K$ for a fixed $a$, but we actually obtain  a measurable process, viewed as a function of $a \in (0,2)$,
relatively to the topology of weak convergence.
Indeed, Lemma \ref{lem:measurability_measures} gives not only the measurability of the measures with respect to the starting points, ending points and domains, but also with respect to the thickness level $a$.
This justifies for instance that the martingale $m_a^K$, defined in Proposition \ref{prop:martingale}, is well defined.

Finally, the same procedure can be applied to define other types of Brownian multiplicative chaos measures associated to loops, such as
$\bigcap_{i=1}^n \Mc_{a_i}^{\wp_i}$ (recall Section \ref{sec:preliminaries_BMC}).

\subsection{Measurability of Brownian multiplicative chaos}\label{S:BMC}

This section deals with some technicalities concerning the measurability of the Brownian chaos measures w.r.t. the starting points, ending points, domains and thickness levels.

Denote by $\mathscr{M}$ the set of Borel measures on $\C$ equipped with the topology of weak convergence,  and by $\mathscr{C}$ the set of continuous trajectories in the plane with finite duration equipped with the topology induced by $d_{\mathrm{paths}}$ \eqref{Eq dist paths}.
Recall the definition \eqref{Eq mi} of $\mi(D)$ and $\Mi(D)$ and the definition \eqref{eq:def_H_S} of the half plane $\H_m$.
Denote by $\mathscr{S}$ the set
\[
\mathscr{S} := \left\{ (m,x_0,z) \in (\mi(D), \Mi(D)) \times D \times D: x_0 \in \H_m, \Im(z) = m \right\}
\]
equipped with its Borel $\sigma$-algebra. 
Let $n \geq 1$. We consider a stochastic process
\[
(m_i,x_i,z_i)_{i=1 \dots n} \in \mathscr{S}^n \mapsto (\wp_{D \cap \H_{m_i}}^{x_i,z_i})_{i=1 \dots n} \in \mathscr{C}
\]
such that for all $(m_i,x_i,z_i)_{i=1 \dots n} \in \mathscr{S}^n$, $\wp_{D \cap \H_{m_i}}^{x_i,z_i}$, $i=1 \dots n$, are independent Brownian trajectories from $x_i$ to $z_i$ in the domain $D \cap \H_{m_i}$, i.e. distributed according to $\mu_{D \cap \H_{m_i}}^{x_i,z_i} / H_{D \cap \H_{m_i}}(x_i,z_i)$ \eqref{Eq mu D z w boundary}. We consider a measurable version of this stochastic process, that is a version such that
\[
(\omega, (m_i,x_i,z_i)_{i=1 \dots n}) \in \Omega \times \mathscr{S}^n \mapsto (\wp_{D \cap \H_{m_i}}^{x_i,z_i})_{i=1 \dots n}(\omega) \in \mathscr{C}
\]
is measurable ($\Omega$ stands here for the underlying probability space).
In the next result, we consider the multiplicative chaos measures associated to the above Brownian paths. The subset $I \subset \{1 \dots n\}$ encodes the trajectories involved and we will need to consider all these measures jointly in $I$.

\begin{lemma}\label{lem:measurability_measures}
The process
\begin{equation}
\label{eq:process_BMC}
(a,(m_i,x_i,z_i)_{i=1 \dots n}) \in (0,2) \times \mathscr{S}^n \mapsto \left( \Mc_a^{ \cap_{i \in I} \wp_{D \cap \H_{m_i}}^{x_i,z_i}} \right)_{I \subset \{1 \dots n\}} \in \prod_{I \subset \{1 \dots n\} } \mathscr{M}^{\# I}
\end{equation}
is measurable.
\end{lemma}

Let us comment that the process \eqref{eq:process_BMC} should actually possess a continuous modification. However, showing such a regularity is actually far from being simple (see Proposition 1.2 and Remark 1.1 of \cite{jegoCritical}). Fortunately, this will not be needed in this article.

\begin{proof}
The Brownian chaos measures are defined as the pointwise limit of measures that are clearly measurable w.r.t. the path (see Section \ref{sec:preliminaries_BMC} or \cite[Proposition 1.1]{jegoRW}). Therefore, the process \eqref{eq:process_BMC} is measurable as a pointwise limit of measurable processes.
\end{proof}

We finish this section by showing that the measure $\Mc_a^K$ on thick points of the massive loop soup $\Lc_D(K)$ is measurable w.r.t. $\sigma (\scalar{ \Lc_D^\theta })$ \eqref{eq:sigma_algebra_admissible}, a $\sigma$-algebra smaller than the one generated by $\Lc_D^\theta$.

\begin{lemma}\label{lem:measurability_intermediate}
The measure $\Mc_a^K$ is measurable w.r.t. $\sigma( \scalar{ \Lc_D^\theta(K) } )$.
\end{lemma}

\begin{proof}
For all $\eps>0$, $n \geq 1$ and pairwise distinct loops $\wp^{(1)}, \dots, \wp^{(n)} \in \Lc_{D,\eps}^\theta \cap \Lc_D^\theta(K)$, the measure
$\Mc_a^{ \wp_{2,\eps}^{(1)} \cap \dots \cap \wp_{2,\eps}^{(n)} }$
is a measurable function of the occupation measures of $\wp_{2,\eps}^{(i)}$, $i=1 \dots n$. This is a consequence of \cite[Proposition 1.1]{jegoRW}. Therefore, for all $\eps>0$, $\Mc_a^{K,\eps}$ is measurable w.r.t. the $\sigma$-algebra $\Fc_\eps$ generated by the occupation measure of $\wp_{2,\eps}, \wp \in \Lc_D^\theta(K)$. We conclude by noticing that $\cap_{\eps >0} \sigma( \Fc_\delta, \delta \in (0,\eps))$ is included in the $\sigma$-algebra generated by the occupation measure of $\wp, \wp \in \Lc_D^\theta(K)$. This proves Lemma \ref{lem:measurability_intermediate} since the occupation measure of a loop $\wp$ is a function of its equivalence class $\scalar{\wp}$.
\end{proof} 

\section{First moment computations and rooted measure}
\label{Sec 1st moment}

The goal of this section will be to give a proof of Proposition \ref{prop:first_moment_killed_loops} and Proposition \ref{prop:martingale}. We will also state and prove in Lemma \ref{lem:Girsanov_K} a generalisation of Proposition \ref{prop:first_moment_killed_loops}, which describes the law of the
loop soup after reweighting by our measure
$\Mc_{a}^K(dz)$ \eqref{eq:def_measure_killed_loops}.

\subsection{Preliminaries}\label{subsec:preliminaries_first}

We will consider a finite number of Brownian-like
trajectories $\wp_{1},\dots,\wp_{n}$
and consider their distribution seen from a typical thick point
$z$ generated by the interaction of the $n$ trajectories.

Recall Definition \ref{def:admissible} where admissible functions are defined. We also recall that
$\Xi^z_{a}$ denotes the loop rooted at $z$ obtained by gluing a Poisson point process of Brownian excursions from $z$ to $z$ with intensity measure $2\pi a \mu_{D}^{z,z}$ \eqref{Eq mu D z w}. The goal of this section is to prove:

\begin{lemma}\label{lem:first_moment_loopmeasure}
For any $n \geq 1$ and any nonnegative measurable function $F$ which is admissible,
\begin{multline}
\label{Eq 1st moment loopmeasure}
\int \loopmeasure_D(d \wp_1) \dots \loopmeasure_D(d \wp_n) F(z,\wp_1, \dots, \wp_n) \Mc_a^{\wp_1 \cap \dots \cap \wp_n}(dz) \\
=
\CR(z,D)^a
\int_{ \mathsf{a} \in E(a,n) } \frac{\d \mathsf{a}}{a_1 \dots a_n} \Expect{ F(z,\Xi_{a_1}^z, \dots, \Xi_{a_n}^z) } \d z,
\end{multline}
where $(\Xi_{a_i}^z)_{1\leq i\leq n}$ are independent.
\end{lemma}

In particular, note that when $n =1$ the expected mass of the Brownian chaos generated by a single loop coming from the Brownian loop soup is finite; however this becomes infinite as soon as $n \ge 2$.

\medskip

Before starting the proof of this lemma, we point out that the emergence of the process $\Xi_a^z$ can be guessed (at least in the case $\theta = 1/2$) thanks to isomorphisms theorems (from \cite[Proposition 3.9]{ALS2}, but see also Corollary \ref{C:iso}) in which the Gaussian free field has nonzero boundary conditions.

We also comment on the method of proof. A natural approach to this lemma would be to exploit the identity \eqref{Eq loop to excursion} which relates the loop measure $\loopmeasure_{D}$ in terms of excursion measures $\mu_D^{z,z}$, and then to approximate these excursion measures $\mu_D^{z,z}$ by the more well-behaved $\mu_D^{z,w}$, then letting $w \to z$. Indeed, Girsanov-type transforms of chaos measures associated to trajectories sampled according to $\mu_D^{z,w}/G_D(z,w)$ have been obtained in \cite{AidekonHuShi2018}, and would lead (formally) relatively quickly and painlessly to formulae such as \eqref{Eq 1st moment loopmeasure}. 

Unfortunately this appealing approach suffers from a subtle but serious technical drawback, which is that this does not tie in well with our chosen definition for $\Mc_a^{\wp_1 \cap \dots \cap \wp_n}$ in Section \ref{sec:def}. The issue is that it is not obvious
that the chaos measures associated to excursions to soups of excursions sampled from $\mu_D^{z,w}$ converge to the chaos measure 
$\Mc_a^{\wp_1 \cap \dots \cap \wp_n}$ defined in Section \ref{sec:def}. Even if such a convergence could be proved (so that one might take this as the definition of $\Mc_a^{\wp_1 \cap \dots \cap \wp_n}$) it would not be clear that the limit would be measurable with respect to the collection of loops $\wp_1, \ldots, \wp_n$. Unfortunately this measurability is a crucial feature, and so a different route must be taken. The approach we use in Section \ref{sec:def} does not suffer from this problem: indeed, although the idea is here again to reduce the loops to excursions, these excursions are measurably defined from $\wp_1, \ldots, \wp_n$.

The proof of Lemma \ref{lem:first_moment_loopmeasure} below may therefore at first sight look a little unnatural and somewhat mysterious: the idea is to start from the answer (i.e., from the right-hand side of \eqref{Eq 1st moment loopmeasure}), write down the explicit law of the decomposition of each loop in $\Xi_a^z$ into excursions according to their point with lowest imaginary part (this is the content of Lemma \ref{lem:ppp_mininum}), and check that this agrees after simplifications with the left hand side of \eqref{Eq 1st moment loopmeasure}.

% and it can be checked that this approach quickly leads to neat formulas. However, a major drawback of this approach is that it uses an approximation of the loop soup $\Lc_D^\theta$ which is not measurable w.r.t. $\Lc_D^\theta$. Showing that the chaos measures associated to such an approximation converge towards our definition of multiplicative chaos associated to the loop soup (Section \ref{sec:def}) requires some hard work.
%Instead of going through this technical point, we use a different approach, closer to our definition, by rooting the loops at the point with minimal imaginary part.

%We start the proof of Lemma \ref{lem:first_moment_loopmeasure} by first stating an intermediate result that we will need. In the following lemma, we describe the law of the loop $\Xi_a^z$ seen from its point with minimal imaginary part. Recall the definitions of $\mi(D)$ and $\H_m$; see Notation \ref{not:mi}.

\begin{lemma}\label{lem:ppp_mininum}
Let $z \in D$, $a >0$ and $F$ be a nonnegative measurable function which is admissible. Then, $\Expect{ F(\Xi_a^{z,D}) }$ is equal to
\begin{align*}
2 \pi a \int_{\mi(D)}^{\Im(z)} \d m \int_{(\R + i m) \cap D} \d z_\bot \frac{\CR(z, D \cap \H_m)^a}{\CR(z,D)^a} H_{D \cap \H_m}(z,z_\bot)^2
\Expect{ F(\wp_{D \cap \H_m}^{z, z_\bot} \wedge \wp_{D \cap \H_m}^{z_\bot, z} \wedge \Xi_a^{z, D \cap \H_m} )}
\end{align*}
where the loops $\wp_{D \cap \H_m}^{z, z_\bot}$, $\wp_{D \cap \H_m}^{z_\bot, z}$ and $\Xi_a^{z, D \cap \H_m}$ are independent and distributed as in Notations \ref{not:xi_a^z} and \ref{not:paths}.
\end{lemma}

In words, this lemma states that the point $z_\bot$ of $\Xi_a^{z,D}$ with minimal imaginary part has a density with respect to Lebesgue measure given by the above expression. Moreover, the law of $\Xi_a^{z,D}$ conditionally on $z_\bot \in \R + im$ is given by the concatenation of two independents paths: the original path $\Xi_a^z$ in the smaller domain $D \cap \H_m$ and a loop $\wp_{D \cap \H_m}^{z, z_\bot} \wedge \wp_{D \cap \H_m}^{z_\bot, z}$ in $D \cap \overline{\H}_m$ joining $z$ and $z_\bot$. We point out that it is not immediately obvious that the right hand side defines a probability law (i.e., is equal to 1 when $F \equiv 1$) but this can be seen directly using variational considerations on the conformal radius of $z$ in $D \cap \H_m$ as $m$ varies.

\begin{proof}[Proof of Lemma \ref{lem:ppp_mininum}]
By density-type arguments, we can assume that $F$ is continuous (recall that the topology on the space of continuous paths is the one associated to the distance $d_{\rm paths}$ \eqref{Eq dist paths}).

We first observe that it is enough to prove Lemma \ref{lem:ppp_mininum} in the case of the upper half plane $\H$. Indeed, let us assume that the result holds in that case and let $D$ be a bounded simply connected domain. By translating $D$ if necessary, we can assume that $D$ is contained in $\H$. It is an easy computation to show the result for $D$ from the result for $\H$ as soon as we know the following two restriction properties:
\begin{equation}
\label{eq:ppp3}
\Expect{ F(\Xi_a^{z,\H}) \indic{\Xi_a^{z,\H} \subset D} }
= \frac{\CR(z,D)^a}{ \CR(z,\H)^a }
\Expect{ F(\Xi_a^{z,D}) }
\end{equation}
and for any $m > \mi(D)$ and $z_\bot \in (\R + im) \cap D$,
\begin{equation}
\label{eq:ppp4}
\Expect{ F(\wp_{\H_m}^{z, z_\bot} \wedge \wp_{\H_m}^{z_\bot, z}) \indic{\wp_{\H_m}^{z, z_\bot} \wedge \wp_{\H_m}^{z_\bot, z} \subset D} }
= \frac{H_{D \cap \H_m}(z,z_\bot)^2}{H_{\H_m}(z,z_\bot)^2} \Expect{ F(\wp_{D \cap \H_m}^{z, z_\bot} \wedge \wp_{D \cap \H_m}^{z_\bot, z}) }.
\end{equation}
\eqref{eq:ppp4} is a mere reformulation of the restriction property \eqref{eq:measure_path_restriction} on measures.
To conclude the transfer of the result to general domains, let us prove \eqref{eq:ppp3}.
%It follows from the fact that
%\begin{equation}
%\label{eq:ppp5}
%\Prob{ \Xi_{a_i}^{z,\H} \subset D} = \frac{\CR(z,D)^a}{ \CR(z,\H)^a },
%\end{equation}
%and the fact that the law of the loop $\Xi_{a_i}^{z,\H}$ conditioned on the event that it stays inside $D$ is the same as the law of $\Xi_{a_i}^{z,D}$.
%To prove \eqref{eq:ppp5}, notice that the number of excursions in $\Xi_{a_i}^{z,\H}$ which exit the domain $D$ is a Poisson random variable with mean
%\[
%2\pi a_i \mu_\H^{z,z}(\wp \notin D) = 2\pi a_i \lim_{w \to z} \mu_\H^{z,w}(\wp \notin D)
%= 2\pi a_i \lim_{w \to z} G_\H(z,w) ( 1 - \Prob{ \wp_\H^{z,w} \subset D } ).
%\]
%$\wp_\H^{z,w}$ is distributed according to \eqref{eq:proof_not} and the probability that the trajectory never leaves $D$ is equal to $G_D(z,w) / G_\H(z,w)$. Therefore
%\[
%2\pi a_i \mu_\H^{z,z}(\wp \notin D) = 2 \pi a_i \lim_{w \to z} (G_\H(z,w) - G_D(z,w)) = a_i \log (\CR(z,\H) / \CR(z,D)).
%\]
%Therefore, the probability $\Prob{ \Xi_{a_i}^{z,\H} \subset D }$ is equal to the probability that a Poisson variable with mean $a_i \log (\CR(z,\H) / \CR(z,D))$ is equal to zero which gives \eqref{eq:ppp5}.
It turns out that it is also a consequence of \eqref{eq:measure_path_restriction}. Indeed, by continuity of $F$,
\begin{align*}
\Expect{ F(\Xi_a^{z,\H}) \indic{\Xi_a^{z,\H} \subset D} }
= \lim_{w \to z} e^{-2\pi a G_\H(z,w)} \sum_{n \geq 0} \frac{(2\pi G_\H(z,w))^n}{n!}
\Expect{ F( \wp^{z,w}_{\H,1} \wedge \dots \wedge \wp^{z,w}_{\H,n} ) \indic{\forall i=1 \dots n, \wp^{z,w}_{\H,i} \subset D } }
\end{align*}
where $\wp^{z,w}_{\H,i}, i=1 \dots n$, are i.i.d. and distributed according to \eqref{eq:proof_not}. By the restriction property \eqref{eq:measure_path_restriction}, we further have
\begin{align*}
\Expect{ F(\Xi_a^{z,\H}) \indic{\Xi_a^{z,\H} \subset D} }
& = \lim_{w \to z} e^{-2\pi a G_\H(z,w)} \sum_{n \geq 0} \frac{(2\pi G_D(z,w))^n}{n!}
\Expect{ F( \wp^{z,w}_{D,1} \wedge \dots \wedge \wp^{z,w}_{D,n} ) } \\
& = \Big( \lim_{w \to z} e^{-2\pi a (G_\H(z,w) - G_D(z,w))} \Big) \Expect{ F(\Xi_a^{z,D}) }
= \frac{\CR(z,D)^a}{ \CR(z,\H)^a }
\Expect{ F(\Xi_a^{z,D}) }
\end{align*}
This shows \eqref{eq:ppp3}.

The rest of the proof is dedicated to showing Lemma \ref{lem:ppp_mininum} in the case of the upper half plane $\H$. By continuity of $F$, we have
\begin{equation}\label{eq:ppp2}
\Expect{ F(\Xi_a^{z,\H}) }
= \lim_{w \to z} e^{-2\pi a G_\H(z,w)} \sum_{n \geq 1} \frac{(2\pi a G_\H(z,w))^n}{n!} \Expect{ F( \wp_{\H,1}^{z,w} \wedge \dots \wedge \wp_{\H,n}^{z,w} ) }.
\end{equation}
By symmetry,
\begin{equation}\label{eq:ppp1}
\Expect{ F( \wp_{\H,1}^{z,w} \wedge \dots \wedge \wp_{\H,n}^{z,w} ) }
= \Expect{ F( \wp_{\H,1}^{z,w} \wedge \dots \wedge \wp_{\H,n}^{z,w} ) \vert \forall i =1 \dots n-1, \mi(\wp_{\H,n}^{z,w} ) < \mi(\wp_{\H,i}^{z,w})  }.
\end{equation}
To make the $n$ trajectories independent, we will condition further on $\min_{i = 1 \dots n} \mi(\wp_{\H,i}^{z,w})$. Let us first compute its distribution. For all $m \in (0,\Im(z))$, we have
\[
\Prob{ \min_{i=1 \dots n} \mi(\wp_{\H,i}^{z,w}) > m } = \Prob{ \wp_\H^{z,w} \subset \H_m }^n = G_{\H_m}(z,w)^n G_{\H}(z,w)^{-n}.
\]
The Green function in the upper half plane is explicit and is equal to
\[
G_\H (z,w) = \frac{1}{2\pi} \log \frac{|z - \bar{w}|}{|z - w|},
\quad
G_{\H_m} (z,w) = \frac{1}{2\pi} \log \frac{|z - \bar{w} - 2im|}{|z - w|}.
\]
By differentiating w.r.t. $m$, we deduce that the density of $\min_{i=1 \dots n} \mi(\wp_{\H,i}^{z,w})$ is given by
\[
\frac{n}{\pi} \frac{G_{\H_m}(z,w)^{n-1}}{G_\H(z,w)^n} \frac{\Im(z-\bar{w}) - 2m}{|z-\bar{w} - 2im|^2} \d m.
\]
We now want to expand \eqref{eq:ppp1}. Conditioned on $\mi(\wp_{\H,n}^{z,w}) = \min_{i=1 \dots n} \mi(\wp_{\H,i}^{z,w}) = m$, the $n$ trajectories are independent with the following distributions: the first $n-1$ trajectories are trajectories from $z$ to $w$ in $\H_m$ with law $\mu_{\H_m}^{z,w} / G_{\H_m}(z,w)$ and the last trajectory $\wp_{\rm min}$ which reaches the lowest level is distributed as follows:
\[
\Expect{f(\wp_{\rm \min})} =
\frac{1}{Z_m(z,w)} \int_{ \R + im} \d z_\bot H_{\H_m}(z,z_\bot) H_{\H_m}(w,z_\bot) \Expect{ f( \wp_{\H_m}^{z,z_\bot} \wedge \wp_{\H_m}^{z_\bot, w} )}.
\]
In the above equation, $Z_m(z,w)$ is the normalising constant
\[
Z_m(z,w) = \int_{ \R + im} H_{\H_m}(z,z_\bot) H_{\H_m}(w,z_\bot)  \d z_\bot.
\]
Overall, this shows that
\begin{align*}
& \Expect{ F( \wp_{\H,1}^{z,w} \wedge \dots \wedge \wp_{\H,n}^{z,w} ) }
= \frac{n}{\pi} \frac{G_{\H_m}(z,w)^{n-1}}{G_\H(z,w)^n} \int_0^{\Im(z)} \d m \frac{\Im(z-\bar{w}) - 2m}{|z-\bar{w} - 2im|^2} \frac{1}{Z_m(z,w)} \\
& \times \int_{ \R + im} \d z_\bot H_{\H_m}(z,z_\bot) H_{\H_m}(z,w) \Expect{ F( \wp_{\H_m,1}^{z,w} \wedge \dots \wedge \wp_{\H_m,n-1}^{z,w} \wedge \wp_{\H_m}^{z,z_\bot} \wedge \wp_{\H_m}^{z_\bot, w} )}.
\end{align*}
Plugging this back in \eqref{eq:ppp2}, we have
\begin{align*}
& \Expect{ F(\Xi_a^{z,\H}) }
= 2a \lim_{w \to z} e^{-2\pi a G_\H(z,w)} \sum_{n \geq 1} \frac{(2\pi a G_{\H_m}(z,w))^{n-1}}{(n-1)!}
\int_0^{\Im(z)} \d m \frac{\Im(z-\bar{w}) - 2m}{|z-\bar{w} - 2im|^2} \frac{1}{Z_m(z,w)} \\
& \times \int_{ \R + im} \d z_\bot H_{\H_m}(z,z_\bot) H_{\H_m}(z,w) \Expect{ F( \wp_{\H_m,1}^{z,w} \wedge \dots \wedge \wp_{\H_m,n-1}^{z,w} \wedge \wp_{\H_m}^{z,z_\bot} \wedge \wp_{\H_m}^{z_\bot, w} )} \\
& = 2a \lim_{w \to z} e^{-2\pi a (G_\H(z,w) - G_{\H_m}(z,w))} \int_0^{\Im(z)} \d m \frac{\Im(z-\bar{w}) - 2m}{|z-\bar{w} - 2im|^2} \frac{1}{Z_m(z,w)} \\
& \times \int_{ \R + im} \d z_\bot H_{\H_m}(z,z_\bot) H_{\H_m}(z,w) \Expect{ F( \Xi_a^{(z,w),\H_m} \wedge \wp_{\H_m}^{z,z_\bot} \wedge \wp_{\H_m}^{z_\bot, w} )}
\end{align*}
where in the last line we wrote $\Xi_a^{(z,w),\H_m}$ for a trajectory which consists in the concatenation (at $z$ say) of all the excursion in a Poisson point process with intensity $2 \pi a \mu_{\H_m}^{z,w}$. At this stage, it is not a loop, but it converges to $\Xi_a^{z,\H_m}$ as $w \to z$.
We are now ready to take the limit $w \to z$. Firstly,
\[
e^{-2\pi a (G_\H(z,w) - G_{\H_m}(z,w))} \to \CR(z,\H_m)^a / \CR(z,\H)^a.
\]
Secondly, since the Poisson kernel is explicit in the upper half plane
\[
H_{\H_m}(z,z_\bot) = \frac{1}{\pi} \frac{\Im(z) - m}{|z-z_\bot|^2},
\]
we can compute
\[
\lim_{w \to z} Z_m(z,w) = \frac{1}{\pi^2} \int_\R \frac{(\Im(z) - m)^2}{(x^2 + (\Im(z) - m)^2 )^2} \d x = \frac{1}{\pi^2} \frac{1}{\Im(z) - m} \int_\R \frac{1}{(x^2+1)^2} \d x = \frac{1}{2\pi} \frac{1}{\Im(z) - m}.
\]
Therefore, as $w \to z$, we have
\[
\frac{\Im(z-\bar{w}) - 2m}{|z-\bar{w} - 2im|^2} \frac{1}{Z_m(z,w)} \to \pi
\]
By dominated convergence theorem, we obtain that
\begin{align*}
& \Expect{ F(\Xi_a^{z,\H}) }
= 2\pi a \frac{\CR(z,\H_m)^a}{\CR(z,\H)^a} \int_0^{\Im(z)} \d m \int_{\R + im} \d z_\bot H_{\H_m}(z,z_\bot)^2 \Expect{ F(\Xi_a^{z, \H_m} \wedge \wp_{\H_m}^{z, z_\bot} \wedge \wp_{\H_m}^{z_\bot, z})}
\end{align*}
which concludes the proof.
\end{proof}

\begin{proof}[Proof of Lemma \ref{lem:first_moment_loopmeasure}]
By density-type arguments, we can assume that $F$ is continuous.
By definition, we can rewrite the left hand side of \eqref{Eq 1st moment loopmeasure} as
\begin{align*}
& \lim_{\eps \to 0} \int \loopmeasure_D(d \wp^1) \dots \loopmeasure_D(d \wp^n) \indic{ \forall i =1 \dots \eps, h(\wp^i) > \eps } F(z,\wp^1, \dots, \wp^n) \Mc_a^{\wp^1_{\eps,2} \cap \dots \cap \wp^n_{\eps,2}}(dz) \\
& = \lim_{\eps \to 0} \loopmeasure_D(h(\wp) > \eps)^n \Expect{ F(z,\wp^1_{\eps,2}, \dots, \wp^n_{\eps,2}) \Mc_a^{\wp^1_{\eps,2} \cap \dots \cap \wp^n_{\eps,2}}(dz) }
\end{align*}
where in the second line, $\wp_{\eps,2}^i$, $i = 1 \dots n$, are i.i.d. trajectories with law \eqref{eq:law_loop_height} described in Lemma \ref{lem:loops_height}. Note also that in the second line we used the continuity of $F$ and the fact that the first portion of the trajectory $\wp_{\eps,1}$ vanishes as $\eps \to 0$. We are going to expand this expression with the help of Lemma \ref{lem:loops_height}. The term $\loopmeasure_D(h(\wp) > \eps)$ and the partition function $Z$ in \eqref{eq:continuum_law_zbot_zeps} will cancel out and we obtain that the left hand side of \eqref{Eq 1st moment loopmeasure} is equal to (we write below with some abuse of notation a product of integrals instead of multiple integrals)
\begin{align}
\label{eq:first1}
& \lim_{\eps \to 0} \prod_{i=1}^n \int_{\mi(D)}^{\Mi(D)-\eps} \d m^i \int_{(\R + i m^i) \cap D} \d z_\bot^i \int_{(\R + i(m^i + \eps))\cap D} \d z_\eps^i H_{S_{m^i, m^i + \eps}}(z_\bot^i, z_\eps^i) H_{D \cap \H_{m^i}} (z_\eps^i, z_\bot^i) \\
\nonumber
& ~~~~~~ \times \Expect{ F(z, (\wp_{D \cap \H_{m^i}}^{z_\eps^i, z_\bot^i})_{i=1 \dots n} ) \Mc_a^{ \cap \wp_{D \cap \H_{m^i}}^{z_\eps^i, z_\bot^i} }(dz) }.
\end{align}
The trajectories $\wp_{D \cap \H_{m^i}}^{z_\eps^i, z_\bot^i}$ are independent Brownian trajectories with law as in \eqref{eq:proof_not}. By \eqref{eq:Girsanov_intersection}, the last expectation above is equal to
\begin{align}
\label{eq:first2}
& (2\pi)^n \int_{ \mathsf{a} \in E(a,n)} \d \mathsf{a}
\prod_{i=1}^n \CR(z,D \cap \H_{m^i})^{a_i} G_{D \cap \H_{m^i}}(z_\eps^i, z) \frac{H_{D \cap \H_{m^i}}(z,z_\bot^i)}{H_{D \cap \H_{m^i}}(z_\eps^i, z_\bot^i)} \\
\nonumber
& \times
\Expect{ F(z, (\wp_{D \cap \H_{m^i}}^{z_\eps^i, z} \wedge \Xi_{a_i}^{z,D \cap \H_{m^i}} \wedge \wp_{D \cap \H_{m^i}}^{z, z_\bot^i})_{i=1 \dots n} ) } dz
\end{align}
where all the trajectories above are independent.
When $\eps \to 0$, $z_\eps^i \to z_\bot^i$ and it is easy to see that $\wp_{ D \cap \H_{m^i}}^{z_\eps^i, z} \wedge \wp_{ D \cap \H_{m^i}}^{z, z_\bot^i}$ converges in distribution to a loop $\wp_{ D \cap \H_{m^i}}^{z_\bot^i, z} \wedge \wp_{ D \cap \H_{m^i}}^{z, z_\bot^i}$ that is the concatenation of two independent paths distributed as in Notations \ref{not:paths}. This loop will play the role of the loop whose imaginary part reaches the minimum among all loops in $\Xi_{a_i}^{z,\H}$ (see Lemma \ref{lem:ppp_mininum}). Coming back to \eqref{eq:first1} and \eqref{eq:first2}, we see that the Poisson kernels $H_{  D \cap \H_{m^i}} (z_\eps^i, z_\bot^i)$ appearing in both equations cancel out. Noticing that as soon as $\Im(z) > m^i+\eps$,
\[
\int_{(\R + i(m^i + \eps)) \cap D} H_{ S_{m^i, m^i + \eps}}(z_\bot^i, z_\eps^i) G_{ D \cap \H_{m^i}}(z_\eps^i, z) \d z_\eps^i = H_{ D \cap \H_{m^i}}(z,z_\bot^i),
\]
we overall obtain that the left hand side of \eqref{Eq 1st moment loopmeasure} is equal to
\begin{align*}
& (2\pi)^n \int_{\mathsf{a} \in E(a,n)} \d \mathsf{a}
\prod_{i=1}^n \int_{\mi(D)}^{\Im(z)} \d m^i \int_{(\R + i m^i) \cap D} \d z_\bot^i \CR(z,  D \cap \H_{m^i})^{a_i} H_{ D \cap \H_{m^i}}(z,z_\bot^i)^2\\
& ~~~~~~~~~~~~~~~~ \times
\Expect{ F(z, (\wp_{ D \cap \H_{m^i}}^{z_\bot^i, z} \wedge \wp_{ D \cap \H_{m^i}}^{z, z_\bot^i} \wedge \Xi_a^{z,  D \cap \H_{m^i}} )_{i=1 \dots n} ) }.
\end{align*}
Lemma \ref{lem:ppp_mininum} identifies this last expression with the right hand side of \eqref{Eq 1st moment loopmeasure}. This concludes the proof.
\end{proof}

\subsection{First moment (Girsanov transform)}

We now start the proof of Proposition \ref{prop:first_moment_killed_loops} as well as describing the way the loop soup changes when one shifts the probability measure by $\Mc_a^K(dz)$. The following result is the analogue of Theorem \ref{th:PD} at the approximation level. It is a quick consequence of Lemma \ref{lem:first_moment_loopmeasure}.

\begin{lemma}\label{lem:Girsanov_K}
For any bounded measurable admissible function $F$,
\begin{align*}
& \Expect{ F(z, \Lc_D^\theta) \Mc_a^K(dz) } =\\
&
\CR(z,D)^a
\sum_{n \geq 1} \frac{\theta^n}{n!}
\int_{\mathsf{a} \in E(a,n)} \frac{\d \mathsf{a}}{a_1 \dots a_n} \Expect{ \prod_{i=1}^n \left( 1 - e^{-K T(\Xi_{a_i}^z) } \right) F(z, \Lc_D^\theta \cup \{ \Xi_{a_i}^z, i = 1 \dots n \} ) } \d z ,
\end{align*}
where $(\Xi_{a_i}^z)_{1\leq i\leq n}$ are independent
and independent of $\Lc_D^\theta$.
\end{lemma}

\begin{proof}[Proof of Lemma \ref{lem:Girsanov_K}]
By definition of $\Mc_a^K$ in \eqref{eq:def_measure_killed_loops}
and monotone convergence,
we want to compute
\begin{equation}
\label{eq:2.1}
\Expect{ \sum_{\substack{\wp_1, \dots, \wp_n \in \Lc_D^\theta\\ \forall i \neq j, \wp_i \neq \wp_j}}
\prod_{i=1}^n \left( 1-e^{-KT(\wp_i)} \right)
F(z,\Lc_D^\theta) \Mc_a^{\wp_1 \cap \dots \cap \wp_n}(dz) }.
\end{equation}
By Palm's formula applied to the
Poisson point process $\Lc_D^\theta$,
we can rewrite \eqref{eq:2.1} as
\[
\theta^n \int \loopmeasure_D(d \wp_1) \dots \loopmeasure_D(d \wp_n) \prod_{i=1}^n \left( 1-e^{-KT(\wp_i)} \right) \E_{\Lc_D^\theta} \left[ F(z,\Lc_D^\theta \cup \{ \wp_1, \dots, \wp_n \}) \right] \Mc_a^{\wp_1 \cap \dots \cap \wp_n}(dz) .
\]
By Lemma \ref{lem:first_moment_loopmeasure}, this is equal to
\[
\CR(z,D)^a
\theta^n \int_{\mathsf{a} \in E(a,n)} \frac{\d \mathsf{a}}{a_1 \dots a_n} \Expect{ \prod_{i=1}^n \left( 1 - e^{-K T(\Xi_{a_i}^z) } \right) F(z, \Lc_D^\theta \cup \{ \Xi_{a_i}^z, i = 1 \dots n \} ) } \d z.
\]
This concludes the proof of Lemma \ref{lem:Girsanov_K}.
\end{proof}

We will get Proposition \ref{prop:first_moment_killed_loops} simply by taking a function $F$ depending only on $z$ in Lemma \ref{lem:Girsanov_K}. Before this, we first state a lemma which shows that (somewhat miraculously, in our opinion) the integrals appearing in Lemma \ref{lem:Girsanov_K} can be computed explicitly in terms of hypergeometric functions; this is where the function $\Fs$ comes from in our results.

\begin{lemma}\label{lem:Fs}
The function $\Fs$ defined in \eqref{eq:def_Fs} can be expressed as follows:
for all $u \geq 0$,
\begin{equation}\label{eq:integral_hypergeom}
\Fs(u) = \sum_{n \geq 1} \frac{1}{n!} \theta^n
\int_{\mathsf{a} \in E(1,n)} \frac{\d a_1 \dots \d a_{n-1}}{a_1 \dots a_n} \prod_{i=1}^n \left( 1 - e^{- u a_i} \right).
\end{equation}
\end{lemma}

\begin{proof}[Proof of Lemma \ref{lem:Fs}]
For all $u \geq 0$, let $\hat \Fs (u)$ denote the right hand side of \eqref{eq:integral_hypergeom}. We will show that $\hat \Fs (u) = \Fs (u)$.
%\[
%\hat{\Fs}(u) := \sum_{n \geq 1} \frac{1}{n!} \theta^n
%\int_{a_i >0, a_1 + \dots + a_n =1} \frac{\d a_1 \dots \d a_{n-1}}{a_1 \dots %a_n} \prod_{i=1}^n \left( 1 - e^{- u a_i} \right).
%\]
%In this proof, we will freely exchange differentiation and sum/integral. One can easily justify these formal computations by dominated convergence theorem.
Note that we have
\[
\frac{\d}{\d u} \prod_{i=1}^n (1-e^{-ua_i}) = \sum_{j=1}^n e^{-u a_j} \prod_{i \neq j} (1-e^{-ua_i})
\]
and by symmetry we deduce that
\begin{align}
\nonumber
\hat{\Fs}'(u) & = \theta e^{-u} + \sum_{n \geq 2} \frac{\theta^n}{(n-1)!} \int_{a_i>0, a_1 + \dots + a_{n-1} < 1} \frac{\d a_1 \dots \d a_{n-1}}{a_1 \dots a_{n-1}} e^{-u(1-(a_1 + \dots + a_{n-1}))} \prod_{i=1}^{n-1} (1 -e^{-ua_i}) \\
& = \theta e^{-u} \left( 1 + \sum_{n \geq 2} \frac{\theta^{n-1}}{(n-1)!} \int_{a_i>0, a_1 + \dots + a_{n-1} < 1} \frac{\d a_1 \dots \d a_{n-1}}{a_1 \dots a_{n-1}} \prod_{i=1}^{n-1} (e^{ua_i} -1) \right).
\label{eq:proof_lem_int1}
\end{align}
Differentiating further,
\begin{align*}
\frac{1}{\theta} \frac{\d}{\d u} (e^u \hat{\Fs}'(u))
& = \theta \sum_{n \geq 2} \frac{\theta^{n-2}}{(n-2)!} \int_{a_i > 0, a_1 + \dots + a_{n-2} <1} \frac{\d a_1 \dots \d a_{n-2}}{a_1 \dots a_{n-2}} \prod_{i=1}^{n-2} (e^{ua_i} -1) \\
& ~~~~~~~~~~~~~~~~~~~~~~\times \int_0^{1-(a_1 + \dots + a_{n-2})} e^{ua_{n-1}} \d a_{n-1} \\
& = \frac{\theta e^u}{u} \sum_{n \geq 2} \frac{\theta^{n-2}}{(n-2)!} \int_{a_i > 0, a_1 + \dots + a_{n-2} <1} \frac{\d a_1 \dots \d a_{n-2}}{a_1 \dots a_{n-2}} \prod_{i=1}^{n-2} (1-e^{-ua_i}) \\
& - \frac{\theta}{u} \sum_{n \geq 2} \frac{\theta^{n-2}}{(n-2)!} \int_{a_i > 0, a_1 + \dots + a_{n-2} <1} \frac{\d a_1 \dots \d a_{n-2}}{a_1 \dots a_{n-2}} \prod_{i=1}^{n-2} (e^{ua_i} - 1).
\end{align*}
By \eqref{eq:proof_lem_int1}, we see that the second term in the right hand side is equal to $-e^u \hat{\Fs}'(u)/u$. We now define the function $\mathsf{G}(u)$ to be the first term in the right hand side, multiplied by $u e^{-u}/\theta$. Thus we have
\begin{equation}
\label{eq:proof_lem_int3}
\frac{1}{\theta} e^u (\hat{\Fs}'(u) + \hat{\Fs}''(u)) = \frac{\theta e^u}{u} \mathsf{G}(u) - \frac{e^u}{u} \hat{\Fs}'(u).
\end{equation}
We further have
\begin{align*}
\mathsf{G}'(u) & = \theta \sum_{n \geq 3} \frac{\theta^{n-3}}{(n-3)!} \int_{a_i>0, a_1 + \dots +a_{n-3} <1} \frac{\d a_1 \dots \d a_{n-3}}{a_1 \dots a_{n-3}} \prod_{i=1}^{n-3} (1-e^{-ua_i}) \\
& ~~~~~~~~~~~~~~~~~~~~~ \times \int_0^{1-(a_1 + \dots + a_{n-3})} e^{-ua_{n-2}} \d a_{n-2} \\
& = \frac{\theta}{u} \sum_{n \geq 3} \frac{\theta^{n-3}}{(n-3)!} \int_{a_i>0, a_1 + \dots +a_{n-3} <1} \frac{\d a_1 \dots \d a_{n-3}}{a_1 \dots a_{n-3}} \prod_{i=1}^{n-3} (1-e^{-ua_i}) \\
& - \frac{\theta e^{-u}}{u} \sum_{n \geq 3} \frac{\theta^{n-3}}{(n-3)!} \int_{a_i>0, a_1 + \dots + a_{n-3} <1} \frac{\d a_1 \dots \d a_{n-3}}{a_1 \dots a_{n-3}} \prod_{i=1}^{n-3} (e^{ua_i}-1) \\
& = \frac{\theta}{u} \mathsf{G}(u) - \frac{1}{u} \hat{\Fs}'(u)
\end{align*}
by definition of $\mathsf{G}$ and \eqref{eq:proof_lem_int1}. Reformulating,
\[
\frac{\d}{\d u} \left( \frac{\mathsf{G}(u)}{u^\theta} \right) = \frac{\mathsf{G}'(u) - \theta u^{-1} \mathsf{G}(u)}{u^{\theta}} = -\frac{\hat{\Fs}'(u)}{u^{\theta+1}}.
\]
Thanks to \eqref{eq:proof_lem_int3}, we deduce that
\begin{align*}
\frac{1}{\theta} \frac{\d}{\d u} \left( u^{1-\theta} (\hat{\Fs}'(u) + \hat{\Fs}''(u)) \right)
= - \frac{1}{u^\theta} \hat{\Fs}''(u)
\end{align*}
and
\[
(1-\theta) \hat{\Fs}'(u) + \hat{\Fs}''(u) + u(\hat{\Fs}''(u) + \hat{\Fs}'''(u)) = 0.
\]
By looking at the solutions of this equation (see \cite[Section 13.1]{special}), we deduce that there exist $c_1, c_2 \in \R$ such that
\[
\hat{\Fs}'(u) = c_1 e^{-u} U(\theta,1,u) + c_2 e^{-u} {}_1F_1(\theta,1,u)
\]
where $U(\theta,1,u) = \frac{1}{\Gamma(\theta)} \int_0^\infty e^{-ut} t^{\theta-1} (1+t)^{-\theta} \d t$ is Tricomi's confluent hypergeometric function and
${}_1F_1(\theta,1,u) = \sum_{n=0}^\infty \frac{\theta (\theta+1) \dots (\theta + n -1)}{n!^2} u^n$ is Kummer's confluent hypergeometric function. With \eqref{eq:proof_lem_int1}, we see that $\hat{\Fs}'(u) \to \theta$ as $u \to 0$. Hence $c_1 = 0$ and $c_2 = \theta$.
We have proven that
\[
\hat{\Fs}'(u) = \theta e^{-u} {}_1F_1(\theta,1,u)
\]
and, therefore, $\hat{\Fs} = \Fs$. This concludes the proof of Lemma \ref{lem:Fs}.
\end{proof}

We can now conclude with a proof of Proposition \ref{prop:first_moment_killed_loops}.

\begin{proof}[Proof of Proposition \ref{prop:first_moment_killed_loops}]
By Lemma \ref{lem:Girsanov_K} applied to the function $F = F(z)$ depending only on $z$,  and by doing the change of variable $b_i = a_i/a$,  we have
\[
\E[\Mc_a^K(dz)]= \frac{1}{a} \sum_{n \geq 1} \frac{\theta^n}{n!}
\int_{\mathsf{b} \in E(1,n)} \frac{\d \mathsf{b}}{b_1 \dots b_n} \prod_{i=1}^n \Expect{ 1 - e^{- K T(\Xi_{a \cdot b_i}^z) } }  \CR(z,D)^a \d z.
\]
By Palm's formula and by recalling the definition \eqref{eq:def_Cpz} of $C_K(z)$, we have
\begin{equation}
\label{eq:proba_killing}
\Expect{ 1 - e^{-KT(\Xi_{a \cdot b_i}^z)} }
= 1 - \exp \left( 2\pi a \cdot b_i \int_0^\infty p_D(t,z,z) ( e^{-Kt} -1) dt \right)
= 1 - \exp \left( - C_K(z) a \cdot b_i \right).
\end{equation}
With Lemma \ref{lem:Fs}, we conclude that
\begin{align*}
\E[\Mc_a^K(dz)]= \frac{1}{a} \Fs( C_K(z) a ) \CR(z,D)^a \d z.
\end{align*}
This concludes the proof.
\end{proof}

\subsection{The crucial martingale}

\label{S:martingale}

We now turn to the proof of Proposition \ref{prop:martingale}. We will see that it is the consequence of the following two lemmas. We will first state these two lemmas, then show how they imply Proposition \ref{prop:martingale}, and then prove the two lemmas.

The first lemma shows that the function $\Fs$, defined in \eqref{eq:def_Fs} and appearing in the first moment of $\Mc_a^K$, solves some integral equation. As we will see, this equation is precisely what is required in order to show that the expectation of the martingale is constant.

\begin{lemma}\label{lem:sanity_check}
For all $a \geq 0$ and $v \geq 0$,
\begin{equation}
\int_0^a \frac{\d {\thk}}{{\thk} (a-{\thk})^{1-\theta}} e^{{\thk} v} \Fs({\thk} v) + \frac{1}{a^{1-\theta}} = \frac{e^{av}}{a^{1-\theta}}.
\end{equation}
\end{lemma}

Let $K' < K$.
The second lemma expresses the measure $\Mc_a^K$ in terms of $\Mc_{\thk}^{K'}$, ${\thk} \in (0,a)$.
Denote by $\Mc_a^{K,K'}$ the measure on $a$-thick points of loops in $\Lc_D^\theta(K) \backslash \Lc_D^\theta(K')$, i.e.
\[
\Mc_a^{K,K'} :=
\sum_{n \geq 1} \frac{1}{n!} \sum_{\substack{\wp_1, \dots, \wp_n \in \Lc_D^\theta(K) \backslash \Lc_D^\theta(K')\\ \forall i \neq j, \wp_i \neq \wp_j}} \Mc_a^{\wp_1 \cap \dots \cap \wp_n}.
\]
For any $\thk \in (0,a)$, denote also $\Mc_{a-{\thk}}^{K,K'} \cap \Mc_{\thk}^{K'}$ the measure on thick points, where the total thickness $a$ comes from a combination of loops in $\Lc_D^\theta (K) \setminus \Lc_D^\theta(K')$ (with thickness $a- \thk$) and loops in $\Lc_D^\theta(K')$ (with thickness $\thk$). More precisely,
\[
\Mc_{a-{\thk}}^{K,K'} \cap \Mc_{\thk}^{K'} := \sum_{n, m \geq 1} \frac{1}{n! m!} \sum_{\substack{\wp_1, \dots, \wp_n \in \Lc_D^\theta(K) \backslash \Lc_D^\theta(K')\\ \forall i \neq j, \wp_i \neq \wp_j}}
\sum_{\substack{\wp_1', \dots, \wp_m' \in \Lc_D^\theta(K') \\ \forall i \neq j, \wp_i' \neq \wp_j'}} \Mc_{a-\thk}^{\wp_1 \cap \dots \cap \wp_n} \cap \Mc_{\thk}^{\wp_1' \cap \dots \cap \wp_m'},
\]
where $\Mc_{a-\thk}^{\wp_1 \cap \dots \cap \wp_n} \cap \Mc_{\thk}^{\wp_1' \cap \dots \cap \wp_m'}$ is defined in Section \ref{sec:preliminaries_BMC}.
We recall that $\Mc_{a-{\thk}}^{K,K'} \cap \Mc_{\thk}^{K'}$ may be viewed as the Brownian chaos generated by $\Mc_{a-{\thk}}^{K,K'}$ with respect to an intensity measure $\sigma$, which is itself an (independent) Brownian chaos generated by $\Mc_{\thk}^{K'}$; see \eqref{eq:intersection_measure}.

We claim:

\begin{lemma}
\label{lem:decomposition_KK'}
Let $K' < K$. We can decompose
\begin{equation}
\label{eq:integral_KK'}
\Mc_a^K =
\Mc_a^{K,K'} + \Mc_a^{K'} +
\int_0^a \d {\thk} ~\Mc_{a-{\thk}}^{K,K'} \cap \Mc_{\thk}^{K'}.
\end{equation}
\end{lemma}

\begin{remark}\label{R:notproved}
By taking $K \to \infty$, and writing $K$ instead of $K'$, it should be possible to deduce from Lemma \ref{lem:decomposition_KK'} and from our results, a posteriori, that we have an identity of the type:
\begin{equation}
\label{eq:not-proved}
\nu_a^K + \int_0^a \d\rho ~\nu_{a - {\thk}}^K \cap \Mc_{\thk}^K = \nu_a.
\end{equation}
Here, the measure $\nu_a^K$, is (informally) the uniform measure on thick points of the non-killed loop soup, and $\nu_{a - {\thk}}^K \cap \Mc_{\thk}^K$ would be a uniform measure on thick points created by both measures;
both would need to be defined carefully. One should further expect that $\nu_a^K$ coincides with the exponential measure on such points except for a factor of the form $1/a^{1-\theta}$ (this can heuristically be understood in the case $\theta =1/2$ as coming from the tail of the Gaussian distribution).

Accepting the above, we see that \eqref{eq:not-proved} is consistent with the martingale in Proposition \ref{prop:martingale}.
 The identity \eqref{eq:not-proved} is in fact what motivated us to define the martingale in Proposition \ref{prop:martingale}.
\end{remark}

Let's see how Proposition \ref{prop:martingale} follows from Lemmas \ref{lem:sanity_check} and \ref{lem:decomposition_KK'}.

\begin{proof}[Proof of Proposition \ref{prop:martingale}]
Let $K' < K$.
We first note that
\begin{equation}
\label{eq:first_moment_KK'}
\Expect{ \Mc_a^{K,K'}(dz) } = \frac{1}{a} \Fs \left( a C_K(z) - a C_{K'}(z) \right) e^{-a C_{K'}(z)} \CR(z,D)^a \d z.
\end{equation}
Indeed, the only difference with the expectation of $\Mc_a^K$ is that loops are required to survive the $K'$-killing, so that $1-e^{-KT(\wp)}$ is replaced by $e^{-K'T(\wp)} - e^{-KT(\wp)}$ and we find that
\begin{align*}
\Expect{ \Mc_a^{K,K'}(dz) } =
\sum_{n \geq 1} \frac{\theta^n}{n!} \int_{\mathsf{a} \in E(a,n)} \d \mathsf{a} \prod_{i=1}^n \frac{e^{-a_iC_{K'}(z)} - e^{-a_iC_K(z)}}{a_i} \CR(z,D)^a \d z.
\end{align*}
\eqref{eq:first_moment_KK'} then follows by factorising by $\prod_i e^{-a_i C_{K'}(z)} = e^{-a C_{K'}(z)}$ and by Lemma \ref{lem:Fs}.

By \eqref{eq:integral_KK'} and properties of the intersection measure (in particular (1.6) in \cite{jegoRW}), we have
\begin{align*}
& \Expect{ \Mc_{\thk}^K(dz) \vert \Fc_{K'} }
= \Mc_{\thk}^{K'} + \Expect{\Mc_{\thk}^{K,K'}} + \int_0^{\thk} \d \beta ~\Expect{ \Mc_{{\thk}-\beta}^{K,K'}(z)} \Mc_\beta^{K'}(dz) \\
& = \Mc_{\thk}^{K'} + \frac{1}{{\thk}} \Fs \left( {\thk} C_K(z) - {\thk} C_{K'}(z) \right) e^{-{\thk} C_{K'}(z)} \CR(z,D)^{\thk} \d z \\
& + \int_0^{\thk} \d \beta \frac{1}{{\thk} - \beta} \Fs \left( ({\thk} - \beta) ( C_K(z) - C_{K'}(z)) \right) e^{-({\thk} - \beta) C_{K'}(z)} \CR(z,D)^{{\thk} - \beta} \Mc_\beta^{K'}(dz).
\end{align*}
Hence the conditional expectation $\Expect{m_a^K \vert \Fc_{K'}}$ is equal to
\begin{align}
\label{eq:proof_prop_martingale}
& a^{\theta-1} \CR(z,D)^a e^{-aC_K(z)} dz
+ \int_0^a \d {\thk} \frac{1}{(a-{\thk})^{1-\theta}} \CR(z,D)^{a-{\thk}} e^{-(a-{\thk}) C_K(z)} \Mc_{{\thk}}^{K'}(dz) \\
\nonumber
& + \int_0^a \frac{\d {\thk}}{{\thk}(a-{\thk})^{1-\theta}} \CR(z,D)^a e^{-(a-{\thk}) C_K(z)-{\thk} C_{K'}(z)} \Fs \left( {\thk} C_K(z) - {\thk} C_{K'}(z) \right) dz
\\
\nonumber
& + \int_0^a \d {\thk} \int_0^{\thk} \d \beta \frac{1}{({\thk}-\beta)(a-{\thk})^{1-\theta}} \CR(z,D)^{a-\beta} e^{-(a-{\thk}) C_K(z) - ({\thk}-\beta) C_{K'}(z)} \\
\nonumber
& ~~~~~~~~~~~~~~~~ \times \Fs \left( ({\thk}-\beta)(C_K(z) - C_{K'}(z)) \right) \Mc_\beta^{K'}(dz).
\end{align}
By Lemma \ref{lem:sanity_check} (with $v = C_K(z) - C_{K'}(z)$), the sum of the first and third terms in \eqref{eq:proof_prop_martingale} is equal to
\[
\frac{1}{a^{1-\theta}} \CR(z,D)^a e^{-a C_{K'}(z)} dz.
\]
On the other hand, by exchanging the two integrals, the fourth term is equal to
\begin{align*}
& \int_0^a \d \beta \CR(z,D)^{a - \beta} e^{-(a-\beta)C_K(z)} \Mc_\beta^{K'}(dz) \\
& \quad \times \int_\beta^a \frac{\d {\thk}}{({\thk} - \beta)(a - {\thk})^{1-\theta}} e^{({\thk}-\beta)(C_K(z) - C_{K'}(z))} \Fs \left( ({\thk}-\beta)(C_K(z) - C_{K'}(z)) \right).
\end{align*}
By Lemma \ref{lem:sanity_check}, the integral with respect to ${\thk}$ is equal to
\[
\frac{e^{(a-\beta)(C_K(z) - C_{K'}(z))} -1}{(a-\beta)^{1-\theta}}
\]
implying that the fourth term of \eqref{eq:proof_prop_martingale} is equal to
\begin{align*}
& \int_0^a \frac{\d \beta}{(a-\beta)^{1-\theta}} \CR(z,D)^{a - \beta} e^{-(a-\beta)C_{K'}(z)} \Mc_\beta^{K'}(dz) \\
& - \int_0^a \frac{\d \beta}{(a-\beta)^{1-\theta}} \CR(z,D)^{a - \beta} e^{-(a-\beta)C_{K}(z)} \Mc_\beta^{K'}(dz).
\end{align*}
This second integral cancels with the second term of \eqref{eq:proof_prop_martingale}. Overall, this shows that
\begin{align*}
& \Expect{m_a^K(dz) \vert \Fc_{K'} } \\
& = \frac{1}{a^{1-\theta}} \CR(z,D)^a e^{-a C_{K'}(z)} dz
+ \int_0^a \frac{\d \beta}{(a-\beta)^{1-\theta}} \CR(z,D)^{a - \beta} e^{-(a-\beta)C_{K'}(z)} \Mc_\beta^{K'}(dz) \\
& = m_a^{K'}(dz).
\end{align*}
This concludes the proof of Proposition \ref{prop:martingale}.
\end{proof}

The rest of the section is devoted to the proofs of Lemmas \ref{lem:sanity_check} and \ref{lem:decomposition_KK'}. We start with Lemma \ref{lem:sanity_check}.

\begin{proof}[Proof of Lemma \ref{lem:sanity_check}]
By doing the change of variable ${\thk} = a \beta$, it is enough to show that
\[
\int_0^1 \d \beta \frac{1}{\beta (1-\beta)^{1-\theta}} e^{av\beta} \Fs(av \beta) = e^{av} -1.
\]
Recall that for all $u \geq 0$,
\[
\Fs(u) = \theta \int_0^u \d t ~e^{-t} \sum_{n=0}^\infty \frac{\theta^{(n)}}{n!^2} t^n
\]
where we have let
\begin{equation}
\label{eq:def_theta_n}
\theta^{(0)} := 1
\quad \quad \text{and} \quad \quad
\theta^{(n)} := \theta (\theta +1) \dots (\theta + n-1), \quad n \geq 1.
\end{equation}
By exchanging the integral and the sum, we find that for all $u \geq 0$,
\[
\Fs(u) = \theta \sum_{n=0}^\infty \frac{\theta^{(n)}}{n!} \left( 1 - e^{-u} \sum_{k=0}^n \frac{u^k}{k!} \right)
= \theta e^{-u} \sum_{n=0}^\infty \frac{\theta^{(n)}}{n!} \sum_{k=n+1}^\infty \frac{u^k}{k!}.
\]
Hence
\begin{align*}
& \int_0^1 \d \beta \frac{1}{\beta (1-\beta)^{1-\theta}} e^{av\beta} \Fs(av \beta)
= \theta \sum_{n=0}^\infty \frac{\theta^{(n)}}{n!} \sum_{k=n+1}^\infty \frac{(av)^k}{k!} \int_0^1 \d \beta \frac{\beta^{k-1}}{(1-\beta)^{1-\theta}}.
\end{align*}
Now, by \eqref{eq:beta_function}, for all $k \geq 1$,
\[
\int_0^1 \d \beta \frac{\beta^{k-1}}{(1-\beta)^{1-\theta}} = \frac{(k-1)! \Gamma(\theta)}{\Gamma(k+\theta)} = \frac{(k-1)!}{\theta^{(k)}},
\]
which implies that
\begin{align*}
\int_0^1 \d \beta \frac{1}{\beta (1-\beta)^{1-\theta}} e^{av\beta} \Fs(av \beta)
& = \theta \sum_{n=0}^\infty \frac{\theta^{(n)}}{n!} \sum_{k=n+1}^\infty \frac{1}{k.\theta^{(k)}} (av)^k \\
& = \theta \sum_{k=1}^\infty \frac{1}{k.\theta^{(k)}} (av)^k \sum_{n=0}^{k-1} \frac{\theta^{(n)}}{n!}.
\end{align*}
Furthermore, we can easily show by induction that
\[
\sum_{n=0}^{k-1} \frac{\theta^{(n)}}{n!}
= \frac{1}{\theta} \frac{\theta^{(k)}}{(k-1)!}.
\]
We can thus conclude that
\[
\int_0^1 \d \beta \frac{1}{\beta (1-\beta)^{1-\theta}} e^{av\beta} \Fs(av \beta)
= \sum_{k=1}^\infty \frac{1}{k!} (av)^k = e^{av} -1
\]
as desired.
\end{proof}

We now turn to the proof of Lemma \ref{lem:decomposition_KK'}.

\begin{proof}[Proof of Lemma \ref{lem:decomposition_KK'}]
We have
\begin{align*}
\Mc_a^K & = \sum_{n \geq 1} \frac{1}{n!} \sum_{\wp_1 \neq \dots \neq \wp_n \in \Lc_D^\theta(K)} \Mc_a^{\wp_1 \cap \dots \cap \wp_n} \\
& = \sum_{n \geq 1} \frac{1}{n!} \sum_{k=0}^n \binom{n}{k} \sum_{\substack{\wp_1 \neq \dots \neq \wp_k \in \Lc_D^\theta(K') \\\wp_{k+1} \neq \dots \neq \wp_n \in \Lc_D^\theta(K) \backslash \Lc_D^\theta(K')}} \Mc_a^{\wp_1 \cap \dots \cap \wp_n}
\end{align*}
The terms $k=0$ and $k=n$ give rise to $\Mc_a^{K,K'}$ and $\Mc_a^{K'}$ respectively.
By Proposition 1.3 in \cite{jegoRW} (applied to Brownian loops instead of Brownian motions, although as explained in Section \ref{sec:def} this is justified), we can disintegrate
\[
\Mc_a^{\wp_1 \cap \dots \cap \wp_n}
= \int_0^a \d {\thk} ~\Mc_{\thk}^{\wp_1 \cap \dots \cap \wp_k} \cap \Mc_{a-{\thk}}^{\wp_{k+1} \cap \dots \cap \wp_n}.
\]
Therefore, letting $m$ be $n-k$,
\begin{align*}
&\Mc_a^K
= \Mc_a^{K,K'} + \Mc_a^{K'} \\
& + \int_0^a \d {\thk} \Big( \sum_{k \geq 1} \tfrac{1}{k!}
\sum_{\wp_1 \neq \dots \neq \wp_k \in \Lc_D^\theta(K') } \Mc_{\thk}^{\wp_1 \cap \dots \cap \wp_k} \Big)
\cap
\Big( \sum_{m \geq 1} \tfrac{1}{m!}
\sum_{\wp_1 \neq \dots \neq \wp_m \in \Lc_D^\theta(K) \backslash \Lc_D^\theta(K')} \Mc_{a-{\thk}}^{\wp_1 \cap \dots \cap \wp_m} \Big) \\
& = \Mc_a^{K,K'} + \Mc_a^{K'} + \int_0^a \d {\thk}~ \Mc_{\thk}^{K'} \cap \Mc_{a-{\thk}}^{K,K'}.
\end{align*}
This concludes the proof of Lemma \ref{lem:decomposition_KK'}.
\end{proof}

\section{Second moment computations and multi-point rooted measure}
\label{Sec 2nd mom}

The goal of this section is to initiate the study of the second moment.

\subsection{Preliminaries}

We start off by giving the analogue of Lemma \ref{lem:first_moment_loopmeasure} in the second moment case. A new process of excursions will come into play, which we describe now. We first introduce the following special function:
\begin{equation}
\label{eq:def_Gs}
\Bs(u) := \sum_{k \geq 1} \frac{u^k}{k! (k-1)!}, \quad u \geq 0.
\end{equation}
$\Bs$ can be expressed in terms of the modified Bessel function of the first kind $I_1$ (see \eqref{eq:modified_Bessel_function}), but it is more convenient to give a name to the function $\Bs$ instead of $I_1$, since it comes up in many places below.

Let $z, z' \in D$ be two distinct points and let $a, a'>0$. We consider the cloud (meaning the point process) of excursions $\Xi_{a,a'}^{z,z'}$ such that for all $k \geq 1$,
\begin{equation}
\Prob{ \# \Xi_{a,a'}^{z,z'} = 2k }
= \frac{1}{\Bs\Big((2\pi)^2 aa' G_D(z,z')^2\Big)} \frac{(2\pi \sqrt{aa'} G_D(z,z'))^{2k}}{k! (k-1)!},
\label{PP}
\end{equation}
and conditionally on $\{ \# \Xi_{a,a'}^{z,z'} = 2k \}$, $\Xi_{a,a'}^{z,z'}$ is composed of $2k$ independent and identically distributed excursions from $z$ to $z'$, with common law $\mu_D^{z,z'} / G_D(z,z')$
\eqref{Eq mu D z w}. Note that $\Xi_{a,a'}^{z,z'}$ is \textbf{not} a Poisson point process of excursions, since 
\eqref{PP} is not the Poisson distribution. However, one can see that it becomes asymptotically Poisson (conditioned to be even), in the limit when $z \to z'$. This fact will not be needed in what follows but is useful to guide the intuition. The parity condition implicit in \eqref{PP} is crucial, since it allows us to combine these excursions into loops that visit both $z$ and $z'$. 

Recall the definition \eqref{eq:def_Ean} of the simplex $E(a,n)$, the notion of admissible functions intoduced in Definition \ref{def:admissible} and also Notation \ref{not:xi_a^z} where the loops $\Xi_a^z$ are defined.

\begin{lemma}
\label{lem:second_moment_loopmeasure}
Let $z, z'\in D$. Let $0 < a , a'<2$. Let $n, m \geq 1$, $l \in \{0, \dots, n \wedge m \}$ and $F = F(z, z', \wp_1, \ldots, \wp_n, \wp'_{l+1}, \ldots, \wp'_m)$ be a bounded measurable admissible function of two points and $n+m - l$ loops. We have
\begin{align}
\label{eq:lem_second_moment_loopmeasure}
& \int \loopmeasure_D(d \wp_1) \dots \loopmeasure_D(d \wp_n) \loopmeasure_D(d \wp_{l+1}') \dots \loopmeasure_D(d \wp_m') \\
\nonumber
& ~~~ \Mc_a^{\wp_1 \cap \dots \wp_n}(dz) \Mc_{a'}^{\wp_1 \dots \cap \wp_l \cap \wp_{l+1}' \cap \dots \wp_m'}(dz') F(z,z', \wp_1, \dots, \wp_n, \wp_{l+1}', \dots, \wp_m') \\
\nonumber
& = 
\CR(z,D)^a \CR(z',D)^{a'}
\int_{\substack{\mathsf{a} \in E(a,n) \\ \mathsf{a}' \in E(a',m) }} \frac{ \d \mathsf{a} }{a_1 \dots a_n} \frac{\d \mathsf{a}'}{a_1' \dots a_m'} \prod_{i=1}^l \Bs \left( (2\pi)^2 a_i a_i' G_D(z,z')^2 \right) \\
\nonumber
& ~~~\times \E \left[ F\Big(z,z',
\{\Xi^{z,z'}_{a_i,a'_i} \wedge \Xi^z_{a_i} \wedge \Xi^{z'}_{a'_i}\}_{i=1}^l ; %\dots, \Xi^{z,z'}_{a_l,a'_l} \wedge \Xi^z_{a_l} \wedge \Xi^{z'}_{a'_l} ,
\{\Xi^z_{a_{i}}\}_{i=l+1}^n ;
\{\Xi^{z'}_{a'_{i}}\}_{i=l+1}^m \Big)   \right] 
\d z \d z'
\end{align}
where all the random variables appearing above are independent and $\wedge$ denotes concatenation in some order (the precise order does not matter by admissibility).
\end{lemma}

\begin{figure}
   \centering
    \def\svgwidth{0.5\columnwidth}
   %% Creator: Inkscape 1.1 (c4e8f9e, 2021-05-24), www.inkscape.org
%% Accompanies image file 'drawing3.pdf' (pdf, eps, ps)
\begingroup%
  \makeatletter%
  \providecommand\color[2][]{%
    \errmessage{(Inkscape) Color is used for the text in Inkscape, but the package 'color.sty' is not loaded}%
    \renewcommand\color[2][]{}%
  }%
  \providecommand\transparent[1]{%
    \errmessage{(Inkscape) Transparency is used (non-zero) for the text in Inkscape, but the package 'transparent.sty' is not loaded}%
    \renewcommand\transparent[1]{}%
  }%
  \providecommand\rotatebox[2]{#2}%
  \newcommand*\fsize{\dimexpr\f@size pt\relax}%
  \newcommand*\lineheight[1]{\fontsize{\fsize}{#1\fsize}\selectfont}%
  \ifx\svgwidth\undefined%
    \setlength{\unitlength}{349.70784833bp}%
    \ifx\svgscale\undefined%
      \relax%
    \else%
      \setlength{\unitlength}{\unitlength * \real{\svgscale}}%
    \fi%
  \else%
    \setlength{\unitlength}{\svgwidth}%
  \fi%
  \global\let\svgwidth\undefined%
  \global\let\svgscale\undefined%
  \makeatother%
  \begin{picture}(1,0.75162248)%
    \lineheight{1}%
    \setlength\tabcolsep{0pt}%
    \put(0,0){\includegraphics[width=\unitlength,page=1]{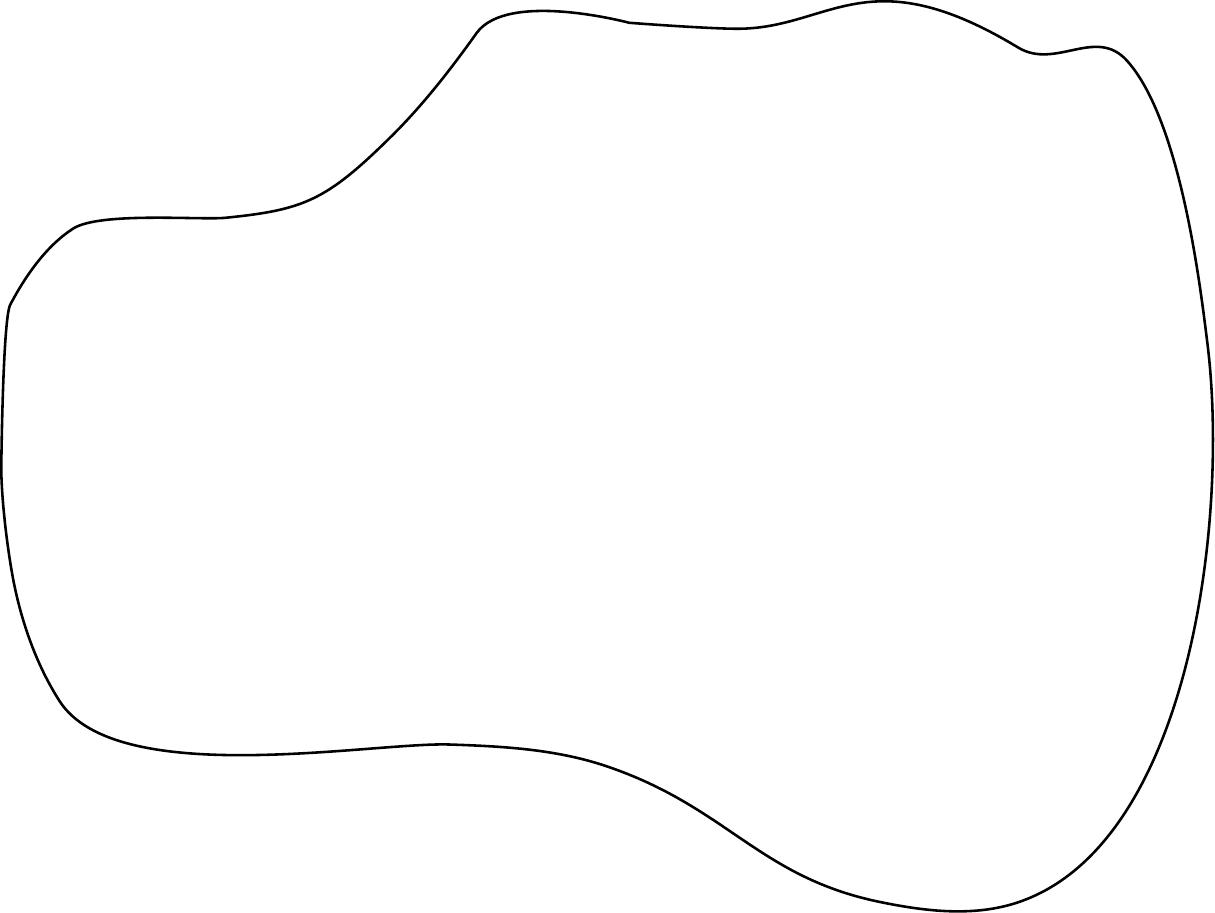}}%
    \put(0.21725297,0.62592664){\makebox(0,0)[lt]{\lineheight{1.25}\smash{\begin{tabular}[t]{l}$D$\end{tabular}}}}%
    \put(0,0){\includegraphics[width=\unitlength,page=2]{drawing3.pdf}}%
    \put(0.73408259,0.18821213){\color[rgb]{0.83137255,0,0}\makebox(0,0)[lt]{\lineheight{1.25}\smash{\begin{tabular}[t]{l}$\Xi_{a'}^{z'}$\end{tabular}}}}%
    \put(0.13899485,0.42781036){\color[rgb]{0,0,1}\makebox(0,0)[lt]{\lineheight{1.25}\smash{\begin{tabular}[t]{l}$\Xi_a^z$\end{tabular}}}}%
    \put(0.49164269,0.49347761){\makebox(0,0)[lt]{\lineheight{1.25}\smash{\begin{tabular}[t]{l}$\Xi_{a,a'}^{z,z'}$\end{tabular}}}}%
    \put(0.29637507,0.37548434){\makebox(0,0)[lt]{\lineheight{1.25}\smash{\begin{tabular}[t]{l}$z$\end{tabular}}}}%
    \put(0.81191832,0.38799545){\makebox(0,0)[lt]{\lineheight{1.25}\smash{\begin{tabular}[t]{l}$z'$\end{tabular}}}}%
  \end{picture}%
\endgroup%

   \caption{Illustration of Lemmas \ref{lem:second_moment_loopmeasure} and \ref{lem:second_crucial}: decomposition of a loop which is $a$-thick at $z$ and $a'$-thick at $z'$.}\label{fig3}
\end{figure}

Before we start with the proof of this lemma, we make a few comments on its meaning. Note that in the left hand side, we can think of $z$ and $z'$ respectively as having been sampled from Brownian chaos measures associated with loops which can overlap: namely, $\wp_1, \ldots, \wp_l$ are common to both collections. The right hand side expresses the law that results from this conditioning (or more precisely reweighting): we get not only the Poisson point processes of excursions $\Xi_{a_i}^z$ and $\Xi_{a'_i}^{z'}$ which already appeared in Lemma \ref{lem:first_moment_loopmeasure}, but also an independent \emph{non-Poissonian} collection of excursions joining $z$ and $z'$ with law given by \eqref{PP}. Figure \ref{fig3} depicts a loop that is both $a$-thick at $z$ and $a'$-thick at $z'$ and illustrates its decomposition into the three parts $\Xi_a^z$, $\Xi_{a'}^{z'}$ and $\Xi_{a,a'}^{z,z'}$.

We encapsulate the heart of the proof Lemma \ref{lem:second_moment_loopmeasure} in Lemma \ref{lem:second_crucial} below.
For $a,a'\in (0,2)$, $z \in D$,
let $\Mc_{a'}^{\Xi^z_{a}}$ denote the measure on
$a'$-thick points generated by the loop $\Xi^z_{a}$ (recall Notation \ref{not:xi_a^z}).
More precisely,
\begin{equation}\label{eq:chaos_ppp}
\Mc_{a'}^{\Xi^z_{a}}
:=
\sum_{k\geq 1}\dfrac{1}{k!}
\sum_{\substack{\wp_1, \dots, \wp_k 
\\ \text{excursions of } \Xi^z_{a}
\\ \forall i \neq j, \wp_i \neq \wp_j}} 
\Mc_{a'}^{\wp_1 \cap \dots \cap \wp_k}.
\end{equation}

\begin{lemma}\label{lem:second_crucial}
Let $z \in D$, $a, a' \in (0,2)$. For any nonnegative measurable admissible function $F$,
\begin{equation}
\label{eq:lem_second_crucial}
\Expect{ \Mc_{a'}^{\Xi_a^z}(dz') F(z',\Xi_a^z) }  = \frac{1}{a'} \CR(z',D)^{a'} \Bs((2\pi)^2 aa' G_D(z,z')^2) \Expect{ F(z', \Xi_a^z \wedge \Xi_{a'}^{z'} \wedge \Xi_{a,a'}^{z,z'} ) } \d z'.
\end{equation}
\end{lemma}

We now explain how Lemma \ref{lem:second_moment_loopmeasure} is obtained from this result. We will then prove Lemma \ref{lem:second_crucial}.

\begin{proof}[Proof of Lemma \ref{lem:second_moment_loopmeasure}]
By Lemma \ref{lem:first_moment_loopmeasure}, we can rewrite the left hand side of \eqref{eq:lem_second_moment_loopmeasure} as
\begin{align*}
& \CR(z,D)^a \d z \int_{\mathsf{a} \in E(a,n)} \frac{\d \mathsf{a}}{a_1 \dots a_n} \int \loopmeasure_D(d \wp'_{l+1}) \dots \loopmeasure_D(d \wp'_m) \\
& \times \E \Big[ \Mc_{a'}^{\Xi_{a_1}^z \cap \dots \cap \Xi_{a_l}^z \cap \wp_{l+1}' \cap \dots \cap \wp_m}(dz')
F(z,z', \{ \Xi_{a_i}^z \}_{i=1 \dots n} ; \{ \wp_i' \}_{i=l+1 \dots m} ) \Big].
\end{align*}
Concluding the proof is then routine: we use the disintegration formula \eqref{eq:disintegration} to specify the thickness of each trajectory, Lemma \ref{lem:first_moment_loopmeasure} and Lemma \ref{lem:second_crucial}. We omit the details.
\end{proof}

The rest of this section is dedicated to the proof of Lemma \ref{lem:second_crucial}.
As in the first moment computations made in Section \ref{subsec:preliminaries_first}, we will need to have an understanding of the processes of loops involved in Lemma \ref{lem:second_crucial} seen from their point with minimal imaginary part. Lemma \ref{lem:ppp_mininum} already achieves such a description for $\Xi_a^z$. We now completes the picture by doing it for loops $\wp_D^{z,z'} \wedge \wp_D^{z',z}$ appearing in the definition of $\Xi_{a,a'}^{z,z'}$.

\begin{lemma}\label{lem:loop_minimum}
Let $z, z' \in D$ be distinct points. For all nonnegative measurable function $F$,
\begin{align*}
& \Expect{F(\wp_D^{z,z'} \wedge \wp_D^{z',z})}
= \frac{1}{G_D(z,z')^2} \int_{\mi(D)}^{\Im(z) \wedge \Im(z')} \d m ~G_{D \cap \H_m}(z,z') \int_{(\R + im) \cap D} \d z_\bot \\
& \times H_{D \cap \H_m}(z',z_\bot) H_{D \cap \H_m}(z,z_\bot) \Expect{
F( \wp_{D \cap \H_m}^{z,z'} \wedge \wp_{D \cap \H_m}^{z',z_\bot} \wedge \wp_{D \cap \H_m}^{z_\bot,z} )
+ F( \wp_{D \cap \H_m}^{z,z_\bot} \wedge \wp_{D \cap \H_m}^{z_\bot,z'} \wedge \wp_{D \cap \H_m}^{z',z} ) }.
\end{align*}
\end{lemma}

We mention that the first (resp. second) term in the above expectation corresponds to the case where the minimum of the loop $\wp_D^{z,z'} \wedge \wp_D^{z',z}$ is achieved by the second piece $\wp_D^{z',z}$ (resp. first piece $\wp_D^{z,z'}$).

\begin{proof}
The proof is similar to the proof of Lemma \ref{lem:ppp_mininum}. We first notice that, by restriction arguments, it is enough to show the result for the upper half plane. We then show it exploiting explicit expressions for the Green function and the Poisson kernel. We do not provide more details.
\begin{comment}
More details:
For all $m \in (0,m_z \wedge m_{z'})$,
\[
\Prob{ \mi(\wp_{\H}^{z,z'} \wedge \wp_{\H}^{z',z}) \geq m } = G_{\H_m}(z,z')^2 / G_{\H}(z,z')^2.
\]
By differentiating w.r.t. $m$, we deduce that the density of $\mi(\wp_{\H}^{z,z'} \wedge \wp_{\H}^{z',z})$ w.r.t. Lebesgue measure is given by
\[
\frac{2}{\pi} \frac{G_{\H_m}(z,z')}{G_\H(z,z')^2} \frac{\Im(z-\bar{z}') - 2m}{|z-\bar{z}' - 2im|^2} \d m.
\]
Furthermore,
\[
\Expect{ F(\wp_{\H}^{z,z'} \wedge \wp_{\H}^{z',z}) \vert \mi(\wp_{\H}^{z,z'} \wedge \wp_{\H}^{z',z}) = m } = \frac{1}{Z_m} \int_{\R + im} \d z_\bot H_{\H_m}(z',z_\bot) H_{\H_m}(z,z_\bot) \Expect{ F( \wp_{\H_m}^{z,z'} \wedge \wp_{\H_m}^{z',z_\bot} \wedge \wp_{\H_m}^{z_\bot,z} ) }
\]
We have
\[
Z_m = \frac{1}{\pi} \frac{\Im(z-\bar{z}') - 2m}{|z-\bar{z}' - 2im|^2}.
\]
Putting everything together,
\begin{align*}
\Expect{ F(\wp_{\H}^{z,z'} \wedge \wp_{\H}^{z',z}) } = 2 \frac{1}{G_\H(z,z')^2} \int_0^{m_z \wedge m_{z'}} G_{\H_m(z,z')} \int_{\R + im} \d z_\bot H_{\H_m}(z',z_\bot) H_{\H_m}(z,z_\bot) \Expect{ F( \wp_{\H_m}^{z,z'} \wedge \wp_{\H_m}^{z',z_\bot} \wedge \wp_{\H_m}^{z_\bot,z} ) }
\end{align*}
\end{comment}
\end{proof}

We finally prove Lemma \ref{lem:second_crucial}.

\begin{figure}
   \centering
    \def\svgwidth{0.5\columnwidth}
   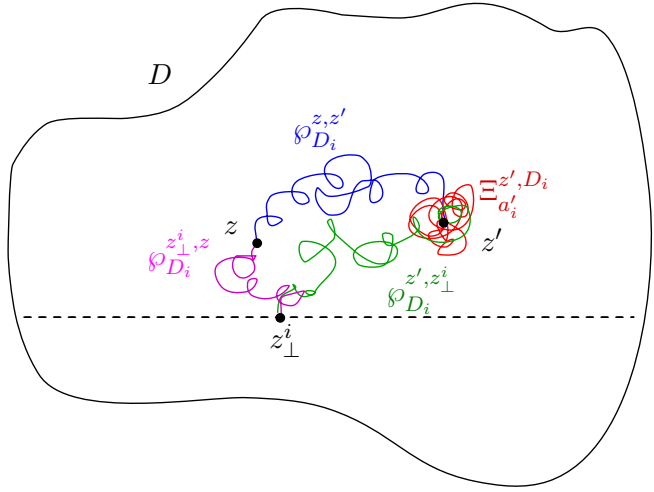
   \caption{Loop rooted at $z$, $a_i'$-thick at $z'$ with minimal imaginary part reached at $z_\bot^i$. In the case depicted above, the minimal imaginary part is reached by the excursion from $z'$ to $z$ which corresponds to the event $E_2^i(\eps)$ defined in the proof of Lemma \ref{lem:second_crucial}. The loop is decomposed into four portions (one for each colour) corresponding to \eqref{eq:cr2}.}\label{fig2}
\end{figure}

\begin{proof}[Proof of Lemma \ref{lem:second_crucial}]
By density-type arguments, we can assume that $F$ is continuous. Let $m_z := \Im(z)$ and let $\Xi_{a,\eps}^z := \{ \wp \text{~excursion~in~} \Xi_a^z: \mi(\wp) < m_z - \eps \}$ be the set of excursions in $\Xi_a^z$ which go below $\H_{m_z-\eps}$ (recall Notation \ref{not:mi}). In this proof, we will, with some abuse of notations, denote by $\Xi_{a,\eps}^z$ both the set of excursions and the loop obtained as the concatenation of all these excursions. $\# \Xi_{a,\eps}^z$ is a Poisson variable with mean $2\pi a \mu_D^{z,z}(\mi(\wp) < m_z-\eps)$ and conditioned on $\# \Xi_{a,\eps}^z = n$, $\Xi_{a,\eps}^z$ is composed of $n$ i.i.d. excursions with common distribution $\wp_\eps$ that we describe now. We root $\wp_\eps$ at the point $z_\bot$ with minimal imaginary part and, recalling Notation \ref{not:paths}, we have for any bounded measurable function $F$,
\begin{equation}
\label{eq:2p4}
\Expect{F(\wp_\eps)} = \frac{1}{Z_\eps} \int_{\mi(D)}^{m_z-\eps} \d m \int_{(\R +im)\cap D} \d z_\bot H_{D \cap \H_m}(z,z_\bot)^2 \Expect{ F(\wp_{D \cap \H_m}^{z,z_\bot} \wedge \wp_{D \cap \H_m}^{z_\bot,z}) },
\end{equation}
where $Z_\eps$ is the normalising constant
\[
Z_\eps = \int_{\mi(D)}^{m_z-\eps} \d m \int_{(\R +im)\cap D} \d z_\bot H_{D \cap \H_m}(z,z_\bot)^2.
\]
Note that $Z_\eps = \mu_D^{z,z}(\mi(\wp) < m_z-\eps)$.
We can now start the computation of the left hand side of \eqref{eq:lem_second_crucial}. By continuity of $F$, it is equal to
\begin{align}
\label{eq:2p7}
& \lim_{\eps \to 0} e^{-2\pi a Z_\eps} \sum_{n=0}^\infty \frac{(2\pi a Z_\eps)^n}{n!}
\Expect{ \Mc_{a'}^{\Xi_{a,\eps}^z}(dz') F(z', \Xi_{a,\eps}^z) \vert \# \Xi_{a,\eps}^z = n } \\
\nonumber
& =
\lim_{\eps \to 0} e^{-2\pi a Z_\eps} \sum_{n=0}^\infty \frac{(2\pi a Z_\eps)^n}{n!} \sum_{k=1}^n \binom{n}{k} \Expect{ \Mc_{a'}^{\wp_\eps^1 \cap \dots \cap \wp_\eps^k}(dz') F(z', \wp_\eps^1 \wedge \dots \wedge \wp_\eps^n )},
\end{align}
where $\wp_{\eps}^i, i = 1 \dots n$, are i.i.d. trajectories distributed according to \eqref{eq:2p4}. The binomial coefficient corresponds to the number of ways to choose $k$ trajectories that actually visit $z'$ among the collection of $n$ trajectories.
We now use the disintegration formula \eqref{eq:disintegration} to specify the contribution of each of the $k$ trajectories. To ease notations, in the following computations we denote by $D_i = D \cap \H_{m^i}$ and we write with some abuse of notation a product of integrals instead of multiple integrals. 
Also, by independence, the shift by the above intersection measure will not have any impact on $\wp_\eps^{k+1}, \dots, \wp_\eps^n$. We will therefore remove these trajectories from the computations and add them back when it will be necessary.
We have
\begin{align}
\label{eq:2p5}
& \Expect{ \Mc_{a'}^{\wp_\eps^1 \cap \dots \cap \wp_\eps^k}(dz') F(z', \wp_\eps^1 \wedge \dots \wedge \wp_\eps^k )}
= \frac{1}{Z_\eps^k} \int_{\mathsf{a'} \in E(a',k)} \d \mathsf{a}' \left( \prod_{i=1}^k \int_{\mi(D)}^{m_z-\eps} \d m^i \int_{(\R +im^i)\cap D} \d z_\bot^i \right) \\
\nonumber
& \times \left( \prod_{i=1}^k H_{D_i}(z,z_\bot^i)^2 \right)
\Expect{ \bigcap_{i=1}^k \Mc_{a_i'}^{\wp_{D_i}^{z,z_\bot^i} \wedge \wp_{D_i}^{z_\bot^i,z}} (dz') F \left(z', \bigwedge_{i=1}^k ( \wp_{D_i}^{z,z_\bot^i} \wedge \wp_{D_i}^{z_\bot^i,z} ) \right) }.
\end{align}
We decompose for all $i=1 \dots k$,
\[
\Mc_{a_i'}^{\wp_{D_i}^{z,z_\bot^i} \wedge \wp_{D_i}^{z_\bot^i,z}} = \Mc_{a_i'}^{\wp_{D_i}^{z,z_\bot^i}} + \Mc_{a_i'}^{\wp_{D_i}^{z_\bot^i,z}} + \Mc_{a_i'}^{\wp_{D_i}^{z,z_\bot^i} \cap \wp_{D_i}^{z_\bot^i,z}}
\]
and we then expand
\begin{align}
\label{eq:cr1}
& \Expect{ \bigcap_{i=1}^k \Mc_{a_i'}^{\wp_{D_i}^{z,z_\bot^i} \wedge \wp_{D_i}^{z_\bot^i,z}} (dz') F \left(z', \bigwedge_{i=1}^k ( \wp_{D_i}^{z,z_\bot^i} \wedge \wp_{D_i}^{z_\bot^i,z} ) \right) } \\
\nonumber
& = \sum_{I_1, I_2, I_3} \Expect{ \bigcap_{i \in I_1} \Mc_{a_i'}^{\wp_{D_i}^{z,z_\bot^i}} \cap \bigcap_{i \in I_2} \Mc_{a_i'}^{\wp_{D_i}^{z_\bot^i,z}} \cap \bigcap_{i \in I_3} \Mc_{a_i'}^{\wp_{D_i}^{z,z_\bot^i} \cap \wp_{D_i}^{z_\bot^i,z}} (dz') F \left(z', \bigwedge_{i=1}^k ( \wp_{D_i}^{z,z_\bot^i} \wedge \wp_{D_i}^{z_\bot^i,z} ) \right) }
\end{align}
where the sum runs over all partition of $\{1, \dots, k\}$ in three subsets $I_1, I_2$ and $I_3$.
Recall that the disintegration formula \eqref{eq:disintegration} yields
\begin{equation}
\label{eq:cr3}
\Mc_{a_i'}^{\wp_{D_i}^{z,z_\bot^i} \cap \wp_{D_i}^{z_\bot^i,z}}
= \int_0^{a_i'} \d \alpha_i'~ \Mc_{a_i' - \alpha_i'}^{\wp_{D_i}^{z,z_\bot^i}} \cap \Mc_{\alpha_i'}^{\wp_{D_i}^{z_\bot^i,z}}.
\end{equation}
Now, by \eqref{eq:Girsanov_intersection}, the expectation in the left hand side of \eqref{eq:cr1} is equal to
\begin{align}
\label{eq:cr2}
& \sum_{I_1, I_2, I_3} \left( \prod_{i \in I_1 \cup I_2} \frac{H_{D_i}(z',z_\bot^i)}{H_{D_i}(z,z_\bot^i)} 2\pi G_{D_i}(z,z') \CR(z',D_i)^{a_i'} \right) \left( \prod_{i \in I_3} a_i' \frac{H_{D_i}(z',z_\bot^i)^2}{H_{D_i}(z,z_\bot^i)^2} (2\pi)^2 G_{D_i}(z,z')^2 \CR(z',D_i)^{a_i'} \right) \\
\nonumber
& \times \E \Big[ F \Big(z', \bigwedge_{i \in I_1} (\wp_{D_i}^{z,z'} \wedge \wp_{D_i}^{z',z_\bot^i} \wedge \wp_{D_i}^{z_\bot^i, z} \wedge \Xi_{a_i'}^{z',D_i} ) \wedge \bigwedge_{i \in I_2} (\wp_{D_i}^{z,z^i_\bot} \wedge \wp_{D_i}^{z_\bot^i, z'} \wedge \wp_{D_i}^{z',z} \wedge \Xi_{a_i'}^{z',D_i} ) \\
\nonumber
& ~~~~~~~~~~~~~~\wedge \bigwedge_{i \in I_3} ( \wp_{D_i}^{z,z'} \wedge \wp_{D_i}^{z',z} \wedge \Xi_{a_i}^{z',D_i} \wedge \wp_{D_i}^{z',z_\bot^i} \wedge \wp_{D_i}^{z_\bot^i,z'} ) \Big) \Big].
\end{align}
See Figure \ref{fig2} for an illustration of the trajectories appearing on the right hand side of the above display.
Note that the $a_i'$ in the product over $i \in I_3$ comes from the integration of $\alpha_i'$ in \eqref{eq:cr3}. 
When we will plug this back in \eqref{eq:2p5}, we will have to multiply everything with the product of Poisson kernel $H_{D_i}(z,z_\bot^i)^2$. This latter product times the two products in parenthesis in \eqref{eq:cr2} can be rewritten as
\begin{align*}
& \Big( \prod_{i \in I_1 \cup I_2} H_{D_i}(z',z_\bot^i) H_{D_i}(z,z_\bot^i) 2\pi G_{D_i}(z,z') \CR(z',D_i)^{a_i'} \Big) \Big( \prod_{i \in I_3} a_i' H_{D_i}(z',z_\bot^i)^2 (2\pi)^2 G_{D_i}(z,z')^2 \CR(z',D_i)^{a_i'} \Big) \\
& = (2\pi)^k \CR(z',D)^{a'} G_D(z,z')^{2k} \Big( \prod_{i \in I_1 \cup I_2} \Big\{ H_{D_i}(z',z_\bot^i) H_{D_i}(z,z_\bot^i) \frac{G_{D_i}(z,z')}{G_D(z,z')^2} \Big\} \frac{\CR(z',D_i)^{a_i'}}{\CR(z',D)^{a_i'}} \Big) \\
& \times \Big( \prod_{i \in I_3} \Big\{ 2\pi a_i' H_{D_i}(z',z_\bot^i)^2 \frac{\CR(z',D_i)^{a_i'}}{\CR(z',D)^{a_i'}} \Big\} \frac{G_{D_i}(z,z')^2}{G_D(z,z')^2} \Big).
\end{align*}
We recognise in the brackets above the density of the point $z_\bot^i$ with minimal imaginary part in the loops $\wp_D^{z,z',i} \wedge \wp_D^{z',z,i}$ ($i \in I_1 \cup I_2$) and $\Xi_{a_i'}^{z,D}$ ($i \in I_3$); see Lemmas \ref{lem:loop_minimum} and \ref{lem:ppp_mininum}. The term $\CR(z',D_i)^{a_i'}/\CR(z',D)^{a_i'}$ (resp. $G_{D_i}(z,z')^2 / G_D(z,z')^2$) corresponds to the probability for $\Xi_{a_i'}^{z',D}$ (resp. $\wp_D^{z,z',i} \wedge \wp_D^{z',z,i}$) to stay in $D_i$. 
To make this more precise, we introduce the following events: for all $i = 1\dots k$, let $E_1^i(\eps)$, resp. $E_2^i(\eps)$ and $E_3^i(\eps)$, be the event that the minimal height among the trajectories $\wp_D^{z,z',i}$, $\wp_D^{z',z,i}$ and $\Xi_{a_i'}^{z',D}$ is smaller than $m_z - \eps = \Im(z) - \eps$ and is reached by the first, resp. second and third.
Lemmas \ref{lem:loop_minimum} and \ref{lem:ppp_mininum} imply that for all $\beta \in \{1,2\}$, for all $i \in I_\beta$,
\begin{align*}
& \int_{\mi(D)}^{m_z-\eps} \d m^i \int_{(\R +im^i)\cap D} \d z_\bot^i
\Big\{ H_{D_i}(z',z_\bot^i) H_{D_i}(z,z_\bot^i) \frac{G_{D_i}(z,z')}{G_D(z,z')^2} \Big\} \frac{\CR(z',D_i)^{a_i'}}{\CR(z',D)^{a_i'}} \\
& \times \Expect{ F(z', \wp_{D_i}^{z,z'} \wedge \wp_{D_i}^{z',z_\bot^i} \wedge \wp_{D_i}^{z_\bot^i, z} \wedge \Xi_{a_i'}^{z',D_i} ) } \\
& = \Expect{ F(z', \wp_D^{z,z',i} \wedge \wp_D^{z',z,i} \wedge \Xi_{a_i'}^{z',D} ) \mathbf{1}_{E_\beta^i(\eps)} }
\end{align*}
and that for all $i \in I_3$,
\begin{align*}
& \int_{\mi(D)}^{m_z-\eps} \d m^i \int_{(\R +im^i)\cap D} \d z_\bot^i
\Big\{ 2\pi a_i' H_{D_i}(z',z_\bot^i)^2 \frac{\CR(z',D_i)^{a_i'}}{\CR(z',D)^{a_i'}} \Big\} \frac{G_{D_i}(z,z')^2}{G_D(z,z')^2} \\
& \times \Expect{ F(z', \wp_{D_i}^{z,z'} \wedge \wp_{D_i}^{z',z} \wedge \Xi_{a_i}^{z',D_i} \wedge \wp_{D_i}^{z',z_\bot^i} \wedge \wp_{D_i}^{z_\bot^i,z'} )}
\\ 
& = \Expect{ F(z', \wp_D^{z,z',i} \wedge \wp_D^{z',z,i} \wedge \Xi_{a_i'}^{z',D} ) \mathbf{1}_{E_3^i(\eps)} }.
\end{align*}
Overall, and going back to \eqref{eq:2p5}, we have obtained that
\begin{align*}
& \Expect{ \Mc_{a'}^{\wp_\eps^1 \cap \dots \cap \wp_\eps^k}(dz') F(z', \wp_\eps^1 \wedge \dots \wedge \wp_\eps^k )}
= \frac{1}{Z_\eps^k} (2\pi)^k G_D(z,z')^{2k} \int_{\mathsf{a'} \in E(a',k)} \d \mathsf{a}' \\
& \times \sum_{I_1, I_2, I_3}
\Expect{ F \Big( z', \bigwedge_{i=1}^k (\wp_D^{z,z',i} \wedge \wp_D^{z',z,i} \wedge \Xi_{a_i'}^{z',D}) \Big) \prod_{i \in I_1} \mathbf{1}_{E_1^i(\eps)} \prod_{i \in I_2} \mathbf{1}_{E_2^i(\eps)} \prod_{i \in I_3} \mathbf{1}_{E_3^i(\eps)}  }.
\end{align*}
Plugging this back in \eqref{eq:2p7} and remembering that we have to add the trajectories $\wp_\eps^{k+1}, \dots, \wp_\eps^n$, we see that the left hand side of \eqref{eq:lem_second_crucial} is equal to
\begin{align*}
& \lim_{\eps \to 0} e^{-2\pi a Z_\eps} \sum_{n=0}^\infty \sum_{k=1}^n  \frac{(2\pi a Z_\eps)^{n-k}}{(n-k)!k!} (2\pi \sqrt{a} G_D(z,z'))^{2k} \int_{\mathsf{a'} \in E(a',k)} \d \mathsf{a}' \\
& \times \sum_{I_1, I_2, I_3}
\Expect{ F \Big( z', \bigwedge_{i=1}^k (\wp_D^{z,z',i} \wedge \wp_D^{z',z,i} \wedge \Xi_{a_i'}^{z',D}) \wedge \wp_\eps^{k+1} \wedge \dots \wedge \wp_\eps^n \Big) \prod_{\beta \in \{1,2,3\}} \prod_{i \in I_\beta} \mathbf{1}_{E_\beta^i(\eps)} } \\
& = \sum_{k=1}^\infty \frac{(2\pi \sqrt{a} G_D(z,z'))^{2k}}{k!} \int_{\mathsf{a'} \in E(a',k)} \d \mathsf{a}' \sum_{I_1, I_2, I_3}
\Expect{ F \Big( z', \bigwedge_{i=1}^k (\wp_D^{z,z',i} \wedge \wp_D^{z',z,i} \wedge \Xi_{a_i'}^{z',D}) \wedge \Xi_a^{z,D} \Big) \prod_{\beta \in \{1,2,3\}} \prod_{i \in I_\beta} \mathbf{1}_{E_\beta^i(\eps=0)} }.
\end{align*}
Since,
\[
\sum_{I_1, I_2, I_3} \prod_{i \in I_\beta} \mathbf{1}_{E_\beta^i(\eps=0)} = 1,
\]
we can use additivity of Poisson point processes and then the fact that the Lebesgue measure of the simplex  $E(a',k)$ is equal to $(a')^{k-1}/(k-1)!$ to obtain that the left hand side of \eqref{eq:lem_second_crucial} is equal to
\begin{align*}
& \sum_{k=1}^\infty \frac{(2\pi \sqrt{a} G_D(z,z'))^{2k}}{k!}
\Expect{ F \Big( z', \bigwedge_{i=1}^k (\wp_D^{z,z',i} \wedge \wp_D^{z',z,i} ) \wedge \Xi_{a'}^{z',D} \wedge \Xi_a^{z,D} \Big) } \int_{\mathsf{a'} \in E(a',k)} \d \mathsf{a}' \\
& = \frac{1}{a'} \sum_{k=1}^\infty \frac{(2\pi \sqrt{aa'} G_D(z,z'))^{2k}}{k!(k-1)!} \Expect{ F \Big( z', \bigwedge_{i=1}^k (\wp_D^{z,z',i} \wedge \wp_D^{z',z,i} ) \wedge \Xi_{a'}^{z',D} \wedge \Xi_a^{z,D} \Big) } .
\end{align*}
This is the right hand side of \eqref{eq:lem_second_crucial} which concludes the proof.
\end{proof}

\subsection{Second moment}

Combining Lemma \ref{lem:second_moment_loopmeasure} with a 
Palm formula type of argument, we obtain the following expression for the second moment of functionals of our measure. See Figure \ref{fig3} for an illustration of the loop $\Xi_{a,a'}^{z,z'} \wedge \Xi_a^z \wedge \Xi_{a'}^{z'}$.

\begin{lemma}
\label{lem:second_moment}
For any bounded measurable admissible function $F = F(z, z', \Lc)$ of a pair of points $z, z'$ and a collection loops $\Lc$, we have:
\begin{align*}
& \E[ F(z,z', \Lc_D^\theta) \Mc_a^K(dz) \Mc_{a'}^K(dz') ] =
\CR(z,D)^a \CR(z',D)^{a'} \\
& \times \sum_{\substack{n,m \geq 1 \\ 0 \leq l \leq n \wedge m}} \frac{\theta^{n+m-l}}{(n-l)!(m-l)!l!}
\int_{\substack{\mathsf{a} \in E(a,n)\\ \mathsf{a}' \in E(a',m)}} \frac{\d \mathsf{a}}{a_1 \dots a_n} \frac{\d \mathsf{a}'}{a_1' \dots a_m'} \prod_{i=1}^l \Bs \left( (2\pi)^2 a_i a_i' G_D(z,z')^2 \right) \\
& \times \E \Bigg[ \prod_{i=1}^l \left( 1 - e^{-K T(\Xi_{a_i,a_i'}^{z,z'} \wedge \Xi_{a_i}^z \wedge \Xi_{a_i'}^{z'} )} \right) \prod_{i=l+1}^n \left( 1 - e^{-KT(\Xi_{a_i}^z)} \right) \prod_{i=l+1}^m \left( 1 - e^{-KT(\Xi_{a_i'}^{z'})} \right) \\
& ~~~~~~ F \left( z,z', \Lc_D^\theta \cup \{ \Xi_{a_i,a_i'}^{z,z'} \wedge \Xi_{a_i}^z \wedge \Xi_{a_i'}^{z'} \}_{i = 1}^l \cup \{ \Xi_{a_i}^z \}_{i=l+1}^{n} \cup \{ \Xi_{a_i'}^{z'} \}_{i=l+1}^{m} \right) \Bigg] \d z \d z'
\end{align*}
where all the above processes are independent.
\end{lemma}

\begin{proof}
In what follows, to shorten notations, we will write with some abuse of notation ``$\wp_1 \neq \dots \neq \wp_n \in \Lc_D^\theta(K)$'' instead of ``$\wp_1, \dots, \wp_n \in \Lc_D^\theta(K)$ and for all $i \neq j$, $\wp_i \neq \wp_j$''. By definition of $\Mc_a^K$, $\E F(z,z', \Lc_D^\theta) \Mc_a^K(dz) \Mc_{a'}^K(dz')$ is equal to
\begin{align*}
& \sum_{n,m \geq 1} \frac{1}{n!m!} \E \sum_{\substack{\wp_1 \neq \dots \neq \wp_n \in \Lc_D^\theta(K)\\\wp_1' \neq \dots \neq \wp_m' \in \Lc_D^\theta(K)}} F(z,z', \Lc_D^\theta) \Mc_a^{\wp_1 \cap \dots \wp_n}(dz) \Mc_{a'}^{\wp_1' \cap \dots \wp_m'}(dz') \\
& = \sum_{n,m \geq 1} \frac{1}{n!m!} \sum_{l=0}^{n \wedge m} l! \binom{n}{l} \binom{m}{l} \\
& \times \E \sum_{\substack{\wp_1 \neq \dots \neq \wp_n \neq \wp_{l+1}' \neq \dots \neq \wp_m' \in \Lc_D^\theta(K)\\\forall i =1 \dots l, \wp_i' = \wp_i}} F(z,z', \Lc_D^\theta) \Mc_a^{\wp_1 \cap \dots \wp_n}(dz) \Mc_{a'}^{\wp_1' \cap \dots \wp_m'}(dz').
\end{align*}
$l$ represents the number of loops that are in both sets of loops. $\binom{n}{l}$ (resp. $\binom{m}{l}$) is the number of ways to choose a subset of $l$ loops in the set of $n$ loops (resp. $m$ loops). $l!$ is then the number of ways to map one subset to the other.
Fix now, $n,m \geq 1$ and $l \in \{0, \dots, n \wedge m \}$. By Palm's formula, the expectation above is equal to $\theta^{n+m-l}$ times
\begin{align*}
& \int \loopmeasure_D(d \wp_1) \dots \loopmeasure(\wp_n) \loopmeasure_D(d \wp_{l+1}') \dots \loopmeasure(\wp_m') \prod_{i=1}^n \left( 1 - e^{-K T(\wp_i)} \right) \prod_{i=l+1}^m \left( 1 - e^{-K T(\wp_i')} \right) \\
& \Expect{ F(z,z',\Lc_D^\theta \cup \{ \wp_i \}_{i=1}^n \cup \{ \wp_i'\}_{i=l+1}^m } \Mc_a^{\wp_1 \cap \dots \wp_n}(dz) \Mc_{a'}^{\wp_1 \dots \cap \wp_l \cap \wp_{l+1}' \cap \dots \wp_m'}(dz')
\end{align*}
where the expectation is only with respect to $\Lc_D^\theta$. Lemma \ref{lem:second_moment_loopmeasure} concludes the proof.
\end{proof}

In particular, Lemma \ref{lem:second_moment} gives an explicit formula for the second moment. Indeed, we have already seen that
\[
\Expect{e^{-KT(\Xi_a^z)}} = e^{-aC_K(z)}.
\]
Moreover,
\begin{align*}
\Expect{ e^{-KT(\Xi^{z,z'}_{a,a'})} }
& = \frac{1}{\Bs((2\pi)^2 a a' G_D(z,z')^2 )} \sum_{k=1}^\infty \frac{(2\pi \sqrt{a a'} G_D(z,z'))^{2k}}{k! (k-1)!} \Expect{ e^{-KT(z \to z')} }^{2k} \\
& = \frac{\Bs \left( (2\pi)^2 a a' G_D(z,z')^2 \Expect{ e^{-KT(z \to z')} }^2 \right)}{\Bs((2\pi)^2 a a' G_D(z,z')^2 )},
\end{align*}
where in the above $T(z \to z')$ is the running time of an excursion from $z$ to $z'$ distributed according to $\mu^{z,z'}_D / G_D(z,z')$. We further notice that
\[
G_D(z,z') \Expect{ e^{-KT(z \to z')} } = G_{D,K}(z,z').
\]
Overall, this shows that $\E[ \Mc_a^K(dz) \Mc_{a'}^K(dz') ]$ is equal to
\begin{align}
\nonumber
& \CR(z,D)^a \CR(z',D)^{a'} \sum_{\substack{n,m \geq 1 \\ 0 \leq l \leq n \wedge m}} \frac{\theta^{n+m-l}}{(n-l)!(m-l)!l!}
\int_{\substack{\mathsf{a} \in E(a,n)\\ \mathsf{a}' \in E(a',m)}} \frac{\d \mathsf{a}}{a_1 \dots a_n} \frac{\d \mathsf{a}'}{a_1' \dots a_m'}  \\
\nonumber
& \times \prod_{i=1}^l \left( \Bs \left( (2\pi)^2 a_i a_i' G_D(z,z')^2 \right) - e^{-a_iC_K(z) - a_i'C_K(z')} \Bs \left( (2\pi)^2 a_i a_i' G_{D,K}(z,z')^2 \right) \right) \\
& \times \prod_{i=l+1}^n \left( 1 - e^{-a_i C_K(z)} \right) \prod_{i=l+1}^m \left( 1 - e^{-a_i' C_K(z')} \right) \d z \d z'.
\label{eq:second_moment_killed_loops}
\end{align}
The purpose of the next section is to study the asymptotic properties of this expression. This will basically conclude the proof of Theorem \ref{th:convergence_continuum} in the $L^2$-phase, but this will also be useful in order to go beyond this phase to cover the whole $L^1$-phase.

\subsection{Simplifying the second moment}

Let $a,a'>0$. Recall the definition \eqref{eq:def_Gs} of $\Bs$ and define for all $u, u', v \geq 0$,
\begin{align}
\Hs_{a,a'}(u,u',v) & := \sum_{\substack{n,m \geq 1\\0 \leq l \leq n \wedge m}} \frac{\theta^{n+m-l}}{(n-l)!(m-l)!l!} \nonumber \\
& \int_{\substack{\mathsf{a} \in E(a,n)\\ \mathsf{a}' \in E(a',m)}} \d \mathsf{a} \d \mathsf{a}'
\prod_{i=1}^l \frac{\Bs(va_i a_i')}{a_i a_i'} \prod_{i=l+1}^n \frac{1-e^{-ua_i}}{a_i} \prod_{i=l+1}^m \frac{1-e^{-u'a_i'}}{a_i'}.
\label{eq:def_Hs}
\end{align}
$v$ will be taken to be a multiple of $G_D(z,z')^2$, whereas $u$ and $u'$ will coincide with $C_K(z)$ and $C_K(z')$ respectively.

To get an upper bound on the second moment, we will start from the expression \eqref{eq:second_moment_killed_loops}, and bound the second line in that expression with a quantity that does not depend on $K$. We do so simply by ignoring the second term in the product, which leads to the expression for $\Hs$ in \eqref{eq:def_Hs}. Intuitively, this amounts to ignoring the requirements that the loops that visit both $z$ and $z'$ are killed. Indeed, since $z$ and $z'$ are typically macroscopically far away, such loops will be killed with high probability and so ignoring the requirement gives us a good upper bound.

\begin{lemma}\label{lem:Hs}
Let $a,a' >0$. There exists $C >0$ such that for all $u, u' \geq 1, v > 0$,
\begin{equation}
\label{eq:lem_h_upper_bound}
\Hs_{a,a'}(u,u',v) \leq C (uu')^\theta v^{1/4-\theta/2} e^{2 \sqrt{vaa'}}.
\end{equation}
Moreover, for all $v > 0$,
\begin{equation}
\label{eq:lem_h_asymptotic}
\lim_{u, u' \to \infty} \frac{\Hs_{a,a'}(u,u',v)}{(uu')^\theta} = \frac{1}{\Gamma(\theta)} \left( \frac{aa'}{v} \right)^{\frac{\theta -1}{2}} I_{\theta-1} \left( 2 \sqrt{v aa'} \right),
\end{equation}
where $I_{\theta-1}$ is given by 
\eqref{eq:modified_Bessel_function}.
In particular, when $\theta =1/2$, for all $v >0$,
\begin{equation}
\label{eq:lem_h_asymptotic_1/2}
\lim_{u, u' \to \infty} \frac{\Hs_{a,a'}(u,u',v)}{\sqrt{uu'}} = \frac{1}{\pi\sqrt{aa'}} \cosh(2 \sqrt{vaa'}).
\end{equation}
\end{lemma}

\begin{proof}[Proof of Lemma \ref{lem:Hs}]
We start off by doing the change of variable $(n,m,l) \leftarrow (n-l,m-l,l)$ and obtain using Lemma \ref{lem:Fs} that $\Hs_{a,a'}(u,u',v)$ is equal to
\begin{align}
\nonumber
& \sum_{\substack{n,m \geq 1 \text{~and~} l \geq 0 \\\text{or~} n=m=0 \text{~and~} l \geq 1}} \frac{\theta^{n+m+l}}{n!m!l!} \int_{\substack{\mathsf{a} \in E(a,n+l)\\ \mathsf{a}' \in E(a',m+l)}} \d \mathsf{a} \d \mathsf{a}'
\prod_{i=1}^l \frac{\Bs(va_i a_i')}{a_i a_i'} \prod_{i=l+1}^{n+l} \frac{1-e^{-ua_i}}{a_i} \prod_{i=l+1}^{m+l} \frac{1-e^{-u'a_i'}}{a_i'} \\
& = \sum_{l \geq 0} \frac{\theta^l}{l!} \int_{\substack{\mathsf{a} \in E(a,l+1) \\ \mathsf{a}' \in E(a',l+1)}} \d\mathsf{a} \d\mathsf{a}' \frac{\Fs(ua_{l+1})}{a_{l+1}} \frac{\Fs(u'a_{l+1}')}{a_{l+1}'} \prod_{i=1}^l \frac{\Bs(va_i a_i')}{a_i a_i'}
+ \sum_{l \geq 1} \frac{\theta^l}{l!} \int_{\substack{\mathsf{a} \in E(a,l) \\ \mathsf{a}' \in E(a',l)}} \d\mathsf{a} \d\mathsf{a}' \prod_{i=1}^l \frac{\Bs(va_i a_i')}{a_i a_i'}.
\label{eq:proof_lem_second_moment1}
\end{align}
Let us explain briefly where this comes from. The first term is the ``off-diagonal'' term corresponding to $n, m \ge 1$ and $l \ge 0$, while the second term is the ``on-diagonal'' term corresponding to $n = m = 0$ and $l \ge 1$. Furthermore, to get the expression of the first term, we reason as follows. The term $\d \mathsf{a}$ in the first line concerns $n+l$ variables, $a_1, \ldots, a_{n+l}$, whose sum is fixed equal to $a$. We freeze $a_1, \ldots, a_l$ and first integrate over $a_{l+1}, \ldots a_{l+n}$. We may call $a_{l+1}$ the sum of these $n$ variables; thus $a_1 + \ldots + a_{l+1} = a$. Summing over $n$ and applying  Lemma \ref{lem:Fs} we recognise the expression for $\Fs (u a_{l+1})/ a_{l+1}$. The same can be done separately for $\d \mathsf{a}'$, leading us to the claimed expression.

Now, let $l \geq 1$, $\alpha, \alpha' >0$ and let us note that by definition of $\mathsf{B}$,
\begin{align*}
& \int_{\substack{\mathsf{a} \in E(\alpha,l)\\ \mathsf{a}' \in E(\alpha',l)}} \d \mathsf{a} \d \mathsf{a}'
\prod_{i=1}^l \frac{\Bs(va_i a_i')}{a_i a_i'}
= \int_{\substack{\mathsf{a} \in E(\alpha,l)\\ \mathsf{a}' \in E(\alpha',l)}} \d \mathsf{a} \d \mathsf{a}' \sum_{k_1, \dots, k_l \geq 1} v^{k_1 + \dots  + k_l} \prod_{i=1}^l \frac{(a_i a_i')^{k_i-1}}{k_i! (k_i-1)!}.
\end{align*}
Using the fact that for all $\beta >0$, $k,k' \geq 1$,
\[
\int_0^\beta \frac{x^{k-1}}{(k-1)!} \frac{(\beta - x)^{k'-1}}{(k'-1)!} \d x = \frac{\beta^{k+k'-1}}{(k+k'-1)!},
\]
we find by induction that
\begin{align*}
\int_{\mathsf{a} \in E(\alpha,l)} \d \mathsf{a} \prod_{i=1}^l \frac{a_i^{k_i-1}}{(k_i-1)!}
= \frac{\alpha^{k_1 + \dots + k_l -1}}{(k_1 + \dots + k_l - 1)!}.
\end{align*}
Hence
\begin{align*}
\int_{\substack{\mathsf{a} \in E(\alpha,l)\\ \mathsf{a}' \in E(\alpha',l)}} \d \mathsf{a} \d \mathsf{a}'
\prod_{i=1}^l \frac{\Bs(va_i a_i')}{a_i a_i'}
& = \sum_{k_1, \dots, k_l \geq 1} v^{k_1 + \dots  + k_l} \frac{(\alpha \alpha')^{k_1 + \dots + k_l -1}}{(k_1 + \dots + k_l -1)!^2} \prod_{i=1}^l \frac{1}{k_i} \\
& = \sum_{k \geq l} v^k \frac{(\alpha \alpha')^{k -1}}{(k -1)!^2} \sum_{\substack{k_1, \dots, k_l \geq 1\\ k_1 + \dots + k_l = k}} \prod_{i=1}^l \frac{1}{k_i}
\end{align*}
and
\begin{align*}
\sum_{l \geq 1} \frac{\theta^l}{l!}
\int_{\substack{\mathsf{a} \in E(\alpha,l)\\ \mathsf{a}' \in E(\alpha',l)}} \d \mathsf{a} \d \mathsf{a}'
\prod_{i=1}^l \frac{\Bs(va_i a_i')}{a_i a_i'}
= \sum_{k \geq 1} v^k \frac{(\alpha \alpha')^{k-1}}{(k-1)!^2} \sum_{l=1}^k \frac{\theta^l}{l!} \sum_{\substack{k_1, \dots, k_l \geq 1\\ k_1 + \dots + k_l = k}} \prod_{i=1}^l \frac{1}{k_i}.
\end{align*}
Looking at the series expansion of $(1-x)^{-\theta}$ near 0 and recalling the definition \eqref{eq:def_theta_n} of $\theta^{(k)}$, we see that for all $k \geq 1$,
\[
\sum_{l=1}^k \frac{\theta^l}{l!} \sum_{\substack{k_1, \dots, k_l \geq 1\\ k_1 + \dots + k_l = k}} \prod_{i=1}^l \frac{1}{k_i}
= \frac{\theta^{(k)}}{k!}.
\]
We deduce that
\begin{align}
\label{eq:proof_integral_product_Bs}
\sum_{l \geq 1} \frac{\theta^l}{l!}
\int_{\substack{\mathsf{a} \in E(\alpha,l)\\ \mathsf{a}' \in E(\alpha',l)}} \d \mathsf{a} \d \mathsf{a}'
\prod_{i=1}^l \frac{\Bs(va_i a_i')}{a_i a_i'}
= \sum_{k \geq 1} v^k \frac{(\alpha \alpha')^{k-1}}{(k-1)!^2} \frac{\theta^{(k)}}{k!}.
\end{align}
Taking $\alpha =a$, $\alpha' =a'$, this gives an expression for the second term of \eqref{eq:proof_lem_second_moment1}.
As for the first term in \eqref{eq:proof_lem_second_moment1}, it can be computed in a similar manner: namely, we get
\begin{align}
\label{eq:proof_lem_second_moment2}
\frac{\Fs(ua)}{a} \frac{\Fs(u'a')}{a'} +
\sum_{k \geq 1} \frac{v^k}{(k-1)!^2} \frac{\theta^{(k)}}{k!} \int_0^a \d \alpha \int_0^{a'} \d \alpha' (\alpha \alpha')^{k-1} \frac{\Fs(u(a-\alpha))}{a-\alpha} \frac{\Fs(u'(a'-\alpha'))}{a'-\alpha'}.
\end{align}
with $\alpha = a- a_{l+1} = a_1 + \ldots a_l$ and, respectively, $\alpha' = a' - a'_{l+1} = a'_1 + \ldots + a'_{l}$. 

We then use \eqref{eq:lem_upper_bound_Fs} and \eqref{eq:beta_function} to bound
\[
\int_0^a \alpha^{k-1} \frac{\Fs(u(a-\alpha))}{a-\alpha} \d \alpha
\leq C u^\theta \int_0^a \frac{\alpha^{k-1}}{(a-\alpha)^{1-\theta}} \d \alpha = Cu^\theta a^{k-1+\theta} \frac{(k-1)!}{\theta^{(k)}}.
\]
We finally find that the first term of \eqref{eq:proof_lem_second_moment1} is at most
\[
C (uu')^\theta (aa')^{\theta-1} + C (uu')^\theta (aa')^{\theta-1} \sum_{k \geq 1} \frac{(vaa')^k}{k! \theta^{(k)}} = C (uu')^\theta (aa'/v)^{\theta/2-1/2} \Gamma(\theta) I_{\theta -1}(2 \sqrt{vaa'}).
\]
The second term of \eqref{eq:proof_lem_second_moment1} can be bounded by $\cosh(2 \sqrt{vaa'})$. This concludes the proof of \eqref{eq:lem_h_upper_bound}. \eqref{eq:lem_h_asymptotic} follows as well by using the asymptotic $\Fs(w) \sim w^\theta / \Gamma(\theta)$ as $w \to \infty$ and by applying dominated convergence theorem in \eqref{eq:proof_lem_second_moment2}. \eqref{eq:lem_h_asymptotic_1/2} follows from \eqref{eq:modified_Bessel_function-1/2} and \eqref{eq:gamma_function_1/2}.
\end{proof}

As a consequence, we obtain the following estimates on the second moment of $\Mc_a^K$.

\begin{corollary}\label{cor:second_moment_simplified}
There exists $C>0$ such that for all $K \geq 1$, $z,z' \in D$, $a, a' \in (0,2)$,
\[
\Expect{\Mc_a^K(dz) \Mc_{a'}^K(dz')} \leq C(aa')^{\theta/2-3/4} (\log K)^{2\theta} G_D(z,z')^{1/2-\theta} \exp \left( 4 \pi \sqrt{aa'} G_D(z,z') \right)
\d z \d z'.
\]
Moreover,
\[
\lim_{K \to \infty} \frac{\Expect{\Mc_a^K(dz) \Mc_{a'}^K(dz')}}{(\log K)^{2\theta}} = \frac{1}{4^{\theta}\Gamma(\theta)} \left( \frac{\sqrt{aa'}}{2 \pi G_D(z,z')} \right)^{\theta-1} I_{\theta-1}\left( 4\pi \sqrt{aa'} G_D(z,z') \right)
\d z \d z'.
\]
In particular, when $\theta=1/2$,
\[
\lim_{K \to \infty} \frac{\Expect{\Mc_a^K(dz) \Mc_{a'}^K(dz')}}{\log K} = \frac{1}{2\pi\sqrt{aa'}} \cosh \left( 4\pi \sqrt{aa'} G_D(z,z') \right)\d z \d z'.
\]
\end{corollary}

\section{Going beyond the \texorpdfstring{$L^2$}{L2}-phase}
\label{sec:beyond_L2}

The goal of this section is to prove Proposition \ref{prop:oscillations_Mc}.
We now describe the proof at a high level. When the thickness parameter $a$ is smaller than 1, we can directly apply Cauchy--Schwarz inequality and control the second moment. (This could be done directly using Corollary \ref{cor:second_moment_simplified}). However when $a \in [1,2)$, the second moment blows up. The broad strategy is by now well understood and consists in introducing ``good events'', similar to \cite{BerestyckiGMC}. In our context, this good event at a given point $z \in D$ will take the following form: we will require that the total number of crossings of each dyadic annulus centred at $z$ is upper bounded at each scale by some given scale-dependent quantity (see \eqref{eq:def_good_event}). On the one hand, adding these events does not change the measure with high probability (Lemma \ref{lem:first_moment_good_event}). On the other hand, the measure restricted to the good events has a finite second moment which varies smoothly with respect to the thickness parameter (Lemma \ref{lem:second_moment_aa'}).

In the entire section, we will fix a set $A \Subset D$ compactly included in $D$. We will always restrict our attention to points lying in $A$ and the estimates that we obtain may depend on $A$. We will only provide the proof of Proposition \ref{prop:oscillations_Mc} in the case $\thk < a$ (which is in fact all that we use). The case $\thk>a$ would be similar as we have assumed $a>0$.

We start by defining the ``\textbf{good events}'' that we will work with.
For any countable collection $\Cc$ of Brownian-like loops and for any $r>0$ and $z \in D$, we define $N_{z,r}^\Cc$ to be the number of crossings from $\partial D(z,r)$ to $\partial D(z,er)$ in $\Cc$ (upward crossings, we do not count the way back). That is, $N_{z,r}^\Cc = \sum_{\wp \in \Cc} N^{\wp}_{z,r}$, and $N^{\wp}_{z,r}$ is the number of upcrossings of the interval $[r,er]$ by the function $|\wp(\cdot) - z|$. Note that this is an admissible functional of $\Cc $ and $z$.

Recall that the parameter $a \in (0,2)$ is the thickness parameter which is fixed throughout this paper. We now choose $a< b< 2$ sufficiently close to $a$ (in a way which will be specified later). Let $r_0 \in (0,1)$ be small. For a given $z \in D$, we consider the good event
\begin{equation}
\label{eq:def_good_event}
\Gc_K(z) := \left\{ \forall r \in \{e^{-n}, n \geq 1\} \cap (0,r_0): N_{z,r}^{\Lc^\theta_D(K)} \leq b (\log r)^2 \right\}.
\end{equation}
As will be clear from what follows from Lemma \ref{lem:first_moment_good_event}, for typical (in the sense of $\Mc_{\thk}^K$) points, we expect the number of crossings to be roughly $\thk (\log r)^2$, since the aspect ratio of the annulus is $e$.
Given these good events, we also define the modified version of $\Mc_{\thk}^K$, ${\thk} \in (0,a]$ as follows:
\begin{equation}
\label{eq:def_measure_good_event}
\tilde{\Mc}_{\thk}^K (dz) := \mathbf{1}_{\Gc_K(z)} \Mc_{\thk}^K(dz).
\end{equation}
Note that we use the same parameter $b$ in the definition of the good event for \emph{all} ${\thk} \le a$ above.

Proposition \ref{prop:oscillations_Mc} will follow quickly from the following intermediate results.

\begin{lemma}\label{lem:first_moment_good_event}
There exists $w_1: (0,1) \to (0,\infty)$ such that $w_1(r_0) \to 0$ as $r_0 \to 0$ and such that for all bounded measurable function $f:D \to \R$ with compact support included in $A$, for all ${\thk} \in [a/2,a]$ and $K \geq 1$,
\[
\Expect{ \abs{ \int_D f(z) \Mc_{\thk}^K(dz) - \int_D f(z) \tilde{\Mc}_{\thk}^K(dz) } } \leq w_1(r_0) \norme{f}_\infty (\log K)^\theta.
\]
\end{lemma}

To analyse the behaviour of $\tilde \Mc_{\thk}^K$, a key role will be played by the following estimate.

\begin{lemma}\label{lem:second_moment_bdd}
Let $\eta \in [0,2-a)$. If $b$ is close enough to $a$, then
\[
\sup_{{\thk} \in [a/2,a]}
\sup_{K \geq 1} \frac{1}{(\log K)^{2\theta}}
\int_{A \times A} \frac{1}{|z-z'|^\eta}
 \Expect{ \tilde{\Mc}_{\thk}^K(dz) \tilde{\Mc}_{\thk}^K(dz') } < \infty.
\]
\end{lemma}

Together with Frostman's lemma, this essentially shows that any set $S$ which supports $\Mc_a$ (or, more precisely, $\tilde \Mc_a$ but this has no impact by Lemma \ref{lem:first_moment_good_event}) has dimension at least $2 - a$.
We will also use this estimate (with $\eta=0$) to show the following control, which is the main required estimate for Proposition \ref{prop:oscillations_Mc}.

\begin{lemma}\label{lem:second_moment_aa'}
Let $r_0 \in (0,1)$ be fixed. If $b$ is close enough to $a$, then
\[
\limsup_{{\thk} \to a^-} \limsup_{K \to \infty} \sup_{f} \norme{f}_\infty^{-2} (\log K)^{-2\theta}
\Expect{ \left( \int_D f(z) \tilde{\Mc}_{\thk}^K(dz) - \int_D f(z) \tilde{\Mc}_a^K(dz) \right)^2 } =0,
\]
where the supremum is over all bounded, non-zero, non-negative measurable function $f : D \to [0,\infty)$ with compact support included in $A$.
\end{lemma}

Let us first briefly check that Lemmas \ref{lem:first_moment_good_event} and \ref{lem:second_moment_aa'} allow us to conclude the proof of Proposition \ref{prop:oscillations_Mc}.

\begin{proof}[Proof of Proposition \ref{prop:oscillations_Mc}]
Let $f:D \to \R$ be a bounded measurable function with compact support included in $A$ and let $K \geq 1$, ${\thk} \in [a/2,a]$.
By Lemma \ref{lem:first_moment_good_event},
\begin{align*}
& \Expect{ \abs{ \int_D f(z) \Mc_{\thk}^K(dz) - \int_D f(z) \Mc_a^K(dz) } } \\
& \leq 2 w_1(r_0) \norme{f}_\infty (\log K)^\theta
+ \Expect{ \abs{ \int_D f(z) \tilde{\Mc}_{\thk}^K(dz) - \int_D f(z) \tilde{\Mc}_a^K(dz) } }.
\end{align*}
Lemma \ref{lem:second_moment_aa'} and Cauchy--Schwarz allow us to control the second right hand side term, so that
\[
\limsup_{{\thk} \to a} \limsup_{K \to \infty} \sup_f
\frac{1}{\norme{f}_\infty (\log K)^\theta} \Expect{ \abs{ \int_D f(z) \Mc_{\thk}^K(dz) - \int_D f(z) \Mc_a^K(dz) } }
\leq 2w_1(r_0).
\]
Since the left hand side term is independent of $r_0$, by letting $r_0 \to 0$, we deduce that it vanishes. This concludes the proof.
\end{proof}

The rest of this section is devoted to the proof of the three intermediate lemmas.

\subsection{Number of crossings in the processes of excursions}

We start by studying the number of crossings in the processes of excursions that appear in the second moment computations. Recall that these processes are defined in Notation \ref{not:xi_a^z} and in \eqref{PP}.
Let $z, z' \in D$ and $r>0$ be such that $|z-z'| > er$.
We are going to study $N_{z,r}^\Cc$ for $\Cc = \Xi_a^z, \Xi_a^{z'}$ or $\Xi_{a,a'}^{z,z'}$.
 We start off with the first two variables.
We can decompose
\[
N_{z,r}^{\Xi_a^z} = \sum_{i=1}^{P} G_i
\quad \text{and} \quad
N_{z,r}^{\Xi_a^{z'}} = \sum_{i=1}^{P'} G_i'
\]
where $P$ (resp. $P'$) is the Poisson random variable corresponding to the number of excursions in $\Xi_a^z$ (resp. $\Xi_a^{z'}$) that touch $\partial D(z,er)$ (resp. $\partial D(z,r)$) and $G_i$ (resp. $G_i'$), $i \geq 1,$ are i.i.d. random variables, independent of $P$ (resp. $P'$), and distributed according to the number of crossings from $\partial D(z,r)$ to $\partial D(z,er)$ in a path distributed according to $\mu^{z,z}_D( \cdot \vert \tau_{z,er} < \infty )$ (resp. $\mu^{z',z'}_D( \cdot \vert \tau_{z,r} < \infty )$)

If the domain $D$ were a disc centred at $z$, then, by rotational invariance and Markov property, the $G_i$'s would be geometric random variables (we use the convention that geometric random variables are larger or equal to 1 with probability 1). In general, this is only asymptotically true as $r \to 0$. We recall that we fix a set $A \Subset D$ compactly included in $D$ during the whole Section \ref{sec:beyond_L2}.

\begin{lemma}
\label{lem:crossingsPPP}
\begin{enumerate}
\item
$P$ and $P'$ are Poisson random variables with means given by
\[
\Expect{P} = a \log \frac{\CR(z,D)}{er}
\quad \text{and} \quad
\Expect{ P'} = a \log \CR(z',D) - a \xi_{D \backslash D(z,r)}(z',z')
\]
where $w \mapsto \xi_{D \backslash D(z,r)} (z',w)$ is the harmonic extension of $w \in \partial D \cup \partial D(z,r) \mapsto \log |z'-w|$ in the domain $D \backslash D(z,r)$. In particular, for all $z,z' \in A, r>0$ such that $e^2 \leq |z-z'|/r \leq e^3$,
\begin{equation}
\label{eq:lem_asymp_P}
\Expect{P} = a \log \frac{1}{r} + O(1)
\quad \text{and} \quad
\Expect{ P' } = a \log \frac{1}{r} + O(1).
\end{equation}
\item
Let $z, z' \in A, r>0$ be such that $e^2 \leq |z-z'| / r \leq e^3$.
The random variable $G_i$ is stochastically dominated by $G_+$ and stochastically dominates $G_-$ where $G_{\pm}$ are geometric random variables with success probabilities
\[
p_\pm = \frac{1+o(1)}{|\log r|}.
\]
There exist $C_+, C_- \in \R$, $u_+(r), u_-(r) \in \R$ that go to zero as $r \to 0$ such that for all $k \geq 1$,
\begin{equation}\label{tailGi'}
\left( 1- \frac{1+u_-(r)}{|\log r|} \right)^{k-1} \left( 1 + \frac{C_-}{\log r} \right) \leq \Prob{ G_i' \geq k } \leq \left( 1- \frac{1+u_+(r)}{|\log r|} \right)^{k-1} \left( 1 + \frac{C_+}{\log r} \right).
\end{equation}
\end{enumerate}
\end{lemma}

(The quantities $C_+, C_-, u_+(r), u_-(r)$ and the implicit constants in $O(1)$ and $o(1)$ may depend on $A$.)

\begin{proof}
1. We will rely on the following (probably well known) fact about Green function in a domain $U$ (which however may be a non simply connected domain) with Dirichlet boundary conditions on $\partial U$: we claim that
\begin{equation}\label{Green}
G_U(z,w) = -\frac1{2\pi} \log |z- w| + \xi_U(z, w),
\end{equation}
where $\xi_U(z, \cdot)$ is the harmonic extension of $ \frac1{2\pi} \log |z- \cdot|$ from $\partial U$ to $U$. Furthermore, when $U$ is simply connected and $z = w$ then $\xi_U (z,z) = \tfrac1{2\pi}\log \CR (z, U)$ (see, e.g., (1.4) in \cite{BerestyckiPowell}). To see this, observe that the difference between the two functions on the left and on the right hand sides of \eqref{Green} is harmonic in $w$ (except possibly at $w = z$) and is at most $o ( \log |z-w|)$ when $w \to z$ (for instance, one may use domain monotonicity to see this). This difference also has zero boundary condition on $\partial U$. An application of the optional stopping theorem therefore shows that this difference is identically zero on $U$.

We obtain the mean of $P$ by considering all trajectories that start from $z$ and leave $D(z,er)$; equivalently we can subtract from all trajectories those that stay in $D(z, er)$ and get the desired asymptotics from Dirichlet Green function asymptotics:
\begin{align*}
2\pi a\mu_D^z( \tau_{z,er} < \infty )
& = 2\pi a\lim_{w \to z} \mu_D^{z,w}( \tau_{z,er} < \infty ) \\
& = 2\pi a\lim_{w \to z} G_D(z,w) - G_{D(z,er)}(z,w)
= a \log \frac{\CR(z,D)}{er}.
\end{align*}
The mean of $P'$ can be computed similarly using \eqref{Green}. \eqref{eq:lem_asymp_P} then follows.

2. Consider a Brownian motion starting from a point on $\partial D(z, er)$, conditioned to hit $z$ before exiting $D$. This is a Markov process (it can be described through a certain $h$-transform, where $h(x) = G_D(x,z)$). By the strong Markov property and elementary properties of $h$-transforms, we can stochastically dominate $G_i$ by a geometric random variable whose success probability is given by
\[
p_+ := 1- \min_{x \in \partial D(z,r)} \frac{ \mu_D^{x,z} (\tau_{z,er} < \infty) }{G_D(x,z)}
= 1 - \min_{x \in \partial D(z,r)} \frac{1}{G_D(x,z)} \EXPECT{x}{ G_D(X_{\tau_{z,er}}, z) \indic{ \tau_{z,er} < \infty } }.
\]
Here $\E_x$ denotes the expectation with respect to a Brownian motion starting from $x$. Hence
\[
p_+ = 1 - \frac{\log \frac{\CR(z,D)}{er} + o(1)}{\log \frac{\CR(z,D)}{r} + o(1)}
= \frac{1}{\log \frac{\CR(z,D)}{r}} + o \left( \frac{1}{\log r} \right)
\]
where the $o(1)$ terms are uniform over $z$ restricted to $A$. The lower bound is similar with minima replaced by maxima.

We now turn to the case of $G_i'$. For all $k \geq 1$, using again elementary properties of the $h$-transform,
\begin{align*}
\Prob{G_i' \geq k}
%& \geq \min_{x \in \partial D(z,er)} \PROB{x}{ \tau_{z,r} < \tau_{\partial D} }^{k-1} \min_{y \in \partial D(z,r)} \frac{\PROB{y}{\tau_{z'} < \tau_{\partial D}}}{\PROB{x}{\tau_{z'} < \tau_{\partial D} }} \\
& \ge \min_{x \in \partial D(z,er)} \PROB{x}{ \tau_{z,r} < \tau_{\partial D} }^{k-1} \min_{y \in \partial D(z,r)} \frac{G_D(y,z')}{G_D(x,z')} \\
& = \left( 1- \frac{1+o(1)}{|\log r|} \right)^{k-1} \left( 1 + \frac{O(1)}{\log r} \right),
\end{align*}
as desired.
\end{proof}

We now state three corollaries of Lemma \ref{lem:crossingsPPP}. The first corollary will be used in the proof of Lemma \ref{lem:first_moment_good_event} whereas the third one will be used in the proof of Lemma \ref{lem:second_moment_bdd}. The second one will be useful in order to show that $\Mc_a$ is supported by $\Tc(a)$ almost surely (Theorem \ref{th:thick_points_continuum}). We will only prove the first corollary, since it is the most difficult one to prove and the proofs of the other two only require small adaptations.

Note that, in Corollary \ref{cor:crossingPPP1}, we will need to take into account the killing associated to the mass. On the other hand, in Corollaries \ref{cor:crossingPPP2} and \ref{cor:crossingPPP3}, this will not be necessary thanks to FKG-inequality for Poisson point processes (see \cite[Lemma 2.1]{Janson84}).

\begin{corollary}\label{cor:crossingPPP1}
Let $u \in (0,1/2)$. There exists $C(u) > 0$ such that for all $z \in A$, $r \in (0,1)$ and $\thk >0$,
\begin{equation}
\label{eq:cor_mixed}
\Expect{ \left( 1 - e^{-KT(\Xi_{\thk}^z)} \right) e^{\frac{u}{|\log r|} N_{z,r}^{\Xi_{\thk}^z} } } \leq \left( 1 - e^{- \thk(3/2 C_K(z) + C(u) |\log r|)} \right) \exp \left( \thk \frac{u}{1-u} (1+o(1)) |\log r| \right)
\end{equation}
where $o(1) \to 0$ as $r \to 0$ and may depend on $u$ and $A$.
\end{corollary}

By FKG-inequality for Poisson point processes, the expectation on the left hand side of \eqref{eq:cor_mixed} is at least the product of the expectation of each of the two terms which behaves like (as we will see in the proof below)
\[
\left( 1 - e^{- \thk C_K(z))} \right) \exp \left( \thk \frac{u}{1-u} (1+o(1)) |\log r| \right).
\]
The content of Corollary \ref{cor:crossingPPP1} is therefore that upper bound matches the lower bound with the only difference that $C_K(z)$ becomes the larger value $3/2 C_K(z) + C(u) |\log r|$.

\begin{proof}
Since $\frac{u}{|\log r|} N_{z,r}^{\Xi_{\thk}^z}$ and $KT(\Xi_{\thk}^z)$ are additive functions of $\Xi_{\thk}^z$,  Palm formula gives that the left hand side of \eqref{eq:cor_mixed} is equal to
\begin{align*}
& \exp \left( 2\pi \thk \int \mu_D^{z,z}(d \wp) \left( e^{ \frac{u}{|\log r|} N_{z,r}^{\wp}} - 1 \right) \right)
-  \exp \left( 2\pi \thk \int \mu_D^{z,z}(d \wp) \left( e^{\frac{u}{|\log r|} N_{z,r}^{\wp} - KT(\wp)} - 1 \right) \right) \\
& = \Expect{ e^{\frac{u}{|\log r|} N_{z,r}^{\Xi_{\thk}^z} } }
\left( 1 - \exp \left( 2\pi \thk \int \mu_D^{z,z}(d \wp) e^{\frac{u}{|\log r|} N_{z,r}^{\wp}} \left( e^{- KT(\wp)} - 1 \right) \right) \right).
\end{align*}
Our goal now is to bound from above
\[
2\pi \int \mu_D^{z,z}(d \wp) e^{\frac{u}{|\log r|} N_{z,r}^{\wp}} \left( 1- e^{- KT(\wp)} \right).
\]
We can rewrite it as
\begin{align*}
C_K(z) + 2\pi \int \mu_D^{z,z}(d \wp) \left( e^{\frac{u}{|\log r|} N_{z,r}^{\wp}} -1 \right) \left( 1- e^{- KT(\wp)} \right)
\end{align*}
and by bounding for $x>1$ and $y \in (0,1)$, $(x-1)(1-y) \leq ((x-1)^2+(1-y)^2)/2 \leq (x^2-1)/2 + (1-y)/2$, we obtain that it is at most
\begin{align*}
C_K(z) + \frac{1}{2} C_K(z) + \frac12 2\pi \int \mu_D^{z,z}(d \wp) \left( e^{\frac{2u}{|\log r|} N_{z,r}^{\wp}} -1 \right).
\end{align*}
We denote $G$ a random variable whose law is given by $N_{z,r}^{\wp}$ where $\wp$ is a trajectory distributed according to $\mu_D^{z,z}( \cdot \vert \tau_{z,er} < \infty ) $.
Thanks to Lemma \ref{lem:crossingsPPP} point 2, an easy computation with geometric random variables shows that
\[
\Expect{ e^{\frac{2u}{|\log r|} G }  - 1} = \frac{1}{1-2u} - 1 +o(1) = \frac{2u}{1-2u} + o(1).
\]
With Lemma \ref{lem:crossingsPPP} point 1, this implies that
\begin{align*}
2\pi \int \mu_D^{z,z}(d \wp) \left( e^{\frac{2u}{|\log r|} N_{z,r}^{\wp}} -1 \right)
= 2\pi \mu_D^{z,z}( \tau_{z,er} < \infty ) \Expect{ e^{\frac{2u}{|\log r|} G }  - 1} \leq C(u) |\log r|.
\end{align*}
The same reasoning shows that
\[
\Expect{ e^{\frac{u}{|\log r|} N_{z,r}^{\Xi_{\thk}^z} } } = \exp \left( \thk \frac{u}{1-u} (1+o(1)) |\log r| \right).
\]
Wrapping things up, we have proven that the left hand side of \eqref{eq:cor_mixed} is at most
\[
\left( 1 - e^{- \thk(3/2 C_K(z) + C(u) |\log r|)} \right) \exp \left( \thk \frac{u}{1-u} (1+o(1)) |\log r| \right).
\]
This concludes the proof.
\end{proof}

\begin{corollary}\label{cor:crossingPPP3}
There exist $\gamma >0$ and $r_0 >0$ that may depend on $a, b$ and $A$ such that for all $r \in (0,r_0)$ and $z \in A$,
\begin{equation}
\label{eq:cor_crossingPPP1-}
\Prob{ N_{z,r}^{\Xi_{a}^z} < \left\{ a - \frac{b-a}{2} \right\} (\log r)^2 } \leq r^\gamma.
\end{equation}
\end{corollary}

\begin{corollary}\label{cor:crossingPPP2}
Let $a>0$, $z,z' \in A, r>0$ be such that $e^2 \leq |z-z'| / r \leq e^3$. Fix a parameter $u>0$. Then,
\begin{equation}
\label{eq:lem_large_deviations_z}
\Expect{ \exp \left( - \frac{u}{|\log r|} N_{z,r}^{\Xi_a^z} \right) } = \exp \left( - a \frac{u}{1+u} (1+o(1)) |\log r| \right)
\end{equation}
and
\begin{equation}
\label{eq:lem_large_deviations_z'}
\Expect{ \exp \left( - \frac{u}{|\log r|} N_{z,r}^{\Xi_a^{z'}} \right) } = \exp \left( - a \frac{u}{1+u} (1+o(1)) |\log r| \right)
\end{equation}
where the $o(1)$ terms tend to 0 as $r \to 0$ and may depend on $A$ and $u$.
\end{corollary}

\medskip

We now move on to the study of $N_{z,r}^{\Xi^{z,z'}_{a,a'}}$, again in the setting where $e^2 \leq |z-z'|/r \leq e^3$. It is convenient to first view the trajectories in $\Xi_{a,a'}^{z,z'}$ as excursions from $z'$ to $z$ (rather than a mixture of equal number of excursions going from $z $ to $ z'$ and vice-versa). When we time-reverse an excursion, an upcrossing becomes a downcrossing. Since two upcrossings are necessarily separated by a downcrossing, the error in counting the upcrossings when we fix the direction of the excursion as being from $z'$ to $z$ is at most 1.
We can decompose
\[
N_{z,r}^{\Xi^{z,z'}_{a,a'}} = \sum_{i=1}^{\# \Xi^{z,z'}_{a,a'}} G_i^z(r)
\]
where $G_i^z(r)$, $i \geq 1$, are i.i.d. random variables, independent of $\# \Xi^{z,z'}_{a,a'}$, that correspond to the number of crossings from $\partial D(z,r)$ to $\partial D(z,er)$ for a trajectory distributed according to $\mu^{z',z}_D/G_D(z',z)$. In the same vein as in Lemma \ref{lem:crossingsPPP}, $1 + G_i^z(r)$ dominates and is dominated  stochastically a geometric random variable with success probability
\begin{equation}\label{psuccess}
p = \frac{1+o(1)}{|\log r|}.
\end{equation}
Note in particular that since $r \to 0$, $p \to 0$.

When we then consider the quantity which is really of interest to us, i.e., the number of crossings of the annulus $D(z, er) \setminus D(z,r)$ associated with the \emph{loops} coming from concatenating the pairs of excursions in $\Xi^{z, z'}_{a, a'}$, the resulting error from having considered the upcrossings of the reverse excursions instead of those of the original excursions in \eqref{psuccess} is therefore negligible.

In particular, we obtain:

\begin{lemma}\label{lem:crossingPPP10}
Let $a,a'>0$, $z,z' \in A, r>0$ be such that $e^2 \leq |z-z'| / r \leq e^3$. Fix a parameter $u>0$. Then
\begin{equation}
\label{eq:lem_large_deviations_zz'}
\Expect{ \exp \left( - \frac{u}{|\log r|} N_{z,r}^{\Xi_{a,a'}^{z,z'}} \right) } = \frac{\Bs \left( (2\pi)^2 aa' G_D(z,z')^2 \frac{1 + o(1)}{(1+u)^2} \right)}{\Bs \left( (2\pi)^2 aa' G_D(z,z')^2 \right)},
\end{equation}
where the $o(1)$ term tends to 0 as $r \to 0$ (and may depend on $A$ and $u>0$).
\end{lemma}

\begin{proof}
For all $c >0$, we have trivially
\begin{align*}
\Expect{c^{\# \Xi^{z,z'}_{a,a'}}} = \frac{\Bs \left( (2\pi)^2 aa' G_D(z,z')^2 c^2 \right)}{\Bs \left( (2\pi)^2 aa' G_D(z,z')^2 \right)},
\end{align*}
where we have used the definition of $\Xi^{z,z'}_{a,a'}$ in \eqref{PP} and the definition of $\Bs$ just above. Therefore, applying this with $c = \E( e^{- u G/(|\log r|)})$ (where $G$ is the geometric random variable coming from \eqref{psuccess}) concludes the proof.
\end{proof}

\subsection{Proof of Lemma \ref{lem:first_moment_good_event} (typical points are not thick)}

%This can be easily seen by rooting the loops at the point whose distance to $z$ is maximal (in a similar way as \cite[Proposition 8]{Lawler04}). We omit the details.

Before we begin the proof of Lemma \ref{lem:first_moment_good_event}, we will require an estimate which says that a Lebesgue-typical, fixed point $z$ is not thick for the measure $\Mc^K_{a}$. We will need to show this in a somewhat quantitative way, and uniformly in $K$. For orientation, the number of crossings $N_{z,r}^{\Lc^\theta_D} $ of the annulus of scale $r$ around $z$ roughly corresponds to the local time regularised at scale $r$ around $z$ accumulated by $\Lc^\theta_D$, and so is roughly of the order of the square of the GFF. For a typical point, we expect this to be roughly $\log 1/r$. For a Liouville typical point, this would instead be of the order of $(\log 1/r)^2$. The deviation probability below may thus be expected to decay polynomially. Let us finally mention that it will be important for us to nail the right exponent in order to obtain the upper bound on the dimension of the set $\Tc(a)$ of $a$-thick points (Theorem \ref{th:thick_points_continuum}).

\begin{lemma}
\label{lem:crossing_soup}
For any $\lambda \in (0,1)$, there exists $r_\lambda >0$ such that for all $r \in (0,r_\lambda)$, $z \in D$ and $u >0$,
\begin{equation}
\Prob{ N_{z,r}^{\Lc_D^\theta} \geq u (\log r)^2 } \leq r^{\lambda u}.
\end{equation}
\end{lemma}

\begin{proof}[Proof of Lemma \ref{lem:crossing_soup}]
First of all, $N_{z,r}^{\Lc_D^\theta}$ is stochastically dominated by $N_{z,r}^{\Lc_U^\theta}$ where $U$ is the disc centred at $z$ with radius being equal to the diameter of $D$. Without loss of generality, we can therefore assume that the domain $D$ is the unit disc $\D$ and that $z$ is the origin.
In the remaining of the proof, we will write $N_r^{\boldsymbol{\cdot}}$ instead of $N_{z,r}^{\boldsymbol{\cdot}}$.

For $0 < r_1 < r_2$, we will denote by $A(r_1,r_2)$ the annulus $r_2 \D \setminus r_1 \D$.
For all $k = 1, \dots, k_{\rm max} := \floor{ -\log r} -1$, consider the set of ``loops at scale $k$''
\[
\Lc_k := \{ \wp \in \Lc_\D^\theta: \wp \subset e^{k+1} r \D, \wp \text{~crosses~} A(e^{k-1/2} r, e^k r) \}.
\]
We can decompose
\[
N_r^{\Lc_D^\theta} = \sum_{k=1}^{k_{\rm max}} \sum_{\wp \in \Lc_k} N_r^\wp.
\]
We now make three observations.
Firstly, by thinning property of Poisson point processes,  $\Lc_k$, $k=1 \dots {k_{\rm max}}$, are independent collections of loops.  Secondly, conditioned on $\# \Lc_k$, $\Lc_k$ is composed of $\# \Lc_k$ i.i.d. loops with law
\begin{equation}
\label{eq:proof_law_loop_k}
\frac{\indic{\wp \subset e^{k+1} r \D, \wp \text{~crosses~} A(e^{k-1/2} r, e^k r)} \loopmeasure_\D (\d \wp)}
{\loopmeasure_\D(\{ \wp \subset e^{k+1} r \D, \wp \text{~crosses~} A(e^{k-1/2} r, e^k r) \})}.
\end{equation}
Finally,  for each $k$, $\# \Lc_k$ is a Poisson random variable whose mean is, by scaling invariance of the Brownian loop measure, given by
\[
\loopmeasure_{e^{k+1} r \D} (\{ \wp \text{~crosses~} A(e^{k-1/2} r, e^k r) \} )
= \loopmeasure_\D (\{ \wp \text{~crosses~} A(e^{-3/2}, e^{-1}) \} ).
\]
Therefore $\E[ \# \Lc_k ]$ is a finite quantity that does not depend on $k$ or $r$. Let $P_k, k=1 \dots {k_{\rm max}}$, be i.i.d.  Poisson random variables with the above mean. We have decomposed
\[
N_r^{\Lc_D^\theta} = \sum_{k=1}^{k_{\rm max}} \sum_{i=1}^{P_k} N_r^{\wp_i^k}
\]
where for all $k$ and $i$, $\wp_i^k$ are independent and distributed according to \eqref{eq:proof_law_loop_k}.
Let $\lambda \in (0,1)$ be a parameter. We have
\begin{equation}
\Expect{ \exp \left( \frac{\lambda}{|\log r|} N_r^{\Lc_D^\theta} \right) }
= \prod_{k=1}^{k_{\rm max}} \exp \left\{ \E[P] \left( \Expect{ e^{ \frac{\lambda}{|\log r|} N_r^{\wp^k} } } - 1 \right) \right\}.
\end{equation}
The rest of the proof is dedicated to showing that for all $k = 1 \dots {k_{\rm max}}$,
\begin{equation}
\label{eq:proof_lem_crossing_soup}
\Expect{e^{ \frac{\lambda}{|\log r|} N_r^{\wp^k} }}
\leq 1 + C_\lambda / |\log r|
\end{equation}
for some constant $C_\lambda$ depending only on $\lambda$. Indeed, this will imply that
\[
\Expect{ \exp \left( \frac{\lambda}{|\log r|} N_r^{\Lc_D^\theta} \right) } \leq e^{\Expect{P} C_\lambda}
\]
and the proof of Lemma \eqref{lem:crossing_soup} will be completed by Markov inequality.

\medskip

We now turn to the proof of \eqref{eq:proof_lem_crossing_soup}. Let $k \in \{1, \dots, {k_{\rm max}}\}$. We are going to describe the law \eqref{eq:proof_law_loop_k} by rooting the loop $\wp^k$ at the unique point $z$ where its modulus is maximal.
We will denote $R = |z|$ and $w$ the first hitting point of
$e^{k-1/2} r \D$.
The law \eqref{eq:proof_law_loop_k} can be disintegrated as
\begin{displaymath}
\frac{1}{Z_{k}}
\int_{e^k r}^{e^{k+1} r} R ~\d R
\int_{R \partial \D} \frac{\d z}{2\pi R}
\indic{\wp \text{~crosses~} A(e^{k-1/2} r, e^k r)}
\mu^{z,z}_{R\D}(d\wp)
,
\end{displaymath}
where the measure $\mu^{z,z}_{R\D}(d\wp)$ is given by
\eqref{Eq boundary exc}
and $Z_{k}$ is the normalising constant.
This decomposition is somewhat similar to
\cite[Proposition 8]{Lawler04}.
Further, the measure
$\indic{\wp \text{~crosses~} A(e^{k-1/2} r, e^k r)}
\mu^{z,z}_{R\D}(d\wp)$
is the image of the measure
\begin{displaymath}
\int_{e^{k-1/2}r \partial \D} \d w ~
\mu^{z,w}_{A(e^{k-1/2} r, R)}(d \wp_{1})
\mu^{w,z}_{R \D}(d \wp_{2})
\end{displaymath}
under the concatenation
$(\wp_{1},\wp_{2})\mapsto \wp_{1}\wedge\wp_{2}$.
This is similar to decompositions appearing in
\cite[Section 5.2]{LawlerConformallyInvariantProcesses}.
Moreover, in this decomposition,
$N_{r}^{\wp_{1}\wedge\wp_{2}} = N_{r}^{\wp_{2}}$.
It follows that for any bounded measurable function $F : \R \to \R$, we have
\begin{equation}
\label{eq:proof_lem_crossing_soup3}
\Expect{F(N_r^{\wp^k})}
= \frac{1}{Z_k} \int_{e^k r}^{e^{k+1} r} R ~\d R \int_{R \partial \D} \frac{\d z}{2\pi R} \int_{e^{k-1/2}r \partial \D} \d w
H_{A(e^{k-1/2} r, R)}(z,w) H_{R \D}(w,z) \E^{w,z}_{R \D}[F(N_r^\wp)]
\end{equation}
where $\E^{w,z}_{R \D}$ is the expectation associated to the law
$\mu_{R \D}^{w,z} (\cdot) / H_{R \D}(w,z)$
\eqref{Eq mu D z w boundary} and $Z_k$ is the normalising constant
\[
Z_k = \int_{e^k r}^{e^{k+1} r} R ~\d R \int_{R \partial \D} \frac{\d z}{2\pi R} \int_{e^{k-1/2}r \partial \D} \d w H_{A(e^{k-1/2} r, R)}(z,w) H_{R \D}(w,z).
\]

Let $n \geq 1$ and denote $\P^w$ the law of planar Brownian motion $(B_t)_{t \geq 0}$ starting from $w$ and $\tau_n$ the first time that $r \partial \D$ is reached after having already crossed the annulus $A(r,er)$ $n-1$ times in the upward direction.  We also denote by $\tau_{R \partial \D}$ the first hitting time of $R \partial \D$. The conditional law $\P^{w,z}_{R \D}$ can be expressed as an $h$-transform of $\P^w$ as follows:
\[
\P^{w,z}_{R \D}(N_r^\wp \geq n)
= \E^w \left[\frac{H_{R\D}(B_{\tau_n} ,z)}{H_{R\D}(w,z)}  \indic{\tau_n < \tau_{R \partial \D}} \right].
\]
Therefore,
\begin{equation}
\label{eq:proof_lem_crossing_soup2}
\P^{w,z}_{R \D}(N_r^\wp \geq n)
\leq \frac{\max_{|y| = r} H_{R \D}(y,z)}{H_{R \D}(w,z)} \P^w \left( \tau_n < \tau_{R \partial \D} \right).
\end{equation}
Since $(\log |B_t|)_{t \geq 0}$ is a martingale, for all $0 < r_1 < r_2 < r_3$, for all $x \in r_2 \partial \D$, we have
\[
\P^x \left( \tau_{r_1 \partial \D} < \tau_{r_3 \partial \D} \right)
= \frac{\log(r_3/r_2)}{\log (r_3/r_1)}
\]
and by strong Markov property, we deduce that
\[
\P^w \left( \tau_n < \tau_{R \partial \D} \right) = \frac{\log (R/e^{k-1/2} r)}{\log (R/r)} \left( \frac{\log (R/er)}{\log (R/r)} \right)^{n-1}.
\]
Moreover, by Harnack inequality,  the ratio of Poisson kernels in \eqref{eq:proof_lem_crossing_soup2} can be bounded by some constant independent of $k$ and $r$. Recalling that $R \in [e^k r, e^{k+1} r]$, this shows that
\[
\P^{w,z}_{R \D}(N_r^\wp \geq n)
\leq \frac{C}{k} \left( 1- \frac{1}{k} \right)^{n-1}.
\]
Going back to \eqref{eq:proof_lem_crossing_soup3}, we have proven that when $\wp^k$ is distributed according to \eqref{eq:proof_law_loop_k}, then for all $n \geq 1$,
\[
\Prob{N_r^{\wp^k} \geq n } \leq \frac{C}{k} \left( 1- \frac{1}{k} \right)^{n-1}.
\]
Since $\lambda <1$ and $k \leq |\log r|$, we deduce from the above bound that
\begin{align*}
\Expect{e^{ \frac{\lambda}{|\log r|} N_r^{\wp^k} }}
\leq 1 + \frac{C}{k} \left( -\frac{|\log r|}{\lambda} \log \left( 1 - \frac{1}{k} \right) - 1 \right)^{-1} \leq 1 + \frac{C_\lambda}{|\log r|}
\end{align*}
for some constant $C_\lambda$ that depends only on $\lambda$. This proves \eqref{eq:proof_lem_crossing_soup} and concludes the proof of Lemma \ref{lem:crossing_soup}.
\end{proof}

We are now ready to prove Lemma \ref{lem:first_moment_good_event}.

\begin{proof}[Proof of Lemma \ref{lem:first_moment_good_event}]
Let $f : D \to \R$ be a bounded measurable function with compact support included in $A$, ${\thk} \in [a/2,a]$ and $K \geq 1$. By a union bound, we have
\begin{align*}
 \Expect{ \abs{ \int_D f(z) \Mc_{\thk}^K(dz) - \int_D f(z) \tilde{\Mc}_{\thk}^K(dz) } } %\\
& \leq
\sum_{\substack{r = e^{-n}\\n \geq \ceil{\log(1/r_0)} }} \int_D \abs{f(z)} \Expect{ \Mc_{\thk}^K(dz) \indic{ N_{z,r}^{\Lc^\theta_D(K)} > b (\log r)^2 } }.
\end{align*}
Let $r = e^{-n}$ for some $n \geq \ceil{\log(1/r_0)}$. By Lemma \ref{lem:Girsanov_K},  we have
\begin{align}
\label{eq:proof_lem_first_moment_good}
& \Expect{ \Mc_{\thk}^K(dz) \indic{ N_{z,r}^{\Lc^\theta_D(K)} > b (\log r)^2 } }
= \CR(z,D)^a
\sum_{n \geq 1} \frac{\theta^n}{n!}
\int_{\mathsf{a} \in E(a,n)} \frac{\d \mathsf{a}}{a_1 \dots a_n} \\
& \times  \Expect{ \prod_{i=1}^n \left( 1 - e^{-K T(\Xi_{a_i}^z) } \right)
\indic{ \sum_{i=1}^n N_{z,r}^{\Xi_{a_i}^z} + N_{z,r}^{\Lc^\theta_D(K)} > b (\log r)^2 }   } \d z .
\nonumber
\end{align}
Let $u \in (0,1/2)$ be a parameter. By an exponential Markov inequality, we can bound the above indicator function by
\begin{align*}
& \indic{ \sum_{i=1}^n N_{z,r}^{\Xi_{a_i}^z} > \left(\thk + \frac{b-a}{2} \right) (\log r)^2 }
+ \indic{ N_{z,r}^{\Lc^\theta_D(K)} > \frac{b-a}{2} (\log r)^2 } \\
& \leq e^{-u \left(\thk + \frac{b-a}{2} \right) |\log r|} \prod_{i=1}^n \exp \left( \frac{u}{|\log r|} N_{z,r}^{\Xi_{a_i}^z} \right)
+ \indic{ N_{z,r}^{\Lc^\theta_D(K)} > \frac{b-a}{2} (\log r)^2 }.
\end{align*}
By Corollary \ref{cor:crossingPPP1} and then by using the fact that $\sum a_i = \thk$, the expectation on the right hand side of \eqref{eq:proof_lem_first_moment_good} is therefore at most
\begin{align}
\nonumber
& e^{-u \left(\thk + \frac{b-a}{2} \right) |\log r|} \prod_{i=1}^n \left( 1 - e^{- a_i(3/2 C_K(z) + C(u) |\log r|)} \right) \exp \left( a_i \frac{u}{1-u} (1+o(1)) |\log r| \right) \\
\nonumber
& + \prod_{i=1}^n \left( 1 - e^{- a_i C_K(z))} \right) \Prob{ N_{z,r}^{\Lc^\theta_D(K)} > \frac{b-a}{2} (\log r)^2 } \\
\label{eq:proof_lem_first_moment_good2}
& = \exp \left\{ - \left( u \left(\thk + \frac{b-a}{2} \right) - \thk \frac{u}{1-u} (1+o(1)) \right) |\log r| \right\} \prod_{i=1}^n \left( 1 - e^{- a_i(3/2 C_K(z) + C(u) |\log r|)} \right)\\
& + \prod_{i=1}^n \left( 1 - e^{- a_i C_K(z))} \right) \Prob{ N_{z,r}^{\Lc^\theta_D(K)} > \frac{b-a}{2} (\log r)^2 }.
\nonumber
\end{align}
By choosing $u$ small enough, we can ensure
\[
u \left(\thk + \frac{b-a}{2} \right) - \thk \frac{u}{1-u} = \frac{b-a}{2} u - \thk \frac{u^2}{1-u}
\]
to be strictly positive. Therefore,  if $r$ is small enough, the first exponential in \eqref{eq:proof_lem_first_moment_good2} can be bounded by $r^\gamma$ for some $\gamma >0$ depending on $u$, $a$ and $b$ (recall that $\thk \in [a/2,a]$). We use Lemma \ref{lem:crossing_soup} to bound the probability in \eqref{eq:proof_lem_first_moment_good2} by $r^\gamma$ for some $\gamma>0$.
With Lemma \ref{lem:Fs} we therefore see that
\begin{align*}
& \Expect{ \Mc_{\thk}^K(dz) \indic{ N_{z,r}^{\Lc^\theta_D(K)} > b (\log r)^2 } }
\leq \frac1{\thk} \CR(z,D)^a r^\gamma \left( \Fs \left( \thk (3/2 C_K(z) + C(u) |\log r|) \right) + \Fs \left( \thk C_K(z) \right)  \right)
\end{align*}
Using the inequality $\Fs(u) \leq C u^\theta$, we conclude that
\[
\Expect{ \Mc_{\thk}^K(dz) \indic{ N_{z,r}^{\Lc^\theta_D(K)} > b (\log r)^2 } }
\leq C (\log K)^\theta r^c
\]
for some $C,c>0$ that may depend on $a$, $b$ and $A$. Finally,
\begin{align*}
& \Expect{ \abs{ \int_D f(z) \Mc_{\thk}^K(dz) - \int_D f(z) \tilde{\Mc}_{\thk}^K(dz) } } \\
& \leq C \norme{f}_\infty (\log K)^\theta
\sum_{\substack{r = e^{-n}\\n \geq \ceil{\log(1/r_0)} }} r^c
\leq C \norme{f}_\infty (\log K)^\theta (r_0)^c.
\end{align*}
This concludes the proof.
\end{proof}

\subsection{Proof of Lemma \ref{lem:second_moment_bdd} (truncated \texorpdfstring{$L^2$}{L2} bound)}

\begin{proof}
Let ${\thk} \in [a/2,a]$ and let $z, z' \in A$. The constants appearing in this proof may depend on $a, b, r_0$ and $A$, but will be uniform in $z, z'$ and ${\thk}$. We want to bound from above
\[
\Expect{\tilde{\Mc}_{\thk}^K(dz) \tilde{\Mc}_{\thk}^K(dz')}.
\]
If $|z-z'| \geq r_0$, we simply bound this by
\[
\Expect{\tilde{\Mc}_{\thk}^K(dz) \tilde{\Mc}_{\thk}^K(dz')} \leq C (\log K)^{2\theta} \d z \d z'
\]
by Corollary \ref{cor:second_moment_simplified}, where $C>0$ is some constant depending on $r_0$. We now assume that $|z-z'| < r_0$ and we let $r \in (0,r_0) \cap \{e^{-n}, n \geq 1\}$ be such that $e^2 \leq |z-z'|/r \leq e^3$.
By Lemma \ref{lem:second_moment},
$\Expect{\tilde{\Mc}_{\thk}^K(dz) \tilde{\Mc}_{\thk}^K(dz')}$ is at most
\begin{align*}
& C \sum_{\substack{n,m \geq 1\\0 \leq l \leq n \wedge m}} \frac{1}{(n-l)! (m-l)! l!} \theta^{n+m-l}
\int_{\substack{\mathsf{a} \in E({\thk},n) \\ \mathsf{a}' \in E({\thk},m)}} \frac{\d \mathsf{a}}{a_1 \dots a_n} \frac{\d \mathsf{a}'}{a_1' \dots a_m'} \prod_{i=1}^l \Bs \left( (2\pi)^2 a_i a_i' G_D(z,z')^2 \right) \\
& \times \Expect{ F \left( (\Xi^{z,z'}_{a_i,a_i'} \wedge \Xi^z_{a_i} \wedge \Xi^{z'}_{a_i'})_{i=1 \dots l}, (\Xi^z_{a_i})_{i =l+1 \dots n}, (\Xi^{z'}_{a_i'})_{i=l+1 \dots m} \right) } \d z \d z'
\end{align*}
with $F(\wp_1, \dots, \wp_n, \wp_{l+1}', \dots, \wp_m')$ being equal to
\[
\prod_{i=1}^n \left( 1 - e^{-KT(\wp_i)} \right) \prod_{i=l+1}^m \left( 1 - e^{-KT(\wp_i')} \right) \indic{ \sum_{i=1}^n N_{z,r}^{\wp_i} + \sum_{i=l+1}^m N_{z,r}^{\wp_i'} \leq b (\log r)^2 }.
\]
Now, let $u = 2 \sqrt{{\thk}/b} - 1$, which is positive if $b$ is close enough to $a$, and observe that $F(\wp_1, \dots, \wp_n,$ $\wp_{l+1}', \dots, \wp_m')$ is bounded from above by
\begin{align*}
F_u(\wp_1, \dots, \wp_n, \wp_{l+1}', \dots, \wp_m') & := e^{bu |\log r|} \exp \left( - \frac{u}{|\log r|} \sum_{i=1}^n N_{z,r}^{\wp_i} - \frac{u}{|\log r|} \sum_{i=l+1}^m N_{z,r}^{\wp_i'} \right) \\
& ~~\times \prod_{i=l+1}^n \left( 1 - e^{-KT(\wp_i)} \right) \prod_{i=l+1}^m \left( 1 - e^{-KT(\wp_i')} \right).
\end{align*}
Here we both neglect the killing part for $\wp_1 \dots \wp_l$ and we bound the indicator function in the spirit of an exponential Markov inequality. We have
\begin{align*}
& \Expect{F_u \left( (\Xi^{z,z'}_{a_i,a_i'} \wedge \Xi^z_{a_i} \wedge \Xi^{z'}_{a_i'})_{i=1 \dots l}, (\Xi^z_{a_i})_{i =l+1 \dots n}, (\Xi^{z'}_{a_i'})_{i=l+1 \dots m} \right) } \\
& = e^{bu |\log r|} \prod_{i=1}^l \Expect{ \exp \left( - \tfrac{u}{|\log r|} N_{z,r}^{\Xi_{a_i,a_i'}^{z,z'}} \right) } \Expect{ \exp \left( - \tfrac{u}{|\log r|} N_{z,r}^{\Xi_{a_i}^{z}} \right) } \Expect{ \exp \left( - \tfrac{u}{|\log r|} N_{z,r}^{\Xi_{a_i'}^{z'}} \right) } \\
& \prod_{i=l+1}^n \Expect{ \left( 1 - e^{-K T (\Xi_{a_i}^z)} \right) \exp \left( - \tfrac{u}{|\log r|} N_{z,r}^{\Xi_{a_i}^{z}} \right) }
\prod_{i=l+1}^m \Expect{ \Big( 1 - e^{-K T (\Xi_{a_i'}^{z'})} \Big) \exp \Big( - \tfrac{u}{|\log r|} N_{z,r}^{\Xi_{a_i}^{z'}} \Big) }.
\end{align*}
\[
1 - e^{-K T (\Xi_{a_i}^z)}
\quad \text{and} \quad
- \exp \left( - \tfrac{u}{|\log r|} N_{z,r}^{\Xi_{a_i}^{z}} \right)
\]
are increasing functions of the loop soup. Therefore, by FKG-inequality for Poisson point processes (see \cite[Lemma 2.1]{Janson84}),
\begin{equation}
\label{eq:proof_FKG}
\Expect{ \left( 1 - e^{-K T (\Xi_{a_i}^z)} \right) \exp \left( - \tfrac{u}{|\log r|} N_{z,r}^{\Xi_{a_i}^{z}} \right) }
\leq \Expect{ 1 - e^{-K T (\Xi_{a_i}^z)} } \Expect{ \exp \left( - \tfrac{u}{|\log r|} N_{z,r}^{\Xi_{a_i}^{z}} \right) }.
\end{equation}
Recall that (see \eqref{eq:lem_large_deviations_z} and \eqref{eq:lem_large_deviations_z'})
\[
\Expect{ \exp \left( - \tfrac{u}{|\log r|} N_{z,r}^{\Xi_{a_i}^{z}} \right) } =
e^{- \frac{u+o(1)}{1+u} a_i |\log r|}
\quad \text{and} \quad
\Expect{ \exp \left( - \tfrac{u}{|\log r|} N_{z,r}^{\Xi_{a_i'}^{z'}} \right) } =
e^{- \frac{u+o(1)}{1+u} a_i' |\log r|}
\]
and (see \eqref{eq:lem_large_deviations_zz'})
\[
\Bs \left( (2\pi)^2 a_i a_i' G_D(z,z')^2 \right) \Expect{ \exp \left( - \tfrac{u}{|\log r|} N_{z,r}^{\Xi_{a_i,a_i'}^{z,z'}} \right) } = \Bs \left( (2\pi)^2 a_i a_i' G_D(z,z')^2 \frac{1+o(1)}{(1+u)^2} \right).
\]
The $o(1)$ above go to zero as $r \to 0$. In what follows, to ease notations, we will not write the $o(1)$. This is of no importance: alternatively, one can increase slightly the value of the thickness parameter ${\thk}$ and absorb the $o(1)$ in doing so.
We continue the computations and find that
\begin{align*}
& \prod_{i=1}^l \Bs \left( (2\pi)^2 a_i a_i' G_D(z,z')^2 \right) \Expect{F \left( (\Xi^{z,z'}_{a_i,a_i'} \wedge \Xi^z_{a_i} \wedge \Xi^{z'}_{a_i'})_{i=1 \dots l}, (\Xi^z_{a_i})_{i =l+1 \dots n}, (\Xi^{z'}_{a_i'})_{i=l+1 \dots m} \right) } \\
& \leq e^{bu |\log r|} \exp  \left( - \frac{u}{1+u} \left( \sum_{i=1}^n a_i + \sum_{i=1}^m a_i' \right) |\log r| \right)
 \prod_{i=l+1}^n \left( 1 - e^{-a_i C_K(z)} \right) \prod_{i=l+1}^m \left( 1 - e^{-a_i' C_K(z')} \right) \\
& \times \prod_{i=1}^l \Bs \left( (2\pi)^2 a_i a_i' G_D(z,z')^2 / (1+u)^2 \right).
\end{align*}
Since $\sum_{i=1}^n a_i = \sum_{i=1}^m a_i' = {\thk}$, we have found that $\Expect{\tilde{\Mc}_{\thk}^K(dz) \tilde{\Mc}_{\thk}^K(dz')}$ is at most
\begin{align*}
& e^{\left( bu - 2{\thk} \frac{u}{1+u} \right) |\log r|} \sum_{\substack{n,m \geq 1\\0 \leq l \leq n \wedge m}} \frac{1}{(n-l)! (m-l)! l!} \theta^{n+m-l}
\int_{\substack{\mathsf{a} \in E({\thk},n) \\ \mathsf{a}' \in E({\thk},m)}} \d \mathsf{a} \d \mathsf{a}' \prod_{i=l+1}^n \frac{1 - e^{-a_i C_K(z)} }{a_i} \\
& \times \prod_{i=l+1}^m \frac{1 - e^{-a_i' C_K(z')} }{a_i'} \prod_{i=1}^l \frac{\Bs \left( (2\pi)^2 a_i a_i' G_D(z,z')^2 / (1+u)^2 \right)}{a_i a_i'} \d z \d z' \\
& = e^{\left( bu - 2{\thk} \frac{u}{1+u} \right) |\log r|} \Hs_{{\thk},{\thk}}(C_K(z),C_K(z'),(2\pi)^2G_D(z,z')^2/(1+u)^2) \d z \d z'
\end{align*}
where the function $\Hs_{{\thk},{\thk}}$ is defined in \eqref{eq:def_Hs}. By \eqref{eq:lem_h_upper_bound}, we can further bound from above the expectation $\Expect{\tilde{\Mc}_{\thk}^K(dz) \tilde{\Mc}_{\thk}^K(dz')}$ by
\[
C (\log K)^{2\theta}
e^{\left( bu - 2{\thk} \frac{u}{1+u} \right) |\log r|} G_D(z,z')^{1/2-\theta} e^{4\pi {\thk} G_D(z,z')/(1+u)} \d z \d z'.
\]
Recalling that $r$ has been chosen in such a way that $2\pi G_D(z,z') = |\log r| + O(1)$ and that $u = 2\sqrt{{\thk}/b} -1$, we conclude that $\Expect{\tilde{\Mc}_{\thk}^K(dz) \tilde{\Mc}_{\thk}^K(dz')}$ is at most
\[
C (\log K)^{2\theta} |\log r|^{1/2-\theta} \exp \left( \left( b - 2 (\sqrt{b} - \sqrt{{\thk}})^2 \right) |\log r| \right) \d z \d z' \leq C (\log K)^{2\theta} |z-z'|^{-b} \d z \d z'.
\]
Since $b$ can be made arbitrary close to $a$, this concludes the proof.
\end{proof}

\subsection{Proof of Lemma \ref{lem:second_moment_aa'} (convergence)}

\begin{proof}
Assume that $b$ is close enough to $a$ so that Lemma \ref{lem:second_moment_bdd} holds for some $\eta >0$. Let $f:D \to [0,\infty)$ be a non-negative bounded measurable function with compact support included in $A$ and let $a' \in [a/2,a]$.
We have
\begin{align*}
& \Expect{ \left( \int f \d \tilde{\Mc}_{a'}^K - \int f \d \tilde{\Mc}_a^K \right)^2 } \\
& = \Expect{ \int f \d \tilde{\Mc}_{a'}^K \left( \int f \d \tilde{\Mc}_{a'}^K - \int f \d \tilde{\Mc}_a^K \right) } + \Expect{ \int f \d \tilde{\Mc}_a^K \left( \int f \d \tilde{\Mc}_a^K - \int f \d \tilde{\Mc}_{a'}^K \right) }.
\end{align*}
Let $\eta>0$ be small. Since $f$ is non-negative, we can bound
\begin{align*}
& \Expect{ \int f \d \tilde{\Mc}_a^K \left( \int f \d \tilde{\Mc}_a^K - \int f \d \tilde{\Mc}_{a'}^K \right) }
\leq \norme{f}_\infty^2 \int_{A \times A} \indic{|z - z'| \leq \eta} \Expect{ \tilde{\Mc}_a^K(dz) \tilde{\Mc}_a^K(dz') }\\
& + \int_{A \times A} f(z) f(z') \indic{|z - z'| > \eta} \Expect{ \tilde{\Mc}_{a}^K(dz) \left( \tilde{\Mc}_{a}^K(dz') - \tilde{\Mc}_{a'}^K(dz') \right) }.
\end{align*}
Thanks to Lemma \ref{lem:second_moment_bdd}, we know that
\[
\lim_{\eta \to 0} \limsup_{K \to \infty} \frac{1}{(\log K)^{2\theta}} \int_{A \times A} \indic{|z - z'| \leq \eta} \Expect{ \tilde{\Mc}_a^K(dz) \tilde{\Mc}_a^K(dz') } = 0.
\]
We now deal with the second term. Let $z, z' \in A$ such that $|z-z'|> \eta$.
We start by claiming that there exist $C>0, r_1 \in (0,r_0)$ that may depend on $\eta$ and $b-a$ such that
\begin{equation}
\label{eq:proof_prop_aa'1}
(\log K)^{-2\theta} \Expect{ \tilde{\Mc}_a^K(dz) \tilde{\Mc}_a^K(dz') }
\leq \eta + (\log K)^{-2\theta} \Expect{ \hat{\Mc}_a^K(dz) \hat{\Mc}_a^K(dz') }
\end{equation}
where $\hat{\Mc}_a^K(dz)$ is defined similarly as $\tilde{\Mc}_a^K(dz)$ but with the good event restricting the number of crossings of annulus for $r \in (r_1,r_0)$ instead of $r \in (0,r_0)$. We omit the proof of this claim since it follows along similar lines as the proof of Lemma \ref{lem:first_moment_good_event}. The point is that since $z$ and $z'$ are at distance macroscopic, there will be only a finite number of excursions between $z$ and $z'$ so that \eqref{eq:proof_prop_aa'1} boils down to Lemma \ref{lem:first_moment_good_event}.
The conclusion of these preliminaries is that we have bounded
\begin{align}
\label{eq:proof_prop_aa'4}
& (\log K)^{-2 \theta} \norme{f}_\infty^{-2} \Expect{ \int f \d \tilde{\Mc}_a^K \left( \int f \d \tilde{\Mc}_a^K - \int f \d \tilde{\Mc}_{a'}^K \right) }  \\
& \leq (\log K)^{-2 \theta} \int_{A \times A} \frac{f(z)f(z')}{\norme{f}_\infty^2} \indic{|z - z'| > \eta} \Expect{ \hat{\Mc}_{a}^K(dz) \left( \hat{\Mc}_{a}^K(dz') - \hat{\Mc}_{a'}^K(dz') \right) }
+ o_{\eta \to 0}(1)
\nonumber
\end{align}
where $o_{\eta \to 0}(1) \to 0$ as $\eta \to 0$, uniformly in $K \geq 1, a' \in [a/2,a]$ and $f$.

Now let $z,z' \in A$ such that $|z-z'| > \eta$. By Lemma \ref{lem:second_moment}, $\Expect{ \hat{\Mc}_{a}^K(dz) \hat{\Mc}_{a'}^K(dz') }$ is equal to
\begin{align*}
& \frac{1}{aa'} \CR(z,D)^a \CR(z',D)^{a'} \sum_{\substack{n,m \geq 1\\0 \leq l \leq n \wedge m}} \frac{\theta^{n+m-l}}{(n-l)! (m-l)! l!}
\int_{\substack{\mathsf{a} \in E(1,n)\\ \mathsf{a}' \in E(1,m)}} \d \mathsf{a} \d \mathsf{a}' \\
& \E \prod_{i=1}^l \frac{1-e^{-KT(\Xi_{aa_i,a'a_i'}^{z,z'}) - KT(\Xi_{aa_i}^z) - KT(\Xi_{a'a_i'}^{z'})}}{a_i a_i'}
 \prod_{i=l+1}^n \frac{1 - e^{-KT(\Xi_{aa_i}^z)} }{a_i} \prod_{i=l+1}^m \frac{1 - e^{-KT(\Xi_{a'a_i'}^{z'})} }{a_i'} \\
& F \left( \bigwedge_{i=1}^l \Xi_{aa_i,a'a_i'}^{z,z'} \wedge \bigwedge_{i=1}^n \Xi_{aa_i}^z  \wedge \bigwedge_{i=1}^m \Xi_{a'a_i'}^{z'} \wedge \Lc^\theta_D \right) \prod_{i=1}^l \Bs \left( (2\pi)^2 a a' a_i a_i' G_D(z,z')^2 \right)
\end{align*}
where
\[
F(\Cc) := \indic{ \forall r \in \{e^{-n}, n \geq 1\} \cap (r_1,r_0), N_{z,r}^\Cc \leq b (\log r)^2 \text{~and~} N_{z',r}^\Cc \leq b (\log r)^2 }.
\]
We develop further this expression according to the number $2k_i$ of excursions in $\Xi_{aa_i,a'a_i'}^{z,z'}, i =1 \dots l$. In particular, $\Xi_{k_i}^{z,z'}$ will denote the concatenation of $2k_i$ i.i.d. trajectories distributed according to $\mu_D^{z,z'} / G_D(z,z')$. $\Expect{ \hat{\Mc}_{a}^K(dz) \hat{\Mc}_{a'}^K(dz') }$ is equal to
\begin{align*}
& \frac{1}{aa'} \CR(z,D)^a \CR(z',D)^{a'} \sum_{\substack{n,m \geq 1\\0 \leq l \leq n \wedge m}} \frac{\theta^{n+m-l}}{(n-l)! (m-l)! l!}
\int_{\substack{\mathsf{a} \in E(1,n)\\ \mathsf{a}' \in E(1,m)}} \d \mathsf{a} \d \mathsf{a}' \sum_{k_1, \dots, k_l \geq 1} \\
& \E \prod_{i=1}^l \frac{1-e^{-KT(\Xi_{k_i}^{z,z'}) - KT(\Xi_{aa_i}^z) - KT(\Xi_{a'a_i'}^{z'})}}{a_i a_i'}
 \prod_{i=l+1}^n \frac{1 - e^{-KT(\Xi_{aa_i}^z)} }{a_i} \prod_{i=l+1}^m \frac{1 - e^{-KT(\Xi_{a'a_i'}^{z'})} }{a_i'} \\
& F \left( \bigwedge_{i=1}^l \Xi_{k_i}^{z,z'} \wedge \bigwedge_{i=1}^n \Xi_{aa_i}^z  \wedge \bigwedge_{i=1}^m \Xi_{a'a_i'}^{z'} \wedge \Lc^\theta_D \right) \prod_{i=1}^l \frac{(2\pi \sqrt{aa'a_ia_i'} G_D(z,z'))^{2k_i}}{k_i! (k_i-1)!}
\end{align*}
In what follows, we will naturally couple the PPP of excursions away of $z'$ by decomposing $\Xi_{aa_i}^{z'} = \Xi_{a'a_i}^{z'} \wedge \Xi_{(a-a') a_i}^{z'}$ (recall that $a' \leq a$).
We can then decompose
\begin{align*}
& \frac{a}{\CR(z,D)^a} \Expect{ \hat{\Mc}_a^K(dz) \left( \frac{a}{\CR(z',D)^a} \hat{\Mc}_a^K(dz') - \frac{a'}{\CR(z',D)^{a'}} \hat{\Mc}_{a'}^K(dz') \right) } = S_1 + S_2 + S_3
\end{align*}
where
\begin{align*}
S_1 & = \sum_{\substack{n,m \geq 1\\0 \leq l \leq n \wedge m}} \frac{\theta^{n+m-l}}{(n-l)! (m-l)! l!}
\int_{\substack{\mathsf{a} \in E(1,n)\\ \mathsf{a}' \in E(1,m)}} \d \mathsf{a} \d \mathsf{a}' \\
& \sum_{k_1, \dots, k_l \geq 1} \left\{ \prod_{i=1}^l \frac{(2\pi a \sqrt{a_ia_i'} G_D(z,z'))^{2k_i}}{k_i! (k_i-1)!} - \prod_{i=1}^l \frac{(2\pi \sqrt{aa'a_ia_i'} G_D(z,z'))^{2k_i}}{k_i! (k_i-1)!} \right\} \\
& \E \prod_{i=1}^l \frac{1-e^{-KT(\Xi_{k_i}^{z,z'}) - KT(\Xi_{aa_i}^z) - KT(\Xi_{aa_i'}^{z'})}}{a_i a_i'}
 \prod_{i=l+1}^n \frac{1 - e^{-KT(\Xi_{aa_i}^z)} }{a_i} \prod_{i=l+1}^m \frac{1 - e^{-KT(\Xi_{aa_i'}^{z'})} }{a_i'} \\
& F \left( \bigwedge_{i=1}^l \Xi_{k_i}^{z,z'} \wedge \bigwedge_{i=1}^n \Xi_{aa_i}^z  \wedge \bigwedge_{i=1}^m \Xi_{aa_i'}^{z'} \wedge \Lc^\theta_D \right),
\end{align*}
\begin{align*}
S_2 & = \sum_{\substack{n,m \geq 1\\0 \leq l \leq n \wedge m}} \frac{\theta^{n+m-l}}{(n-l)! (m-l)! l!}
\int_{\substack{\mathsf{a} \in E(1,n)\\ \mathsf{a}' \in E(1,m)}} \d \mathsf{a} \d \mathsf{a}' \prod_{i=1}^l \Bs \left( (2\pi)^2 aa'a_ia_i' G_D(z,z')^2 \right) \\
& \E \Bigg\{ \prod_{i=1}^l \frac{1-e^{-KT(\Xi_{aa_i,a'a_i'}^{z,z'}) - KT(\Xi_{aa_i}^z) - KT(\Xi_{aa_i'}^{z'})}}{a_i a_i'}
 \prod_{i=l+1}^n \frac{1 - e^{-KT(\Xi_{aa_i}^z)} }{a_i} \prod_{i=l+1}^m \frac{1 - e^{-KT(\Xi_{aa_i'}^{z'})} }{a_i'} \\
& - \prod_{i=1}^l \frac{1-e^{-KT(\Xi_{aa_i,a'a_i'}^{z,z'}) - KT(\Xi_{aa_i}^z) - KT(\Xi_{a'a_i'}^{z'})}}{a_i a_i'}
 \prod_{i=l+1}^n \frac{1 - e^{-KT(\Xi_{aa_i}^z)} }{a_i} \prod_{i=l+1}^m \frac{1 - e^{-KT(\Xi_{a'a_i'}^{z'})} }{a_i'} \Bigg\} \\
& \times F \left( \bigwedge_{i=1}^l \Xi_{aa_i,a'a_i'}^{z,z'} \wedge \bigwedge_{i=1}^n \Xi_{aa_i}^z  \wedge \bigwedge_{i=1}^m \Xi_{aa_i'}^{z'} \wedge \Lc^\theta_D \right)
\end{align*}
and
\begin{align}
\nonumber
S_3 & = \sum_{\substack{n,m \geq 1\\0 \leq l \leq n \wedge m}} \frac{\theta^{n+m-l}}{(n-l)! (m-l)! l!}
\int_{\substack{\mathsf{a} \in E(1,n)\\ \mathsf{a}' \in E(1,m)}} \d \mathsf{a} \d \mathsf{a}' \prod_{i=1}^l \Bs \left( (2\pi)^2 aa'a_ia_i' G_D(z,z')^2 \right) \\
\nonumber
& \E \prod_{i=1}^l \frac{1-e^{-KT(\Xi_{aa_i,a'a_i'}^{z,z'}) - KT(\Xi_{aa_i}^z) - KT(\Xi_{a'a_i'}^{z'})}}{a_i a_i'}
 \prod_{i=l+1}^n \frac{1 - e^{-KT(\Xi_{aa_i}^z)} }{a_i} \prod_{i=l+1}^m \frac{1 - e^{-KT(\Xi_{a'a_i'}^{z'})} }{a_i'} \\
& \times \Bigg\{ F \left( \bigwedge_{i=1}^l \Xi_{aa_i,a'a_i'}^{z,z'} \wedge \bigwedge_{i=1}^n \Xi_{aa_i}^z  \wedge \bigwedge_{i=1}^m \Xi_{aa_i'}^{z'} \wedge \Lc^\theta_D \right) - F \left( \bigwedge_{i=1}^l \Xi_{aa_i,a'a_i'}^{z,z'} \wedge \bigwedge_{i=1}^n \Xi_{aa_i}^z  \wedge \bigwedge_{i=1}^m \Xi_{a'a_i'}^{z'} \wedge \Lc^\theta_D \right)  \Bigg\}.
\label{eq:proof_prop_aa'2}
\end{align}
We now claim that for all $i \in \{1, 2, 3\}$, uniformly in $z,z' \in A$ with $|z-z'|>\eta$,
\begin{equation}
\label{eq:proof_prop_aa'3}
\limsup_{a' \to a} \limsup_{K \to \infty} \frac{1}{(\log K)^{2\theta}} S_i = 0.
\end{equation}
For $S_1$ and $S_2$, this follows by bounding the function $F$ by one and then by noting that we obtained explicit expressions for the limit in $K$ that are continuous with respect to the thickness parameters. See Corollary \ref{cor:second_moment_simplified}.
We now explain how to deal with $S_3$. We notice that on the event that none of the excursions of $\bigwedge_{i=1}^l \Xi_{(a-a')a_i'}^{z'}$ hits the circle $\partial D(z',r_1)$, the difference of the function $F$ appearing in \eqref{eq:proof_prop_aa'2} vanishes. Since $0 \leq F \leq 1$, we can therefore bound this difference by the indicator of the complement of this event. After applying a union bound, we find that
\begin{align*}
& \E \prod_{i=1}^l \frac{1-e^{-KT(\Xi_{aa_i,a'a_i'}^{z,z'}) - KT(\Xi_{aa_i}^z) - KT(\Xi_{a'a_i'}^{z'})}}{a_i a_i'}
 \prod_{i=l+1}^n \frac{1 - e^{-KT(\Xi_{aa_i}^z)} }{a_i} \prod_{i=l+1}^m \frac{1 - e^{-KT(\Xi_{a'a_i'}^{z'})} }{a_i'} \\
& \times \Bigg\{ F \left( \bigwedge_{i=1}^l \Xi_{aa_i,a'a_i'}^{z,z'} \wedge \bigwedge_{i=1}^n \Xi_{aa_i}^z  \wedge \bigwedge_{i=1}^m \Xi_{aa_i'}^{z'} \wedge \Lc^\theta_D \right) - F \left( \bigwedge_{i=1}^l \Xi_{aa_i,a'a_i'}^{z,z'} \wedge \bigwedge_{i=1}^n \Xi_{aa_i}^z  \wedge \bigwedge_{i=1}^m \Xi_{a'a_i'}^{z'} \wedge \Lc^\theta_D \right)  \Bigg\} \\
& \leq \E \prod_{i=1}^l \frac{1-e^{-KT(\Xi_{aa_i,a'a_i'}^{z,z'}) - KT(\Xi_{aa_i}^z) - KT(\Xi_{a'a_i'}^{z'})}}{a_i a_i'}
 \prod_{i=l+1}^n \frac{1 - e^{-KT(\Xi_{aa_i}^z)} }{a_i} \prod_{i=l+1}^m \frac{1 - e^{-KT(\Xi_{a'a_i'}^{z'})} }{a_i'} \\
& \times \sum_{j=1}^m \sum_{\wp \in \Xi_{(a-a')a_j'}^{z'}} \indic{\wp \text{~hits~} \partial D(z',r_1)} \\
& \leq \prod_{i=1}^l \frac{1}{a_i a_i'} \prod_{i=l+1}^n \frac{1 - \E e^{-KT(\Xi_{aa_i}^z)} }{a_i} \prod_{i=l+1}^m \frac{1 - \E e^{-KT(\Xi_{a'a_i'}^{z'})} }{a_i'}
\sum_{j=1}^m \E \sum_{\wp \in \Xi_{(a-a')a_j'}^{z'}} \indic{\wp \text{~hits~} \partial D(z',r_1)} .
\end{align*}
Since
\begin{align*}
\sum_{j=1}^m \E \sum_{\wp \in \Xi_{(a-a')a_j}^{z'}} \indic{\wp \text{~hits~} \partial D(z',r_1)} & = \sum_{j=1}^m 2 \pi (a-a') a_j' \mu_D^{z',z'}(\tau_{\partial D(z',r_1)} < \infty) \\
& = 2 \pi (a-a') a' \mu_D^{z',z'}(\tau_{\partial D(z',r_1)} < \infty) \leq C (a-a')
\end{align*}
for some constant $C>0$ which may depend on $r_1$, we have obtained that
\begin{align*}
S_3 & \leq C (a-a') \sum_{\substack{n,m \geq 1\\0 \leq l \leq n \wedge m}} \frac{\theta^{n+m-l}}{(n-l)! (m-l)! l!}
\int_{\substack{\mathsf{a} \in E(1,n)\\ \mathsf{a}' \in E(1,m)}} \d \mathsf{a} \d \mathsf{a}' \prod_{i=1}^l \frac{\Bs \left( (2 \pi)^2 aa' a_i a_i' G_D(z,z')^2 \right)}{a_ia_i'} \\
& \times \prod_{i=l+1}^n \frac{1-e^{-aa_iC_K(z)}}{a_i}
\prod_{i=l+1}^m \frac{1-e^{-a'a_i'C_K(z')}}{a_i'}
\end{align*}
By Lemma \ref{lem:Hs}, this is at most $C(a-a') (\log K)^{2\theta}$ for some constant $C >0$ that may depend on $r_1$ and $\eta$. This finishes the proof of \eqref{eq:proof_prop_aa'3} for $S_3$.

To conclude, we have proven that
\[
\limsup_{a' \to a} \limsup_{K \to \infty} (\log K)^{-2\theta} \Expect{ \hat{\Mc}_a^K(dz) \left( \frac{a}{\CR(z',D)^a} \hat{\Mc}_a^K(dz') - \frac{a'}{\CR(z',D)^{a'}} \hat{\Mc}_{a'}^K(dz') \right) } = 0,
\]
uniformly over $z,z' \in A$ with $|z-z'|> \eta$.
Hence
\[
\limsup_{a' \to a} \limsup_{K \to \infty} (\log K)^{-2\theta} \Expect{ \hat{\Mc}_a^K(dz) \left( \hat{\Mc}_a^K(dz') - \hat{\Mc}_{a'}^K(dz') \right) } = 0.
\]
Coming back to \eqref{eq:proof_prop_aa'4}, this implies that
\[
\limsup_{a' \to a} \limsup_{K \to \infty}
(\log K)^{-2 \theta} \norme{f}_\infty^{-2} \Expect{ \int f \d \tilde{\Mc}_a^K \left( \int f \d \tilde{\Mc}_a^K - \int f \d \tilde{\Mc}_{a'}^K \right) } \leq o_{\eta \to 0}(1).
\]
Since the left hand side term does not depend on $\eta$, it has to be non positive. Similarly, the same statement holds true when one exchanges $a$ and $a'$ in the expectation above so that
\[
\limsup_{a' \to a} \limsup_{K \to \infty} (\log K)^{-2\theta} \norme{f}_\infty^{-2} \Expect{ \left( \int f \d \tilde{\Mc}_{a'}^K - \int f \d \tilde{\Mc}_a^K \right)^2 } \leq 0.
\]
This concludes the proof.
\end{proof}

\section{Properties of the loop soup chaos}\label{sec:measurability_etc}

The aim of this section is to study the properties of the measure $\Mc_a$.
We will start in Section \ref{sec:PD} by proving \eqref{eq:th_PD} concerning the local structure of the loop soup around a typical thick point. This section will also give a new perspective on the martingale $(m_a^K(dz), K \geq 0)$. Indeed, it is likely that Theorem \ref{th:PD} could be alternatively proven as a consequence of the discrete approximation of $\Mc_a$ (Theorem \ref{th:convergence_discrete}) and as a consequence of Proposition \ref{Prop Le Jan subordinator}. We decided to take another route which remains in the continuum setting. The advantage of this approach is that it gives an independent proof of the fact that $(m_a^K(dz), K \geq 0)$ is a martingale. This is close in spirit to Lyons' approach \cite{Lyons97} to the Biggins martingale convergence theorem for spatial branching processes originally established by Biggins \cite{Biggins77}.

Section \ref{sec:thick_points} is then dedicated to the proof of Theorem \ref{th:thick_points_continuum} concerning the Hausdorff dimension of the set of thick points.

We will then address in Section \ref{sec:measurability} the measurability of $\Mc_a$ with respect to the Brownian loop soup. This will prove Point \ref{it:measurability} in Theorem \ref{th:convergence_continuum}. From this, we will obtain in Section \ref{sec:charac} the characterisation of the joint law of $(\Lc_D^\theta, \Mc_a)$ as stated in Theorem \ref{th:PD}. This characterisation will allow us to get in Section \ref{sec:conformal_covariance} the conformal covariance of the measure (actually a stronger version of it) which is the content of Point \ref{it:conformal_covariance} in Theorem \ref{th:convergence_continuum}. In the last part of this section, we will use the conformal invariance of the measure to deduce its almost sure positivity, i.e. Theorem \ref{th:convergence_continuum}, Point \ref{it:nondegenerate}.

\subsection{Poisson--Dirichlet distribution}\label{sec:PD}

The aim of this section is to prove \eqref{eq:th_PD}.
Recall from   Section \ref{sec:preliminaries_BLS} that, conditionally on $\Lc_D^\theta$, $\{U_\wp, \wp \in \Lc_D^\theta\}=:{\mathcal U}$ denotes a collection of i.i.d. uniform random variables on $[0,1]$. We will prove that for any nonnegative measurable admissible function $F$,
\begin{equation}
\label{eq:th_PDbis}
\Expect{ \int_D F(z, \Lc_D^\theta,\mathcal U) \Mc_a(dz) }
= \frac{1}{2^{\theta} a^{1-\theta} \Gamma(\theta)} \int_D \Expect{ F(z, \Lc_D^\theta \cup \Xi_{\underline{a}}, \mathcal U \cup {\mathcal U}_{\underline{a}}  ) } \CR(z,D)^a \d z
\end{equation}

\noindent where in the RHS, the two collections of loops $\Lc_D^\theta$ and $\Xi_{\underline{a}} := \{ \Xi_{a_i}^z, i \geq 1 \}$ are independent, and, conditionally on everything else, ${\mathcal U}_{\underline{a}}$ denotes a collection of i.i.d. uniform random variables on $[0,1]$ indexed by $\Xi_{\underline{a}}$. This equation may seem stronger but is actually equivalent to \eqref{eq:th_PD}. We recall that $\Lc_D^\theta(K) = 
\left\{ \wp \in \Lc_D^\theta : U_\wp < 
1-e^{- K T(\wp)} \right\}$ denotes the loops killed at rate $K$ and we further introduce ${\mathcal U}(K):=\{U_\wp, \wp \in \Lc_D^\theta(K)\}$. Conditionally on $\Lc_D^{\theta}(K)$, we see that ${\mathcal U}(K)$ is a collection of independent random variables where $U_\wp$ is uniformly distributed in $[0,1-e^{-K T(\wp)}]$. By the monotone class theorem, it suffices to prove \eqref{eq:th_PDbis} for $F(z,\Lc_D^\theta,{\mathcal U})=\indic{z\in A}G(\Lc_D^\theta(K),{\mathcal U}(K))$ for  $G$ an arbitrary nonnegative measurable function, $A\subset D$ a Borel set and $K>0$.

Recall the definition of $\Mc_a(A)$ in Theorem \ref{th:convergence_continuum}. By Proposition \ref{prop:m_vs_Mc}, $\Mc_a(A)$ is the  ($L^1$ by Proposition \ref{prop:martingale}) limit of $\frac{1}{2^\theta \Gamma(\theta)}m_a^K(A)$ where we recall that
\[
m_a^K(dz) := \frac{1}{a^{1-\theta}} \CR(z,D)^a e^{-aC_K(z)} dz +
\int_0^a \d {\thk} \frac{1}{(a-{\thk})^{1-\theta}} \CR(z,D)^{a-{\thk}} e^{-(a-{\thk}) C_K(z)} \Mc_{\thk}^K(dz).
\]

\noindent We want to compute the LHS of \eqref{eq:th_PDbis} for $\indic{z\in A}G(\Lc_D^\theta(K),{\mathcal U}(K))$ instead of $F(z, \Lc_D^\theta,\mathcal U)$.  Since $(m_a^K(A),\, K\ge 0)$ is a uniformly integrable martingale by Proposition \ref{prop:martingale}, 
\begin{equation}\label{eq:PD_GK}
\int_D \Expect{ \indic{z\in A}G(\Lc_D^\theta(K),{\mathcal U}(K))  \Mc_a(\d z) } = \frac{1}{2^\theta \Gamma(\theta)}\Expect{G(\Lc_D^\theta(K),{\mathcal U}(K))  m_a^K(A)}.
\end{equation}

\noindent We use the fact that $(m_a^K(\d z),\, K\ge 0)$ is a martingale as a motivation for computing the right hand side of \eqref{eq:PD_GK}. Starting from the right hand side of \eqref{eq:PD_GK}, we will now obtain an independent proof of the fact that $(m_a^K(\d z),\, K\ge 0)$ is a martingale and work towards establishing \eqref{eq:th_PDbis} as a byproduct.

%Independently of \eqref{eq:PD_GK}, we are going to prove the following identity
%\begin{equation}
%\label{eq:PD_end}
%\Expect{G^{(K)}m_a^K(\d z)} = \frac{1}{a^{1-\theta}} \CR(z,D)^{a} \Expect{ G( \Lc_D^\theta(K) \cup  \widehat \Xi_{\underline{\hat a}}^z , {\mathcal U}(K)\cup \widehat {\mathcal U}_{\underline{\hat a}} )} \d z,
%\end{equation}
%where $\widehat \Xi_{\underline{\hat a}}^z$ and $\widehat {\mathcal U}_{\underline{\hat a}}$ will be introduced in the proof.
%On the one hand, this will show that $m_a^K(\d z)$ is the Radon--Nikodym derivative of the law $\Lc_D^\theta \cup  \widehat \Xi_{\underline{\hat a}}^z$ with respect to $\Lc_D^\theta$, restricted to the sigma algebra generated by $\Lc_D^\theta(K)$. In particular it will give an alternative proof of the fact that $(m_a^K(\d z),K>0)$ is a martingale.
%On the other hand, combining \eqref{eq:PD_GK} and \eqref{eq:PD_end} will conclude the proof of \eqref{eq:th_PDbis}.

Set $G^{(K)}:=G(\Lc_D^\theta(K),{\mathcal U}(K))$ for concision. From the definition recalled above of $m_a^K(\d z)$ (recall we do not assume here it is a martingale), we have 
\begin{equation}\label{eq:GK}
\Expect{G^{(K)} m_a^K(\d z)} = \mathbb{E}_1 + \mathbb{E}_2
\end{equation}

\noindent with
\begin{align}
\mathbb{E}_1
&:=
\frac{1}{a^{1-\theta}} \CR(z,D)^a e^{-aC_K(z)} \Expect{G^{(K)}}dz, \label{eq:E1}
\\
\mathbb{E}_2
&:=
\int_0^a \d {\thk} \frac{1}{(a-{\thk})^{1-\theta}} \CR(z,D)^{a-{\thk}} e^{-(a-{\thk}) C_K(z)} \Expect{G^{(K)}\Mc_{\thk}^K(dz)}. \label{eq:E2}
\end{align}

\noindent To compute $\mathbb{E}_2$, we first need to compute $ \Expect{G^{(K)} \Mc_{\thk}^K(dz) }$ for $\thk \in [0,a]$. Recall that by Lemma \ref{lem:Girsanov_K}, for any nonnegative measurable function $F$,
\begin{align*}
& \Expect{ F(z, \Lc_D^\theta) \Mc_{\thk}^K(dz) } \\
& = 
\CR(z,D)^{\thk}
\sum_{n \geq 1} \frac{\theta^n}{n!} 
\int_{\underline{\thk} \in E(\thk,n)} \frac{\d \underline{\thk}}{\thk_1 \dots \thk_n} \Expect{ \prod_{i=1}^n \left( 1 - e^{-K T(\Xi_{\thk_i}^z) } \right) F\left(z, \Lc_D^\theta \cup \Xi_{\underline{\thk}}^z\right) } \d z ,
\end{align*}

\noindent where $\underline{\thk}=(\thk_1,\ldots,\thk_n)$, $\d\underline{\thk}=\d\thk_1\ldots \d\thk_{n-1}$ and $\Xi_{\underline{\thk} }^z:=(\Xi_{\thk_i}^z)_{1\leq i\leq n}$ is independent of $\Lc_D^\theta$.  We rewrite the RHS in a slightly different form. First, the term of index $n$ in the sum is equal to 
\[
\theta^n \int_{\underline{\thk} \in E(\thk,n),\thk_1<\dots <\thk_n} \frac{\d \underline{\thk}}{\thk_1 \dots \thk_n} \Expect{ \prod_{i=1}^n \left( 1 - e^{-K T(\Xi_{\thk_i}^z) } \right) F(z, \Lc_D^\theta \cup \Xi_{\underline{\thk}}^z) } \d z.
\]

\noindent Secondly,  recall from \eqref{eq:proba_killing} that $E[ 1- e^{-K T(\Xi_{\thk_i}^z)}]=1- e^{- \thk_i C_K(z)}$. Hence 
\[
\Expect{ \prod_{i=1}^n \left( 1 - e^{-K T(\Xi_{\thk_i}^z) } \right) F(z, \Lc_D^\theta \cup \Xi_{\underline{\thk}}^z ) }
= 
\prod_{i=1}^n \left( 1 - e^{- \thk_i C_K(z) } \right)\Expect{F(z, \Lc_D^\theta \cup  \widehat \Xi_{\underline{\thk}}^z ) }
\]
where $\widehat \Xi_{\underline{\thk}}^z := \{\hat \Xi_{\thk_i}^z, i=1\dots n\}$ is a collection of independent loops independent of $\Lc_D^\theta$, and  $\hat \Xi_{\thk_i}^z$ has the distribution of $\Xi_{\thk_i}^z$ biased by $1-e^{-K T(\Xi_{\thk_i}^z)}$.
Combining the two, we see that
\begin{align*}
& \Expect{ F(z, \Lc_D^\theta) \Mc_{\thk}^K(dz) } =
\CR(z,D)^{\thk}
\sum_{n \geq 1} \theta^n 
\int_{\underline{\thk} \in E(\thk,n),\thk_1<\ldots<\thk_n} \d \underline{\thk}\prod_{i=1}^n \frac{1 - e^{- \thk_i C_K(z)}}{\thk_i} \Expect{ F\left(z, \Lc_D^\theta \cup  \widehat \Xi_{\underline{\thk}}^z  \right) } \d z.
\end{align*}

\noindent Again, one can actually take a function $F(z, \Lc_D^\theta, {\mathcal U})$. From the proof of Lemma \ref{lem:Girsanov_K} and the previous lines, one can check that the function $F$ in the RHS  will turn into $ F(z, \Lc_D^\theta \cup  \widehat \Xi_{\underline{\thk}}^z,{\mathcal U}\cup \widehat {\mathcal U}_{\underline{\thk}} )$ where $\widehat {\mathcal U}_{\underline{\thk}}=\{ \widehat{U}_\wp,\wp\in  \widehat \Xi_{\underline{\thk}}^z \}$ is conditionally on everything else a collection of independent random variables, with $\widehat{U}_{\wp}$ being  uniform in $[0,1-e^{- K T(\wp)}]$.  Taking for $F(z, \Lc_D^\theta, {\mathcal U})$ the function $G^{(K)}=G(\Lc_D^\theta(K), {\mathcal U}(K))$, it implies that
\begin{align*}
& \Expect{ G^{(K)} \Mc_{\thk}^K(dz) } =\\
& 
\CR(z,D)^{\thk}
\sum_{n \geq 1} \theta^n 
\int_{\underline{\thk} \in E(\thk,n),\thk_1<\ldots<\thk_n} \d \underline{\thk}\prod_{i=1}^n \frac{1 - e^{- \thk_i C_K(z)}}{\thk_i} \Expect{ G(\Lc_D^\theta(K) \cup  \widehat \Xi_{\underline{\thk}}^z , {\mathcal U}(K)\cup \widehat{\mathcal U}_{\underline{\thk}} ) } \d z.
\end{align*}

\noindent From the expression of ${\mathbb E}_2$ in \eqref{eq:E2}, after a  change of variables $(\thk_1,\ldots,\thk_{n-1},\thk) \to (\thk_1,\ldots,\thk_n)$, we get
\begin{align*}
\mathbb{E}_2& =\CR(z,D)^{a}  
\sum_{n \geq 1} \theta^n  \\
&
\int_{\thk_1<\ldots<\thk_n, \thk < a} \frac{e^{-(a-\thk)C_K(z)}}{(a-\thk)^{1-\theta}} \prod_{i=1}^n \d \thk_i \frac{1 - e^{- \thk_i C_K(z)}}{\thk_i} \Expect{ G(\Lc_D^\theta(K) \cup  \widehat \Xi_{\underline{\thk}}^z , {\mathcal U}(K)\cup \widehat{\mathcal U}_{\underline{\thk}} ) } \d z
\end{align*}

\noindent where $\thk:=\thk_1+\ldots +\thk_n$ in the integral. We will reinterpret this equality via the following lemma whose proof is deferred to the end of this section.
\begin{lemma}\label{lem:PDthinning}
Let $\{a_1, a_2, \dots\}$ be a random partition of $[0,a]$ distributed according to a Poisson-Dirichlet distribution with parameter $\theta$. Let $u>0$. Remove each atom $a_i$ independently with probability $e^{-u a_i}$. Denote by $\widehat a_1 <\ldots < \widehat a_N$  the remaining atoms (there are only a finite number of them). Then 
\[
P(N=0) =  e^{-ua}
\]

\noindent and for any integer $n\ge 1$, and $0< \thk_1<\ldots<\thk_n$ with $\thk:=\thk_1+\ldots+\thk_n$,
\begin{equation}\label{eq:PD_lem}
P(N=n,\widehat a_1\in \d \thk_1,\ldots, \widehat a_n\in \d \thk_n) =  \frac{a^{1-\theta}}{(a-\thk)^{1-\theta}} e^{-u(a-\thk)}\prod_{i=1}^n \frac{\theta}{\thk_i} (1-e^{-u \thk_i}) \d \thk_i.
\end{equation}
\end{lemma}

\noindent Using the lemma with $u=C_K(z)$ and with the notation of the lemma, we get that 
\[
\mathbb{E}_2 = \frac{1}{a^{1-\theta}} \CR(z,D)^{a} \Expect{ G(\Lc_D^\theta(K) \cup  \widehat \Xi_{\underline{\hat a}}^z , {\mathcal U}(K)\cup \widehat{\mathcal U}_{\underline{\hat a}} ) \indic{N\ge 1}} \d z
\]

\noindent where $\widehat \Xi_{{\underline{\hat  a}}}^z=\{ \hat \Xi_{\hat a_i}^z, i = 1 \dots N \}$ and $\widehat{\mathcal U}_{\underline{\hat a}}:=\{ \widehat{U}_\wp,\wp\in  \widehat \Xi_{\underline{\hat a}}^z \}$ with natural notation. We also have
\[
\mathbb{E}_1 = \frac{1}{a^{1-\theta}} \CR(z,D)^{a} \Expect{G(\Lc_D^\theta(K),{\mathcal U}(K))  \indic{N=0} }.
\]

\noindent From \eqref{eq:GK}, we get 
\begin{equation}
\label{eq:PD_end}
\Expect{G^{(K)}m_a^K(\d z)} = \frac{1}{a^{1-\theta}} \CR(z,D)^{a} \Expect{ G( \Lc_D^\theta(K) \cup  \widehat \Xi_{\underline{\hat a}}^z , {\mathcal U}(K)\cup \widehat {\mathcal U}_{\underline{\hat a}} )} \d z,
\end{equation}
with the convention that $\widehat \Xi_{\underline{\hat a}}^z$ and $\widehat {\mathcal U}_{\underline{\hat a}}$ are empty when $N=0$. We observe that in the expectation in the RHS, the loop soup $\Lc_D^\theta(K) \cup  \widehat \Xi_{\underline{\hat a}}^z$ and the random variables ${\mathcal U}(K)\cup \widehat {\mathcal U}_{\underline{\hat a}}$ are distributed respectively as the collection of loops killed at rate $K$ in the loop soup $\Lc_D^\theta \cup \Xi_{\underline{a}}^z$, and the random variables ${\mathcal U}\cup {\mathcal U}_{\underline{a}}$ restricted to the killed loops.
Rephrasing, this proves that $a^{\theta-1} \CR(z,D)^{-a} m_a^K(\d z)$ is the Radon--Nikodym derivative of the law of $\Lc_D^\theta \cup \Xi_{\underline{a}}^z$ with respect to the law of $\Lc_D^\theta$, restricted to the sigma algebra generated by the loops killed at rate $K$.
This provides another proof of the fact that $(m_a^K(\d z),K>0)$ is a martingale. Furthermore, integrating \eqref{eq:PD_end} over $z\in A$, 
this shows that the right hand side of \eqref{eq:PD_GK} equals the right hand side of \eqref{eq:th_PDbis} for $F(z,\Lc_D,{\mathcal U})= \indic{z\in A}G(\Lc_D(K),{\mathcal U}(K))$. Using \eqref{eq:PD_GK} (which holds since we know the martingale property),  this concludes the proof of \eqref{eq:th_PDbis} and thus \eqref{eq:th_PD}.

\begin{proof}[Proof of Lemma \ref{lem:PDthinning}] That $\P(N=0)=e^{-u a}$ is clear so we only prove \eqref{eq:PD_lem}. Let $\{a_1, a_2, \dots\}$ be a random partition of $[0,a]$ distributed according to a Poisson-Dirichlet distribution with parameter $\theta$. The atoms $\{a_1, a_2, \dots\}$ can be constructed via the jumps of a Gamma subordinator. More precisely, consider a Poisson point process $\{p_1,p_2,\dots\}$ on $\R_+$ with intensity $\indic{x>0}\frac{\theta}{x}e^{-x} \d x$. Let $\Sigma:=\sum_{i\ge 1} p_i$ be the sum of the atoms of the PPP. Then, the collection $\{a\frac{p_1}{\Sigma}, a\frac{p_2}{\Sigma},\ldots\}$ is independent of $\Sigma$ and  distributed as $\{a_1,a_2, \dots\}$. One can also say that the atoms $\{p_1,p_2,\dots\}$ conditioned on $\Sigma=a$ are distributed as $\{a_1, a_2, \dots\}$. Using this representation, we remove each atom $p_i$ of the PPP independently with probability $e^{-u p_i}$.  The remaining atoms form a PPP of intensity $\indic{x>0}\frac{\theta}{x}e^{-x}(1-e^{-ux}) \d x$.  Notice that
$$
\int_0^\infty \frac{\theta}{x}e^{-x}(1-e^{-ux}) \d x 
=
\theta \ln(u+1).
$$

\noindent In particular, the set of remaining atoms is finite a.s. Let 
$N_{p}$ be its cardinality, and when $N_p\ge 1$, let $\hat p_1<\ldots <\hat p_{N_p}$ these atoms ordered increasingly.  For  $n\ge 1$, and $0< \thk_1<\ldots<\thk_n$, 
\[
P(N_p=n,\hat p_1\in \d \thk_1,\ldots, \hat p_n\in \d \thk_n) =  (u+1)^{-\theta}\prod_{i=1}^n \frac{\theta}{\thk_i} e^{-\thk_i}(1-e^{-u \thk_i}) \d \thk_i.
\]

\noindent The removed atoms are independent of the remaining atoms and form a PPP of intensity $\frac{\theta}{x}e^{-(u+1)x} \d x$. It is the L\'evy measure of a Gamma($\theta,u+1$) subordinator. In particular, the sum of all these atoms, which is $\Sigma -\sum_{i=1}^{N_p} \hat p_i$, has the Gamma($\theta,u+1$) distribution, with density $\frac{(u+1)^\theta}{\Gamma(\theta)}s ^{\theta-1}e^{-(u+1)s}\d s$. It implies that , with $\thk:=\sum_{i=1}^n \thk_i$,
\[
P(N_p=n,\hat p_1\in \d \thk_1,\ldots, \hat p_n\in \d \thk_n,\, \Sigma \in \d a) =  \frac{1}{\Gamma(\theta)}(a-\thk)^{\theta-1}e^{-(u+1)(a-\thk)}  \prod_{i=1}^n \frac{\theta}{\thk_i} e^{-\thk_i}(1-e^{-u \thk_i}) \d \thk_i \d a .
\]

\noindent Dividing by the probability that $\Sigma$ is in $\d a$, which is $\frac{1}{\Gamma(\theta)}a^{\theta-1}e^{- a}\d a$, we proved that 
\[
P(N_p =n,\hat p_1\in \d \thk_1,\ldots, \hat p_n\in \d \thk_n\, \mid\,  \Sigma = a) =  \Big(1-\frac{\thk}{a}\Big)^{\theta-1}e^{-u(a-\thk)}  \prod_{i=1}^n \frac{\theta}{\thk_i} (1-e^{-u \thk_i}) \d \thk_i. 
\]

 \noindent By the discussion at the beginning of the proof, we  know that the distribution of $(N_p, \hat p_1,\ldots,\hat p_{N_p})$ conditionally on $\Sigma=a$ is the one of $(N,\hat a_1,\ldots, \hat a_N)$. The lemma follows.
\end{proof}

\subsection{Proof of Theorem \ref{th:thick_points_continuum} (thick points)}\label{sec:thick_points}

We conclude this section with a proof of Theorem \ref{th:thick_points_continuum}.

\begin{proof}[Proof of Theorem \ref{th:thick_points_continuum} and Point \ref{it:dim} of Theorem \ref{th:convergence_continuum}]
We first start by showing that $\Mc_a$ is supported on the set of thick points $\Tc(a)$. To this end, let us denote
\[
\Tc(a,r_0,\eta) := \left\{ z \in D: \forall r \in \{e^{-n}, n \geq 1 \} \cap (0,r_0): \abs{ \frac1{n^2} N_{z,r}^{\Lc_D^\theta} - a } \leq \eta \right\}.
\]
By Theorem \ref{th:PD}, $\Expect{ \Mc_a(D \setminus \Tc(a,r_0,\eta)) }$ is equal to
\begin{equation}
\label{eq:proof_thm_thick1}
\frac{1}{2^\theta a^{1-\theta} \Gamma(\theta)} \int_D \CR(z,D)^a \Prob{ \exists r \in \{e^{-n}, n \geq 1\} \cap (0,r_0): \abs{ \frac1{n^2} N_{z,r}^{\Lc_D^\theta \cup \Xi_a^z} - a } > \eta } \d z.
\end{equation}
A union bound shows that the probability in the integrand is at most
\begin{align*}
\sum_{\substack{r = e^{-n}\\n \geq |\log r_0|}} \left\{ \Prob{ \frac1{n^2} N_{z,r}^{\Xi_a^z} < a - \eta} + \Prob{ \frac1{n^2} N_{z,r}^{\Xi_a^z} > a + \eta/2 } + \Prob{ \frac1{n^2} N_{z,r}^{\Lc_D^\theta} > \eta/2 } \right\}.
\end{align*}
Each of the three probabilities appearing in the sum decays polynomially in $r$: see Corollary \ref{cor:crossingPPP3} for the first one, a slight variant of Corollary \ref{cor:crossingPPP1} for the second one, and Lemma \ref{lem:crossing_soup} for the last one. Therefore, the probability appearing in \eqref{eq:proof_thm_thick1} converges pointwise to zero as $r_0 \to 0$. By dominated convergence theorem, we obtain that $\Expect{ \Mc_a(D \setminus \Tc(a,r_0,\eta)) } \to 0$ as $r_0 \to 0$. In other words,
$\Mc_a$ is almost surely supported by
\[
\bigcup_{r_0>0} \Tc(a,r_0,\eta) = \left\{ z \in D: \exists r_0 >0, \forall r \in \{e^{-n}, n \geq 1 \} \cap (0,r_0): \abs{ \frac1{n^2} N_{z,r}^{\Lc_D^\theta} - a } \leq \eta \right\}.
\]
Since this is true for all $\eta>0$, this concludes the proof that $\Mc_a(D \setminus \Tc(a)) = 0$ a.s.

We now turn to the proof of the claims concerning the carrying dimension of $\Mc_a$ and the Hausdorff dimension of $\Tc(a)$. We start with the lower bound and we let $\eta \in [0,2-a)$, $A \Subset D$ and we assume that $b$ is close enough to $a$ so that Lemma \ref{lem:second_moment_aa'} holds.
Recall the definition \eqref{eq:def_measure_good_event} of $\tilde{\Mc}_a^K$.
Let us denote by $\tilde{\Mc}_{a,r_0}$ the limit of $(\log K)^{-\theta} \tilde{\Mc}_a^K$ (we keep track of the dependence in $r_0$). By Lemma \ref{lem:second_moment_bdd} and by Fatou's lemma, the energy
\[
e_{r_0}(A) := \int_{A \times A} \frac{1}{|z-z'|^\eta} \tilde{\Mc}_{a,r_0}(dz) \tilde{\Mc}_{a,r_0}(dz')
\]
has finite expectation and is therefore almost surely finite.  Moreover, Lemma \ref{lem:first_moment_good_event} and Fatou's lemma also show that
\[
\lim_{r_0 \to 0} \Expect{ \Mc_a(A) - \tilde{\Mc}_{a,r_0}(A) } =0.
\]
The following event has therefore full probability measure
\[
E := \bigcap_A \left\{ \liminf_{n \to \infty} \Mc_a(A) - \tilde{\Mc}_{a,e^{-n}}(A) = 0
\quad \text{and} \quad
\forall r_0 \in \{ e^{-n}, n \geq 1\}, e_{r_0}(A) < \infty \right\}
\]
where the intersection runs over all set $A$ of the form $\{ z \in D, \mathrm{dist}(z,\partial D) > e^{-n} \}, n \geq 1$. Now, let $B \subset D$ be a Borel set such that $\Mc_a(B) >0$.  There exists some set $A$ of the above form such that $\Mc_a(B \cap A) >0$. Moreover, since for all $r_0>0$,
\[
\Mc_a(B \cap A) - \tilde{\Mc}_{a,r_0}(B \cap A) \leq \Mc_a(A) - \tilde{\Mc}_{a,r_0}(A),
\]
we see that on the event $E$, we can find $r_0 \in \{e^{-n}, n \geq 1 \}$, such that $\tilde{\Mc}_{a,r_0}(B \cap A) >0$. But because on the event $E$, the energy $e_{r_0}(A)$ is finite, Frostman's lemma \cite[Theorem 3.4.2]{bishop2017fractals} implies that the Hausdorff dimension of $B \cap A$ is at least $\eta$. To wrap things up, we have proven that almost surely, for all Borel set $B$ such that $\Mc_a(B) >0$, the Hausdorff dimension of $B$ is at least $\eta$. Since $\eta$ can be made arbitrary close to $2-a$, this concludes the lower bound on the carrying dimension of $\Mc_a$. The lower bound on the dimension of $\Tc(a)$ follows since we have already proven that $\Mc_a$ is almost surely supported on $\Tc(a)$.

We now turn to the upper bound. We will show that the Hausdorff dimension of $\Tc(a)$ is almost surely at most $2-a$.
Since $\Mc_a(D \setminus \Tc(a))=0$ a.s., this will also provide the upper bound on the carrying dimension of $\Mc_a$ and it will conclude the proof.  Let $\delta>0$ and denote by $\Hc_{2-a+\delta}$ the $(2-a+\delta)$-Hausdorff measure.  Let $\eta>0$ be much smaller than $\delta$.
We first notice that
\begin{align*}
\Tc(a) \subset \bigcup_{N \geq 1} \bigcap_{n \geq N} \left\{ z \in D: \frac1{n^2} N_{z,e^{-n}}^{\Lc_D^\theta} > a - \eta \right\}
\subset \bigcap_{N \geq 1} \bigcup_{n \geq N} \left\{ z \in D: \frac1{n^2} N_{z,e^{-n}}^{\Lc_D^\theta} > a - \eta \right\}
\end{align*}
and we can therefore bound,
\[
\Hc_{2-a+\delta}(\Tc(a)) \leq \lim_{N \to \infty} \sum_{n \geq N} \Hc_{2-a+\delta} \left( \left\{ z \in D: \frac1{n^2} N_{z,e^{-n}}^{\Lc_D^\theta} > a - \eta \right\} \right).
\]
Now, let $n \geq 1$ be large and denote by $r_n = e^{-n}$. Let $\{z_i, i \in I\} \subset D$ be a maximal $r_n^{1+\eta}$-net of $D$ (in particular, $\# I \asymp r_n^{-2(1+\eta)}$). If $z \in D$ is such that $|z-z_i| < r_n^{1+\eta}$, we notice that the annulus $D(z,e r_n) \setminus D(z,r_n)$ contains the annulus $D(z_i, er_n - r_n^{1+\eta}) \setminus D(z_i, r_n + r_n^{1+\eta})$, and therefore, the number of crossings in $\Lc_D^\theta$ of the former annulus is smaller or equal than the number of crossings of the latter. This shows that we can cover
\[
\left\{ z \in D: \frac1{n^2} N_{z,e^{-n}}^{\Lc_D^\theta} > a - \eta \right\}
\subset \bigcup_{i \in I} \left\{ z \in D(z_i, e^{-(1+\eta)n}), \frac1{n^2} N_{z_i,r_n, \eta}^{\Lc_D^\theta} > a - \eta \right\}
\]
where we have denoted by $N_{z_i,r_n, \eta}^{\Lc_D^\theta}$ the number of upcrossings of $D(z_i, er_n - r_n^{1+\eta}) \setminus D(z_i, r_n + r_n^{1+\eta})$ in $\Lc_D^\theta$.
Let $\lambda \in (0,1)$ be close to 1. An immediate adaptation of Lemma \ref{lem:crossing_soup} to annuli with slightly different radii, shows that if $n$ is large enough, then
\[
\Prob{ \frac1{n^2} N_{z_i,r_n, \eta}^{\Lc_D^\theta} > a - \eta }
\leq r_n^{\lambda(a-\eta)}.
\]
Therefore
\begin{align*}
& \Expect{ \Hc_{2-a+\delta} \left( \left\{ z \in D: \frac1{n^2} N_{z,e^{-n}}^{\Lc_D^\theta} > a - \eta \right\} \right) } \\
& \leq
C \Expect{ \sum_{i \in I} \left( r_n^{1+\eta} \right)^{2-a+\delta} \indic{\frac1{n^2} N_{z_i,r_n, \eta}^{\Lc_D^\theta} > a - \eta }  } \leq C r_n^{(1+\eta)(-a+\delta)+\lambda(a-\eta)}.
\end{align*}
By choosing $\lambda$ and $\eta$ close enough to $1$ and $0$, respectively,  we can ensure the above power to be larger than $\delta/2$ ($\delta$ is fixed for now).  We have proven that
\[
\Expect{ \Hc_{2-a+\delta}(\Tc(a))} \leq C \lim_{N \to \infty} \sum_{n \geq N} r_n^{\delta/2} = 0
\]
and the Hausdorff dimension of $\Tc(a)$ is at most $2-a+\delta$ a.s. This concludes the proof.
\end{proof}

%\section{Measurability, conformal covariance and positivity}

\subsection{Measurability}\label{sec:measurability}

The purpose of this section is to prove Theorem \ref{th:convergence_continuum}, Point \ref{it:measurability}. In Section \ref{S:BMC}, we show that, essentially by definition, for all $K>0$, $\Mc_a^K$ is measurable w.r.t. $\sigma(\scalar{ \Lc_D^\theta(K) })$; see Lemma \ref{lem:measurability_intermediate}. Hence, this section consists in showing that the limiting measure $\Mc_a$ does not depend on the labels underlying the definition of killed loops.

Consider the Brownian loop soup
$\Lc^{\theta}_{D}$.
Since $D$ is bounded, one can order the loops in decreasing 
order of their diameter, $(\hat{\wp}_{i})_{i\geq 1}$.
Let $(U_{i})_{i\geq 1}$ be an i.i.d. sequence
of uniform r.v.s in $[0,1]$,
independent from $\Lc^{\theta}_{D}$.
Given $K>0$,
we consider that
$\Lc^{\theta}_{D}(K)$ is constructed according to \eqref{Eq L K},
with the r.v. $U_{i}$ associated to the loop $\hat{\wp}_{i}$.
It is the Borel $\sigma$-algebra for the topology on
collections of unrooted loops described in
Section \ref{sec:preliminaries_BLS}.
For $m\geq 1$, denote $\Fc_{m}$ the $\sigma$-algebra
generated by $\scalar{\Lc^{\theta}_{D}}$ and the r.v.s
$(U_{i})_{1\leq i\leq m}$,
and $\check{\Fc}_{m}$ the
$\sigma$-algebra
generated by $\scalar{\Lc^{\theta}_{D}}$ and the r.v.s
$(U_{i})_{i > m}$.
By Lemma \ref{lem:measurability_intermediate}, the random measure $\Mc_{a}$ 
is measurable with respect to $\check{\Fc}_{1}$.
We want to show that $\Mc_{a}$  admits a modification
coinciding a.s. with $\Mc_{a}$ which is measurable with respect to
$\sigma( \scalar{ \Lc_D^\theta } )$.

\begin{lemma}
\label{Lem meas m}
For every $m\geq 1$,
$\Mc_{a}$ is measurable with respect to $\check{\Fc}_{m}$.
\end{lemma}

\begin{proof}
For $K>0$, denote
\begin{displaymath}
\Mc_{a}^{K,m} :=
\sum_{n \geq 1} \frac{1}{n!} \sum_{\substack{\wp_1, \dots, \wp_n 
\\
\in \Lc_D^\theta(K)\cup \{\hat{\wp}_{i}, i=1\dots m-1\}
\\ \forall i \neq j, \wp_i \neq \wp_j}} \Mc_a^{\wp_1 \cap \dots \cap \wp_n}
.
\end{displaymath}
Introducing this measure is useful since $\Mc_a^{K,m}$ is independent of the first $m$ labels $U_i, i =1 \dots m$: the $m$ biggest loops will be always included, without having to check whether $U_i < 1 - e^{-K T(\hat{\wp}_i) }$ or not.
By Lemma \ref{lem:measurability_intermediate}, the random measure $\Mc_{a}^{K,m}$ is measurable with respect
to $\check{\Fc}_{m}$.
Moreover, a.s. for $K$ large enough, we have for all $i \in \{1, \dots, m \}$, $U_i < 1-e^{-KT(\hat{\wp}_i)}$. Thus, if $K$ is large enough, 
$\Mc_{a}^{K,m} = \Mc_{a}^{K}$ and $(\log K)^{-\theta}\Mc_{a}^{K,m}$
converges in probability as $K\to +\infty$
to $\Mc_{a}$.
This shows that $\Mc_{a}$ is $\check{\Fc}_{m}$-measurable.
\end{proof}

\begin{proposition}
\label{Prop Meas}
A.s., we have that
$\E[\Mc_{a}\vert \scalar{ \Lc_D^\theta } ] = \Mc_{a}$.
In particular, $\Mc_{a}$  admits a modification
coinciding a.s. with $\Mc_{a}$ which is measurable with respect to
$\scalar{ \Lc_D^\theta }$.
\end{proposition}

\begin{proof}
Lemma \eqref{Lem meas m} ensures that
for every $m\geq 1$, 
$\E[\Mc_{a}\vert \Fc_{m} ]=\E[\Mc_{a}\vert \scalar{ \Lc_D^\theta } ]$ a.s.
Further, as $m\to +\infty$,
$\E[\Mc_{a}\vert \Fc_{m} ]$ converges to
$\Mc_{a}$ a.s. and in $L^1$.
This concludes.
\end{proof}

\subsection{Proof of Theorem \ref{th:PD} (characterisation)}\label{sec:charac}

\begin{proof}[Proof of Theorem \ref{th:PD}]
\eqref{eq:th_PD} was proved in Section \ref{sec:PD}. Together with Proposition \ref{Prop Meas} from the previous section, it shows that the couple $(\Lc_D^\theta,\Mc_a)$ satisfies the three points of Theorem \ref{th:PD}. The fact that they characterize the law is standard, see \cite{bass1994, AidekonHuShi2018, jegoBMC}. Fix $\Lc_D^\theta$. We need to show that if $\widetilde \Mc_a$ is another Borel measure which is measurable with respect to $<\Lc_D^\theta>$ and verifies \eqref{eq:th_PD}, then $\widetilde \Mc_a= \Mc_a$ a.s. We  define $\widehat \Mc_a:=\widetilde \Mc_a-\Mc_a$. By \eqref{eq:th_PD} applied to $\widetilde \Mc_a$ and $\Mc_a$, the expectation of $\int_D F(z,\Lc_D^\theta)\widehat \Mc_a$ is zero for any bounded measurable admissible function $F$. Take $F(z,\Lc_D^\theta)=\widehat \Mc_a(A)\indic{z\in A}\indic{|\widehat \Mc_a(A)|<c}$ where $c>0$ and $A$ is a Borel set. We get that  $\Expect{\widehat \Mc_a(A)^2\indic{|\widehat \Mc_a(A)|<c}}$ vanishes, and therefore that $\Expect{\widehat \Mc_a(A)^2}=0$ by monotone convergence, so that $\widetilde \Mc_a(A)=\Mc_a(A)$ a.s. It completes the proof.
\end{proof}

\subsection{Conformal covariance}\label{sec:conformal_covariance}

Let $\psi : D \to \tilde{D}$ be a conformal map between two bounded simply connected domains. 
Recall that in Section \ref{sec:preliminaries_BLS}, we introduced the transformation $\Tc_\psi$ on paths defined by
\[
\Tc_\psi : (\wp(t), 0 \leq t \leq T(\wp)) \mapsto \left( \psi(\wp(S_{\psi, \wp}^{-1}(t)), 0 \leq t \leq S_{\psi,\wp}(T(\wp)) \right)
\]
where
\[
S_{\psi, \wp}(t) = \int_0^t |\psi'(\wp(s))|^2 \d s.
\]
For any collection $\Cc$ of loops in $D$, we define $\Tc_\psi \Cc := \{ \Tc_\psi \wp, \wp \in \Cc \}$.

\begin{theorem}\label{th:conformal_covariance_couple}
$(\Tc_\psi \Lc_D^\theta, |(\psi^{-1})'(\tilde{z})|^{-2-a} \Mc_{a,D} \circ \psi^{-1}(d \tilde{z}))$ and $(\Lc_{\tilde{D}}^\theta, \Mc_{a, \tilde{D}})$ have the same joint distribution, where loops are considered unrooted.
\end{theorem}

\begin{proof}
We are going to use the characterisation of the joint law of $(\Lc_{\tilde{D}}^\theta, \Mc_{a, \tilde{D}})$ given in Theorem \ref{th:PD} and we need to check that $(\Tc_\psi \Lc_D^\theta, |\psi'(\psi^{-1}(\tilde{z}))|^{2+a} \Mc_{a,D} \circ \psi^{-1})$ satisfies the three properties therein. By conformal invariance of the unrooted loop measure $\mu_D^{\text{loop}*}$, $\Tc_\psi \Lc_D^\theta$ has the same law as $\Lc_{\tilde{D}}^\theta$. This shows the first property. The second property concerning the measurability is clear since it is stable under conformal transformations. To conclude, we need to check the third property. Let $F : \tilde{D} \times \mathfrak{L}_{\tilde{D}} \to \R$ be a nonnegative measurable admissible function. By definition of the pushforward of $\Mc_{a,D}$, we have
\begin{align}
\label{eq:proof_conformal_covariance1}
& \Expect{ \int_{\tilde{D}} F(\tilde{z}, \Tc_\psi \Lc_D^\theta) |\psi'(\psi^{-1}(\tilde{z}))|^{2+a} \Mc_{a,D} \circ \psi^{-1}(d \tilde{z})  } \\
& = \Expect{ \int_D F(\psi(z), \Tc_\psi \Lc_D^\theta) |\psi'(z)|^{2+a} \Mc_{a,D}(dz) }. \nonumber
\end{align}
Since $(z, \Lc) \in D \times \mathfrak{L}_D \mapsto F(\psi(z), \Tc_\psi \Lc) |\psi'(z)|^{2+a} \in \R$ is a nonnegative measurable admissible function, we can apply Theorem \ref{th:PD} to obtain that the left hand side of \eqref{eq:proof_conformal_covariance1} is equal to
\begin{align*}
\frac{1}{2^{\theta} a^{1-\theta} \Gamma(\theta)} \int_D \Expect{ F(\psi(z), \Tc_\psi (\Lc_D^\theta \cup \{ \Xi_{a_i,D}^z, i \geq 1 \} ) ) } \CR(z,D)^a |\psi'(z)|^{2+a}  \d z.
\end{align*}
Above, we wrote $\Xi_{a_i,D}^z$ instead of $\Xi_{a_i}^z$ to emphasise that the underlying domain is $D$.
By doing the change of variable $\tilde{z} = \psi(z)$, and because $\CR(z,D) |\psi'(z)| = \CR(\tilde{z}, \tilde{D})$, we obtain that the left hand side of \eqref{eq:proof_conformal_covariance1} is equal to 
\[
\frac{1}{2^{\theta} a^{1-\theta} \Gamma(\theta)} \int_{\tilde{D}} \Expect{ F(\tilde{z}, \Tc_\psi (\Lc_D^\theta \cup \{ \Xi_{a_i,D}^{\psi^{-1}(\tilde{z})}, i \geq 1 \} ) ) } \CR(\tilde{z},\tilde{D})^a  \d \tilde{z}.
\]
Since the image of the measure $\mu_D^{\psi^{-1}(\tilde{z}), \psi^{-1}(\tilde{z})}$ under $\Tc_\psi$ is the measure $\mu_{\tilde{D}}^{\tilde{z}, \tilde{z}}$ (see \cite[Proposition 5.5]{LawlerConformallyInvariantProcesses}), and by conformal invariance of $\Lc_D^\theta$, we can rewrite
\[
\Expect{ F(\tilde{z}, \Tc_\psi (\Lc_D^\theta \cup \{ \Xi_{a_i,D}^{\psi^{-1}(\tilde{z})}, i \geq 1 \} ) ) }
= \Expect{ F(\tilde{z},  \Lc_{\tilde{D}}^\theta \cup \{ \Xi_{a_i,\tilde{D}}^{\tilde{z}}, i \geq 1 \} ) }.
\]
To wrap things up, we have proven that
\begin{align*}
& \Expect{ \int_{\tilde{D}} F(\tilde{z}, \Tc_\psi \Lc_D^\theta) |\psi'(\psi^{-1}(\tilde{z}))|^{2+a} \Mc_{a,D} \circ \psi^{-1}(d \tilde{z})  } \\
& = \frac{1}{2^{\theta} a^{1-\theta} \Gamma(\theta)} \int_{\tilde{D}} \Expect{ F(\tilde{z},  \Lc_{\tilde{D}}^\theta \cup \{ \Xi_{a_i,\tilde{D}}^{\tilde{z}}, i \geq 1 \} ) } \CR(\tilde{z},\tilde{D})^a  \d \tilde{z}
\end{align*}
which is the third property characterising the joint law of $(\Lc_{\tilde{D}}^\theta, \Mc_{a, \tilde{D}})$. This concludes the proof.
\end{proof}

\subsection{Positivity}

We conclude this section with the proof of Theorem \ref{th:convergence_continuum}, Point \ref{it:nondegenerate}.

\begin{proof}[Proof of Theorem \ref{th:convergence_continuum}, Point \ref{it:nondegenerate}]
The claim that, for all open set $A \subset D$, $\Mc_a(A)$ is finite almost surely, is clear since the total mass of $\Mc_a$ has finite expectation. We will therefore focus on proving that for all open set $A \subset D$, $\Mc_a(A) >0$ almost surely.  Let $A$ be such a set and let $A_1$ and $A_2$ be two disjoint subsets of $A$ that are scaled copies of $A$, i.e. we can write $A_i = f_i(A)$ where $f_i$, $i=1,2$, are affine functions.  In what follows, we keep track of the domain $D$ where the loop soup lives by writing $\Mc_{a,D}$ instead of $\Mc_a$. By only keeping loops that are contained in $A$, and by restriction property of Brownian loop soup, we see that $\Mc_{a,D}(A)$ stochastically dominates $\Mc_{a,A}(A)$. It is therefore sufficient to show that
$
\Prob{ \Mc_{a,A}(A) = 0 } = 0.
$
Similarly, by only keeping loops that are contained in $A_1 \cup A_2$, we obtain that
\[
\Prob{ \Mc_{a,A}(A) = 0 } \leq \Prob{ \Mc_{a,A_1 \cup A_2}(A_1 \cup A_2) = 0 }.
\]
Since $\Mc_{a,A_1 \cup A_2}(A_1 \cup A_2)$ is distributed like the independent sum $\Mc_{a,A_1}(A_1) + \Mc_{a,A_2}(A_2)$, we can rewrite the probability on the right hand side as the product of $\Prob{ \Mc_{a,A_i}(A_i) = 0 }$, $i=1,2$.
Now, by conformal covariance and because the $A_i$'s are affine transformations of $A$, $\Prob{ \Mc_{a,A_i}(A_i) = 0 } = \Prob{ \Mc_{a,A}(A) = 0 }$, $i=1,2$. We have therefore shown that
\[
\Prob{ \Mc_{a,A}(A) = 0 } \leq \Prob{ \Mc_{a,A}(A) = 0 }^2.
\]
Since this probability is strictly smaller than one (the expectation $\Mc_{a,A}(A)$ is positive), it has to vanish. This concludes the proof.
\end{proof}

\part*{Part Two: Discrete}

\addcontentsline{toc}{part}{Part Two: Discrete}

\section{Reduction}\label{sec:discrete_high_level}

The purpose of this section is to explain the main lines of the proof of Theorem \ref{th:convergence_discrete}. The idea is to use the result of \cite{jegoRW} that shows that the scaling limit of the set of thick points of planar random walk, killed upon exiting for the first time a given domain, is described by Brownian multiplicative chaos. In order to use this result, we will first compare the discrete measures and the continuum measures at the ``approximation level'', i.e. for loops killed by the mass. We will then show that the discrete measures with and without mass can be compared.

For $K>0$, define the set of $a$-thick points of $\Lc_{D_N}^\theta(K)$ by
\begin{equation}
\label{eq:def_thick_points_discrete_mass}
\Tc_{N,K}(a) := \left\{ z \in D_N:
\ell_{z}(\Lc_{D_N}^\theta(K))
\geq \g a (\log N)^2 \right\}
\end{equation}
(see \eqref{Eq Occup field} for the definition of $\ell_{z}(\Lc_{D_N}^\theta(K))$)
and the associated point measure
\begin{equation}
\label{eq:def_measure_discrete_mass}
\Mc_a^{N,K}(A) := \frac{\log N}{N^{2-a}} \sum_{z \in \Tc_{N,K}(a)} \indic{z\in A}.
\end{equation}
Importantly, the normalisation of the measure is the same as in the case of a single random walk trajectory (see (1.1) of \cite{jegoRW}). Without the mass cutoff, there are much more loops and the measure has to be tamed a bit more (see \eqref{eq:def_measure_discrete}).

Our first step will be to establish:

\begin{proposition}\label{prop:discrete_vs_continuum}
There exists a universal constant
$c_0 >0$, given by \eqref{eq:c0},
such that, for any fixed $K>0$, $(\Lc_{D_N}^\theta(K),\Mc_a^{N,K})$ converges in distribution as $N \to \infty$ towards $(\Lc_D^\theta(K),c_0^a \Mc_a^K)$.  The underlying topologies are the topology induced by the distance $d_\mathfrak{L}$ \eqref{eq:d_frac_L} and the topology of weak convergence of measures.
\end{proposition}

The second step will be to control the effect of the mass on the measure:

\begin{proposition}\label{prop:discrete_mass_infty}
For any Borel set $A \subset \C$,
\[
\limsup_{K \to \infty} \limsup_{N \to \infty} \Expect{ \abs{ \Mc_a^N(A) - \frac{2^\theta}{(\log K)^\theta} \Mc_a^{N,K}(A) } } = 0.
\]
Moreover, $d_\mathfrak{L}( \Lc_{D_N}^\theta, \Lc_{D_N}^\theta(K) )$ goes to zero in probability as $N \to \infty$ and then $K \to \infty$.
\end{proposition}

We can now prove Theorems \ref{th:convergence_discrete} and \ref{th:identification1/2}.

\begin{proof}[Proof of Theorem \ref{th:convergence_discrete}]
This is an immediate consequence of Theorem \ref{th:convergence_continuum} and Propositions \ref{prop:discrete_vs_continuum} and \ref{prop:discrete_mass_infty}.
\end{proof}

\begin{proof}[Proof of Theorem \ref{th:identification1/2}]
By Le Jan's isomorphism (Theorem \ref{Thm Iso Le Jan}), we can couple a discrete GFF $\varphi_N$ in $D_N$ and a random walk loop soup $\Lc_{D_N}^{1/2}$ with critical intensity in such a way that the occupation field $\ell(\Lc_{D_N}^{1/2})$ and $\tfrac{1}{2} \varphi_N^2$ coincide. Let $\Mc_a^N$ be the measure defined as in \eqref{eq:def_measure_discrete} and $C_0(D)$ be the space of continuous functions on $\bar{D}$ that vanish on $\partial D$. We view $\varphi_N$ as a random element of $\R^{C_0(D)}$ by setting for all $f \in C_0(D)$,
\[
\varphi_N f := \frac{1}{N^2} \sum_{z \in D_N} \varphi_N(z) f(z).
\]
We are going to show that $(\varphi_N, \Lc_{D_N}^{1/2}, 2^{-\theta} c_0^{-a} \Mc_a^N)$ converges in distribution (along a subsequence) towards a triplet that satisfies all the relations required by Theorem \ref{th:identification1/2}. The topologies associated to $\Lc_{D_N}^{1/2}$ and $\Mc_a^N$ are the same ones as in Theorem \ref{th:convergence_discrete} and the topology associated to $\varphi_N$ is the product topology on $\R^{C_0(D)}$. To establish such a result, we only need to argue that
\begin{enumerate}[(i)]
\item
$(\Lc_{D_N}^{1/2}, 2^{-\theta} c_0^{-a} \Mc_a^N) \xrightarrow[]{(d)} (\Lc_D^{1/2}, \Mc_a)$ where $\Mc_a$ is the multiplicative chaos associated to $\Lc_D^{1/2}$ from Theorem \ref{th:convergence_continuum};
\item
$(\varphi_N, 2^{-\theta} c_0^{-a} \Mc_a^N) \xrightarrow[]{(d)} (\varphi, \tfrac{1}{\sqrt{2\pi a}} \cosh(\gamma h))$ where $\cosh(\gamma h)$ is the hyperbolic cosine associated to $h = \sqrt{2\pi} \varphi$;
\item
$(\Lc_{D_N}^{1/2}, \varphi_N) \xrightarrow[]{(d)} (\Lc_D^{1/2}, \varphi)$ along a subsequence $(N_k)_{k \geq 1}$, where the Brownian loop soup $\Lc_D^{1/2}$ and the GFF $\varphi$ satisfy Le Jan's identity: $ :\! \ell(\Lc_D^{1/2}) \!: \ = \tfrac{1}{2} :\!\varphi^2\!:$.
\end{enumerate}
Indeed, assume these three convergences. The law of $(\varphi_{N_k}, \Lc_{D_{N_k}}^{1/2}, 2^{-\theta} c_0^{-a} \Mc_a^{N_k})$ is tight since each of the three components converges. Let $(\varphi_\infty, \Lc_{D,\infty}^{1/2}, \Mc_{a,\infty})$ be any subsequential limit. The three pairwise convergences above suffice to identify the law of this triplet: $\varphi_\infty$ is a GFF in $D$ and $\Lc_{D,\infty}^{1/2}$ is a critical Brownian loop soup in $D$ related by Le Jan's identity; $\Mc_{a,\infty}$ is measurable w.r.t. $\Lc_{D,\infty}^\theta$ and is the associated multiplicative chaos; $\Mc_{a,\infty}$ is measurable w.r.t. $\varphi_\infty$ and is the associated hyperbolic cosine.

To conclude the proof, we need to explain where (i)-(iii) come from. (i) is the content of Theorem \ref{th:convergence_discrete}. (ii) is a quick consequence of \cite{BiskupLouidor} as we are about to explain. By definition \eqref{eq:def_measure_discrete} of $\Mc_a^N$ and because the occupation field of $\Lc_{D_N}^{1/2}$ is equal to $\varphi_N^2 /2$ (without any normalisation), $\Mc_a^N$ is equal to $\eta_\gamma^N + \eta_{-\gamma}^N$ where $\gamma = \sqrt{2a}$ and $\eta_{\pm \gamma}^N$ are the measures defined by
\[
\eta_{\pm \gamma}^N(A) := \frac{\sqrt{\log N}}{N^{2-\gamma^2/2}} \sum_{z \in D_N} \indic{z \in A} \indic{ \pm \varphi_N(z) \geq \frac{\gamma}{\sqrt{2} \pi} \log N},
\quad A \subset \R^2 \text{~Borel~set}.
\]
By \cite[Theorems 2.1 and 2.5]{BiskupLouidor}, there exists some universal constant $c_*>0$ such that
$\eta_\gamma^N$ converges in distribution to $c_* e^{\gamma h}$ where $h = \sqrt{2\pi} \varphi$ is a Gaussian free field in $D$. This convergence can be easily extended to the joint convergence of $(\varphi_N,\eta_\gamma^N)$ to $(\varphi, c_* e^{\gamma h})$. Indeed, this extension follows from a simple use of Girsanov's theorem and the details can be found in the proof of \cite[Lemma 6.9]{2010.13939} in a slightly different setting. The two convergences $(\varphi_N,\eta_\gamma^N) \to (\varphi, c_* e^{\gamma h})$ and $(\varphi_N,\eta_{-\gamma}^N) \to (\varphi, c_* e^{-\gamma h})$, plus the fact that the limiting measures are measurable w.r.t. the underlying GFF $\varphi$, imply the joint convergence $(\varphi_N,\eta_\gamma^N, \eta_{-\gamma}^N) \to (\varphi, c_* e^{\gamma h}, c_* e^{-\gamma h})$. In particular, $(\varphi_N, \eta_{\gamma}^N + \eta_{-\gamma}^N) \to (\varphi, 2c_* \cosh(\gamma h))$ as desired in (ii). The value of the constant $c_*$ can be computed looking at the first moment.

Finally, let us prove (iii). This is an immediate consequence of the two joint convergences 
$$(\Lc_{D_N}^{1/2}, \ell(\Lc_{D_N}^{1/2}) - \E \ell(\Lc_{D_N}^{1/2}) ) \xrightarrow[]{(d)} (\Lc_D^{1/2}, :\! \ell(\Lc_D^{1/2}) \!: )$$ 
and 
$$(\varphi_N, \varphi_N^2 - \E \varphi_N^2 ) \xrightarrow[]{(d)} (\varphi, :\! \varphi^2 \!:)$$
along a subsequence $(N_k')_{k \geq 1}$ (see the proof of \cite[Lemma 6]{QianWerner19Clusters} and Lemma \ref{L:Wick_discrete_continuous}, respectively). Indeed, these convergences implies tightness of the quadruple $(\Lc_{D_N}^{1/2}, \ell(\Lc_{D_N}^{1/2}) - \E \ell(\Lc_{D_N}^{1/2}), \varphi_N, 1/2 \varphi_N^2 - 1/2\E \varphi_N^2)$ along $(N_k')$. Let $(N_k)_{k \geq 1}$ be a further subsequence of $(N_k')_{k \geq 1}$ such that the quadruple above converges towards some 
$$(\Lc_D^{1/2}, :\! \ell(\Lc_D^{1/2}) \!:, \varphi, 1/2 :\! \varphi^2 \!:)$$ along $(N_k)$. To conclude, we only need to make sure that the second and fourth components of the limiting variable agree. Our specific choice of coupling between $\Lc_{D_N}^{1/2}$ and $\varphi_N$ ensures that this is always true at the discrete level. Therefore, it is also true in the limit.
\end{proof}

\begin{remark}
  Notice that the above argument does not establish convergence since Le Jan's isomorphism (as noted earlier) does not uniquely determine the joint law of the free field and loop soup. Nevertheless, the subsequential limit satisfies the relations stated in Theorem \ref{th:identification1/2}.
\end{remark}

The remaining of Part Two is organised as follows. In Section \ref{sec:discrete_exact}, we give exact expressions for the first two moments associated to $\Mc_a^N$ and $\Mc_a^{N,K}$, as well as describing the associated conditional laws of the random walk loop soup $\Lc_{D_N}^\theta$. These exact formulae will be instrumental in the proof of Proposition \ref{prop:discrete_mass_infty} which is achieved in Section \ref{sec:discrete_convergence}. Finally, Section \ref{sec: massivechaoslimit} is dedicated to the proof of Proposition \ref{prop:discrete_vs_continuum}.

\section{Exact expressions}\label{sec:discrete_exact}

In this section we will give the expressions of the first and second moments for $\Mc_a^{N}$ and $\Mc_a^{N,K}$,
as well as give the corresponding conditional laws of the
random walk loop soup $\Lc_{D_N}^\theta$.

\subsection{First moment (discrete Girsanov)}\label{sec:discrete_exact1}

Recall the definition \eqref{eq:c0} of the constant $c_0$ which appears in the asymptotic of the Green function on the diagonal; see \eqref{eq:lem_Green_discrete3}.
%is the constant order term in the expansion of the $0$-potential on $\Z^{2}_{N}$; see \cite[Theorem~4.4.4]{LawlerLimic2010RW}.

In this section $D_{N}$ will be just a subset of
$\Z^{2}_{N}$,
with both $D_{N}$ and $\Z^{2}_{N}\setminus D_{N}$
non-empty.
For $z\in D_{N}$, denote
\begin{displaymath}
\CR_{N}(z,D_{N})
:= N c_{0}^{-1}
e^{-(\log N)^{2}/(2\pi G_{D_{N}}(z,z))}.
\end{displaymath}
As the notation suggests, we will use $\CR_{N}(z,D_{N})$
in a situation where it converges to a conformal radius as
$N\to +\infty$; see \eqref{eq:lem_Green_discrete3}.
Let $q_{N}(z)$ be the ratio
\begin{displaymath}
q_{N}(z) := \dfrac{\log N}{2\pi G_{D_{N}}(z,z)}.
\end{displaymath}
If $N\to +\infty$ and the Euclidean distance from $z$
to $\Z^{2}_{N}\setminus D_{N}$ is bounded away from $0$,
then $q_{N}(z)\to 1$;
see Lemma \ref{lem:Green_discrete}.

Given $z,w \in D_{N}$, we will denote by $\tilde{\mu}_{D_{N}}^{z,w}$
the renormalised measure
\begin{equation}
\label{Eq tilde mu}
\tilde{\mu}_{D_{N}}^{z,w} :=
\Big(\dfrac{1}{2\pi}\log N\Big)^{2}
\check{\mu}_{D_{N}}^{z,w},
\end{equation}
where $\check{\mu}_{D_{N}}^{z,w}$ is given by
\eqref{Eq check mu}.
Given $z\in D_{N}$ and $a>0$,
we will denote by $\Xi^{z}_{N, a}$
the random loop in $D_{N}$, obtained by
concatenating a Poisson point process of continuous time random walk excursions from $z$ to $z$
of intensity $2\pi a \tilde{\mu}_{D_{N}}^{z,z}$,
and having a local time in $z$
\begin{displaymath}
\ell_{z}(\Xi^{z}_{N, a})
= \dfrac{1}{2\pi} a (\log N)^{2} .
\end{displaymath}
As in Section \ref{Sec 1st moment},
we will consider admissible functions $F$
which do not depend on the order of excursions in a loop.

The following proposition is merely a rephrasing of
Proposition \ref{Prop Le Jan subordinator}
in terms of the random discrete measure
$\Mc_a^N$ given by \eqref{eq:def_measure_discrete}.
It is to be compared to
Theorem \ref{th:PD} for the continuum setting.

\begin{proposition}
\label{Prop 1st mom Mc N}
Fix $z\in D_{N}$ and $a>0$.
For any bounded measurable admissible function $F$,
\begin{align*}
&\Expect{
F(\Lc_{D_{N}}^{\theta}) \Mc_a^N(\{ z\})
}
\\
& =
\dfrac{1}{N^{2} \Gamma(\theta)}
q_{N}(z)^{\theta}
\log N
\int_{a}^{+\infty}
\rho^{\theta - 1}
\dfrac{c_{0}^{\rho}}{N^{\rho - a}}
\CR_{N}(z,D_{N})^{\rho}
\Expect{
F(\Lc_{D_{N}\setminus\{ z\}}^{\theta}
\cup
\{\Xi^{z}_{N, a_{i}}, i\geq 1\}
)
} \d \rho,
\end{align*}
where on the right-hand side,
$\Lc_{D_{N}\setminus\{ z\}}^{\theta}$ and
$\{\Xi^{z}_{N, a_{i}}, i\geq 1\}$ are
independent,
the $(a_{i})_{i\geq 1}$
is a Poisson-Dirichlet partition $\PD (0,\theta)$ of
$[0,\rho]$,
and the $\Xi^{z}_{N, a_{i}}$
are conditionally independent given $(a_{i})_{i\geq 1}$.
\end{proposition}

\begin{proof}[Proof of Proposition \ref{Prop 1st mom Mc N}]
By definition of $\Mc_a^N$, we have
\begin{align*}
&\Expect{
F(\Lc_{D_{N}}^{\theta}) \Mc_a^N(\{ z\})
}
= \frac{(\log N)^{1-\theta}}{N^{2-a}} \Expect{
F(\Lc_{D_{N}}^{\theta}) \indic{ \ell_z( \Lc_{D_N}^\theta ) \geq \frac{1}{2\pi} a (\log N)^2 } } \\
& = \frac{(\log N)^{1-\theta}}{N^{2-a}} \int_a^\infty
\Prob{ 2\pi (\log N)^{-2} \ell_z( \Lc_{D_N}^\theta ) \in \d \rho } \Expect{
F(\Lc_{D_{N}}^{\theta}) \big\vert \ell_z( \Lc_{D_N}^\theta ) = \frac{1}{2\pi} \rho (\log N)^2 }.
\end{align*}
By Proposition \ref{Prop Le Jan subordinator}, the local time $\ell_z( \Lc_{D_N}^\theta )$ follows a Gamma distribution with parameter $\theta$. The conditional law of the loop soup $\Lc_{D_N}^\theta$, conditionally on its local time at $z$, is described in that same proposition. The Poisson-Dirichlet partition that appears above
comes from the Gamma subordinator \eqref{Eq Gamma subord}; see \cite{Kingman}. Performing a change of variable and rearranging the terms concludes the proof of the proposition.
\end{proof}

Now let us consider the massive case. Fix $K>0$ a constant.
For $z\in D_{N}$, denote
\begin{displaymath}
q_{N, K}(z) := \dfrac{\log N}{2\pi G_{D_{N}, K}(z,z)},
\qquad
C_{N,K}(z):=
2\pi(G_{D_{N}}(z,z)-G_{D_{N}, K}(z,z)),
\end{displaymath}
\begin{displaymath}
J_{N,K}(z) :=
\int_{0}^{+\infty}
(
e^{-t/G_{D_{N}}(z,z)}
-
e^{-t/G_{D_{N}, K}(z,z)}
)
\dfrac{\d t}{t} .
\end{displaymath}
Again, $q_{N, K}(z)$ tends to $1$ if $N \to +\infty$
and the Euclidean distance from $z$
to $\Z^{2}_{N}\setminus D_{N}$ is bounded away from $0$.

\begin{lemma}
\label{Lem prob kill discr}
For every $z\in D_{N}$ and $a>0$,
\begin{align*}
\Expect{e^{-K T(\Xi^{z}_{N, a})}}&=
e^{-(2\pi)^{-1} a (\log N)^{2}(G_{D_{N}, K}(z,z)^{-1}-G_{D_{N}}(z,z)^{-1})}
\\&=
e^{- q_{N}(z) q_{N, K}(z) C_{N,K}(z) a}.
\end{align*}
\end{lemma}

\begin{proof}
The expectation above is simply given by the ratio
between \eqref{Eq Gamma subord K} and \eqref{Eq Gamma subord}
for $t=\frac{1}{2\pi} a (\log N)^{2}$.
\end{proof}

The following proposition is to be compared to
Lemma \ref{lem:Girsanov_K}
and Proposition \ref{prop:first_moment_killed_loops}.

\begin{proposition}
\label{Prop 1st mom Mc N K}
For any bounded measurable admissible function $F$,
\begin{align}
\label{Eq Giranov Mc N K}
&\Expect{
F(\Lc_{D_{N}}^{\theta}) \Mc_a^{N, K}(\{ z\})
}
\\
\nonumber
& =
\dfrac{\log N}{N^{2}}
e^{-\theta J_{N,K}(z)}
\int_{a}^{+\infty}
\d \rho
\dfrac{c_{0}^{\rho}}{N^{\rho - a}}
\CR_{N}(z,D_{N})^{\rho}
\sum_{n \geq 1} \frac{\theta^n}{n!}
\int_{\mathsf{a} \in E(\rho,n)} \frac{\d \mathsf{a}}{a_1 \dots a_n}
\\
\nonumber
& \times
\Expect{
\prod_{i=1}^n \left( 1 - e^{-K T(\Xi_{N, a_i}^z) } \right)
F(\Lc_{D_{N}\setminus \{ z\}}^\theta
\cup \widetilde{\Lc}_{D_{N},K,z}^{\theta}
\cup \{ \Xi_{N, a_i}^z, i = 1 \dots n \} )
}
,
\end{align}
where on the right-hand side,
the three collections of loops
$\Lc_{D_{N}\setminus \{ z\}}^\theta$,
$\widetilde{\Lc}_{D_{N},K,z}^{\theta}$
and
$\{ \Xi_{N, a_i}^z, i = 1 \dots n \}$
are independent,
the different $\Xi_{N, a_i}^z$ are independent,
and $\widetilde{\Lc}_{D_{N},K,z}^{\theta}$
is distributed as the loops in
$\Lc_{D_{N}}^\theta\setminus\Lc_{D_{N}}^\theta(K)$
visiting $z$.
In particular,
\begin{equation}
\label{eq:prop_first_discrete}
\Expect{
\Mc_a^{N, K}(\{ z\})
}
=
\dfrac{\log N}{N^{2}}
e^{-\theta J_{N,K}(z)}
\int_{a}^{+\infty}
\dfrac{c_{0}^{\rho}}{N^{\rho - a}}
\CR_{N}(z,D_{N})^{\rho}
\Fs(q_{N}(z) q_{N, K}(z) C_{N,K}(z) \rho)
\dfrac{\d \rho}{\rho},
\end{equation}
where $\Fs$ is given by \eqref{eq:def_Fs}.
\end{proposition}

\begin{proof}
The second identity follows from the first one and
Lemma \ref{Lem prob kill discr}, by taking $F=1$.
See also Lemma \ref{lem:Fs}.

Regarding the identity \eqref{Eq Giranov Mc N K},
observe that the loops in $\Lc_{D_{N}}^\theta(K)$
visiting $z$ form a Poisson point process which is a.s.
finite,
regardless of $D_{N}$ being finite or not.
For instance, the intensity measure for
$(\ell_z(\wp))_{\wp\in \Lc^\theta_{D_{N}}(K),
\wp \text{ visits } z}$
is
\begin{displaymath}
\indic{t>0}\theta
(
e^{-t/G_{D_{N}}(z,z)}
-
e^{-t/G_{D_{N}, K}(z,z)}
)
\dfrac{\d t}{t} ,
\end{displaymath}
which is the difference between \eqref{Eq Gamma subord}
and \eqref{Eq Gamma subord K}.
Its total mass is finite, equal to
$\theta J_{N,K}(z)$.
We obtain \eqref{Eq Giranov Mc N K} by summing over the values of
$\# \{\wp\in \Lc^\theta_{D_{N}}(K) : \wp \text{ visits } z \}$.
We skip the details.
\end{proof}

As a corollary,

\begin{corollary}\label{cor:discrete_first}
Let $D$ be an open bounded simply connected domain, $(D_N)_N$ be a discrete approximation of $D$ as in \eqref{eq:DN} and $f: D \to [0,\infty)$ be a nonnegative bounded continuous function. Then
\[
\sup_{N \geq 1} \Expect{ \scalar{ \Mc_a^N, f} } < \infty
\quad \mathrm{and} \quad
\lim_{N \to \infty} \Expect{ \scalar{ \Mc_a^N, f} } =
c_0^a \frac{a^{\theta-1}}{\Gamma(\theta)} \int_D f(z) \CR(z,D)^a \d z.
\]
Moreover,
\[
\sup_{N \geq 1} \Expect{ \scalar{ \Mc_a^{N,K}, f} } < \infty
\quad \mathrm{and} \quad
\lim_{N \to \infty} \Expect{ \scalar{ \Mc_a^{N,K}, f} } = c_0^a \frac{\Fs(C_K(z) a)}{a} \int_D f(z) \CR(z,D)^a \d z.
\]
\end{corollary}

\begin{proof}
With Propositions \ref{Prop 1st mom Mc N} and \ref{Prop 1st mom Mc N K} in hand, checking this corollary is a simple computation.
\end{proof}

\subsection{Second moment (two-point discrete Girsanov)}\label{sec:discrete_exact2}

Here we will deal with the second moments of
$\Mc_a^{N}$ and $\Mc_a^{N,K}$.

Given $z\in D_{N}$,
the Green function on
$D_{N}\setminus\{ z\}$ can be expressed as follows:
\begin{equation}
\label{eq:Green_DN-z}
G_{D_{N}\setminus\{ z\}}(z',w')
= G_{D_{N}}(z',w') -
\dfrac{G_{D_{N}}(z',z)G_{D_{N}}(z,w')}
{G_{D_{N}}(z,z)} .
\end{equation}
Given $z\neq w \in D_{N}$,
denote
\begin{displaymath}
q_{N}(z,w)
:=
\dfrac{(\log N)^{2}}
{4 \pi^{2}(G_{D_{N}}(z,z) G_{D_{N}}(w,w)
-G_{D_{N}}(z,w)^{2})} .
\end{displaymath}
Let $\widetilde{G}_{D_{N}}(z,w)$ denote the total mass
of the measure $\tilde{\mu}_{D_{N}}^{z,w}$.

\begin{lemma}
\label{Lem tilde G}
Let $z,w\in D_{N}$ such that the graph distance on $\Z^{2}_{N}$
between $z$ and $w$ is at least $2$,
i.e. $\vert w -z\vert > \frac{1}{N}$.
Then,
\begin{displaymath}
\widetilde{G}_{D_{N}}(z,w)
= q_{N}(z,w) G_{D_{N}}(z,w).
\end{displaymath}
\end{lemma}

\begin{proof}
From (3) in Lemma \ref{Lem Markov discrete}
follows that the total mass of
$\check{\mu}_{D_{N}}^{z,w}$ equals
\begin{displaymath}
\dfrac{G_{D_{N}}(z,w)}
{G_{D_{N}}(z,z)G_{D_{N}\setminus\{ z\}}(w,w)}
=
\dfrac{G_{D_{N}}(z,w)}
{G_{D_{N}}(z,z) G_{D_{N}}(w,w)
-G_{D_{N}}(z,w)^{2}} .
\qedhere
\end{displaymath}
\end{proof}

Given $z\neq z' \in D_{N}$, denote
\begin{displaymath}
\CR_{N,z}(z',D_{N})
:=
N c_{0}^{-1}
e^{-(\log N)^{2}/(2\pi G_{D_{N}\setminus\{ z\}}(z',z'))}.
\end{displaymath}
Given $z\neq z' \in D_{N}$ and $a'>0$,
let $\Xi_{N, z, a'}^{z'}$ denote the
the random loop in $D_{N}\setminus\{ z\}$, obtained by
concatenating a Poisson point process of continuous time random walk excursions from $z'$ to $z'$
of intensity $2\pi a' \tilde{\mu}_{D_{N}\setminus\{ z\}}^{z',z'}$,
and having a local time in $z'$
\begin{displaymath}
\ell_{z'}(\Xi_{N, z, a'}^{z'})
= \dfrac{1}{2\pi} a' (\log N)^{2} .
\end{displaymath}
By construction, $\Xi_{N, z, a'}^{z'}$
does not visit $z$. Applying Lemma \ref{Lem prob kill discr} to $D_N \setminus \{ z \}$, we have that
\begin{equation}
\label{eq:discrete_CK1}
\Expect{ e^{-K T(\Xi_{N, z, a'}^{z'})} } = e^{-a' C_{N,K,z}(z')}
\end{equation}
where
\begin{equation}
\label{eq:discrete_CK2}
C_{N,K,z}(z') := \frac{(\log N)^2}{4\pi^2 G_{D_N \setminus \{z\}}(z',z') G_{D_N \setminus \{z\},K}(z',z')} 2\pi \left( G_{D_N \setminus \{z \} }(z',z') - G_{D_N \setminus \{z\},K }(z',z') \right)
\end{equation}
where we recall that the massive Green function is defined in \eqref{eq:def_massive_green_discrete}.

\begin{lemma}
\label{Lem Markov tilde mu}
Let $z,z'\in D_{N}$ such that the graph distance on $\Z^{2}_{N}$
between $z$ and $z'$ is at least $2$,
i.e. $\vert z' -z\vert > \frac{1}{N}$.
Let $a>0$.
Then for any bounded measurable function $F$,
\begin{multline*}
\int
\indic{\wp \text{ visits } z'}
F(\wp)
\tilde{\mu}_{D_{N}}^{z,z}(d\wp)
\\ =
2\pi
\int_{0}^{+\infty} \d a'
\dfrac{c_{0}^{a'}}{N^{a'}}
\CR_{N,z}(z',D_{N})^{a'}
\int \tilde{\mu}_{D_{N}}^{z,z'}(d\wp_{1})
\int \tilde{\mu}_{D_{N}}^{z',z}(d\wp_{2})
\Expect{F(\wp_{1}\wedge\Xi_{N, z, a'}^{z'}\wedge \wp_{2})} ,
\end{multline*}
where $\wedge$ denotes the concatenation of paths.
\end{lemma}

\begin{proof}
This is a consequence of Lemma \ref{Lem Markov discrete}.
By applying (2) in Lemma \ref{Lem Markov discrete}
in the case $z=w$,
and then (1) in Lemma \ref{Lem Markov discrete}
for the measure $\mu_{D_{N}\setminus\{z\}}^{z',z'}$,
we get that
\begin{align}
\label{Eq tilde mu intrmdt}
& \int
\indic{\wp \text{ visits } z'}
F(\wp)
\tilde{\mu}_{D_{N}}^{z,z}(d\wp)
\\
\nonumber
& =
\dfrac{4\pi^{2}}{(\log N)^{2}}
\int_{0}^{+\infty} \d t
e^{-t/ G_{D_{N}\setminus\{ z\}}(z',z')}
\int \tilde{\mu}_{D_{N}}^{z,z'}(d\wp_{1})
\int \tilde{\mu}_{D_{N}}^{z',z}(d\wp_{2})
\\
\nonumber
& \times
\Expect{F(\wp_{1}\wedge
\Xi_{N, z, 2\pi t (\log N)^{-2}}^{z'}\wedge \wp_{2})} .
\end{align}
We conclude by performing the change of variables
$a'=2\pi t (\log N)^{-2}$.
\end{proof}

Given $z\neq z'\in D_{N}$ and $a, a'>0$,
let $\Xi^{z,z'}_{N,a,a'}$ denote the random collection of
an even number of excursions from $z$ to $z'$ with the following law.
For all $k \geq 1$,
\begin{equation}
\label{eq:discrete_def_PPPzz'}
\Prob{ \# \Xi_{N,a,a'}^{z,z'} = 2k }
= \dfrac{1}{\Bs\Big((2\pi)^2 aa' \widetilde{G}_{D_{N}}(z,z')^2\Big)}
\dfrac{(2\pi \sqrt{aa'} \widetilde{G}_{D_{N}}(z,z'))^{2k}}{k! (k-1)!},
\end{equation}
where $\Bs$ is given by \eqref{eq:def_Gs},
and conditionally on $\{ \# \Xi_{N,a,a'}^{z,z'} = 2k \}$,
$\Xi_{N,a,a'}^{z,z'}$ is composed of $2k$ i.i.d. excursions with common law $\tilde{\mu}_{D_{N}}^{z,z'} /\widetilde{G}_{D_{N}}(z,z')$.
As in Section \ref{Sec 2nd mom},
we will consider admissible function,
invariant under reordering of excursions.

\begin{lemma}
\label{Lem decomp flower discrete}
Let $z,z'\in D_{N}$ such that the graph distance on $\Z^{2}_{N}$
between $z$ and $z'$ is at least $2$,
i.e. $\vert z' -z\vert > \frac{1}{N}$.
Let $a>0$.
Then for any bounded measurable admissible function $F$,
\begin{align*}
&
\Expect{
\indic{\Xi_{N,a}^{z} \text{ visits } z'}
F(\Xi_{N,a}^{z})}
\\ & =
e^{- a (2\pi)^{3} \widetilde{G}_{D_{N}}(z,z')^2
G_{D_{N}\setminus \{z\}}(z',z')
/ (\log N)^{2}
}
\int_{0}^{+\infty}
\dfrac{\d a'}{a'}
\dfrac{c_{0}^{a'}}{N^{a'}}
\CR_{N,z}(z',D_{N})^{a'}
\Bs\big((2\pi)^2 aa' \widetilde{G}_{D_{N}}(z,z')^2\big)
\\ & \times
\Expect{F(
\Xi_{N,a,a'}^{z,z'}
\wedge\Xi_{N, z', a}^{z}
\wedge\Xi_{N, z, a'}^{z'})} ,
\end{align*}
where $\Xi_{N,a,a'}^{z,z'}$, $\Xi_{N, z', a}^{z}$
and $\Xi_{N, z, a'}^{z'}$ are independent.
\end{lemma}

\begin{proof}
In $\Xi_{N,a}^{z}$,
the excursions away from $z$ visiting $z'$
are independent from those not visiting $z'$.
The concatenation of the excursions not visiting $z$' is
distributed as $\Xi_{N, z', a}^{z}$.
The excursions visiting $z'$ form a Poisson point process which is
a.s. finite.
According to \eqref{Eq tilde mu intrmdt},
the total mass of the corresponding intensity measure is
\begin{displaymath}
2\pi
a \times
\dfrac{4 \pi^{2}}{(\log N)^{2}}
\widetilde{G}_{D_{N}}(z,z')^2
G_{D_{N}\setminus \{z\}}(z',z')
.
\end{displaymath}
According to Lemma \ref{Lem Markov tilde mu},
and excursion that goes $k$ times there and back between
$z$ and $z'$ can be decomposed into $2k$ excursion
between $z$ and $z'$ and $k$ excursions
$\Xi_{N, z, a_{1}'}^{z'},
\dots,\Xi_{N, z, a_{k}'}^{z'}$
from $z'$ to $z'$ not visiting $z$.
The "thicknesses" (i.e. renormalised local times)
$a_{1}',\dots, a_{k}'$
are random and i.i.d.
The excursions
$\Xi_{N, z, a_{1}'}^{z'},
\dots,\Xi_{N, z, a_{k}'}^{z'}$
are conditionally independent given
$(a_{1}',\dots, a_{k}')$.
The concatenation
$\Xi_{N, z, a_{1}'}^{z'}\wedge
\dots \wedge\Xi_{N, z, a_{k}'}^{z'}$
is distributed as
$\Xi_{N, z, a'}^{z'}$
where $a'=a_{1}' + \dots + a_{k}'$.
The $2k$ excursions from $z$ to $z$'
are i.i.d., independent from
$a_{1}',\dots, a_{k}'$ and
$\Xi_{N, z, a_{1}'}^{z'},
\dots,\Xi_{N, z, a_{k}'}^{z'}$,
each one distributed according to
$\tilde{\mu}_{D_{N}}^{z,z'} /\widetilde{G}_{D_{N}}(z,z')$.
The distribution of $(a_{1}',\dots, a_{k}')$
on the event that
$\Xi_{N,a}^{z}$ performs $k$ travels from $z$ to $z'$ (and $k$ back)
is
\begin{align*}
&\indic{a_{1}',\dots, a_{k}'>0}
e^{- a (2\pi)^{3} \widetilde{G}_{D_{N}}(z,z')^2
G_{D_{N}\setminus \{z\}}(z',z')
/ (\log N)^{2}
}
\dfrac{(2\pi)^{2k} a^{k}\widetilde{G}_{D_{N}}(z,z')^{2k}}{k!}
\\
& \times
\prod_{i=1}^{k}
\Big(
\dfrac{c_{0}^{a_{i}'}}{N^{a_{i}'}}
\CR_{N,z}(z',D_{N})^{a_{i}'}
\d a_{i}'
\Big).
\end{align*}
The induced distribution on
$a'=a_{1}' + \dots + a_{k}'$ is
\begin{displaymath}
\indic{a'>0}
e^{- a (2\pi)^{3} \widetilde{G}_{D_{N}}(z,z')^2
G_{D_{N}\setminus \{z\}}(z',z')
/ (\log N)^{2}
}
\dfrac{(2\pi)^{2k} (a a')^{k}\widetilde{G}_{D_{N}}(z,z')^{2k}}{k! (k-1)!}
\dfrac{c_{0}^{a'}}{N^{a'}}
\CR_{N,z}(z',D_{N})^{a'}
\dfrac{\d a'}{a'}.
\end{displaymath}
One recognizes above the $k$-th term in the expansion of
$\Bs\big((2\pi)^2 aa' \widetilde{G}_{D_{N}}(z,z')^2\big)$;
see \eqref{eq:def_Gs}.
This concludes.
\end{proof}

Next we consider the loop measure
$\loopmeasure_{D_{N}}$ \eqref{Eq discr loops}
and the decomposition of loops that visit
two given vertices $z$ and $z'$.

\begin{lemma}
\label{Lem loops 2 pts}
Let $z,z'\in D_{N}$ such that the graph distance on $\Z^{2}_{N}$
between $z$ and $z'$ is at least $2$,
i.e. $\vert z' -z\vert > \frac{1}{N}$.
Then for any bounded measurable admissible function $F$,
\begin{align*}
&\int
\indic{\wp \text{ visits } z \text{ and } z'}
F(\gamma) \loopmeasure_{D_{N}}(d\wp)
\\ & =
\int_{0}^{+\infty}
\dfrac{\d a}{a}
\CR_{N,z'}(z,D_{N})^{a}
\int_{0}^{+\infty}
\dfrac{\d a'}{a'}
\CR_{N,z}(z',D_{N})^{a'}
\dfrac{c_{0}^{a+a'}}{N^{a+a'}}
\Bs\big((2\pi)^2 aa' \widetilde{G}_{D_{N}}(z,z')^2\big)
\\
& =
\Expect{F(
\Xi_{N,a,a'}^{z,z'}
\wedge\Xi_{N, z', a}^{z}
\wedge\Xi_{N, z, a'}^{z'})} .
\end{align*}
\end{lemma}

\begin{proof}
From Proposition \ref{Prop Le Jan subordinator}
it follows that
\begin{multline*}
\int
\indic{\wp \text{ visits } z \text{ and } z'}
F(\gamma) \loopmeasure_{D_{N}}(d\wp) \\ =
\int_{0}^{+\infty}
\dfrac{\d a}{a}
e^{- a (\log N)^{2}/( 2\pi G_{D_{N}}(z,z))}
\Expect{
\indic{\Xi_{N,a}^{z} \text{ visits } z'}
F(\Xi_{N,a}^{z})}
.
\end{multline*}
By combining with Lemma \ref{Lem decomp flower discrete},
we get that this further equals to
\begin{align*}
& \int_{0}^{+\infty}
\dfrac{\d a}{a}
e^{- a (\log N)^{2}/( 2\pi G_{D_{N}}(z,z))}
e^{- a (2\pi)^{3} \widetilde{G}_{D_{N}}(z,z')^2
G_{D_{N}\setminus \{z\}}(z',z')
/ (\log N)^{2}
}
\\ & \times
\int_{0}^{+\infty}
\dfrac{\d a'}{a'}
\dfrac{c_{0}^{a'}}{N^{a'}}
\CR_{N,z}(z',D_{N})^{a'}
\Bs\big((2\pi)^2 aa' \widetilde{G}_{D_{N}}(z,z')^2\big)
\\ & \times
\Expect{F(
\Xi_{N,a,a'}^{z,z'}
\wedge\Xi_{N, z', a}^{z}
\wedge\Xi_{N, z, a'}^{z'})} .
\end{align*}
We have that
\begin{align*}
&\dfrac{(\log N)^{2}}
{2\pi G_{D_{N}}(z,z)}
+
\dfrac{(2\pi)^{3} \widetilde{G}_{D_{N}}(z,z')^2
G_{D_{N}\setminus \{z\}}(z',z')}{(\log N)^{2}}
\\ & =
\dfrac{(\log N)^{2}}
{2\pi G_{D_{N}}(z,z)}
+
\dfrac{(\log N)^{2} G_{D_{N}}(z,z')^2}
{2\pi G_{D_{N}}(z,z)^2 G_{D_{N}\setminus \{z\}}(z',z')}
\\ & =
\dfrac{(\log N)^{2}
(
G_{D_{N}}(z,z)G_{D_{N}\setminus \{z\}}(z',z')
+
G_{D_{N}}(z,z')^2
)}
{2\pi G_{D_{N}}(z,z)^2 G_{D_{N}\setminus \{z\}}(z',z')}
\\ & =
\dfrac{(\log N)^{2}}
{2\pi G_{D_{N}}(z,z)}
+
\dfrac{(\log N)^{2} G_{D_{N}}(z,z')^2}
{2\pi G_{D_{N}}(z,z)^2 G_{D_{N}\setminus \{z\}}(z',z')}
\\ & =
\dfrac{(\log N)^{2} G_{D_{N}}(z,z)G_{D_{N}}(z',z')}
{2\pi G_{D_{N}}(z,z)^2 G_{D_{N}\setminus \{z\}}(z',z')}
=
\dfrac{(\log N)^{2} G_{D_{N}}(z',z')}
{2\pi G_{D_{N}}(z,z) G_{D_{N}\setminus \{z\}}(z',z')}
\\ & =
\dfrac{(\log N)^{2}}
{2\pi G_{D_{N}\setminus \{z'\}}(z,z)}
=
\log N -\log(c_{0}) -\log(\CR_{N,z'}(z,D_{N})) .
\end{align*}
This concludes.
\end{proof}

Given $z\neq z'\in D_{N}$,
denote
\begin{equation}
\label{eq:qNz}
q_{N,z}(z') :=
\dfrac{\log N}{2\pi G_{D_{N}\setminus\{ z\}}(z',z')},
\end{equation}
\begin{align*}
J_{N,K,z}(z') : & =
\int
\indic{\wp \text{ visits } z'}
\big(1-e^{-K T(\wp)}\big)
\loopmeasure_{D_{N}\setminus\{ z\}}(d\wp)
\\ & =
\int_{0}^{+\infty}
\big(
e^{-t/G_{D_{N}\setminus\{ z\}}(z',z')}
-
e^{-t/G_{D_{N}\setminus\{ z\}, K}(z',z')}
\big)
\dfrac{\d t}{t}
,
\end{align*}
\begin{align}
\label{eq:JNzz'}
J_{N}(z,z') :&=
\int
\indic{\wp \text{ visits } z \text{ and } z'}
\loopmeasure_{D_{N}}(d\wp)
,
\\
J_{N,K}(z,z') :&=
\int
\indic{\wp \text{ visits } z \text{ and } z'}
\big(1-e^{-K T(\wp)}\big)
\loopmeasure_{D_{N}}(d\wp)
.
\end{align}

The following proposition is to be compared to
Lemma \ref{lem:second_moment} in the continuum setting.

\begin{proposition}
\label{Prop 2nd mom discr}
Let $z,z'\in D_{N}$ such that the graph distance on $\Z^{2}_{N}$
between $z$ and $z'$ is at least $2$,
i.e. $\vert z' -z\vert > \frac{1}{N}$.
Let $a, a'>0$ and $K>0$.
Then for any bounded measurable admissible function $F$
the following holds.

1. The case massless-massless :
\begin{align}
\label{eq:prop_exact1}
& \E[ F(\Lc_{D_{N}}^\theta) \Mc_a^{N}(\{ z\})
\Mc_{a'}^{N}(\{ z'\}) ] =
\dfrac{q_{N,z'}(z)^{\theta}q_{N,z}(z')^{\theta} (\log N)^{2}}
{N^{4}\Gamma(\theta)^{2}}
e^{-\theta J_{N}(z,z')}
\\
\nonumber
&
\int_{\substack{\rho,\tilde{\rho}>0 \\ \rho+\tilde{\rho}\geq a}}
\d\rho \tilde{\rho}^{\theta -1}\d\tilde{\rho}
\CR_{N,z'}(z,D_{N})^{\rho +\tilde{\rho}}
\int_{\substack{\rho',\tilde{\rho}'>0 \\ \rho'+\tilde{\rho}'\geq a'}}
\d\rho' \tilde{\rho}'^{\theta -1}\d\tilde{\rho}'
\CR_{N,z'}(z,D_{N})^{\rho' +\tilde{\rho}'}
\dfrac{c_{0}^{\rho+\tilde{\rho}+\rho'+\tilde{\rho}'}}
{N^{\rho+\tilde{\rho}+\rho'+\tilde{\rho}' - a - a'}}
\\
\nonumber
& \times
\sum_{l\geq 0} \dfrac{\theta^{l}}{l!}
\int_{\substack{\mathsf{a} \in E(\rho,l)\\
\mathsf{a}' \in E(\rho',l)}}
\frac{\d \mathsf{a}}{a_1 \dots a_l}
\frac{\d \mathsf{a}'}{a_1' \dots a_l'}
\prod_{i=1}^l
\Bs \left( (2\pi)^2 a_i a_i' \widetilde{G}_{D_{N}}(z,z')^2 \right)
\\
\nonumber
& \times
\Expect{
F \Big(\Lc_{D_{N}\setminus\{z,z'\}}^\theta
\cup \{ \Xi_{N,a_i,a_i'}^{z,z'} \wedge \Xi_{N,z',a_i}^z
\wedge \Xi_{N,z,a_i'}^{z'} \}_{i = 1}^l
\cup
\{ \Xi_{N,z',\tilde{a}_i}^z, i\geq 1 \} \cup
\{ \Xi_{N,z,\tilde{a}_i'}^{z'}, i\geq 1 \} \Big)
},
\end{align}
where
on the right-hand side the
four collections of loops
$\Lc_{D_{N}\setminus\{z,z'\}}^\theta$,
$\{ \Xi_{N,a_i,a_i'}^{z,z'} \wedge \Xi_{N,z',a_i}^z
\wedge \Xi_{N,z,a_i'}^{z'} \}_{i = 1}^l $,
$\{ \Xi_{N,z',\tilde{a}_i}^z, i\geq 1 \}$
and
$\{ \Xi_{N,z,\tilde{a}_i'}^{z'}, i\geq 1 \}$
are independent,
$(\tilde{a}_i)_{i\geq 1}$
and $(\tilde{a}_i')_{i\geq 1}$
are two independent Poisson-Dirichlet partitions
PD$(0,\theta)$
of respectively $[0,\tilde{\rho}]$
and $[0,\tilde{\rho}']$,
the $\Xi_{N,z',\tilde{a}_i}^z$,
respectively $\Xi_{N,z,\tilde{a}_i'}^{z'}$,
are independent conditionally on $(\tilde{a}_i)_{i\geq 1}$,
respectively $(\tilde{a}_i')_{i\geq 1}$,
and the $\Xi_{N,a_i,a_i'}^{z,z'}$, $\Xi_{N,z',a_i}^z $
and $\Xi_{N,z,a_i'}^{z'}$ are all independent.

2. The case massless-massive :
\begin{align}
\label{eq:prop_exact2}
& \E[ F(\Lc_{D_{N}}^\theta) \Mc_a^{N}(\{ z\})
\Mc_{a'}^{N,K}(\{ z'\}) ] =
\dfrac{q_{N,z'}(z)^{\theta}(\log N)^{2}}
{N^{4}\Gamma(\theta)}
e^{-\theta (J_{N,K,z}(z')+J_{N,K}(z,z'))}
\\
\nonumber
&
\int_{\substack{\rho,\tilde{\rho}>0 \\ \rho+\tilde{\rho}\geq a}}
\d\rho \tilde{\rho}^{\theta -1} \d\tilde{\rho}
\CR_{N,z'}(z,D_{N})^{\rho +\tilde{\rho}}
\int_{a'}^{+\infty}
\d\rho'
\CR_{N,z'}(z,D_{N})^{\rho'}
\dfrac{c_{0}^{\rho+\tilde{\rho}+\rho'}}
{N^{\rho+\tilde{\rho}+\rho' - a - a'}}
\\
\nonumber
& \times
\sum_{\substack{m \geq 1 \\ 0 \leq l\leq m}}
\frac{\theta^{m}}{(m-l)!l!}
\int_{\substack{\mathsf{a} \in E(\rho,l)\\
\mathsf{a}' \in E(\rho',m)}}
\frac{\d \mathsf{a}}{a_1 \dots a_l}
\frac{\d \mathsf{a}'}{a_1' \dots a_m'}
\prod_{i=1}^l
\Bs \left( (2\pi)^2 a_i a_i' \widetilde{G}_{D_{N}}(z,z')^2 \right)
\\
\nonumber
& \times \E
\Bigg[ \prod_{i=1}^l
\left( 1 - e^{-K T(\Xi_{N,a_i,a_i'}^{z,z'}
\wedge \Xi_{N,z',a_i}^z
\wedge \Xi_{N,z,a_i'}^{z'} )} \right)
\prod_{i=l+1}^m \left( 1 - e^{-KT(\Xi_{N,z,a_i'}^{z'})} \right) \\
\nonumber
& F \Big(\Lc_{D_{N}\setminus\{z,z'\}}^\theta
\cup \widetilde{\Lc}_{D_{N},K,z'}^{\theta}
\cup \{ \Xi_{N,a_i,a_i'}^{z,z'} \wedge \Xi_{N,z',a_i}^z
\wedge \Xi_{N,z,a_i'}^{z'} \}_{i = 1}^l
\cup
\{ \Xi_{N,z',\tilde{a}_i}^z, i\geq 1 \} \cup
\{ \Xi_{N,z,a_i'}^{z'} \}_{i=l+1}^{m} \Big) \Bigg]
,
\end{align}
where on the right-hand side the
five collections of loops
$\Lc_{D_{N}\setminus\{z,z'\}}^\theta$,
$\widetilde{\Lc}_{D_{N},K,z'}^{\theta}$,
$\{ \Xi_{N,a_i,a_i'}^{z,z'} \wedge \Xi_{N,z',a_i}^z
\wedge \Xi_{N,z,a_i'}^{z'} \}_{i = 1}^l $,
$\{ \Xi_{N,z',\tilde{a}_i}^z, i\geq 1 \}$
and
$\{ \Xi_{N,z,a_i'}^{z'} \}_{i=l+1}^{m}$
are independent,
$(\tilde{a}_i)_{i\geq 1}$ is a Poisson-Dirichlet partition
PD$(0,\theta)$ of $[0,\tilde{\rho}]$,
the $\Xi_{N,z',\tilde{a}_i}^z$
are independent conditionally on $(\tilde{a}_i)_{i\geq 1}$,
the $\Xi_{N,a_i,a_i'}^{z,z'}$,
$\Xi_{N,z',a_i}^z$ and $\Xi_{N,z,a_i'}^{z'}$
are all independent, and
$\widetilde{\Lc}_{D_{N},K,z'}^{\theta}$
is distributed as the loops in
$\Lc_{D_{N}}^\theta\setminus\Lc_{D_{N}}^\theta(K)$
visiting $z'$.

3. The case massive-massive :
\begin{align}
\label{eq:prop_exact3}
& \E[ F(\Lc_{D_{N}}^\theta) \Mc_a^{N,K}(\{ z\})
\Mc_{a'}^{N,K}(\{ z'\}) ] =
\dfrac{(\log N)^{2}}{N^{4}}
e^{-\theta(J_{N,K,z'}(z)+J_{N,K,z}(z')+J_{N,K}(z,z'))}
\\
\nonumber
& \times
\int_{a}^{+\infty} \d \rho
\CR_{N,z'}(z,D_{N})^{\rho}
\int_{a'}^{+\infty} \d \rho'
\CR_{N,z}(z',D_{N})^{\rho'}
\dfrac{c_{0}^{\rho +\rho'}}{N^{\rho +\rho'-a-a'}}
\\
\nonumber
& \times \sum_{\substack{n,m \geq 1 \\ 0 \leq l \leq n \wedge m}} \frac{\theta^{n+m-l}}{(n-l)!(m-l)!l!}
\int_{\substack{\mathsf{a} \in E(\rho,n)\\
\mathsf{a}' \in E(\rho',m)}} \frac{\d \mathsf{a}}{a_1 \dots a_n} \frac{\d \mathsf{a}'}{a_1' \dots a_m'} \prod_{i=1}^l
\Bs \left( (2\pi)^2 a_i a_i' \widetilde{G}_{D_{N}}(z,z')^2 \right) \\
\nonumber
& \times \E
\Bigg[ \prod_{i=1}^l
\left( 1 - e^{-K T(\Xi_{N,a_i,a_i'}^{z,z'}
\wedge \Xi_{N,z',a_i}^z
\wedge \Xi_{N,z,a_i'}^{z'} )} \right)
\prod_{i=l+1}^n \left( 1 - e^{-KT(\Xi_{N,z',a_i}^z)} \right)
\prod_{i=l+1}^m \left( 1 - e^{-KT(\Xi_{N,z,a_i'}^{z'})} \right) \\
\nonumber
& F \Big(\Lc_{D_{N}\setminus\{z,z'\}}^\theta
\cup \widetilde{\Lc}_{D_{N},K,z,z'}^{\theta}
\cup \{ \Xi_{N,a_i,a_i'}^{z,z'} \wedge \Xi_{N,z',a_i}^z
\wedge \Xi_{N,z,a_i'}^{z'} \}_{i = 1}^l
\cup
\{ \Xi_{N,z',a_i}^z \}_{i=l+1}^{n} \cup
\{ \Xi_{N,z,a_i'}^{z'} \}_{i=l+1}^{m} \Big) \Bigg]
,
\end{align}
where on the right-hand side the
five collections of loops
$\Lc_{D_{N}\setminus\{z,z'\}}^\theta$,
$\widetilde{\Lc}_{D_{N},K,z,z'}^{\theta}$,
$\{ \Xi_{N,a_i,a_i'}^{z,z'} \wedge \Xi_{N,z',a_i}^z
\wedge \Xi_{N,z,a_i'}^{z'} \}_{i = 1}^l$,
$\{ \Xi_{N,z',a_i}^z \}_{i=l+1}^{n}$
and $\{ \Xi_{N,z,a_i'}^{z'} \}_{i=l+1}^{m}$,
are independent,
the $\Xi_{N,a_i,a_i'}^{z,z'}$, $\Xi_{N,z',a_i}^z$
and $\Xi_{N,z,a_i'}^{z'}$ are all independent,
and $\widetilde{\Lc}_{D_{N},K,z,z'}^{\theta}$
is distributed as the loops in
$\Lc_{D_{N}}^\theta\setminus\Lc_{D_{N}}^\theta(K)$
visiting $z$ or $z'$.
\end{proposition}

\begin{proof}
Let us first consider the case 1. massless-massless.
We divide the random walk loop soup
$\Lc_{D_{N}}^\theta$ into four independent Poisson point processes:
\begin{itemize}
\item The loops visiting neither $z$ nor $z'$.
These correspond to $\Lc_{D_{N}\setminus\{z,z'\}}^\theta$.
\item The loops visiting $z$ but not $z'$.
We apply to these Proposition \ref{Prop 1st mom Mc N}
in the domain $D_{N}\setminus\{z'\}$.
These loops correspond to
$\{ \Xi_{N,z',\tilde{a}_i}^z, i\geq 1 \}$.
\item The loops visiting $z'$ but not $z$.
We apply to these Proposition \ref{Prop 1st mom Mc N}
in the domain $D_{N}\setminus\{z\}$.
These loops correspond to
$\{ \Xi_{N,z,\tilde{a}_i'}^{z'}, i\geq 1 \}$.
\item The loops visiting both $z$ and $z'$.
These form an a.s. finite Poisson point process.
The corresponding intensity measure is described,
up to the factor $\theta$,
by Lemma \ref{Lem loops 2 pts}.
The corresponding total mass is $\theta  J_{N}(z,z')$.
These loops correspond to
$\{ \Xi_{N,a_i,a_i'}^{z,z'} \wedge \Xi_{N,z',a_i}^z
\wedge \Xi_{N,z,a_i'}^{z'} \}_{i = 1}^l $.
\end{itemize}
By combining the above, we obtain our expression.

Now let us consider the case 3. massive-massive.
We divide the random walk loop soup
$\Lc_{D_{N}}^\theta$ into five independent Poisson point processes:
\begin{itemize}
\item The loops visiting neither $z$ nor $z'$.
These correspond to $\Lc_{D_{N}\setminus\{z,z'\}}^\theta$.
\item The loops visiting $z$ or $z'$
and surviving to the killing rate $K$.
These correspond to $\widetilde{\Lc}_{D_{N},K,z,z'}^{\theta}$.
\item The loops visiting $z$ but not $z'$, and killed by $K$.
These form an a.s. finite Poisson point process.
The total mass of the corresponding intensity measure is
$\theta J_{N,K,z'}(z)$.
We apply to these Proposition \ref{Prop 1st mom Mc N K}
in the domain $D_{N}\setminus\{z'\}$.
These loops correspond to
$\{ \Xi_{N,z',a_i}^z \}_{i=l+1}^{n}$.
\item The loops visiting $z'$ but not $z$, and killed by $K$.
These form an a.s. finite Poisson point process.
The total mass of the corresponding intensity measure is
$\theta J_{N,K,z}(z')$.
We apply to these Proposition \ref{Prop 1st mom Mc N K}
in the domain $D_{N}\setminus\{z\}$.
These loops correspond to
$\{ \Xi_{N,z,a_i'}^{z'} \}_{i=l+1}^{m}$.
\item The loops visiting both $z$ and $z'$, and killed by $K$.
These form an a.s. finite Poisson point process.
The total mass of the corresponding intensity measure is
$\theta J_{N,K}(z,z')$.
We apply to these Lemma \ref{Lem loops 2 pts}.
These loops correspond to
$\{ \Xi_{N,a_i,a_i'}^{z,z'} \wedge \Xi_{N,z',a_i}^z
\wedge \Xi_{N,z,a_i'}^{z'} \}_{i = 1}^l$.
\end{itemize}
By combining the above, we obtain our expression.

The case 2. massless-massive is similar to and intermediate between the cases 1. and 3. We will not detail it.
\end{proof}

We finish this section with an elementary lemma that we state for ease of reference. We omit its proof since it can be easily checked.

\begin{lemma}\label{lem:discrete_asymp_exact}
Let $D \subset \R^2$ be a bounded simply connected domain, $z, z'$ be two distinct points of $D$. Consider a discrete approximation $(D_N)_{N}$ of $D$ in the sense of \eqref{eq:DN} and let $z_N$ and $z_N'$ be vertices of $D_N$ which converge to $z$ and $z'$ respectively. Then
\begin{equation}
1-q_{N,z_N}(z_N'), \quad J_N(z_N,z_N'), \quad J_{N,K}(z_N,z_N') \quad \mathrm{and} \quad J_{N,K,z_N}(z_N')
\end{equation}
all converge to 0. Moreover,
\begin{equation}
\label{eq:discrete_CK3}
C_{N,K,z_N}(z'_N) \to C_K(z').
\end{equation}
\end{lemma}

\subsection{Convergence of excursion measures}

The goal of this section is to prepare the proof of Proposition \ref{prop:discrete_mass_infty} by establishing the convergence of the various measures on discrete paths that appear in the formulas obtained in Sections \ref{sec:discrete_exact1} and \ref{sec:discrete_exact2} towards their continuum analogues.

Consider $D\subset \C$ an open bounded simply connected domain containing the origin
and $(D_{N})_{N}$ a discrete approximation of $D$,
with $D_{N}\subset \Z_{N}^{2}$. See \eqref{eq:DN}.
First we deal with the convergence of probability measures
$\tilde{\mu}^{z_{N},w_{N}}_{D_{N}}/
\widetilde{G}_{D_{N}}(z_{N},w_{N})$
with $w_{N}\neq z_{N}$.

\begin{lemma}
\label{Lem Approx mu 1}
Let $z,w\in D$ with $z\neq w$.
Consider sequences $(z_{N})_{N\geq 1}$
and $(w_{N})_{N\geq 1}$,
with $z_{N}, w_{N}\in D_{N}$ and
\begin{displaymath}
\lim_{N\to +\infty} z_{N} = z,
\qquad
\lim_{N\to +\infty} w_{N} = w.
\end{displaymath}
Then the probability measures
$\mu^{z_{N},w_{N}}_{D_{N}}/G_{D_{N}}(z_{N},w_{N})$
\eqref{Eq discr paths} converge weakly as $N\to +\infty$,
for the metric $d_{\rm paths}$ \eqref{Eq dist paths},
towards
$\mu^{z,w}_{D}/G_{D}(z,w)$ \eqref{Eq mu D z w}.
\end{lemma}

\begin{proof}
Since $D$ is bounded, by performing a translation we can reduce to the case when $D\subset \H$, where $\H$ is the upper half-plane
\begin{displaymath}
\H = \{z\in\C : \Im(z)>0 \},
\end{displaymath}
and that for every $N\geq 1$,
$D_{N} \subset \Z_{N}^{2}\cap \H$.
First, we have that
$\mu^{z_{N},w_{N}}_{\Z_{N}^{2}\cap \H}/
G_{\Z_{N}^{2}\cap \H}(z_{N},w_{N})$
converges weakly towards
$\mu^{z,w}_{\H}/G_{\H}(z,w)$.
This follows from the following two points:
\begin{itemize}
\item For every $t>0$,
the bridges probability measures
$\P^{z_{N},w_{N}}_{\Z_{N}^{2}\cap \H,t}$
converges towards
the Brownian bridge measure $\P^{z,w}_{\H,t}$.
\item The transition densities
$p_{\Z_{N}^{2}\cap \H}(t,z_{N},w_{N})$
converge to
$p_{\H}(t,z,w)$
uniformly in $t\in[0,+ \infty)$
(local central limit theorem);
see \cite[Theorem 2.5.6]{LawlerLimic2010RW}.
\item The discrete Green function
$G_{\Z_{N}^{2}\cap \H}(z_{N},w_{N})$
converges to $G_{\H}(z,w)$.
This follows from \cite[Theorem 4.4.4]{LawlerLimic2010RW}
and the reflection principle.
\end{itemize}
Further, the measure $\mu^{z,w}_{D}/G_{D}(z,w)$
is obtained by conditioning a path under
$\mu^{z,w}_{\H}/G_{\H}(z,w)$ to stay in $D$.
Similarly, $\mu^{z_{N},w_{N}}_{D_{N}}/G_{D_{N}}(z_{N},w_{N})$
is obtained by conditioning a path under
$\mu^{z_{N},w_{N}}_{\Z_{N}^{2}\cap \H}/
G_{\Z_{N}^{2}\cap \H}(z_{N},w_{N})$
to stay in $D_N$.
Moreover, on the event that the path under
$\mu^{z,w}_{\H}/G_{\H}(z,w)$
exits $D$, a.s. there is $\varepsilon>0$ such that any continuous deformation of the path of size less than $\varepsilon$
also has to exit $D$.
This is because a Brownian path exiting $D$ will a.s. create a loop around the point where it first exits $D$.
We refer to
\cite[Lemma 2.6]{Lupu2014LoopsHalfPlane}
for details.
Thus, one gets the convergence of
$\mu^{z_{N},w_{N}}_{D_{N}}/G_{D_{N}}(z_{N},w_{N})$
towards $\mu^{z,w}_{D}/G_{D}(z,w)$.
\end{proof}

\begin{proposition}
\label{Prop Approx mu 2}
Let $z,w\in D$ with $z\neq w$.
Consider sequences $(z_{N})_{N\geq 1}$
and $(w_{N})_{N\geq 1}$,
with $z_{N}, w_{N}\in D_{N}$ and
\begin{displaymath}
\lim_{N\to +\infty} z_{N} = z,
\qquad
\lim_{N\to +\infty} w_{N} = w.
\end{displaymath}
Then the probability measures
$\tilde{\mu}^{z_{N},w_{N}}_{D_{N}}/
\widetilde{G}_{D_{N}}(z_{N},w_{N})$
\eqref{Eq tilde mu} converges weakly as $N\to +\infty$,
for the metric $d_{\rm paths}$ \eqref{Eq dist paths},
towards
$\mu^{z,w}_{D}/G_{D}(z,w)$ \eqref{Eq mu D z w}.
\end{proposition}

\begin{proof}
According to the Markovian decomposition of
Lemma \ref{Lem Markov discrete},
a path $\wp$ under
$\mu^{z_{N},w_{N}}_{D_{N}}/G_{D_{N}}(z_{N},w_{N})$
has the same law as a concatenation
of three independent paths
$\wp_{1}\wedge\tilde{\wp}\wedge \wp_{2}$,
with $\tilde{\wp}$ following the distribution
$\tilde{\mu}^{z_{N},w_{N}}_{D_{N}}/
\widetilde{G}_{D_{N}}(z_{N},w_{N})$,
$\wp_{1}$ following the distribution
$\mu_{D_{N}}^{z_{N},z_{N}}/G_{D_{N}}(z_{N},z_{N})$,
and
$\wp_{2}$ following the distribution
$\mu_{D_{N}\setminus \{ z_N \}}^{w_{N},w_{N}}/
G_{D_{N}\setminus \{ z_N \}}(w_{N},w_{N})$.
Moreover, it is easy to see that
$\diam(\wp_{1})$, $T(\wp_{1})$,
$\diam(\wp_{2})$ and $T(\wp_{2})$ converge in probability to $0$
as $N\to +\infty$.
Thus, the convergence of
$\tilde{\mu}^{z_{N},w_{N}}_{D_{N}}/
\widetilde{G}_{D_{N}}(z_{N},w_{N})$
is equivalent to the convergence of
$\mu^{z_{N},w_{N}}_{D_{N}}/G_{D_{N}}(z_{N},w_{N})$,
and the latter converge to
$\mu^{z,w}_{D}/G_{D}(z,w)$
according to Lemma \ref{Lem Approx mu 1}.
\end{proof}

Next we deal with the convergence of measures
$\tilde{\mu}^{z_{N},z_{N}}_{D_{N}}$.
Given $z\in D$ and $r>0$, 
let $E_{z,r}$ denote the event that a path goes at distance at least $r$ from $z$.
If $r<d(z,\partial D)$, then
$\mu^{z,z}_{D}(E_{z,r})<+\infty$.

\begin{lemma}
\label{Lem Approx mu 3}
Let $z\in D$ and $r\in (0,d(z,\partial D))$.
Consider a sequence $(z_{N})_{N\geq 1}$,
$z_{N}\in D_{N}$, converging to $z$.
Then
\begin{displaymath}
\lim_{N\to +\infty}
\mu^{z_{N},z_{N}}_{D_{N}}
(E_{z_{N},r})
=
\mu^{z,z}_{D}(E_{z,r}).
\end{displaymath}
Moreover,
the probability measures
$\indic{E_{z_{N},r}}\mu^{z_{N},z_{N}}_{D_{N}}/
\mu^{z_{N},z_{N}}_{D_{N}}(E_{z_{N},r})$
converge weakly as $N\to +\infty$,
for the metric $d_{\rm paths}$,
towards
$\indic{E_{z,r}}\mu^{z,z}_{D}/\mu^{z,z}_{D}(E_{z,r})$.
\end{lemma}

\begin{proof}
Since $D$ is bounded, by performing a translation we can reduce to the case when $D\subset \H$
and that for every $N\geq 1$,
$D_{N} \subset \Z_{N}^{2}\cap \H$.
We only need to show that
\begin{equation}
\label{Eq conv H}
\lim_{N\to +\infty}
\mu^{z_{N},z_{N}}_{\Z_{N}^{2}\cap \H}
(E_{z_{N},r})
=
\mu^{z,z}_{\H}(E_{z,r})
\end{equation}
and that the probability measures
$\indic{E_{z_{N},r}}\mu^{z_{N},z_{N}}_{\Z_{N}^{2}\cap \H}/
\mu^{z_{N},z_{N}}_{\Z_{N}^{2}\cap \H}(E_{z_{N},r})$
converge weakly towards the measure
$\indic{E_{z,r}}\mu^{z,z}_{\H}/\mu^{z,z}_{\H}(E_{z,r})$.
Indeed, the measure $\mu^{z_{N},z_{N}}_{D_{N}}$
is a restriction of $\mu^{z_{N},z_{N}}_{\Z_{N}^{2}\cap \H}$
to the paths that stay in $D_{N}$
and $\mu^{z,z}_{D}$ is a restriction of
$\mu^{z,z}_{\H}$ to the paths that stay in $D$.
Using that, one can conclude as in the
proof of Lemma \ref{Lem Approx mu 1}.

Now, consider $(B_{t})_{t\geq 0}$ a Brownian motion starting from $z$
and let $\tau_{r}$ be the stopping time
\begin{displaymath}
\tau_{r} : = \min \{ t\geq 0 : \vert B_{t} -z\vert = r \}.
\end{displaymath}
Also consider $(X^{(N)}_{t})_{t\geq 0}$ the Markov jump process on
$\Z_{N}^{2}$ (see Section \ref{sec:preliminaries_RWLS}) starting from $z_{N}$ and let $\tau_{N,r}$ be the stopping time
\begin{displaymath}
\tau_{N,r} : = \min \{ t\geq 0 : \vert X^{(N)}_{t} -z\vert \geq r \}.
\end{displaymath}
The following holds:
\begin{displaymath}
\mu^{z_{N},z_{N}}_{\Z_{N}^{2}\cap \H}
(E_{z_{N},r})
= \E^{z_{N}}\big[
G_{\Z_{N}^{2}\cap \H}(X^{(N)}_{\tau_{N,r}},z_{N})
\big],
\qquad
\mu^{z,z}_{\H}(E_{z,r})
= \E^{z}\big[ G_{\H}(B_{\tau_{r}},z) \big].
\end{displaymath}
So \eqref{Eq conv H} follows from the convergence in law
of $X^{(N)}_{\tau_{N,r}}$ to $B_{\tau_{r}}$
and the convergence of $G_{\Z_{N}^{2}\cap \H}(w,z_{N})$
to $G_{\H}(w,z)$ uniformly for $w$ away from $z$.
Further, a path $\wp$ under the probability
$\indic{E_{z,r}}\mu^{z,z}_{\H}/\mu^{z,z}_{\H}(E_{z,r})$
can de decomposed as a concatenation 
$\wp_{1}\wedge \wp_{2}$ with the following distribution.
The distribution of $\wp_{1}$
is that of $(B_{t})_{0\leq t\leq \tau_{r}}$
tilted by the density
\begin{displaymath}
\dfrac{G_{\H}(B_{\tau_{r}},z)}{\mu^{z,z}_{\H}(E_{z,r})}.
\end{displaymath}
Conditionally on $\wp_{1}$,
$\wp_{2}$ follows the distribution
$\mu^{w,z}_{\H}/G_{\H}(w,z)$,
where $w$ is the endpoint of $\wp_{1}$.
A similar decomposition holds for a path under
$\indic{E_{z_{N},r}}\mu^{z_{N},z_{N}}_{\Z_{N}^{2}\cap \H}/
\mu^{z_{N},z_{N}}_{\Z_{N}^{2}\cap \H}(E_{z_{N},r})$,
with $X^{(N)}_{t}$ instead of $B_{t}$,
$G_{\Z_{N}^{2}\cap \H}$ instead of $G_{\H}$ and
$\mu^{w,z_{N}}_{\Z_{N}^{2}\cap \H}$
instead of $\mu^{w,z}_{\H}$.
So the desired convergence of measures follows from
the convergence in law of
$(X^{(N)}_{t})_{0\leq t\leq \tau_{N,r}}$
to $(B_{t})_{0\leq t\leq \tau_{r}}$,
the convergence of the Green functions $G_{\Z_{N}^{2}\cap \H}$
to $G_{\H}$,
and from Lemma \ref{Lem Approx mu 1}.
\end{proof}

\begin{proposition}
\label{Prop Approx mu 4}
Let $z\in D$ and $r\in (0,d(z,\partial D))$.
Consider a sequence $(z_{N})_{N\geq 1}$,
$z_{N}\in D_{N}$, converging to $z$.
Then
\begin{displaymath}
\lim_{N\to +\infty}
\tilde{\mu}^{z_{N},z_{N}}_{D_{N}}
(E_{z_{N},r})
=
\mu^{z,z}_{D}(E_{z,r}).
\end{displaymath}
Moreover,
the probability measures
$\indic{E_{z_{N},r}}\tilde{\mu}^{z_{N},z_{N}}_{D_{N}}/
\tilde{\mu}^{z_{N},z_{N}}_{D_{N}}(E_{z_{N},r})$
converge weakly as $N\to +\infty$,
for the metric $d_{\rm paths}$,
towards
$\indic{E_{z,r}}\mu^{z,z}_{D}/\mu^{z,z}_{D}(E_{z,r})$.
\end{proposition}

\begin{proof}
Denote
\begin{displaymath}
B_{N} :=
\{
w\in D_{N} :
\vert w-z_{n}\vert <r
\}.
\end{displaymath}
According to the Markovian decomposition of
Lemma \ref{Lem Markov discrete},
a path $\wp$ under
$\indic{E_{z_{N},r}}\mu^{z_{N},z_{N}}_{D_{N}}/
\mu^{z_{N},z_{N}}_{D_{N}}(E_{z_{N},r})$
has the same law as a concatenation
of three independent paths
$\wp_{1}\wedge\tilde{\wp}\wedge \wp_{2}$,
with $\tilde{\wp}$ following the distribution
$\indic{E_{z_{N},r}}\tilde{\mu}^{z_{N},z_{N}}_{D_{N}}/
\tilde{\mu}^{z_{N},z_{N}}_{D_{N}}(E_{z_{N},r})$,
$\wp_{1}$ following the distribution
$\mu_{D_{N}}^{z_{N},z_{N}}/G_{D_{N}}(z_{N},z_{N})$,
and
$\wp_{2}$ following the distribution
$\mu_{B_{N}}^{z_{N},z_{N}}/G_{B_{N}}(z_{N},z_{N})$.
Further,
$\diam(\wp_{1})$, $T(\wp_{1})$,
$\diam(\wp_{2})$ and $T(\wp_{2})$ converge in probability to $0$
as $N\to +\infty$.
Thus, the convergence of
$\indic{E_{z_{N},r}}\tilde{\mu}^{z_{N},z_{N}}_{D_{N}}/
\tilde{\mu}^{z_{N},z_{N}}_{D_{N}}(E_{z_{N},r})$
is equivalent to the convergence of
$\indic{E_{z_{N},r}}\mu^{z_{N},z_{N}}_{D_{N}}/
\mu^{z_{N},z_{N}}_{D_{N}}(E_{z_{N},r})$,
and the latter converge to
$\indic{E_{z,r}}\mu^{z,z}_{D}/\mu^{z,z}_{D}(E_{z,r})$
according to Lemma \ref{Lem Approx mu 3}.
Moreover,
\begin{displaymath}
\mu^{z_{N},z_{N}}_{D_{N}}(E_{z_{N},r})
=
\dfrac{G_{D_{N}}(z_{N},z_{N})G_{B_{N}}(z_{N},z_{N})}
{\Big(\dfrac{1}{2\pi} \log N \Big)^{2}}
\tilde{\mu}^{z_{N},z_{N}}_{D_{N}}(E_{z_{N},r}).
\end{displaymath}
Thus, 
$\tilde{\mu}^{z_{N},z_{N}}_{D_{N}}(E_{z_{N},r})$
and 
$\mu^{z_{N},z_{N}}_{D_{N}}(E_{z_{N},r})$
have the same limit.
\end{proof}

\begin{corollary}
\label{Cor Approx mu 5}
Let $z\neq z'\in D$ and $r\in (0,d(z,\partial D))$.
Consider sequences $(z_{N})_{N\geq 1}$
and $(z'_{N})_{N\geq 1}$,
with $z_{N}, z'_{N}\in D_{N}$ and
\begin{displaymath}
\lim_{N\to +\infty} z_{N} = z,
\qquad
\lim_{N\to +\infty} z'_{N} = z'.
\end{displaymath}
Then
\begin{displaymath}
\lim_{N\to +\infty}
\tilde{\mu}^{z_{N},z_{N}}_{D_{N}\setminus\{z'_{N}\}}
(E_{z_{N},r})
=
\mu^{z,z}_{D}(E_{z,r}).
\end{displaymath}
Moreover,
the probability measures
$\indic{E_{z_{N},r}}
\tilde{\mu}^{z_{N},z_{N}}_{D_{N}\setminus\{z'_{N}\}}/
\tilde{\mu}^{z_{N},z_{N}}_{D_{N}}(E_{z_{N},r})$
converge weakly as $N\to +\infty$,
for the metric $d_{\rm paths}$,
towards
$\indic{E_{z,r}}\mu^{z,z}_{D}/\mu^{z,z}_{D}(E_{z,r})$.
\end{corollary}

\begin{proof}
The measure $\tilde{\mu}^{z_{N},z_{N}}_{D_{N}\setminus\{z'_{N}\}}$
is obtained by restricting 
$\tilde{\mu}^{z_{N},z_{N}}_{D_{N}}$
to the paths that do not visit $z'_{N}$.
Given that almost every path under $\mu^{z,z}_{D}$
stays at positive distance from $z'$,
the result follows from
Proposition \ref{Prop Approx mu 4}.
\end{proof}

\section{Controlling the effect of mass in discrete loop soup}\label{sec:discrete_convergence}

The purpose of this section is to prove Proposition \ref{prop:discrete_mass_infty}. As in the continuum, the proof relies on a careful analysis of truncated first and second moments. We start off by introducing the good events that we will work with.

Let $z \in N^{-1} \Z^2$ and $r >0$. We will denote by $\partial \D_N(z,r)$ the discrete circle defined as the outer boundary of the discrete disc
\[
\D_N(z,r) := z+ \{ y \in N^{-1} \Z^2: |y| < r \}.
\]
If $\wp$ is a discrete trajectory on $N^{-1} \Z^2$ and if $\Cc$ is a collection of such trajectories, we will denote by $N_{z,r}^{\wp}$ the number of upcrossings from $\partial \D_N(z,r)$ to $\partial \D_N(z,er)$ in $\wp$ and $N_{z,r}^\Cc = \sum_{\wp \in \Cc} N_{z,r}^\wp$. We will not keep track of the dependence in the mesh size $N^{-1}$ in the notations of the number of crossings since it will be clear from the context.

Let $\eta \in (0,1-a/2)$ be a small parameter, $b>a$ be close to $a$ and define the discrete analogues of the good event $\Gc_K(z)$:
\begin{equation}
\Gc_N(z) := \left\{ \forall r \in \{e^{-n}, n \geq 1\} \cap (N^{-1+\eta},r_0): N_{z,r}^{\Lc_{D_N}^\theta} \leq b |\log r|^2  \right\}
\label{eq:def_good_discrete1}
\end{equation}
and
\begin{equation}
\Gc_{N,K}(z) := \left\{ \forall r \in \{e^{-n}, n \geq 1\} \cap (N^{-1+\eta},r_0): N_{z,r}^{\Lc_{D_N}^\theta(K)} \leq b |\log r|^2 \right\}.
\label{eq:def_good_discrete2}
\end{equation}
We emphasise here that we only restrict the number of crossings of annuli at scales $r > N^{-1+\eta}$. We will see that it is enough to turn the measure into a measure bounded in $L^2$ and it will simplify the analysis since we will always look at scales at least mesoscopic ($N^{-1+\beta}$, for some $\beta>0$).

Once the good events are defined, we consider the modified versions of $\Mc_a^N$ and $\Mc_a^{N,K}$:
\[
\tilde{\Mc}_a^N(dz) := \mathbf{1}_{\Gc_N(z)} \Mc_a^N(dz)
\quad \text{and} \quad
\tilde{\Mc}_a^{N,K}(dz) := \mathbf{1}_{\Gc_{N,K}(z)} \Mc_a^{N,K}(dz).
\]

\medskip

In the remaining of Section \ref{sec:discrete_convergence}, we will fix a Borel set $A$ compactly included in $D$ and the constants underlying our estimates will implicitly be allowed to depend on $A$, $\eta$, $a$ and $b$.

The proof of Proposition \ref{prop:discrete_mass_infty} relies on three lemmas that are the discrete analogues of Lemmas \ref{lem:first_moment_good_event}, \ref{lem:second_moment_bdd} and \ref{lem:second_moment_aa'}. We first state these lemmas without proof and explain how the proof of Proposition \ref{prop:discrete_mass_infty} is obtained from them.

We will first need to show that the introduction of the good events almost does not change the first moment:

\begin{lemma}\label{lem:discrete_first_moment_event}
We have
\begin{equation}
\label{eq:lem_discrete_first_moment_event_no_mass}
\lim_{r_0 \to 0} \limsup_{N \to \infty} \Expect{ \abs{ \tilde{\Mc}_a^N (A) - \Mc_a^N(A) } } = 0
\end{equation}
and
\begin{equation}
\label{eq:lem_discrete_first_moment_event_with_mass}
\lim_{r_0 \to 0} \limsup_{K \to \infty} (\log K)^{-\theta} \limsup_{N \to \infty} \Expect{ \abs{ \tilde{\Mc}_a^{N,K} (A) - \Mc_a^{N,K}(A) } } = 0.
\end{equation}
\end{lemma}

Once the good events are introduced, the second moment becomes finite:

\begin{lemma}\label{lem:discrete_second_moment_bdd}
For $z \in D$ and $N \geq 1$, denote by $z_N$ some element of $D_N$ closest to $z$ (with some arbitrary rule).
If $b>a$ is close enough to $a$, then
\begin{equation}
\label{eq:lem_discrete_second_moment_bdd_no_mass}
\int_{A \times A} \sup_{N \geq 1} N^4 \Expect{ \tilde{\Mc}_a^N(\{z_N\}) \tilde{\Mc}_a^N(\{z'_N\}) } \d z \d z' < \infty
\end{equation}
and
\begin{equation}
\label{eq:lem_discrete_second_moment_bdd_with_mass}
\int_{A \times A} \sup_{K \geq 1} (\log K)^{-2\theta} \sup_{N \geq 1} N^4 \Expect{ \tilde{\Mc}_a^{N,K}(\{z_N\}) \tilde{\Mc}_a^{N,K}(\{z'_N\}) } \d z \d z' < \infty.
\end{equation}
\end{lemma}

Finally,

\begin{lemma}\label{lem:discrete_second_moment_convergence}
If $b>a$ is close enough to $a$, then
\begin{equation}
\limsup_{K \to \infty} \limsup_{N \to \infty} \Expect{ \left( \tilde{\Mc}_a^N(A) - \frac{2^\theta}{(\log K)^\theta} \tilde{\Mc}_a^{N,K}(A) \right)^2 } = 0.
\end{equation}
\end{lemma}

\begin{proof}[Proof of Proposition \ref{prop:discrete_mass_infty}]
Proposition \ref{prop:discrete_mass_infty} follows from Lemmas \ref{lem:discrete_first_moment_event} and \ref{lem:discrete_second_moment_convergence} in a very similar way as Proposition \ref{prop:oscillations_Mc} follows from Lemmas \ref{lem:first_moment_good_event} and \ref{lem:second_moment_aa'}; see below Lemma \ref{lem:second_moment_aa'}. Note that we can first restrict ourselves to a Borel set $A$ compactly included in $D$ since the contribution of points near the boundary to the measures is negligible. We omit the details.
\end{proof}

The remaining of Section \ref{sec:discrete_convergence} is organised as follows.
We will start in Section \ref{subsec:discrete_simplifying} by analysing the lengthy formulas appearing in Proposition \ref{Prop 2nd mom discr} in the same spirit as what we did in Lemma \ref{lem:Hs}. We will then study in Section \ref{subsec:discrete_crossing} the number of crossings in the processes of excursions that appear in Propositions \ref{Prop 1st mom Mc N}, \ref{Prop 1st mom Mc N K} and \ref{Prop 2nd mom discr}. The proofs of Lemmas \ref{lem:discrete_first_moment_event}, \ref{lem:discrete_second_moment_bdd} and \ref{lem:discrete_second_moment_convergence} will then be given in Sections \ref{subsec:discrete_first_moment}, \ref{subsec:discrete_second_moment_bdd} and \ref{subsec:discrete_second_moment_convergence} respectively.

\subsection{Simplifying the second moment}\label{subsec:discrete_simplifying}

Define for all $\lambda, \lambda'>0$, $0 < v < \lambda^2 \wedge {\lambda'}^2$ and $u,u' \geq 1$,
\begin{align}
\label{eq:def_H_hat}
\hat{\Hs}_a(\lambda,\lambda',v) & := \int_{\substack{\thk, \tilde{\thk}>0 \\ \thk + \tilde{\thk} \geq a}} \d \thk \d \tilde{\thk} ~e^{-\lambda(\thk + \tilde{\thk})} \tilde{\thk}^{\theta-1}
\int_{\substack{\thk', \tilde{\thk}'>0 \\ \thk' + \tilde{\thk}' \geq a}} \d \thk' \d \tilde{\thk}' e^{-\lambda'(\thk' + \tilde{\thk}')} ( \tilde{\thk}' )^{\theta-1} \\
& \times
\sum_{l \geq 1} \frac{\theta^l}{l!} \int_{\substack{\mathsf{a} \in E(\thk,l) \\ \mathsf{a}' \in E(\thk',l)}} \d \mathsf{a} \d \mathsf{a}' \prod_{i=1}^l \frac{\Bs (a_i a_i' v)}{a_i a_i'}
\nonumber
\end{align}
and
\begin{align}
\dhat{\Hs}_a(\lambda, \lambda', u, u', v) & := \int_{a}^{+\infty} \d \rho
~e^{-\lambda \rho}
\int_{a'}^{+\infty} \d \rho'
e^{-\lambda' \rho'} \sum_{\substack{n,m \geq 1 \\ 0 \leq l \leq n \wedge m}} \frac{\theta^{n+m-l}}{(n-l)!(m-l)!l!} \\
& \times \int_{\substack{\mathsf{a} \in E(\rho,n) \\ \mathsf{a}' \in E(\rho',m)}} \d \mathsf{a} \d \mathsf{a}' \prod_{i=1}^l \frac{ \Bs \left( a_i a_i' v \right) }{a_ia_i'} \prod_{i=l+1}^n \frac{1-e^{-ua_i}}{a_i} \prod_{i=l+1}^m \frac{1-e^{-u'a_i'}}{a_i'}.
\nonumber
\end{align}
By Proposition \ref{Prop 2nd mom discr}, $\hat{\Hs}_a$ is related to the second moment of $\Mc_a^N$ as follows:
\begin{align*}
\Expect{ \Mc_a^N(\{z\}) \Mc_a^N(\{z'\}) } = \dfrac{q_{N,z'}(z)^{\theta}q_{N,z}(z')^{\theta} (\log N)^{2}}
{N^{4-2a}\Gamma(\theta)^{2}}
e^{-\theta J_{N}(z,z')} \hat{\Hs}_a(\lambda, \lambda', v)
\end{align*}
with
\begin{equation}
\label{eq:def_lambda_v}
\lambda = \log N - \log \CR_{N,z'}(z,D_N) - \log c_0, \quad \lambda' = \log N - \log \CR_{N,z}(z',D_N) - \log c_0
\end{equation}
and $v = (2\pi \tilde{G}_{D_N}(z,z'))^2$.
On the other hand, by neglecting the killing for loops that visit both $z$ and $z'$, we see that $\dhat{\Hs}_a$ provides a good upper bound on the second moment of $\Mc_a^{N,K}$ (see Proposition \ref{Prop 2nd mom discr}):
\[
\Expect{ \Mc_a^{N,K}(\{ z\}) \Mc_a^{N,K}(\{ z'\}) } \leq
\dfrac{(\log N)^{2}}{N^{4-2a}}
e^{-\theta(J_{N,K,z'}(z)+J_{N,K,z}(z')+J_{N,K}(z,z'))} \dhat{\Hs}_a(\lambda, \lambda', u, u', v)
\]
where $\lambda, \lambda'$ and $v$ are as above and, recalling the definition \eqref{eq:discrete_CK2} of $C_{N,K,z}(z')$,
\[
u = C_{N,K,z'}(z)
\quad \text{and} \quad
u' = C_{N,K,z}(z').
\]
In the following lemma, which is the discrete counterpart of Lemma \ref{lem:Hs}, we give exact expressions for $\hat{\Hs}_a$ and $\dhat{\Hs}_a$ and deduce upper bounds. These upper bounds will be crucial for us in order to prove Lemma \ref{lem:discrete_second_moment_bdd}. Recall the definition \eqref{eq:def_Hs} of $\Hs_{\thk, \thk'}$.

\begin{lemma}
We have
\begin{equation}
\label{eq:lem_H_massless_explicit}
\hat{\Hs}_a(\lambda,\lambda',v) = \Gamma(\theta) v^{(1-\theta)/2} \int_{[a,\infty)^2} (tt')^{(\theta-1)/2} e^{-\lambda t} e^{-\lambda' t'} I_{\theta-1} \left( 2 \sqrt{v tt'} \right) \d t \d t'
\end{equation}
and
\begin{equation}
\label{eq:lem_H_massive_explicit}
\dhat{\Hs}_a(\lambda, \lambda', u, u', v) = \int_{[a,\infty)^2} e^{-\lambda \thk} e^{-\lambda' \thk'} \Hs_{\thk, \thk'}(u,u',v) \d \thk \d \thk'.
\end{equation}
Moreover, if $\lambda \wedge \lambda' \geq \sqrt{v} +1$ and if $u,u' \geq 1$, then
\begin{equation}
\label{eq:lem_H_massless_bound}
\hat{\Hs}_a(\lambda,\lambda',v)
\leq C v^{1/4-\theta/2} \frac{1}{(\lambda - \sqrt{v})(\lambda' - \sqrt{v})} e^{(2 \sqrt{v} - \lambda - \lambda')a}
\end{equation}
and
\begin{equation}
\label{eq:lem_H_massive_bound}
\dhat{\Hs}_a(\lambda,\lambda', u, u',v)
\leq C (uu')^\theta v^{1/4-\theta/2} \frac{1}{(\lambda - \sqrt{v})(\lambda' - \sqrt{v})} e^{(2 \sqrt{v} - \lambda - \lambda')a}.
\end{equation}
\end{lemma}

\begin{proof}
In \eqref{eq:proof_integral_product_Bs}, we noticed that
\[
\sum_{l \geq 1} \frac{\theta^l}{l!} \int_{\substack{\mathsf{a} \in E(\thk,l) \\ \mathsf{a}' \in E(\thk',l)}} \d \mathsf{a} \d \mathsf{a}' \prod_{i=1}^l \frac{\Bs (a_i a_i' v)}{a_i a_i'} = \sum_{k \geq 1} \frac{v^k (\thk \thk')^{k-1} \theta^{(k)}}{(k-1)!^2 k!}.
\]
We can therefore rewrite
\begin{align*}
\hat{\Hs}_a(\lambda,\lambda',v) = \sum_{k \geq 1} \frac{v^k \theta^{(k)}}{k!} \left( \int_{\substack{\thk, \tilde{\thk}>0 \\ \thk + \tilde{\thk} \geq a}} \d \thk \d \tilde{\thk} e^{-\lambda(\thk + \tilde{\thk})} \frac{\tilde{\thk}^{\theta-1} \thk^{k-1} }{(k-1)!} \right) \Bigg( \lambda \leftrightarrow \lambda' \Bigg)
\end{align*}
where the second term in parenthesis is equal to the first one with $\lambda$ replaced by $\lambda'$. We can further compute (see \eqref{eq:beta_function})
\begin{align*}
\int_{\substack{\thk, \tilde{\thk}>0 \\ \thk + \tilde{\thk} \geq a}} \d \thk \d \tilde{\thk} e^{-\lambda(\thk + \tilde{\thk})} \frac{\tilde{\thk}^{\theta-1} \thk^{k-1} }{(k-1)!}
= \int_a^\infty \d t ~e^{-\lambda t} \int_0^t \d \thk \frac{(t-\thk)^{\theta-1} \thk^{k-1}}{(k-1)!}
= \frac{1}{\theta^{(k)}} \int_a^\infty t^{\theta+k-1} e^{-\lambda t} \d t.
\end{align*}
We have obtained that
\[
\hat{\Hs}_a(\lambda,\lambda',v) = \int_a^\infty \d t ~t^{\theta-1} e^{-\lambda t} \int_a^\infty \d t' ~{t'}^{\theta-1} e^{-\lambda' t'} \sum_{k \geq 1} \frac{(vtt')^k}{k! \theta^{(k)}}.
\]
We recognise here a modified Bessel function \eqref{eq:modified_Bessel_function} concluding the proof of \eqref{eq:lem_H_massless_explicit}.

To obtain the upper bound \eqref{eq:lem_H_massless_bound}, we first bound (thanks to \eqref{eq:modified_Bessel_function_asymptotic})
\[
I_{\theta-1}(u) \leq C e^u/ \sqrt{u}, \quad u>0,
\]
which gives
\begin{align*}
\hat{\Hs}_a(\lambda,\lambda',v) \leq C v^{1/4-\theta/2} \int_{[a,\infty)^2} (tt')^{\theta/2 -3/4} e^{-\lambda t - \lambda' t'} e^{2\sqrt{vtt'}} \ \d t \ \d t'.
\end{align*}
We next bound $2\sqrt{tt'} \leq t + t'$ and
\begin{align*}
\hat{\Hs}_a(\lambda,\lambda',v) & \leq C v^{1/4-\theta/2} \left( \int_a^\infty t^{\theta/2 -3/4} e^{-\lambda t} e^{\sqrt{v}t} \d t \right) \Big( \lambda \leftrightarrow \lambda' \Big) \\
& \leq C v^{1/4-\theta/2} \frac{1}{(\lambda - \sqrt{v})(\lambda' - \sqrt{v})} e^{(2 \sqrt{v} - \lambda - \lambda')a}
\end{align*}
where we used the assumption that $\lambda \wedge \lambda' \geq \sqrt{v} + 1$ to obtain the last inequality. This concludes the proof of \eqref{eq:lem_H_massless_bound}.

\eqref{eq:lem_H_massive_explicit} directly follows from the definition \eqref{eq:def_Hs} of $\Hs_{\thk, \thk'}$. The bound \eqref{eq:lem_H_massive_bound} then follows from Lemma \ref{lem:Hs}. This concludes the proof.
\end{proof}

\begin{remark}
We now state a generalisation of \eqref{eq:lem_H_massless_explicit} that will be needed in the proof of Lemma \ref{lem:discrete_second_moment_convergence}. This generalisation is proven in a very similar way and we omit its proof. Let $p : [0,\infty)^2 \to [0,1]$ be a measurable function. Then for all $k \geq 1$,
\begin{align}
\label{eq:rmk_massless}
& \int_{\substack{\thk, \tilde{\thk}>0 \\ \thk + \tilde{\thk} \geq a}} \d \thk \d \tilde{\thk} ~e^{-\lambda(\thk + \tilde{\thk})} \tilde{\thk}^{\theta-1}
\int_{\substack{\thk', \tilde{\thk}'>0 \\ \thk' + \tilde{\thk}' \geq a}} \d \thk' \d \tilde{\thk}' e^{-\lambda'(\thk' + \tilde{\thk}')} ( \tilde{\thk}' )^{\theta-1} p(\rho + \tilde{\rho}, \rho' + \tilde{\rho}') \\
\nonumber
& \times
\sum_{l = 1}^k \frac{\theta^l}{l!} \int_{\substack{\mathsf{a} \in E(\thk,l) \\ \mathsf{a}' \in E(\thk',l)}} \d \mathsf{a} \d \mathsf{a}' \prod_{i=1}^l \sum_{\substack{k_1, \dots, k_l \geq 1 \\ k_1 + \dots + k_l = k}} \frac{v^{k_i}(a_i a_i')^{k_i-1}}{k_i! (k_i-1)!} \\
& = \frac{v^k}{\theta^{(k)}k!} \int_{(a,\infty)^2} e^{-\lambda t - \lambda' t'} p(t,t') (tt')^{\theta+k-1} \d t ~\d t'. \nonumber
\end{align}
To recover \eqref{eq:lem_H_massless_explicit} from \eqref{eq:rmk_massless}, one simply needs to take the function $p = 1$ and sum over $k \geq 1$.
Similarly, one can prove that
\begin{align}
\label{eq:rmk_mixed}
& \int_{\substack{\thk, \tilde{\thk}>0 \\ \thk + \tilde{\thk} \geq a}} \d \thk \d \tilde{\thk} ~e^{-\lambda(\thk + \tilde{\thk})} \tilde{\thk}^{\theta-1}
\int_a^\infty \d t' e^{-\lambda' t' } \sum_{\substack{m \geq 1 \\ 1 \leq l \leq m}} \frac{\theta^m}{(m-l)!l!}  \\
\nonumber
& \times
\int_{\substack{\mathsf{a} \in E(\thk,l) \\ \mathsf{a}' \in E(t',m)}} \d \mathsf{a} \d \mathsf{a}'
\sum_{k \geq l} v^k p(\rho + \tilde{\rho}, \rho', k) \sum_{\substack{k_1, \dots, k_l \geq 1 \\ k_1 + \dots + k_l = k}} \prod_{i=1}^l \frac{(a_i a_i')^{k_i-1}}{k_i ! (k_i - 1)!} \prod_{i=l+1}^m \frac{1-e^{-a_i' u'}}{a_i'} \\
\nonumber
& = \sum_{k \geq 1} v^k \int_{(a,\infty)^2} \d t ~ \d t' ~e^{-\lambda t} e^{-\lambda' t' } \frac{t^{\theta+k-1}}{k!(k-1)!} p(t,t',k) \Big( \int_0^{t'} \d \rho' {\rho'}^{k-1} \frac{\Fs(u' (t'-  \rho'))}{t' - \rho'} + {t'}^{k-1} \Big).
\end{align}
This latter equality will be useful in the mixed case massless--massive.
\end{remark}

\subsection{Number of crossings in the processes of excursions}\label{subsec:discrete_crossing}

In addition to $A$, $\eta$ and $b$, we will also fix some large $M>0$ throughout Section \ref{subsec:discrete_crossing}.

Let $z, z' \in A$ and $r>0$ be such that $D(z,er) \subset A$ and $1 \leq Mr/|z-z'| < e$.  In view of Propositions \ref{Prop 1st mom Mc N}, \ref{Prop 1st mom Mc N K} and \ref{Prop 2nd mom discr}, we will need to study the number of crossings $N_{z,r}^\Cc$ for $\Cc = \Xi_{N,a}^z, \Xi_{N,z',a}^z$ and $\Xi_{N,z,a}^{z'}$. By property of Poisson point processes, in each cases, we can decompose
\[
N_{z,r}^\Cc = \sum_{i=1}^P G_i
\]
where $G_i, i \geq 1$, are i.i.d. random variables independent of $P$ with the following distributions: $P$ is a Poisson random variable with mean
\begin{align*}
& \E[P] = 2\pi a \tilde{\mu}^{z,z}_{D_N}(\tau_{\partial \D_N(z,er)} < \infty),  \quad \Cc = \Xi_{N,a}^z,\\
& \E[P] = 2\pi a \tilde{\mu}^{z,z}_{D_N \setminus \{z'\} }(\tau_{\partial \D_N(z,er)} < \infty),  \quad \Cc = \Xi_{N,z',a}^z,\\
& \E[P] = 2\pi a \tilde{\mu}^{z',z'}_{D_N \setminus \{z\} }(\tau_{\partial \D_N(z,er)} < \infty), \quad \Cc = \Xi_{N,z,a}^{z'}
\end{align*}
and the common distribution of the $G_i$'s is the law of $N_{z,r}^\wp$ where $\wp$ is distributed according to
\[
\frac{\indic{ \wp \text{~hits~} \partial \D_N(z,er)}\tilde{\mu}^{w,w}_{D_N'}( d \wp)}{\tilde{\mu}^{w,w}_{D_N'}(\tau_{\partial \D_N(z,er)} < \infty)}
\]
and
\begin{equation}
\label{eq:DN'_w}
\begin{array}{ll}
D_N' = D_N \quad \text{and} \quad w=z,  & \quad \Cc = \Xi_{N,a}^z,\\
D_N' = D_N \setminus \{z'\} \quad \text{and} \quad w=z,  & \quad \Cc = \Xi_{N,z',a}^z,\\
D_N' = D_N \setminus \{z\} \quad \text{and} \quad w=z',  & \quad \Cc = \Xi_{N,z,a}^{z'}.
\end{array}
\end{equation}

In what follows, we will refer to the variables $P$ and $G_i$ in ``\textbf{Cases 1, 2 and 3}'' when we mean that we consider the number of excursions $N_{z,r}^\Cc$ in the cases $\Cc = \Xi_{N,a}^z, \Xi_{N,z',a}^z$ and $\Xi_{N,z,a}^{z'}$, respectively. In the upcoming Lemmas \ref{lem:discrete_mean_Poisson} and \ref{lem:discrete_geo}, we will respectively estimate the mean of $P$ and show that the $G_i$'s can be well approximated by geometric random variable. These lemmas are to be compared with Lemma \ref{lem:crossingsPPP} in the continuum, but we will see that the discrete setting leads to some technical difficulties.

\begin{lemma}\label{lem:discrete_mean_Poisson}
Let $z \in A$ and $r \in \{e^{-n}, n \geq 1 \}, r > N^{-1+\eta}$ be such that $D(z,er) \subset A$. We have, in Case 1,
\begin{equation}\label{eq:lem_mean_Poisson1}
\E[P] =
a \left( 1 + O \left( \frac1{|\log r|} \right) \right) \left( \frac{1}{|\log r|} - \frac{1}{\log(N)} \right)^{-1}.
\end{equation}
Let $z' \in A$ be such that $1 \leq Mr/|z-z'| < e$ and denote $\beta = 1 - \frac{|\log r|}{\log N}$, so that $r=N^{-1+\beta}$. Then, in Cases 2 and 3,
\begin{equation}\label{eq:lem_mean_Poisson2}
\E[P] =
a \left( 1 + O \left( \frac{1}{|\log r|} \right) \right) \frac{1}{1-(1-\beta)^2} |\log r|.
\end{equation}
\end{lemma}

\begin{proof}[Proof of Lemma \ref{lem:discrete_mean_Poisson}]
We will show upper bounds on $\E[P]$ as stated in the lemma. The matching lower bounds will follow from the same proof: one simply has to replace maxima by minima below.

Let us first start by showing the following intermediate result:
in Case 1,
\begin{equation}
\label{eq:discrete_mean_Poisson1}
\E[P]
\leq a \frac{(\log N)^2}{2\pi G_{D_N}(z,z) G_{D(z,er)}(z,z)} \max_{y \in \partial \D_N(z,er)} G_{D_N}(y,z),
\end{equation}
in Case 2,
\begin{equation}
\label{eq:discrete_mean_Poisson2}
\E[P]
\leq a \frac{(\log N)^2}{2\pi G_{D_N \setminus \{z'\} }(z,z) G_{D(z,er)}(z,z)} \max_{y \in \partial \D_N(z,er)} G_{D_N \setminus \{z'\}}(y,z)
\end{equation}
and in Case 3,
\begin{equation}
\label{eq:discrete_mean_Poisson3}
\E[P]
\leq a \frac{(\log N)^2}{2\pi G_{D_N \setminus \{z\} }(z',z')} 4 \P^{z'} \left( \tau_{\partial \D_N(z,er)} < \tau_{z'}^+ \wedge \tau_{\partial D_N} \right) \max_{y \in \partial \D_N(z,er)} G_{D_N \setminus \{z\}}(y,z').
\end{equation}
We will show \eqref{eq:discrete_mean_Poisson1} and we will then explain what needs to be changed in order to have \eqref{eq:discrete_mean_Poisson2} and \eqref{eq:discrete_mean_Poisson3}.
First of all, the total mass of $2\pi \tilde{\mu}^{z,z}_{D_N}$ is given by
\begin{align*}
& \frac{1}{2\pi} (\log N)^2 \sum_{w_1, w_2 \sim z} G_{D_N \setminus \{z\} }(w_1,w_2) \\
& = \frac{1}{2\pi} (\log N)^2 \sum_{w_1, w_2 \sim z} \left( G_{D_N}(w_1,w_2) - \frac{G_{D_N}(w_1,z) G_{D_N}(z,w_2)}{G_{D_N}(z,z)} \right)
\end{align*}
where we used \eqref{eq:Green_DN-z} to obtain the last equality.
For $w_1$ fixed, $G_{D_N}(w_1, \cdot)$ is harmonic outside of $w_1$ which implies that
\[
\sum_{w_2 \sim z} G_{D_N}(w_1, w_2) = 4 G_{D_N}(w_1,z).
\]
Since
\[
\sum_{w_1 \sim z} G_{D_N}(w_1,z) = 4(G_{D_N}(z,z) - 1/4),
\]
we obtain that the total mass of $2\pi \tilde{\mu}^{z,z}_{D_N}$ is equal to
\begin{align*}
& \frac{1}{2\pi} (\log N)^2 \left( 4 \sum_{w_1 \sim z} G_{D_N}(w_1,z) - \frac{1}{G_{D_N}(z,z)} \left( \sum_{w_1 \sim z} G_{D_N}(w_1,z) \right)^2 \right) \\
& = \frac{1}{2\pi} (\log N)^2 \left( 16 (G_{D_N}(z,z) -1/4) - \frac{16}{G_{D_N}(z,z)} \left( G_{D_N}(z,z)-1/4 \right)^2 \right) \\
& = \frac{1}{2\pi} (\log N)^2 \left( 4 - \frac{1}{G_{D_N}(z,z)} \right).
\end{align*}
Moreover,  $\tilde{\mu}_{D_N}^{z,z}$ normalised by its total mass is the law of a random walk $(X_t)_{0 \leq t \leq \tau^+_z}$ starting at $z$, killed upon returning at $z$ for the first time:
\[
\tau_z^+ := \inf \{ t>0: X_t = z, \exists s \in (0,t), X_s \neq z\}
\]
and conditioned to stay in $D_N$. We wish to compute the probability for such a walk to visit $\partial \D_N(z,er)$.  By strong Markov property, we have
\begin{align*}
\P^z \left( \tau_{\partial \D_N(z,er)} < \tau_z^+ \vert \tau_z^+ < \tau_{\partial \D_N} \right)
= \frac{\P^z \left( \tau_{\partial \D_N(z,er)} < \tau_z^+ < \tau_{\partial \D_N} \right)}{\P^z \left( \tau_z^+ < \tau_{\partial \D_N} \right) } \\
\leq \frac{\P^z \left( \tau_{\partial \D_N(z,er)} < \tau_z^+ \right)}{\P^z \left( \tau_z^+ < \tau_{\partial \D_N} \right)}
\max_{y \in \partial \D_N(z,er)} \P^y \left( \tau_z < \tau_{\partial \D_N} \right).
\end{align*}
We can express these probabilities in terms of Green functions as follows:
\[
\P^z \left( \tau_z^+ < \tau_{\partial \D_N} \right) = 1 - \frac{1}{4G_{D_N}(z,z)}  = \frac{G_{D_N}(z,z) - 1/4}{G_{D_N}(z,z)},
\]
\[
\max_{y \in \partial \D_N(z,er)} \P^y \left( \tau_z < \tau_{\partial \D_N} \right) = \frac{\max_{y \in \partial \D_N(z,er)} G_{D_N}(y,z)}{G_{D_N}(z,z)},
\]
and
\begin{equation}
\label{eq:proof_mean_Poisson1}
\P^z \left( \tau_{\partial \D_N(z,er)} < \tau_z^+ \right) = \frac{1}{4G_{D(z,er)}(z,z)}.
\end{equation}
Overall, we have shown that
\begin{align*}
2\pi a \tilde{\mu}_{D_N}^{z,z} ( \tau_{\partial \D_N(z,er)} < \infty )
& \leq \frac{1}{2\pi} (\log N)^2 \left( 4 - \frac{1}{G_{D_N}(z,z)} \right) \frac{\max_{y \in \partial \D_N(z,er)} G_{D_N}(y,z)}{4 G_{D(z,er)}(z,z) (G_{D_N}(z,z)-1/4)} \\
& = \frac{(\log N)^2}{2\pi G_{D_N}(z,z) G_{D(z,er)}(z,z)} \max_{y \in \partial \D_N(z,er)} G_{D_N}(y,z),
\end{align*}
which is the desired upper bound \eqref{eq:discrete_mean_Poisson1}. The proof of \eqref{eq:discrete_mean_Poisson2} follows along the exact same lines. To prove \eqref{eq:discrete_mean_Poisson3}, the only thing that needs to be changed is that now, instead of \eqref{eq:proof_mean_Poisson1}, we have
\[
\P^{z'} \left( \tau_{\partial \D_N(z,er)} < \tau_{z'}^+ \wedge \tau_{\partial D_N} \right).
\]
We leave it as it is and directly obtain \eqref{eq:discrete_mean_Poisson3}.

We now move on to explaining how \eqref{eq:lem_mean_Poisson1} and \eqref{eq:lem_mean_Poisson2} follow from \eqref{eq:discrete_mean_Poisson1}, \eqref{eq:discrete_mean_Poisson2} and \eqref{eq:discrete_mean_Poisson3}. We start with \eqref{eq:lem_mean_Poisson1}.
By Lemma \ref{lem:Green_discrete}, we have
\[
2\pi G_{D_N}(z,z) = \log N + O(1),
\quad 2\pi G_{D(z,er)}(z,z) = \log (Nr) + O(1)
\]
and
\[
2\pi \max_{y \in \partial \D_N(z,er)} G_{D_N}(y,z) = |\log r| + O(1)
\]
and therefore, in Case 1,
\[
\E[P] \leq a \left( 1 + O \left( \frac1{\log r} \right) \right) \frac{\log N |\log r|}{\log (Nr)} = a \left( 1 + O \left( \frac1{\log r} \right) \right) \left( \frac{1}{|\log r|} - \frac{1}{\log N} \right)^{-1}.
\]
This concludes the proof of \eqref{eq:lem_mean_Poisson1}. We now prove \eqref{eq:lem_mean_Poisson2} in Case 2. Recall that $\beta = 1 - |\log r|/\log N$. Using the expression \eqref{eq:Green_DN-z} of the Green function in $D_N \setminus \{z'\}$ and then Lemma \ref{lem:Green_discrete}, we see that
\[
G_{D_N \setminus \{z'\}}(z,z) = G_{D_N}(z,z) - \frac{G_{D_N}(z,z')^2}{G_{D_N}(z',z')}
= \frac{1}{2\pi} \left( 1 - (1-\beta)^2 \right) \log N + O(1),
\]
and if $y \in \partial \D_N(z,er)$,
\begin{align*}
G_{D_N \setminus \{z'\}}(y,z) & = G_{D_N}(y,z) - \frac{G_{D_N}(y,z') G_{D_N}(z,z')}{G_{D_N}(z',z')} \\
& = \frac{1}{2\pi} \left( (1-\beta) - (1-\beta)^2 \right) \log N + O(1)
= \frac{\beta}{2\pi} |\log r| + O(1).
\end{align*}
Recall also that
\[
G_{D(z,er)}(z,z) = \frac{1}{2\pi} \log (Nr) + O(1) = \frac{\beta}{2\pi} \log N + O(1).
\]
Plugging these three estimates in \eqref{eq:discrete_mean_Poisson2} concludes the proof of \eqref{eq:lem_mean_Poisson2} in Case 2.

To conclude the proof of Lemma \ref{lem:discrete_mean_Poisson}, it remains to prove \eqref{eq:lem_mean_Poisson2} in Case 3. We need to work a bit more and we need to estimate precisely $\P^{z'} \left( \tau_{\partial \D_N(z,er)} < \tau_{z'}^+ \wedge \tau_{\partial D_N} \right)$.
In view of what we did, in order to conclude, it is enough to show that
\begin{equation}
\label{eq:discrete_mean_Poisson5}
\P^{z'} \left( \tau_{\partial \D_N(z,er)} < \tau_{z'}^+ \wedge \tau_{\partial D_N} \right) = \left( 1 + O \left( \frac1{\log r} \right) \right) \frac{\pi}{2\beta \log N}.
\end{equation}
The rest of the proof is dedicated to this estimate.
We claim that the probability on the left hand side of \eqref{eq:discrete_mean_Poisson5} is at most equal to
\begin{equation}
\label{eq:discrete_mean_Poisson4}
\frac{1 - \P^{z'} \left( \tau_{z'}^+ < \tau_{\partial D_N} \right)}{1 - \P^{z'} \left( \tau_{\partial \D_N(z,er)} < \tau_{\partial D_N} \right) \max_{y \in \partial \D_N(z,er)} \P^y \left( \tau_{z'} < \tau_{\partial D_N} \right) } \P^{z'} \left( \tau_{\partial \D_N(z,er)} < \tau_{\partial D_N} \right).
\end{equation}
Indeed, if we denote by
\[
p = \P^{z'} \left( \tau_{\partial \D_N(z,er)} < \tau_{z'}^+ \wedge \tau_{\partial D_N} \right)
\quad \mathrm{and} \quad
q = \P^{z'} \left( \tau_{z'}^+ < \tau_{\partial \D_N(z,er)} \wedge \tau_{\partial D_N} \right),
\]
the strong Markov property shows that
\begin{align*}
\P^{z'} \left( \tau_{\partial \D_N(z,er)} < \tau_{\partial D_N} \right)
= p
+ q \P^{z'} \left( \tau_{\partial \D_N(z,er)} < \tau_{\partial D_N} \right)
\end{align*}
and also
\begin{align*}
\P^{z'} \left( \tau_{z'}^+ < \tau_{\partial D_N} \right) & = q + p \P^{z'} \left( \tau_{z'}^+ < \tau_{\partial D_N} \vert \tau_{\partial \D_N(z,er)} < \tau_{z'}^+ \wedge \tau_{\partial D_N} \right) \\
& \leq q + p \max_{y \in \partial \D_N(z,er)} \P^y \left( \tau_{z'} < \tau_{\partial D_N} \right).
\end{align*}
Combining the two above estimates yields the claim \eqref{eq:discrete_mean_Poisson4}.
Now, by \cite[Proposition 6.4.1]{LawlerLimic2010RW},
\[
\P^{z'} \left( \tau_{\partial \D_N(z,er)} < \tau_{\partial \D_N} \right)  = 1 - \frac{\log |z'-z|/(er) + O(1/\log r)}{|\log r| + O(1)} = 1 + O \left( \frac1{\log r} \right).
\]
Moreover, for all $y \in \partial \D_N(z,er)$,
\[
\P^y \left( \tau_{z'} < \tau_{\partial D_N} \right)  = \frac{G_{D_N}(y,z')}{G_{D_N}(z',z')} = 1-\beta + O \left( \frac1{\log N} \right)
\]
and
\[
1 - \P^{z'} \left( \tau_{z'}^+ < \tau_{\partial D_N} \right) = \frac{1}{4G_{D_N}(z',z')} = \frac{\pi}{2 \log N} \left( 1 + O \left( \frac1{\log r} \right) \right).
\]
Plugging those three estimates into \eqref{eq:discrete_mean_Poisson4} shows \eqref{eq:discrete_mean_Poisson5} (or more precisely, the upper bound, but the lower bound is similar). This concludes the proof.
\end{proof}

We now turn to the study of the variables $G_i$.

\begin{lemma}\label{lem:discrete_geo}
Let $z \in A$ and $r \in \{e^{-n}, n \geq 1 \}, r > N^{-1+\eta}$ be such that $D(z,er) \subset A$. In Case 1, we have for all $k \geq 1$,
\begin{equation}\label{eq:lem_discrete_geo1}
\Prob{G_i \geq k} = \left( 1 + O \left( \frac1{\log r} \right) \right) \left( 1 - \frac{1 + O(1/\log r)}{|\log r|} - \frac{1 + O(1/\log(Nr))}{\log(Nr)} \right)^{k-1} .
\end{equation}
Let $z' \in A$ be such that $1 \leq Mr/|z-z'| < e$ and denote $\beta = 1 - \frac{|\log r|}{\log N}$, so that $r=N^{-1+\beta}$. There exists $M_\eta >0$ such that if $M > M_\eta$, then we have in Cases 2 and 3, for all $k \geq 1$,
\begin{equation}\label{eq:lem_discrete_geo2}
\Prob{G_i \geq k} =
\left( 1 + O \left( \frac{1}{|\log r|} \right) \right) \left( 1 - \frac{2 - \beta \pm \eta^2 + O(1/\log r)}{\beta |\log r|} \right)^{k-1}.
\end{equation}
\end{lemma}

\begin{proof}
We start with the following claim: let $D_N'$ be a finite subset of  $\Z_N^2$, $w \in D_N' \setminus \D_N(z,er)$ and let $\wp$ be distributed according to
\[
\frac{\indic{ \wp \text{~hits~} \partial \D_N(z,er)}\tilde{\mu}^{w,w}_{D_N'}( d \wp)}{\tilde{\mu}^{w,w}_{D_N'}(\tau_{\partial \D_N(z,er)} < \infty)}.
\]
Then for all $k \geq 1$, $\Prob{N_{z,r}^\wp \geq k}$ is at most
\begin{align}
\label{eq:proof_geo_claim}
\max_{y_1, y_2 \in \partial \D_N(z,er)} \frac{G_{D_N'}(y_1,w)}{G_{D_N'}(y_2,w)}
& \Big( \max_{y \in \partial \D_N(z,er)} \P^y \left( \tau_{ \partial D(z,r)} < \tau_{\partial D_N'} \wedge \tau_{w} \right) \\
& \times \max_{y \in \partial \D_N(z,r)} \P^y \left( \tau_{\partial \D_N(z,er)} < \tau_{\partial D_N'} \wedge \tau_{w} \right)
\Big)^{k-1}
\nonumber
\end{align}
and at least the same quantity with maxima replaced by minima.
We will apply this with $D_N'$ and $w$ given as in \eqref{eq:DN'_w}. The proof of this claim is a quick consequence of strong Markov property. Indeed, the trajectory $\wp$, after hitting for the first time $\partial \D_N(z,er)$, has the law of a random walk starting at some vertex of $\partial \D_N(z,er)$ (with some law that is irrelevant to us), stopped upon reaching $w$ and conditioned to hit $w$ before exiting $D_N'$; and we wish to estimate the probability for such a trajectory to cross the annulus at least $k-1$ times. We omit the details.

We now explain how the proof of Lemma \ref{lem:discrete_geo} follows from \eqref{eq:proof_geo_claim}. Recall that we will apply the above claim with $D_N'$ and $w$ given as in \eqref{eq:DN'_w}. In all cases, one can show that the ratio of the Green functions equals $1+O(1/\log r)$. In all cases, we also have that the second probability in \eqref{eq:proof_geo_claim} is equal to
\[
\max_{y \in \partial \D_N(z,r)} \P^y \left( \tau_{\partial \D_N(z,er)} < \tau_z \right) =
\max_{y \in \partial \D_N(z,r)} \left( 1 - \frac{G_{D(z,er)}(y,z)}{G_{D(z,er)}(z,z)} \right).
\]
By \cite[Propositions 1.6.6 and 1.6.7]{LawlerIntersections},  we deduce that
\[
\max_{y \in \partial \D_N(z,r)} \P^y \left( \tau_{\partial \D_N(z,er)} < \tau_z \right)
= 1 - \frac{1 + O((Nr)^{-1})}{\log(Nr) + O(1)}
= 1 - \frac{1+O(1/\log(Nr))}{\log (Nr)}.
\]
Now, in Case 1, the first probability in \eqref{eq:proof_geo_claim} is equal to
\[
\max_{y \in \partial \D_N(z,er)} \P^y \left( \tau_{ \partial D(z,r)} < \tau_{\partial D_N} \right)
\]
which is estimated in \cite[Proposition 6.4.1]{LawlerLimic2010RW} and is equal to
\[
1 - \frac{1 + O(1/\log r)}{|\log r| + O(1)} = 1 - \frac{1 + O(1/\log r)}{|\log r|}.
\]
This concludes the upper bound \eqref{eq:lem_discrete_geo1}. The lower bound is similar. In Cases 2 and 3, the first probability in \eqref{eq:proof_geo_claim} is equal to
\[
\max_{y \in \partial \D_N(z,er)} \P^y \left( \tau_{ \partial D(z,r)} < \tau_{\partial D_N} \wedge \tau_{z'} \right).
\]
To conclude the proof of \eqref{eq:lem_discrete_geo2}, a small computation shows that it is sufficient to prove that for all $y \in \partial \D_N(z,er)$,
\begin{equation}
\label{eq:proof_geo_p}
p := \P^y \left( \tau_{ \partial D(z,r)} < \tau_{\partial D_N} \wedge \tau_{z'} \right) = 1 - \frac{1 + O(1/M) + O(1/\log r)}{\beta |\log r|}.
\end{equation}
The rest of the proof is dedicated to this estimate. The strategy is very similar to the one we used to prove \eqref{eq:discrete_mean_Poisson5}.
Let us denote
\[
q = \P^y \left( \tau_{z'} < \tau_{\partial D_N} \wedge \tau_{\partial D(z,r)} \right).
\]
By the strong Markov property, we have
\[
\P^y \left( \tau_{\partial \D_N(z,r)} < \tau_{\partial D_N} \right) = p + q \P^{z'} \left( \tau_{\partial \D_N(z,r)} < \tau_{\partial D_N} \right)
\]
and also
\[
\P^y \left( \tau_{z'} < \tau_{\partial D_N} \right) = q + p \P^y \left( \tau_{z'} < \tau_{\partial D_N} \vert \tau_{\partial \D_N(z,r)} < \tau_{z'} \wedge \tau_{\partial D_N} \right).
\]
Combining these two equalities yields
\begin{align*}
p & = \frac{\P^y \left( \tau_{\partial \D_N(z,r)} < \tau_{\partial D_N} \right) - \P^y \left( \tau_{z'} < \tau_{\partial D_N} \right) \P^{z'} \left( \tau_{\partial \D_N(z,r)} < \tau_{\partial D_N} \right)}{1 - \P^{z'} \left( \tau_{\partial \D_N(z,r)} < \tau_{\partial D_N} \right) \P^y \left( \tau_{z'} < \tau_{\partial D_N} \vert \tau_{\partial \D_N(z,r)} < \tau_{z'} \wedge \tau_{\partial D_N} \right)}
\\
& = 1 - \frac{1 - \P^y \left( \tau_{\partial \D_N(z,r)} < \tau_{\partial D_N} \right) }{***} \\
& + \frac{\left(\P^y \left( \tau_{z'} < \tau_{\partial D_N} \vert \tau_{\partial \D_N(z,r)} < \tau_{z'} \wedge \tau_{\partial D_N} \right) - \P^y \left( \tau_{z'} < \tau_{\partial D_N} \right) \right)\P^{z'} \left( \tau_{\partial \D_N(z,r)} < \tau_{\partial D_N} \right)}{***}
\end{align*}
where the denominator did not change from the first identity to the second one.
The probability $p$ increases with the domain $D_N$. By including a macroscopic disc centred at $z$ inside $D_N$ ($z$ is in the bulk of $D$), we will obtain a lower bound on $p$ and by including $D_N$ in a disc centred at $z$ ($D$ is bounded) we will obtain an upper bound. Therefore, assume that $D=D(z,R)$ for some $R>0$.
Now, by \cite[Proposition 6.4.1]{LawlerLimic2010RW},
\[
\P^{z'} \left( \tau_{\partial \D_N(z,r)} < \tau_{\partial D_N} \right)  = 1 - \frac{\log |z'-z|/r + O(1/\log r)}{\log (R/r)}
\]
and
\[
\P^y \left( \tau_{\partial \D_N(z,r)} < \tau_{\partial D_N} \right) = 1 - \frac{1 + O(1/\log r)}{\log (R/r)}.
\]
Moreover,
\[
\P^y \left( \tau_{z'} < \tau_{\partial D_N} \vert \tau_{\partial \D_N(z,r)} < \tau_{z'} \wedge \tau_{\partial D_N} \right) \geq \min_{x \in \partial \D_N(z,r)} \frac{G_{D_N}(x,z')}{G_{D_N}(z',z')} =  1 - \beta + O(1/\log N).
\]
This shows that the denominator is equal to $\beta + O(1/\log r)$.
Since for all $x \in \partial D(z,r)$, we can bound
\[
1 - \frac{C}{M} \leq \frac{|x-z'|}{|y-z'|} \leq 1 + \frac{C}{M},
\]
we have
\begin{align*}
& \abs{ \P^y \left( \tau_{z'} < \tau_{\partial D_N} \vert \tau_{\partial \D_N(z,r)} < \tau_{z'} \wedge \tau_{\partial D_N} \right) - \P^y \left( \tau_{z'} < \tau_{\partial D_N} \right) } \\
& \leq \max_{x \in \partial \D_N(z,r)} \frac{|G_{D_N}(y,z') - G_{D_N}(x,z')|}{G_{D_N}(z',z')} \leq O(1) \frac{1}{M \log N} .
\end{align*}
We obtain that
\[
p = 1 - \frac{1 + O(1/M) + O(1/\log r)}{\beta |\log r|}
\]
which concludes the proof of \eqref{eq:proof_geo_p}. This finishes the proof of Lemma \ref{lem:discrete_geo}.
%In particular:
%
%1) When $D_N' = D_N$, $z = w$,
%$\Prob{N_{z,r}^\wp \geq k}$ is at most
%\begin{align*}
%\max_{y_1, y_2 \in \partial \D_N(z,er)} \frac{G_{D_N}(y_1,z)}{G_{D_N}(y_2,z)}
%\left( \max_{y \in \partial \D_N(z,er)} \P^y \left( \tau_{ \partial D(z,r)} < \tau_{\partial \D_N} \right) \max_{y \in \partial \D_N(z,r)} \left( 1 - \frac{G_{D(z,er)}(y,z)}{G_{D(z,er)}(z,z)} \right)
%\right)^{k-1}
%\end{align*}
%
%2) When $D_N' = D_N \setminus \{z' \}$ for some $z' \in D_N \setminus D(z,er)$, $w=z$,
%$\Prob{N_{z,r}^\wp \geq k}$ is at least
%\begin{align*}
%\min_{y_1, y_2 \in \partial \D_N(z,er)} \frac{G_{D_N \setminus \{z' \}}(y_1,z)}{G_{D_N \setminus \{z' \}}(y_2,z)}
%\left( \min_{y \in \partial \D_N(z,er)} \P^y \left( \tau_{ \partial D(z,r)} < \tau_{\partial \D_N} \wedge \tau_{z'} \right) \min_{y \in \partial \D_N(z,r)} \left( 1 - \frac{G_{D(z,er)}(y,z)}{G_{D(z,er)}(z,z)} \right)
%\right)^{k-1}
%\end{align*}
%
%3) When $D_N' = D_N \setminus \{z \}$ and $w=z'$ for some $z' \in D_N \setminus D(z,er)$,
%$\Prob{N_{z,r}^\wp \geq k}$ is at least
%\begin{align*}
%\min_{y_1, y_2 \in \partial \D_N(z,er)} \frac{G_{D_N \setminus \{z \}}(y_1,z')}{G_{D_N \setminus \{z \}}(y_2,z')}
%\left( \min_{y \in \partial \D_N(z,er)} \P^y \left( \tau_{ \partial D(z,r)} < \tau_{\partial \D_N} \wedge \tau_{z'} \right) \min_{y \in \partial \D_N(z,r)} \left( 1 - \frac{G_{D(z,er)}(y,z)}{G_{D(z,er)}(z,z)} \right)
%\right)^{k-1}
%\end{align*}
\end{proof}

From Lemmas \ref{lem:discrete_mean_Poisson} and \ref{lem:discrete_geo}, we obtain the discrete analogues of Corollaries \ref{cor:crossingPPP1} and \ref{cor:crossingPPP2} that we state below. We provide them without proofs since they follow from Lemmas \ref{lem:discrete_mean_Poisson} and \ref{lem:discrete_geo} in the same way as the two aforementioned Corollaries in the continuum follow from Lemma \ref{lem:crossingsPPP}.

Note that in Case 1, although $\E[P]/a$ and $\E[G_i]$ differ from $(1+o(1)) |\log r|$ which contrasts the continuous setting, the product $\E[P] \E[G_i]$ is still equal to $a (1+o(1)) |\log r|^2$ like in the continuum.

\begin{corollary}\label{cor:discrete_crossingPPP1}
Let $u \in (0,1/2)$.
There exists $C(u)>0$, $r_0>0$ and $c>0$ (which may depend on $a, b, \eta, A$) such that for all $z \in A$ and $r = N^{-1+\beta} \in (N^{-1+\eta},r_0)$,
\begin{equation}
\label{eq:cor_discrete_crossingPPP1a}
\Prob{N_{z,r}^{\Xi_{N,a}^z} > (a + (b-a)/2) |\log r|^{2} } \leq r^c
\end{equation}
and
\begin{align}
\label{eq:cor_discrete_crossingPPP1b}
\Expect{ \left( 1 - e^{-KT(\Xi_{N,a}^z)} \right) e^{\frac{u}{|\log r|} N_{z,r}^{\Xi_{N,a}^z} } } & \leq \left( 1 - e^{- a(3/2 q_N(z) q_{N,K}(z) C_{N,K}(z) + C(u) |\log r|)} \right) \\
& \times \exp \left( a \frac{u}{1-u \beta} (1+o(1)) |\log r| \right) .
\nonumber
\end{align}
\end{corollary}

To quickly see why we have $\frac{u}{1-u \beta}$ instead of $\frac{u}{1-u}$ as in Corollary \ref{cor:crossingPPP1}, we compute
\begin{align*}
\Expect{ e^{\frac{u}{|\log r|} N_{z,r}^{\Xi_{N,a}^z} } }  & = \exp \left( \E[P] \Expect{ e^{\frac{u}{|\log r}G_i} - 1} \right) \\
& = \exp \Big( (1+o(1)) \frac{a}{\beta} |\log r| \Big( \frac{1}{1-u\beta} -1 \Big) \Big)
= \exp \Big( (1+o(1)) a \frac{u}{1-u\beta} |\log r| \Big).
\end{align*}

\begin{corollary}\label{cor:discrete_crossingPPP2}
Let $u >0$.  There exists $M_\eta >0$ such that for all $M \geq M_\eta$, for all $z, z' \in D_N \cap A$ and $r=N^{-1+\beta}>N^{-1+\eta}$ being such that $1 \leq M r / |z-z'| < e$, we have
\begin{equation}
\label{eq:cor_discrete_crossingPPP2a}
\Expect{ \exp \left( - \frac{u}{|\log r|} N_{z,r}^{\Xi^z_{N,z',a}} \right) } \leq \exp \left( - a \frac{u}{(2-\beta)(2-\beta + \beta u  + \eta^2)} (1+o(1)) |\log r| \right)
\end{equation}
and
\begin{equation}
\label{eq:cor_discrete_crossingPPP2b}
\Expect{ \exp \left( - \frac{u}{|\log r|} N_{z,r}^{\Xi^{z'}_{N,z,a}} \right) } \leq \exp \left( - a \frac{u}{(2-\beta)(2-\beta + \beta u  + \eta^2)} (1+o(1)) |\log r| \right).
\end{equation}
\end{corollary}

Finally, we will need a control on the number of excursions in the process $\Xi^{z,z'}_{N,a,a'}$ \eqref{eq:discrete_def_PPPzz'}. The following lemma is to be compared with Lemma \ref{lem:crossingPPP10}.

\begin{lemma}\label{lem:discrete_crossing_not_PPP}
Let $u >0$.  There exists $M = M(\eta)$ large enough, so that if $z, z' \in D_N \cap A$ and $r=N^{-1+\beta}>0$ are such that $1 \leq M r / |z-z'| < e$, then
\[
\Expect{ \exp \left( - \frac{u}{|\log r|} N_{z,r}^{\Xi^{z,z'}_{N,a,a'}} \right) } \leq \frac{\Bs \left( (2\pi)^2 aa' \widetilde{G}_{D_N}(z,z')^2 (1+o(1)) \left(1+\frac{\beta u}{2-\beta + \eta^2}\right)^{-2} \right)}{\Bs \left( (2\pi)^2 aa' \widetilde{G}_{D_N}(z,z')^2 \right)} .
\]
\end{lemma}

\begin{proof}
The proof follows from the definition \eqref{eq:discrete_def_PPPzz'} of $\Xi^{z,z'}_{N,a,a'}$ and from Lemma \ref{lem:discrete_geo} in a very similar way as Lemma \ref{lem:crossingPPP10} was a consequence of Lemma \ref{lem:crossingsPPP}. We omit the details.
\end{proof}

\subsection{Proof of Lemma \ref{lem:discrete_first_moment_event} and localised KMT coupling}\label{subsec:discrete_first_moment}

We remind the reader that Lemma \ref{lem:discrete_first_moment_event} shows that the restriction to good events comes essentially for free in an $L^1$ sense. To do this, a crucial argument is that a typical (deterministic) point $z$ is not thick for the discrete loop soups. In the continuum, the corresponding large deviation estimate followed from Lemma \ref{lem:crossing_soup}. The proof of that lemma could probably be adapted with some tedious but ultimately superficial difficulties coming from the fact that we cannot easily condition on the maximum modulus of a loop when the space is discrete.
However, we find it more instructive to deduce Lemma \ref{lem:discrete_first_moment_event} from a coupling argument between discrete (random walk) loops and continuous (Brownian) loops. This coupling is a relatively simple modification of an argument put forward by Lawler and Trujillo-Ferreras \cite{LawlerFerreras07RWLoopSoup}, in which discrete random walks loop soups were in fact first introduced, with however one major difference. Indeed, \cite{LawlerFerreras07RWLoopSoup} shows that discrete and continuous loops are in one-to-one correspondence provided that they are not too small (essentially, of discrete duration at least $ N^{\kappa}$ with $\kappa > 2/3$, corresponding to loops of mesoscopic diameter $N^{\kappa/2} /N = N^{-2/3}$ when we scale the lattice so that the mesh size is $1/N$). In this correspondence, Lawler and Trujillo-Ferreras show furthermore that such loops are then not more than $\log N/N$ apart from one another with overwhelming probability, similar to a KMT approximation rate from which the result of \cite{LawlerFerreras07RWLoopSoup} follows.

While the KMT approximation is excellent (we in fact do not need the full power of the logarithmic KMT rate), the restriction to mesoscopic loops of sufficiently large polynomial diameter is problematic for us. It would indeed prevent us from getting any meaningful estimate concerning the crossings of annuli of diameter $r \ll N^{-2/3}$. This would place a restriction on the thickness parameter $a$ or equivalently $\gamma$; in order to treat the whole range of values $\gamma \in (0,2)$ we need to be able to consider crossings of annuli of any polynomial diameter $r \ge N^{-1+ \eta}$, with $\eta>0$ arbitrarily small (depending on $\gamma<2$).

On the other hand, it is fairly clear from the proof of \cite{LawlerFerreras07RWLoopSoup} that their result is sharp, and that the coupling described above cannot hold without the restriction $\kappa >2/3$; that is, at all scales smaller than $N^{-2/3}$ some discrete and continuous loops \emph{somewhere} will be quite different from one another. The lemma below shows however that if one is interested in the behaviour of small mesoscopic loops \emph{locally} (close to a given point $z$) then discrete and continuous loops \emph{at all polynomial scales} may be coupled to be close to one another. In this sense, Lemma \ref{L:coupling} below is a localised strengthening of Theorem 1.1 of \cite{LawlerFerreras07RWLoopSoup}.

This lemma may be of independent interest, and we state it now.
Let $\tilde \cL_{D_N}^\theta$ denote the discrete skeleton of $\cL^{\theta}_{D_N}$, which is formed by turning the continuous-time loops of $\cL_{D_N}^\theta$ into discrete-time ones, which consist of the ordered (rooted) sequence of successive vertices visited by each loop.
If $\wp \in \cL_{D}^\theta \cup \tilde \cL_{D_N}^\theta$, let $T({\wp})$ denote the lifetime of $\wp$ (which is an integer if $\wp \in \tilde \cL^\theta_{D_N}$). With a small abuse of notation, we will consider a path $\tilde \wp \in \tilde \cL^\theta_{D_N}$ as being defined over the entire interval of time $[0, T({\tilde \wp})]$ via linear interpolation. Note that with our conventions, the time variable $T({\tilde \wp})$ is typically of order $N^2$ for a macroscopic discrete random walk loop $\tilde \wp$, while its space variable is of order $1$ (i.e., the mesh size is $1/N$ and $\tilde \wp$ takes values in $(\Z/N)^2$). The following will be applied with $r$ of order $N^{-1 + \eta}$ for some $\eta>0$.

\begin{lemma}
  \label{L:coupling} Fix $\theta>0$ and let $\eta>0$. There exists $c>0$ (depending on the intensity $\theta$ and on $\eta$) such that the following holds. Let $z \in D$. For all $N^{-1 + \eta} \le r \le \diam(D)$ we can define on the same probability space $\cL_{D}^\theta $ and $\tilde \cL_{D_N}^\theta$ in such a way that :
  \begin{align*}
  \cA_{r,z}& = \{\wp \in \cL^\theta_{D}; T({\wp}) \ge \tfrac{r^2}{(\log N)^2}; |\wp(0) -z | \le \sqrt{T(\wp)} \log N;  \}  \\
 \text{ and }  \tilde \cA_{r, z,N}& =\{\tilde \wp \in \tilde \cL^\theta_{D_N};  T({\tilde \wp}) \ge \tfrac{r^2 N^2}{(\log N)^2}; |\tilde \wp(0) -z | \le \sqrt{\tfrac{T({\tilde \wp})}{N^2}} \log N\}
  \end{align*}
  are in one-to-one correspondence with probability at least $ 1- c (\log N)^6 / (rN) \ge 1- cN^{- \eta/2}$. Furthermore, if $\wp$ and $\tilde \wp$ are paired in this correspondence,
  \begin{align}
\Big\vert \frac{T({\tilde \wp})}{N^2} - T({\wp}) \Big\vert &\le (5/8) N^{-2}; \label{E:time} \\
\sup_{0 \le s \le 1} |\wp (s T({\wp}) )- \tilde \wp ( s T({\tilde \wp})) | &\le c N^{-1} \log N \label{E:space}
  \end{align}
  on an event of probability at least $1 - cN^{-4}$.
\end{lemma}

\begin{proof}
  We observe that the law of $\tilde \cL^\theta_{D_N}$ is that of a discrete random walk loop soup (in the sense of \cite{LawlerFerreras07RWLoopSoup}, i.e., in discrete time) with intensity $\theta$. Using the notations from \cite{LawlerFerreras07RWLoopSoup}, let $\tilde q_n$ denote the mass of discrete random walk loops with duration exactly $n$ (rooted at a specific point), and let $q_n$ denote the total mass of Brownian loops whose duration falls in the interval $[n- 3/8, n + 5/8]$ starting from a region of unit area (see top of p. 773 in \cite{LawlerFerreras07RWLoopSoup}). These constants are chosen so that the length of this interval is 1 (needed for coupling) and $q_n$ and $\tilde q_n$ are as close as possible: that is, they coincide not only in their first but also their second order, so that
  $$
  |q_n - \tilde q_n | \le C n^{-4}.
  $$
  To do the coupling it is easier to start with a random walk loop soup on the usual (unscaled) lattice $\Z^2$ and then apply Brownian scaling. That is, the Poisson processes of discrete loops emanating from each possible $x \in \Z^2$ and of duration $n \ge (rN)^2/ (\log N)^2$ with $| Nz- x |_{\Z^2} \le \sqrt{n}( \log N)$, can then be put in one-to-one correspondence for each $n\ge (rN)^2$ with a Poisson point processes of continuous Brownian loops of duration $t \in [ n - 3/8, n+5/8]$ starting in a unit square centered at $x$. This coupling fails with a probability at most
  \begin{align*}
& \le  C \sum_{n \ge r^2N^2/ (\log N)^2} \sum_{\substack{x \in \Z^2;\\ |x -Nz |_{\Z^2} \le  \sqrt{n}  \log N}}  | q_n - \tilde q_n|\\
& \le C  \sum_{n \ge (rN)^2/ (\log N)^2} n (\log N)^2  n^{-4} \\
& = C (\log N)^6/(rN)^4 .
  \end{align*}
  We then apply Brownian scaling to the above Brownian loops (this leaves the Brownian loop soup invariant in law), and scale the space variable of the discrete random walk loops, which provides the desired correspondence between $\cA_{r,z}$ and $\tilde \cA_{r,z,N}$.

By definition, the loops in this correspondence satisfy \eqref{E:time}. We now finish the argument in a similar manner to \cite{LawlerFerreras07RWLoopSoup}, coupling the discrete random walk and continuous Brownian loops of a given duration and starting point in the manner of Corollary 3.3 in \cite{LawlerFerreras07RWLoopSoup}, but with exponent $n^{-k}$ instead of $n^{-30}$ (as remarked in \cite[Corollary 3.2]{LawlerFerreras07RWLoopSoup}, the exponent $30$ was arbitrary, and can be replaced with any number $k $ with a suitably chosen constant $c = c_k$). Let $A$ be the event that in this coupling,
\begin{align*}
A &= \Big\lbrace \sup_{0 \le s \le 1} | \wp (s T({\wp}))  - \tilde \wp ( s T({\tilde \wp}))| \ge c_k \frac{\log (N^6)}{N} % & \quad \quad \quad
\text{ for some } \wp \in \cA_{r,z}, \tilde \wp \in \tilde \cA_{r,z,N}\Big\rbrace.
\end{align*}
Then we get (similar to \cite{LawlerFerreras07RWLoopSoup}, except we cannot take advantage of the fact that the duration of loops is at least $N^{2/3}$, and we use an error bound on the coupling which is $O(\text{duration})^{-k}$ instead of $O( \text{duration})^{-30}$):
\begin{align*}
\P( A) &\le c \theta r^2 N^{-4} + r^2N^2 N^6N^5c_k(\frac{r^2N^2}{(\log N)^2})^{-k} \\
&\le  c \theta r^2N^{-4} + c_kN^{11} (\log N)^{2k}(r^2N^2)^{1- k}\\
& \le c (N^{-4} + N^{11  + 2 \eta(1- k)} (\log N)^{2k})
\end{align*}
where $c$ depends on $k$ and $\theta$. If we choose $k$ large enough that $2 \eta(1-k) +11 < - 4$, we obtain
$$
\P( A) \le c N^{-4},
$$
where $c$ depends on $\theta$ and $\eta$, as desired.
\end{proof}

\begin{lemma}\label{lem:not_thick_discrete}
Fix $u>0$ and $\eta>0$. There exists $c>0$, such that for all $r \ge N^{-1+\eta}$ (and $r \le \diam (D)$, say), for all $z \in D_N$,
\[
\Prob{ N_{z,r}^{\Lc_{D_N \setminus \{ z \} }^\theta } > u |\log r|^2 } \leq r^{c}.
\]
\end{lemma}

\begin{proof}
We first dominate $N_{z,r}^{\Lc_{D_N \setminus \{ z \} }^\theta }$ by  $N_{z,r}^{\Lc_{D_N  }^\theta}$; that is, we forget about the restriction that the loops must not visit $z$ itself. We then apply the coupling of Lemma \ref{L:coupling}. Note that to each crossing of $A(z, r, er)$ by a discrete loop must correspond a crossing of the slightly smaller annulus $A' = A(z, 1.01 r, 0.99 er)$ by a continuous Brownian loop to which it is paired; let $N^{\Lc_D}_{z,r'}$ denote the number of crossing of the annulus $A'$ by the Brownian loop soup $\Lc_D$.

We now show that with overwhelming probability all possible loops that cross the annulus $A (z, r, er)$ are accounted for in the one-to-one correspondence of Lemma \ref{L:coupling}. To see this, observe that in order for a loop $\tilde \wp$ to cross the annulus $A(z, r, er)$ and not to be accounted for in the set $\tilde \cA_{z,r, N}$, the loop $\tilde \wp$ must either be extremely short or start far away from $z$: more precisely, its duration $T({\tilde \wp})$ should be less than
\begin{equation}\label{E:timebad}
T({\tilde \wp}) \le \tfrac{r^2N^2}{
(\log N)^2},
\end{equation}
or its starting point should be at a distance at least
\begin{equation}\label{E:spacebad}
|\tilde \wp (0) - z| \ge \sqrt{\tfrac{T({\tilde \wp})}{N^2}} ( \log N)
\end{equation}
from $z$. Either possibility is of course very unlikely since it requires the loop to travel a great distance in a short span of time. Let $\tilde \cB_1$ (resp. $\tilde \cB_2$) denote the set of (discrete) loops which verify \eqref{E:timebad} (resp. \eqref{E:spacebad}) and cross the annulus $A(z, r, er)$.

Let us show first $\E( \tilde \cB_1)$ decays faster than any polynomial. Fix $n \le r^2 N^2 / (\log N)^2$ and a starting point $x$. For a discrete random walk loop $\tilde \wp$ of duration $T({\tilde \wp}) = n $ and started at $x$, the probability to cross an annulus of width $r$ in time $n$ is bounded by
$$
C n \exp ( -  c\tfrac{r^2N^2}{n}) \le Cn \exp( - c(\log N)^2),
$$
for some universal constants $c, C>0$.
The exponential term above is obtained from elementary large deviation estimates (e.g. Hoeffding inequality) for discrete unconditioned random walk via a maximal inequality, and the factor $n$ in front accounts for the conditioning to return to the starting point in time $n$.
Summing over $n \le r^2 N^2$, and multiplying by the intensity of loops of duration $n$ (which is at most polynomial) we see that $\E( \tilde \cB_1) \le N^C \exp ( - c(\log N)^2)$ and so decays faster than any polynomial.

Let us turn to $\tilde \cB_2$, which we can handle similarly. Fix $n \ge r^2 N^2 / (\log N)^2$, and a starting point $x \in D \cap (\Z/N)^2$ such that $|x - z | \ge \sqrt{n} \log N/N$ (note that this means $n \le \diam (D) (N/ \log N)^2 \le C N^2$. In order for a random walk loop $\tilde \wp$ starting from $x$ and of duration $n$ to cross $A$, it must touch $A$ and so travel a distance at least $\sqrt{n} \log N/(2N)$ in time $n$. This is also bounded by
$$
Cn \exp ( - c\tfrac{n (\log N)^2}{n}) \le Cn \exp( - c (\log N)^2).
$$
Summing again over all possible values of $x$ and $n \le C N^2$, we get
$ \E( \tilde \cB_2) \le N^C \exp( - c (\log N)^2) $ and so also decays faster than any polynomial.

Thus, except on an event of probability at most $C N^{-\eta/2}$, $N^{\Lc_{D_N\setminus\{z\}}}_{z,r} \le N^{\Lc_D}_{z,r'}$. We can now
 use Lemma \ref{lem:crossing_soup} to bound the probability that the continuous loop soup has many crossing of the annulus $A'= A(z, 1.01 r, 0.99 er)$. Since the right hand side of the bound in Lemma \ref{lem:crossing_soup} is of the desired form (in fact, is more precise), we deduce
$$
\Prob{ N_{z,r}^{\Lc_{D_N \setminus \{ z \} }^\theta } > u |\log r|^2 }  \le CN^{-\eta/2} + r^c ,
$$
for some $c>0$. Since $r \ge N^{-1+ \eta}$, the right hand side above is at most $r^c$ for some (possibly different) value of $c$ (depending on $\eta$ and $u$ only).
\end{proof}

We now have all the ingredients we need to prove Lemma \ref{lem:discrete_first_moment_event}.

\begin{proof}[Proof of Lemma \ref{lem:discrete_first_moment_event}]
By Proposition \ref{Prop 1st mom Mc N}, we have
\begin{align*}
& \Expect{ \abs{ \tilde{\Mc}_a^N (A) - \Mc_a^N(A) } }
= \Expect{ \int_A \mathbf{1}_{\Gc_N(z)^c} \Mc_a^N(dz) } \\
& = \frac{\log N}{\Gamma(\theta) N^2} \sum_{z \in D_N \cap A} q_N(z)^\theta \int_a^\infty \d \thk \frac{c_0^{\thk} \thk^{\theta-1}}{N^{\thk-a}} \CR_N(z,D_N)^{\thk} \\
& \times \Prob{ \exists r \in \{e^{-n}, n \geq 1 \} \cap (N^{-1+\eta},r_0), N_{z,r}^{\Lc^\theta_{D_N \setminus \{z \} } \cup \Xi_{N,\thk}^z} > b |\log r|^2  } .
\end{align*}
Let $z \in D_N \cap A$. By Lemma \ref{lem:Green_discrete}, we can bound $q_N(z) \leq C$ and $\CR_N(z,D_N) \leq C$ for some constant $C>0$. We divide the integral over $\thk \in (a,\infty)$ into two parts corresponding to the integrals from $a$ to $a+(b-a)/2$ and from $a+(b-a)/2$ to infinity respectively. To bound the latter contribution, we simply bound the probability in the integrand by 1 and observe that
\begin{align*}
\int_{a+(b-a)/2}^\infty \frac{C^{\thk} \thk^{\theta-1}}{N^{\thk-a}} \d \thk
\leq \frac{C}{ N^{(b-a)/2} \log N}.
\end{align*}
To bound the contribution of the integral for $\thk \in (a,a+(b-a)/2)$, we notice that the probability in the integrand can be bounded by its value at $\thk = a + (b-a)/2$. Because
\[
\int_a^{a+(b-a)/2} \d \thk \frac{C^{\thk} \thk^{\theta-1}}{N^{\thk-a}} \leq \frac{C}{\log N},
\]
this leads to
\begin{align*}
& \int_a^{a+(b-a)/2} \d \thk \frac{C^{\thk} \thk^{\theta-1}}{N^{\thk-a}} \Prob{ \exists r \in \{e^{-n}, n \geq 1 \} \cap (N^{-1+\eta},r_0), N_{z,r}^{\Lc^\theta_{D_N \setminus \{z \} } \cup \Xi_{N,\thk}^z} > b |\log r|^2  } \\
& \leq \frac{C}{\log N} \Prob{ \exists r \in \{e^{-n}, n \geq 1 \} \cap (N^{-1+\eta},r_0), N_{z,r}^{\Lc^\theta_{D_N \setminus \{z \} } \cup \Xi_{N,a+(b-a)/2}^z} > b |\log r|^2  } .
\end{align*}
A union bound, Corollary \ref{cor:discrete_crossingPPP1} and Lemma \ref{lem:not_thick_discrete} show that the above probability is bounded by $C r_0^c$ for some $C,c>0$. This concludes the proof of \eqref{eq:lem_discrete_first_moment_event_no_mass}.
The proof of \eqref{eq:lem_discrete_first_moment_event_with_mass} is an interpolation of the proofs of \eqref{eq:lem_discrete_first_moment_event_no_mass} and Lemma \ref{lem:first_moment_good_event}. Note that we use \eqref{eq:cor_discrete_crossingPPP1b} instead of \eqref{eq:cor_discrete_crossingPPP1a}. We leave the details to the reader.
\end{proof}

\subsection{Proof of Lemma \ref{lem:discrete_second_moment_bdd} (truncated \texorpdfstring{$L^2$}{L2} bound)}\label{subsec:discrete_second_moment_bdd}

\begin{proof}[Proof of Lemma \ref{lem:discrete_second_moment_bdd}]
Let $z,z' \in A \cap D_N$. Assume for now that $|z-z'| < N^{-1+\eta}$. By forgetting the good events and the requirement that $z'$ is $a$-thick, we can simply bound
\begin{align*}
\Expect{ \tilde{\Mc}_a^N(\{z\}) \tilde{\Mc}_a^N(\{z'\}) } \leq \frac{(\log N)^{2-2\theta}}{N^{4-2a}} \Prob{ z \in \Tc_N(a) }
\leq C \frac{(\log N)^{1-\theta}}{N^{4-a}}.
\end{align*}
Since $|z-z'|<N^{1-\eta}$, we can further bound
\[
(\log N)^{1-\theta} N^a \leq \log(N) N^a \leq \frac{1}{1-\eta} \log \left( \frac{1}{|z-z'|}  \right) \frac{1}{|z-z'|^{a/(1-\eta)}}.
\]
Since $\eta$ is smaller than $1-a/2$, $a/(1-\eta)$ is smaller than $2$ which guarantees that
\[
\int_{A \times A} \log \left( \frac{1}{|z-z'|}  \right) \frac{1}{|z-z'|^{a/(1-\eta)}} \d z \d z' < \infty.
\]
The remaining of the proof consists in controlling the contribution when $|z-z'| \geq N^{-1+\eta}$. We will denote $|z-z'|=N^{-1+\beta}$ and $\beta$ is therefore at least $\eta$. Let $M>0$ be a large parameter. Let $r \in \{e^{-n}, n \geq 1\} \cap (0,r_0)$ be such that
\[
\frac{|z-z'|}{M} \leq r < e \frac{|z-z'|}{M}.
\]
We choose $M$ large enough to ensure that $r < r_0$, but it will be also important to take $M$ large enough to ensure that we can use Corollary \ref{cor:discrete_crossingPPP2} and Lemma \ref{lem:discrete_crossing_not_PPP}.
For any collection $\Cc$ of discrete loops, define
\[
F(\Cc) := \indic{N_{z,r}^\Cc < b |\log r|^2}.
\]
By only keeping the requirement on the number of crossings of $\D_N(z,er) \setminus \D_N(z,r)$, we can bound
\[
\Expect{ \tilde{\Mc}_a^N( \{z\} ) \tilde{\Mc}_a^N(\{z'\}) } \leq \Expect{ F(\Lc_{D_N}^\theta) \Mc_a^N( \{z\} ) \Mc_a^N(\{z'\}) }.
\]
As in the proof of Lemma \ref{lem:second_moment_bdd}, we will bound $F$ in the spirit of an exponential Markov inequality: define
\[
F_1(\Cc) :=r^{-b} \exp \left( - \frac{1}{|\log r|} N_{z,r}^\Cc \right).
\]
We have $F \leq F_1$. We use Proposition \ref{Prop 2nd mom discr} and the notations therein to bound the expectation of $F_1(\Lc_{D_N}^\theta) \Mc_a^N( \{z\} ) \Mc_a^N(\{z'\})$. We end up with the following expectation to bound:
\begin{equation}\label{eq:proof_exp_to_bound}
\Expect{
F_1 \Big(\Lc_{D_{N}\setminus\{z,z'\}}^\theta
\cup \{ \Xi_{N,a_i,a_i'}^{z,z'} \wedge \Xi_{N,z',a_i}^z
\wedge \Xi_{N,z,a_i'}^{z'} \}_{i = 1}^l
\cup
\{ \Xi_{N,z',\tilde{a}_i}^z, i\geq 1 \} \cup
\{ \Xi_{N,z,\tilde{a}_i'}^{z'}, i\geq 1 \} \Big)
}.
\end{equation}
This expectation does not increase when one forgets
$\Lc_{D_N \setminus \{z,z'\}}^\theta$ above and we bound it by
\begin{align*}
&r^{-b}
\prod_{i=1}^l \Expect{ \exp \left( - \frac{1}{|\log r|} N_{z,r}^{\Xi_{N,a_i,a_i'}^{z,z'}} \right) } \\
& \times \Expect{ \Expect{ \left. \exp \left( - \frac{1}{|\log r|} \left( \sum_{i=1}^l N_{z,r}^{\Xi_{N,z',a_i}^z} + \sum_{i \geq 1}  N_{z,r}^{\Xi_{N,z',\tilde{a}_i}^z} \right) \right) \right\vert \tilde{a_i}, i \geq 1 }
} \\
& \times \left( z \leftrightarrow z' \right) ,
\end{align*}
where in the above, we wrote informally that the last line corresponds to the second line with the processes of excursions around $z$ replaced by the corresponding processes of excursions around $z'$.
By superposition property of Poisson point processes and because $\sum_{i \geq 1} \tilde{a}_i = \tilde{\thk}$ and $\sum_{i=1}^l a_i = \thk$,
\[
\bigcup_{i=1}^l \Xi_{N,z',a_i}^z \cup \bigcup_{i \geq 1} \Xi_{N,z',\tilde{a_i}}^z \overset{\mathrm{(d)}}{=} \Xi_{N,z',\thk + \tilde{\thk}}^z
\]
and a similar result for $z'$.  By Corollary \ref{cor:discrete_crossingPPP2} and by taking $M$ large enough (depending on $\eta$), the expectation in the second line is bounded by
\[
\exp \left( - \frac{\thk + \tilde{\thk}}{(2-\beta)(2 + \eta^2)} (1+o(1)) |\log r| \right)
\]
The expectation in the third line can be bounded by the same quantity with $\thk + \tilde{\thk}$ replaced by $\thk' + \tilde{\thk}'$ (see \eqref{eq:cor_discrete_crossingPPP2b}). Lemma \ref{lem:discrete_crossing_not_PPP} allows us to bound the expectation in the first line by
\[
\prod_{i=1}^l \frac{\Bs \left( (2\pi)^2 a_ia_i' \widetilde{G}_{D_N}(z,z')^2 (1+o(1)) \left( \frac{2-\beta+\eta^2}{2+\eta^2} \right)^{2} \right)}{\Bs \left( (2\pi)^2 a_ia_i' \widetilde{G}_{D_N}(z,z')^2 \right)}.
\]
To wrap things up, we have obtained that \eqref{eq:proof_exp_to_bound} is at most
\begin{align*}
&r^{-b} \prod_{i=1}^l \frac{\Bs \left( (2\pi)^2 a_ia_i' \widetilde{G}_{D_N}(z,z')^2 (1+o(1)) \left( \frac{2-\beta+\eta^2}{2+\eta^2} \right)^{2} \right)}{\Bs \left( (2\pi)^2 a_ia_i' \widetilde{G}_{D_N}(z,z')^2 \right)} \exp \left( - \frac{\thk + \tilde{\thk} + \thk' + \tilde{\thk}'}{(2-\beta)(2+\eta^2)} (1+o(1)) |\log r| \right)
\end{align*}
Plugging this into Proposition \ref{Prop 2nd mom discr} and using the function $\hat{\Hs}$ defined in \eqref{eq:def_H_hat}, we obtain that
\begin{align*}
\Expect{ \tilde{\Mc}_a^N( \{z\} ) \tilde{\Mc}_a^N(\{z'\}) }
\leq \dfrac{q_{N,z'}(z)^{\theta}q_{N,z}(z')^{\theta} (\log N)^{2}}
{N^{4-2a}\Gamma(\theta)^{2}}
e^{-\theta J_{N}(z,z')}r^{-b}
\hat{\Hs}_a(\lambda, \lambda', v)
\end{align*}
where
\[
\lambda = \log N - \log \CR_{N,z'}(z,D_N) - \log c_0 + \frac{1}{(2-\beta)(2+\eta^2)} (1+o(1)) |\log r|,
\]
\[
\lambda' = \log N - \log \CR_{N,z}(z',D_N) - \log c_0 + \frac{1}{(2-\beta)(2+\eta^2)} (1+o(1)) |\log r|,
\]
and
\[
v = (2\pi)^2 \widetilde{G}_{D_N}(z,z')^2 (1+o(1)) \left(\frac{2-\beta + \eta^2}{2+\eta^2}\right)^{2}.
\]
Since $J_N(z,z')$ \eqref{eq:JNzz'} is nonnegative and $q_{N,z}(z')$ \eqref{eq:qNz} is bounded from above (this follows from \eqref{eq:Green_DN-z} and Lemma \ref{lem:Green_discrete}),
we further bound
\begin{equation}
\label{eq:proof_exp_to_bound2}
\Expect{ \tilde{\Mc}_a^N( \{z\} ) \tilde{\Mc}_a^N(\{z'\}) }
\leq C N^{-4+2a} (\log N)^2 r^{-b}
\hat{\Hs}_a(\lambda, \lambda', v) ,
\end{equation}
and it remains to estimate $\hat{\Hs}_a(\lambda, \lambda', v)$.
We have
\begin{align*}
\lambda & = \frac{(\log N)^2}{2\pi G_{D_N \setminus \{z'\}}(z,z)} + \frac{1}{(2-\beta)(2+\eta^2)} (1+o(1)) |\log r| \\
& =  \frac{1}{\beta(2-\beta)} \log N + (1+o(1)) \frac{1}{(2-\beta)(2+\eta^2)} |\log r|
= (1+o(1)) \lambda' ,
\end{align*}
and
\begin{align*}
\sqrt{v} = (1+o(1)) \frac{2-\beta + \eta^2}{(2+\eta^2)\beta (2-\beta)} |\log r|.
\end{align*}
We see that $\lambda - \sqrt{v}$ is always of order $\log N$. In particular, $\lambda > \sqrt{v} + 1$ so that we can use \eqref{eq:lem_H_massless_bound} and bound
\begin{align*}
\hat{\Hs}_a(\lambda,\lambda',v)
\leq C v^{1/4-\theta/2} \frac{1}{(\lambda - \sqrt{v})(\lambda' - \sqrt{v})} e^{(2 \sqrt{v} - \lambda - \lambda')a}
\leq \frac{C}{(\log N)^2} r^{o(1)} e^{(2 \sqrt{v} - \lambda - \lambda')a}.
\end{align*}
Coming back to \eqref{eq:proof_exp_to_bound2}, we have obtained that
\[
\Expect{ \tilde{\Mc}_a^N( \{z\} ) \tilde{\Mc}_a^N(\{z'\}) } \leq N^{-4} r^{-b+o(1)} \exp \left( a (2 \sqrt{v} - \lambda - \lambda' + 2 \log N ) \right) .
\]
An elementary computation shows that
\begin{align*}
\sqrt{v} - \lambda + \log N & = (1+o(1)) \frac{1}{\beta(2-\beta)} \left( \frac{2-2\beta + \eta^2}{2+\eta^2} - (1-\beta) \right) |\log r| \\
& \leq (1+o(1)) \frac{\eta^2}{\beta (2-\beta)} |\log r| \leq (1+o(1)) \eta |\log r| ,
\end{align*}
where we use the fact that $\beta \in [\eta, 1]$ to obtain the last inequality.
By choosing $\eta$ and $b-a$ small enough, we can therefore ensure that
\[
b |\log r| + a (2 \sqrt{v} - \lambda - \lambda' + 2 \log N ) \leq c |\log r|
\]
for some constant $c$ smaller than 2.
To conclude, we have proven that
\[
\Expect{ \tilde{\Mc}_a^N( \{z\} ) \tilde{\Mc}_a^N(\{z'\}) } \leq C N^{-4} |z-z'|^{-c}
\]
for some $c<2$. This provides an integrable domination as stated in \eqref{eq:lem_discrete_second_moment_bdd_no_mass}.

The proof of \eqref{eq:lem_discrete_second_moment_bdd_with_mass} is very similar. Note that we use \eqref{eq:lem_H_massive_bound} instead of \eqref{eq:lem_H_massless_bound} and, as in the proof of Lemma \ref{lem:second_moment_bdd} (specifically \eqref{eq:proof_FKG}), we use FKG-inequality for Poisson point processes (see \cite[Lemma 2.1]{Janson84}) in order to decouple, on the one hand, the killing associated to the mass and, on the other hand, the negative exponential of the number of crossings. We do not give more details.
\end{proof}

\subsection{Proof of Lemma \ref{lem:discrete_second_moment_convergence} (convergence)}\label{subsec:discrete_second_moment_convergence}

In this section, we assume that the parameter $b$, used in the definitions \eqref{eq:def_good_discrete1} and \eqref{eq:def_good_discrete2} of the good events, is close enough to $a$ so that the conclusions of Lemma \ref{lem:discrete_second_moment_bdd} hold.
By developing the product, we have
\begin{align*}
& \Expect{ \left( \tilde{\Mc}_a^N(A) - \frac{2^\theta}{(\log K)^\theta} \tilde{\Mc}_a^{N,K}(A) \right)^2 } \\
& = \int_{A \times A} N^4 \Expect{ \tilde{\Mc}_a^N(z)\left( \tilde{\Mc}_a^N(z') - \frac{2^\theta}{(\log K)^\theta} \tilde{\Mc}_a^{N,K}(z') \right) } \d z \d z' \\
& + \int_{A \times A} N^4 \Expect{\frac{2^\theta}{(\log K)^\theta} \tilde{\Mc}_a^{N,K}(z)\left(\frac{2^\theta}{(\log K)^\theta} \tilde{\Mc}_a^{N,K}(z') - \tilde{\Mc}_a^N(z') \right) } \d z \d z'.
\end{align*}
Lemma \ref{lem:discrete_second_moment_bdd} provides the domination we need in order to apply dominated convergence theorem and it only remains to show that for \emph{fixed} distinct points $z, z' \in A$,
\begin{equation}
\label{eq:proof_discrete_convergence1}
\limsup_{K \to \infty} \limsup_{N \to \infty} N^4 \Expect{ \tilde{\Mc}_a^N(z)\left( \tilde{\Mc}_a^N(z') - \frac{2^\theta}{(\log K)^\theta} \tilde{\Mc}_a^{N,K}(z') \right) } \leq 0
\end{equation}
and
\begin{equation}
\label{eq:proof_discrete_convergence2}
\limsup_{K \to \infty} \limsup_{N \to \infty} N^4 \Expect{\frac{2^\theta}{(\log K)^\theta} \tilde{\Mc}_a^{N,K}(z)\left(\frac{2^\theta}{(\log K)^\theta} \tilde{\Mc}_a^{N,K}(z') - \tilde{\Mc}_a^N(z') \right) } \leq 0.
\end{equation}
We emphasise that, since $z$ and $z'$ are fixed points of the continuous set $A$, they are at a macroscopic distance from each other.
We will sketch the proof of \eqref{eq:proof_discrete_convergence1}. Since the proof of \eqref{eq:proof_discrete_convergence2} is very similar, we will omit it.
Let $r_1>0$ be much smaller than $|z-z'| \vee r_0$ and consider the good events $\Gc_N'(z)$ and $\Gc_{N,K}'(z)$ defined in the same way as $\Gc_N(z)$ and $\Gc_{N,K}(z)$ (see \eqref{eq:def_good_discrete1} and \eqref{eq:def_good_discrete2}) except that the restriction on the number of crossings of annuli is only on radii $r \in (r_1, r_0)$ instead of $(N^{-1+\eta}, r_0)$. The advantage of the event $\Gc_N'(z)$, compared to $\Gc_N(z)$, is that it is a macroscopic event which is well suited to study asymptotics as the mesh size goes to zero (see \eqref{eq:proof_discrete_convergence6}).
Since $z$ and $z'$ are at a distance much larger than $r_1$, one can show that
\begin{align*}
\liminf_{K \to \infty} & \liminf_{N \to \infty} \frac{N^4}{(\log K)^\theta} \Expect{ \tilde{\Mc}_a^N(z) \tilde{\Mc}_a^{N,K}(z') } \\
& \geq - o_{r_1 \to 0}(1) +
\liminf_{K \to \infty} \liminf_{N \to \infty} \frac{N^4}{(\log K)^\theta} \Expect{ \Mc_a^N(z) \Mc_a^{N,K}(z') \mathbf{1}_{\Gc_N'(z) \cap \Gc_{N,K}'(z')} } ,
\end{align*}
where $o_{r_1 \to 0}(1) \to 0$ as $r_1 \to 0$ and may depend on $z, z', a, b, \eta, r_0$. This estimate is in the same spirit as Lemma \ref{lem:discrete_first_moment_event} and we omit the details. We can therefore bound the left hand side of \eqref{eq:proof_discrete_convergence1} by
\begin{equation}
\label{eq:proof_discrete_convergence7}
o_{r_1 \to 0}(1) +
\limsup_{K \to \infty} \limsup_{N \to \infty} N^4 \E \Big[ \Mc_a^N(z) \mathbf{1}_{\Gc_N'(z)} \Big( \Mc_a^N(z') \mathbf{1}_{\Gc_N'(z')} - \frac{2^\theta}{(\log K)^\theta} \Mc_a^{N,K}(z') \mathbf{1}_{\Gc_{N,K}'(z')} \Big) \Big].
\end{equation}
The rest of the proof is dedicated to showing that the second term above vanishes. Letting $r_1 \to 0$ will conclude the proof of \eqref{eq:proof_discrete_convergence1}.

Proposition \ref{Prop 2nd mom discr} gives an exact expression for the expectation in \eqref{eq:proof_discrete_convergence7}. We use the notations therein that we recall for the reader's convenience. The loops visiting $z$ are divided into two collections of loops: the ones that also visit $z'$ and the ones that do not. $l \geq 0$ corresponds to the number of loops in the first collections and $a_i$, $i = 1 \dots l$, are the thicknesses at $z$ of each individual loop in that collection. $\tilde{a_i}$, $i \geq 1$, are the thicknesses at $z$ of the loops which visit $z$ but not $z'$. Finally, $\rho = \sum_{i=1}^l a_i$ and $\tilde{\rho} = \sum_{i \geq 1} \tilde{a}_i$ are the overall thicknesses of the two above sets of loops. Similar notations are used for the point $z'$.
We define $E_N(a_i, a_i', i=1 \dots l, \tilde{\rho}, \tilde{\rho}')$ the event that for all $r \in \{e^{-n}, n \geq 1\} \cap (r_1, r_0)$ and $w \in \{z,z'\}$, the number $N_{w,r}^\Cc$ of discrete crossings in the collection
\begin{align*}
\Cc := \Lc_{D_N \setminus \{z,z'\}}^\theta \cup \{ \Xi_{N,a_i,a_i'}^{z,z'} \wedge \Xi_{N,z',a_i}^z \wedge \Xi_{N,z,a_i'}^{z'} \}_{i=1 \dots l} \cup \{ \Xi_{N,z',\tilde{a}_i}^z \}_{i \geq 1} \cup \{ \Xi_{N,z,\tilde{a}_i'}^{z'} \}_{i \geq 1}
\end{align*}
is at most $b (\log r)^2$. We also define $p_N(a_i, a_i', i=1 \dots l, \tilde{\rho}, \tilde{\rho}')$ the probability of the event $E_N(a_i, a_i', i=1 \dots l, \tilde{\rho}, \tilde{\rho}')$. Note that, by superposition property of Poisson point processes, this probability only depends on the $\tilde{a}_i$ via their sum $\sum \tilde{a}_i = \tilde{\rho}$.
When $l=0$, this probability degenerates to the probability $p_N'( \tilde{\rho} , \tilde{\rho}', 0 )$ where the restriction concerns the number of crossings of $\Lc_{D_N \setminus \{z,z'\}}^\theta \cup \{ \Xi_{N,z',\tilde{a}_i}^z \}_{i \geq 1} \cup \{ \Xi_{N,z,\tilde{a}_i'}^{z'} \}_{i \geq 1} $. The notation $p_N'(\tilde{\rho}, \tilde{\rho}', 0)$ is justified by the fact that it corresponds to the case $k=0$ of the probability $p_N'(\tilde{\rho}, \tilde{\rho}', k)$ that will be defined in \eqref{eq:def_pN'} below.
By Proposition \ref{Prop 2nd mom discr}, the expectation $\Expect{ \Mc_a^N(z) \mathbf{1}_{\Gc_N'(z)} \Mc_a^N(z') \mathbf{1}_{\Gc_N'(z')} }$ is then equal to \eqref{eq:prop_exact1} where the expectation of the function $F$ has to be replaced by $p_N(a_i, a_i', i=1 \dots l, \tilde{\rho}, \tilde{\rho}')$.
In the display below, we develop this last probability according to the number $2 k_i$ of trajectories that were used to form the $i$-th loop $\Xi_{N,a_i,a_i'}^{z,z'}$.
By superposition of Poisson point processes and by definition of $\Xi_{N,a_i,a_i'}^{z,z'}$ (see \eqref{eq:discrete_def_PPPzz'}), we can rewrite $p_N(a_i, a_i', i=1 \dots l, \tilde{\rho}, \tilde{\rho}')$ as
\begin{equation}
\label{eq:proof_discrete_convergence3}
\prod_{i=1}^l \frac{1}{\Bs \left( (2\pi)^2 a_i a_i' \tilde{G}_{D_N}(z,z')^2 \right) } \sum_{k \geq l} \left( 2 \pi \tilde{G}_{D_N}(z,z') \right)^{2k} \sum_{\substack{k_1, \dots, k_l \geq 1 \\ k_1 + \dots + k_l = k}} \prod_{i=1}^l \frac{(a_i a_i')^{k_i}}{k_i ! (k_i - 1)!} p_N'(\rho + \tilde{\rho}, \rho' + \tilde{\rho}', k) ,
\end{equation}
where
\begin{align}
\label{eq:def_pN'}
p_N'(\rho + \tilde{\rho}, \rho' + \tilde{\rho}', k)
& := \P \Big(
\forall r \in \{e^{-n}, n \geq 1\} \cap (r_1, r_0), \forall w \in \{z, z'\}, \\
& N_{w,r}^{\Lc_{D_N \setminus \{z,z'\} }^\theta} + \sum_{i=1}^{2k} N_{w,r}^{\wp_i} + N_{w,r}^{\Xi_{N,z',\thk + \tilde{\thk}}^z} + N_{w,r}^{\Xi_{N,z,\thk' + \tilde{\thk}'}^{z'}} \leq b (\log r)^2 \Big) ,
\nonumber
\end{align}
and where $\wp_i, i = 1 \dots 2k$, are i.i.d. trajectories with common law $\tilde{\mu}_{D_N}^{z,z'} / \tilde{G}_{D_N}(z,z')$. When one plugs this in \eqref{eq:prop_exact1}, the products of the functions $\Bs$ cancel out and, by using the notations $\lambda, \lambda'$ and $v$ as in \eqref{eq:def_lambda_v}, we deduce that
\begin{align}
\label{eq:proof_discrete_convergence5}
& N^4 \Expect{ \Mc_a^N(z) \mathbf{1}_{\Gc_N'(z)} \Mc_a^N(z') \mathbf{1}_{\Gc_N'(z')} }
= \dfrac{q_{N,z'}(z)^{\theta}q_{N,z}(z')^{\theta} (\log N)^{2}}
{\Gamma(\theta)^{2}}
e^{-\theta J_{N}(z,z')} N^{2a}  \\
\nonumber
&  \quad \quad \times \int_{\substack{\thk, \tilde{\thk}>0 \\ \thk + \tilde{\thk} \geq a}} \d \thk \d \tilde{\thk} ~e^{-\lambda(\thk + \tilde{\thk})} \tilde{\thk}^{\theta-1}
\int_{\substack{\thk', \tilde{\thk}'>0 \\ \thk' + \tilde{\thk}' \geq a}} \d \thk' \d \tilde{\thk}' e^{-\lambda'(\thk' + \tilde{\thk}')} ( \tilde{\thk}' )^{\theta-1} \\
\nonumber
& \quad \quad \times
\sum_{l \geq 1} \frac{\theta^l}{l!} \int_{\substack{\mathsf{a} \in E(\thk,l) \\ \mathsf{a}' \in E(\thk',l)}} \d \mathsf{a} \d \mathsf{a}' \sum_{k \geq l} v^k p_N'(\rho + \tilde{\rho}, \rho' + \tilde{\rho}', k) \sum_{\substack{k_1, \dots, k_l \geq 1 \\ k_1 + \dots + k_l = k}} \prod_{i=1}^l \frac{(a_i a_i')^{k_i-1}}{k_i ! (k_i - 1)!}  \\
& \quad \quad +\dfrac{q_{N,z'}(z)^{\theta}q_{N,z}(z')^{\theta} (\log N)^{2}}
{\Gamma(\theta)^{2}}
e^{-\theta J_{N}(z,z')} N^{2a}
\int_{(a,\infty)^2} e^{-\lambda t - \lambda' t'}
p_N'(t,t', 0) (tt')^{\theta-1} \d t \d t', \nonumber
\end{align}
where the last term corresponds to the case $l=0$.
%\begin{align*}
%\end{align*}
By Lemma \ref{lem:discrete_asymp_exact}, the multiplicative factor in front of the first integral in \eqref{eq:proof_discrete_convergence5} is asymptotic to $(\log N)^2 N^{2a} \Gamma(\theta)^{-2}$.
\eqref{eq:rmk_massless} gives a simple expression for the remaining part of the right hand side of \eqref{eq:proof_discrete_convergence5} and
\begin{align*}
& N^4 \Expect{ \Mc_a^N(z) \mathbf{1}_{\Gc_N'(z)} \Mc_a^N(z') \mathbf{1}_{\Gc_N'(z')} } \\
& \sim \dfrac{(\log N)^{2} N^{2a}}
{\Gamma(\theta)^{2}}
 \sum_{k \geq 0} \frac{v^k}{\theta^{(k)} k!} \int_{(a,\infty)^2} e^{-\lambda t - \lambda' t'}
p_N'(t,t', k) (tt')^{\theta+k-1} \d t \d t'.
\end{align*}

We now argue that for any fixed $k \geq 0$, $t, t' \geq a$,
\begin{equation}
\label{eq:proof_discrete_convergence6}
p_N'(t,t',k) \xrightarrow[N \to \infty]{} p'(t,t',k) := \Prob{ \forall r \in \{ e^{-n}, n \geq 1 \} \cap (r_1,r_0), \forall w \in \{z,z'\}, N_{w,r}^{\Cc} \leq b (\log r)^2 } ,
\end{equation}
where
\[
\Cc := \Lc_D^\theta \cup \{ \wp_i \}_{i=1 \dots 2k} \cup \{ \Xi_{t}^z, \Xi_{t'}^{z'} \}
\]
with $\wp_i, i =1 \dots 2k$, i.i.d. trajectories distributed according to $\mu_D^{z,z'} / G_D(z,z')$ and the above collections of trajectories are all independent. This follows from 1) the convergence of $\tilde{\mu}_{D_N}^{z,z'} / \tilde{G}_{D_N}(z,z')$ towards $\mu_D^{z,z'} / G_D(z,z')$ established in Proposition \ref{Prop Approx mu 2}, 2) the convergence of $\Lc_{D_N \setminus \{z,z'\} }^\theta$ towards $\Lc_D^\theta$ \cite{LawlerFerreras07RWLoopSoup}, and 3) the convergence of $\tilde{\mu}_{D_N \setminus \{z'\}}^{z,z}$ towards $\mu_D^{z,z}$ stated in Corollary \ref{Cor Approx mu 5}.
It is then a simple verification that the integral concentrates around $t=t'=a$ as $N \to \infty$ (recall that $\lambda$ and $\lambda'$ are defined in \eqref{eq:def_lambda_v} and go to infinity) and
\begin{align}
\label{eq:proof_discrete_convergence4}
& N^4 \Expect{ \Mc_a^N(z) \mathbf{1}_{\Gc_N'(z)} \Mc_a^N(z') \mathbf{1}_{\Gc_N'(z')} }
\sim \frac1{\Gamma(\theta)^2} \dfrac{(\log N)^{2}}{\lambda \lambda'} N^{2a} e^{-a(\lambda + \lambda')}
\sum_{k \geq 0} \frac{v^k a^{2\theta+2k-2}}{\theta^{(k)} k!}
p'(a,a, k) \\
\nonumber
& \sim \frac{(c_0)^{2a}}{\Gamma(\theta)^2} \CR(z,D)^a \CR(z',D)^a \sum_{k \geq 0} \frac{(2\pi G_D(z,z'))^{2k} a^{2\theta+2k-2}}{\theta^{(k)} k!}
p'(a,a, k).
\end{align}

For the mixed case, the situation is slightly different. Because of the killing, the expectation of $\Mc_a^N(z) \mathbf{1}_{\Gc_N'(z)} \Mc_a^{N,K}(z') \mathbf{1}_{\Gc_{N,K}'(z')}$ is expressed in terms of (see Proposition \ref{Prop 2nd mom discr})
\[
\Expect{ \prod_{i=1}^l
\left( 1 - e^{-K T(\Xi_{N,a_i,a_i'}^{z,z'}
\wedge \Xi_{N,z',a_i}^z
\wedge \Xi_{N,z,a_i'}^{z'} )} \right)
\prod_{i=l+1}^m \left( 1 - e^{-KT(\Xi_{N,z,a_i'}^{z'})} \right) \mathbf{1}_{E_N(a_i, a_i', i=1 \dots l, \tilde{\rho}, \tilde{\rho}')} }.
\]
Since the points $z$ and $z'$ are macroscopically far apart, the durations of the loops $\Xi_{N,a_i,a_i'}^{z,z'}$, $i=1 \dots l$, are macroscopic and one can show that the first product is very close to 1. With an argument very similar to what was done in Corollary \ref{cor:crossingPPP1}, one can show that the expectation of the second product times the indicator function is well approximated by
\begin{align*}
& \Expect{ \prod_{i=l+1}^m \left( 1 - e^{-KT(\Xi_{N,z,a_i'}^{z'})} \right) } \Prob{ E_N(a_i, a_i', i=1 \dots l, \tilde{\rho}, \tilde{\rho}') } \\
& = \prod_{i=l+1}^m \left( 1 - e^{-a_i' C_{N,K,z}(z') } \right) p_N(a_i, a_i', i=1 \dots l, \tilde{\rho}, \tilde{\rho}')
\end{align*}
where $C_{N,K,z}(z')$ is defined in \eqref{eq:discrete_CK2}.
Using \eqref{eq:prop_exact2} together with \eqref{eq:proof_discrete_convergence3}, we obtain that the expectation $N^4 \Expect{ \Mc_a^N(z) \mathbf{1}_{\Gc_N'(z)} \Mc_a^{N,K}(z') \mathbf{1}_{\Gc_{N,K}'(z')} }$ has the same asymptotics as
\begin{align*}
& \dfrac{q_{N,z'}(z)^{\theta} (\log N)^{2}}
{\Gamma(\theta)}
e^{-\theta(J_{N,K,z}(z')+ J_{N,K}(z,z'))} N^{2a} \\
\times \Bigg( &
\int_{\substack{\thk, \tilde{\thk}>0 \\ \thk + \tilde{\thk} \geq a}} \d \thk \d \tilde{\thk} ~e^{-\lambda(\thk + \tilde{\thk})} \tilde{\thk}^{\theta-1}
\int_a^\infty \d \thk' e^{-\lambda' \thk' }
\sum_{\substack{m \geq 1 \\ 1 \leq l \leq m}} \frac{\theta^m}{(m-l)!l!} \int_{\substack{\mathsf{a} \in E(\thk,l) \\ \mathsf{a}' \in E(\thk',m)}} \d \mathsf{a} \d \mathsf{a}' \\
& \sum_{k \geq l} v^k p_N'(\rho + \tilde{\rho}, \rho', k) \sum_{\substack{k_1, \dots, k_l \geq 1 \\ k_1 + \dots + k_l = k}} \prod_{i=1}^l \frac{(a_i a_i')^{k_i-1}}{k_i ! (k_i - 1)!} \prod_{i=l+1}^m \frac{1-e^{-a_i' C_{N,K,z}(z')}}{a_i'} \\
& + \int_a^\infty \d \tilde{\rho} \tilde{\rho}^{\theta-1} e^{-\lambda \tilde{\rho}} \int_a^\infty \d \rho' e^{-\lambda' \rho'} p_N'(\tilde{\rho}, \rho', 0) \sum_{m \geq 1} \frac{\theta^m}{m!} \int_{\mathsf{a}' \in E(\rho',m)} \d \mathsf{a}' \prod_{i=1}^m \frac{1-e^{-a_i' C_{N,K,z}(z')}}{a_i'}
\Bigg).
\end{align*}
The second term of the sum in parenthesis corresponds to the case $l=0$.
The front factor is asymptotic to $\Gamma(\theta)^{-1} (\log N)^2 N^{2a}$, whereas the first term in parenthesis can be simplified thanks to \eqref{eq:rmk_mixed} and the second term in parenthesis can be directly expressed in terms of the function $\Fs$ (see \eqref{eq:def_Fs}). Overall, we obtain that $N^4 \Expect{ \Mc_a^N(z) \mathbf{1}_{\Gc_N'(z)} \Mc_a^{N,K}(z') \mathbf{1}_{\Gc_{N,K}'(z')} }$ has the same asymptotics as
\begin{align*}
\dfrac{(\log N)^{2}}
{\Gamma(\theta)} N^{2a}
\Bigg( &
\sum_{k \geq 1} v^k \int_{(a,\infty)^2} \d t ~ \d t' ~e^{-\lambda t} e^{-\lambda' t' } \frac{t^{\theta+k-1}}{k!(k-1)!} p_N'(t,t',k) \\
& \times \Big( \int_0^{t'} \d \rho' {\rho'}^{k-1} \frac{\Fs(C_{N,K,z}(z') (t'-  \rho'))}{t' - \rho'} + {t'}^{k-1} \Big) \\
& + \int_{(a,\infty)^2} \d t ~ \d t' ~e^{-\lambda t} e^{-\lambda' t' } t^{\theta-1} p_N'(t,t',0) \frac{\Fs \left( C_{N,K,z}(z') t' \right)}{t'}
\Bigg).
\end{align*}
By dominated convergence theorem, Lemma \ref{L:asymptoticsF_C_K} and the convergence \eqref{eq:discrete_CK3} of $C_{N,K,z}(z')$ towards $C_K(z')$, we have
\[
\lim_{K \to \infty} \lim_{N \to \infty} \frac{2^\theta}{(\log K)^\theta} \int_0^{t'} \d \rho' {\rho'}^{k-1} \frac{\Fs(C_{N,K,z}(z') (t'-  \rho'))}{t' - \rho'} = \frac{1}{\Gamma(\theta)} \int_0^{t'} \d \rho' {\rho'}^{k-1} (t'-\rho')^{\theta -1}.
\]
The right hand side term can be computed thanks to \eqref{eq:beta_function} and is equal to $\frac{(k-1)!}{\Gamma(\theta)\theta^{(k)}} {t'}^{k+\theta-1}$.
From this and the asymptotic behaviour \eqref{eq:proof_discrete_convergence6} of $p_N'(t,t',k)$, one can easily deduce that the asymptotics of $2^\theta (\log K)^{-\theta} N^4 \Expect{ \Mc_a^N(z) \mathbf{1}_{\Gc_N'(z)} \Mc_a^{N,K}(z') \mathbf{1}_{\Gc_{N,K}'(z')} }$ is given by
\begin{align*}
& \frac{1}{\Gamma(\theta)^2}
\frac{(\log N)^2}{\lambda \lambda'} N^{2a} e^{-a(\lambda +\lambda')} \sum_{k \geq 0} \frac{v^k a^{2\theta+2k-2}}{k! \theta^{(k)}} p'(a,a,k) \\
& \sim \frac{(c_0)^{2a}}{\Gamma(\theta)^2} \CR(z,D)^a \CR(z',D)^a \sum_{k \geq 0} \frac{(2\pi G_D(z,z'))^{2k} a^{2\theta+2k-2}}{\theta^{(k)} k!}
p'(a,a, k).
\end{align*}
Since we obtain the same limit as in \eqref{eq:proof_discrete_convergence4}, it concludes the proof of \eqref{eq:proof_discrete_convergence1}. The proof of \eqref{eq:proof_discrete_convergence2} follows from a very similar line of argument. This concludes the proof of Lemma \ref{lem:discrete_second_moment_convergence}.

\section{Scaling limit of massive random walk loop soup thick points}
\label{sec: massivechaoslimit}

The goal of this section is to prove Proposition \ref{prop:discrete_vs_continuum}. As already alluded to, it relies heavily on an analogous statement from \cite{jegoRW} about thick points of finitely many random walk trajectories running from internal to boundary points that we state now.

Let $(D_i, x_i, z_i), i \in I$, be a finite collection of bounded simply connected domains $D_i \subset \C$ with internal points $x_i \in D_i$ and boundary points $z_i \in \partial D_i$. Assume that the boundary points $z_i$ are pairwise distinct ($i \neq j \implies z_i \neq z_j$) and that for all $i \in I$, the boundary of $D_i$ is locally analytic near $z_i$ (below we will apply this result to boundaries that are locally flat at $z_i$). Let $\wp_i, i \in I$, be independent Brownian trajectories that start at $x_i$ and are conditioned to exit $D_i$ at $z_i$, i.e. $\wp_i \sim \mu_{D_i}^{x_i,z_i} / H_{D_i}(x_i,z_i)$; see \eqref{Eq mu D z w boundary}. Let $D_{i,N}$ be a discrete approximation of $D_i$ by a portion of the square lattice with mesh size $1/N$ as in \eqref{eq:DN} (take $x_i$ as a reference point instead of the origin) and let $x_{i,N} \in D_{i,N}$ and $z_{i,N} \in \partial D_{i,N}$ be such that $x_{i,N} \to x_i$ and $z_{i,N} \to z_i$ as $N \to \infty$. Let $\wp_{i,N}, i \in I,$ be independent random walk trajectories starting at $x_{i,N}$ and conditioned to exit $D_{i,N}$ at $z_{i,N}$.

For all subset $J$ of the set of indices $I$, let $\Mc_a^{\cap_{j \in J} \wp_{j,N}}$ be the measure supported on $a$-thick points coming from the interaction of all the trajectories $\wp_{j,N}, j \in J$: for all Borel set $A \subset \C$,
\begin{equation}\label{eq:discrete_measure_thick_cap}
\Mc_a^{\cap_{j \in J} \wp_{j,N}}(A) := \frac{\log N}{N^{2-a}}
\sum_{z \in \cap_{j \in D_{j,N}}} \indic{z \in A} \indic{ \sum_{j \in J} \ell_z( \wp_{j,N} ) \geq \frac{a}{2\pi} (\log N)^2 } \indic{ \forall j \in J, \ell_z(\wp_{j,N}) > 0}.
\end{equation}
Recall also that $\Mc_a^{ \cap_{j \in J} \wp_j }$ denotes the Brownian chaos associated to $\wp_j, j \in J$, where each trajectory is required to contribute to the thickness; see Section \ref{sec:preliminaries_BMC}. Of course, when $\cap_{j \in J} D_j = \varnothing$, these measures degenerate to zero. \cite{jegoRW} shows that:

\begin{theorem}[Theorem 5.1 of \cite{jegoRW}]\label{th:jegoRW}
As $N \to \infty$, the joint convergence
\[
\left( \Mc_a^{\cap_{j \in J} \wp_{j,N}}, J \subset I, \wp_{i,N}, i \in I \right) \to
\left( c_0^a \Mc_a^{\cap_{j \in J} \wp_j}, J \subset I, \wp_i, i \in I \right)
\]
holds in distribution where the topology associated to $\Mc_a^{\cap_{j \in J} \wp_{j,N}}$ is the topology of vague convergence on $\cap_{j \in J} D_j$ and the topology associated to $\wp_{i,N}$ is the one induced by $d_{\rm paths}$ \eqref{Eq dist paths}.
\end{theorem}

To use this result, we will first need to describe a decomposition of the loop soup similar to the one described in Lemma \ref{lem:loops_height} that holds in the discrete setting.

\subsection{Decomposition of random walk loop soup}\label{sec:decomposition_RWLS}

Let $D_{N}\subset\Z_{N}^{2}$ be such that
both $D_{N}$ and $\Z_{N}^{2}\setminus D_{N}$
are non-empty.
Denote
\begin{displaymath}
\mi(D_{N}) := \inf \{ \Im(z): z \in D_{N} \}
\quad \text{and} \quad
\Mi(D_{N}) := \sup \{ \Im(z): z \in D_{N} \}.
\end{displaymath}
Consider the random walk loop soup
$\Lc^{\theta}_{D_{N}}$.
For $\wp\in\Lc^{\theta}_{D_{N}}$,
we will use the same notations
$\mi(\wp)$, $\Mi(\wp)$
\eqref{Eq mi} and $h(\wp)$ \eqref{Eq h}
as in the continuum case.
Unlike in the continuum case,
a loop $\wp\in\Lc^{\theta}_{D_{N}}$
can travel several times
back and forth between
$\R + i \mi(\wp)$ and $\R + i \Mi(\wp)$.
So we will restrict to loops
$\wp\in\Lc^{\theta}_{D_{N}}$ that do this only once in each direction.
We will root such a loop at the first time (for the circular order) it visits $\R + i \mi(\wp)$ after having visited
$\R + i \Mi(\wp)$ (see Figure \ref{fig1} for an illustration in the continuum setting). This time is well defined provided
$\wp$ travels only once back and forth between
$\R + i \mi(\wp)$ and $\R + i \Mi(\wp)$,
and after rerooting it is set to $0$.
We will denote by $z_\bot$ the position of
$\wp$ at this time, as in the continuum case.
We have that
$z_\bot\in \Z_{N} + i \mi(\wp)$.
Note however that in discrete $\wp$ may also visit other
points in $\Z_{N} + i \mi(\wp)$.
Given $\eps>0$,
we will denote
\begin{align}
\label{eq:RWLS_height}
\Lc_{D_{N},\eps}^\theta & := \{ \wp \in \Lc_{D_{N}}^\theta:
h(\wp) \geq \eps
\text{ and }
\\
& \wp
\text{ travels only once back and forth between }
\R + i \mi(\wp)
\text{ and }
\R + i (\mi(\wp)+\lceil\eps\rceil_{N})
\},
\nonumber
\end{align}
where
$\lceil\varepsilon\rceil_{N} :=
N^{-1} \lceil N\eps\rceil$.
Note that in discrete we add this condition of a single
round trip between $\R + i \mi(\wp)$
and $\R + i (\mi(\wp)+\lceil\eps\rceil_{N})$.
Recall that we root the loops $\wp\in \Lc_{D_{N},\eps}^\theta$
at $z_\bot$.
Denote
\begin{displaymath}
\tau_\eps(\wp) := \inf \{ t \in [0,T(\wp)]: \Im (\wp(t)) \geq \mi(\wp) + \lceil\eps\rceil_{N} \}
\end{displaymath}
for $\wp\in \Lc_{D_{N},\eps}^\theta$.
As in the continuum, we decompose the loop into two parts
\begin{equation}
\label{eq:discrete_loops12}
\wp_{\eps,1} := (\wp(t))_{0 \leq t \leq \tau_\eps}
\quad \text{and} \quad
\wp_{\eps,2} := (\wp(t))_{\tau_\eps \leq t \leq T(\wp)}.
\end{equation}
Denote $z_{\eps}:=\wp(\tau_\eps)$. Recall the notations $\H_y$ and $S_{y,y'}$ for upper half planes and horizontal strips \eqref{eq:def_H_S}.

\begin{lemma}
\label{Lem Decomp min discr}
$\# \Lc_{D_{N},\eps}^\theta$ is a Poisson random variable with mean given by
\begin{multline*}
\theta
\dfrac{1}{N}
\sum_{\mi(D_{N})\leq m\leq \Mi(D_{N})}
\dfrac{1}{N}
\sum_{z_{1}\in D_{N}\cap(\R + im)}
\dfrac{1}{N}
\sum_{z_{2}\in D_{N}\cap(\R + i(\mi+\lceil\eps\rceil_{N}))}
\\
\big(
N H_{D_{N} \cap S_{m - N^{-1},m+\lceil\eps\rceil_{N}}}
(z_1,z_2)
\big)
H_{D_{N}\cap \H_m}(z_{2},z_{1}),
\end{multline*}
where
$H_{D_{N} \cap S_{m - N^{-1},m+\lceil\eps\rceil_{N}}}
(z_1,z_2)$
and
$H_{D_{N}\cap \H_m}(z_{2},z_{1})$
are the discrete Poisson kernels \eqref{Eq discr Pk}
in
$D_{N} \cap S_{m - N^{-1},m+\lceil\eps\rceil_{N}}$,
respectively
$D_{N}\cap \H_m$.
Conditionally on $\# \Lc_{D_{N},\eps}^\theta$,
the loops in $\Lc_{D_{N},\eps}^\theta$ are i.i.d.
Moreover, for each
$\wp\in \Lc_{D_{N},\eps}^\theta$,
the joint law of
$(z_\bot, z_{\eps},\wp_{\eps,1},\wp_{\eps,2})$
can be described as follows:
\begin{enumerate}
\item Conditionally on $(z_\bot, z_{\eps})$,
$\wp_{\eps,1}$ and $\wp_{\eps,2}$ are two independent trajectories distributed according to
\begin{equation}
\label{eq:discrete_law_wp_12}
\mu_{D_{N} \cap S_{m - N^{-1},m+\lceil\eps\rceil_{N}}}^{z_\bot,z_\eps}
/
H_{D_{N} \cap S_{m - N^{-1},m+\lceil\eps\rceil_{N}}}
(z_\bot,z_\eps)
\quad \text{and} \quad
\mu_{D_{N} \cap \H_m}^{z_\eps,z_\bot}
/
H_{D_{N} \cap \H_m}(z_\eps,z_\bot )
\end{equation}
respectively,
where $\mu_{D_{N} \cap S_{m - N^{-1},m+\lceil\eps\rceil_{N}}}^{z_\bot,z_\eps}$
and
$\mu_{D_{N} \cap \H_m}^{z_\eps,z_\bot}$
follow the definition \eqref{Eq discr int to bound}.
\item
The joint law of $(z_\bot, z_{\eps})$
is given by: for all $z_1, z_2 \in D_N$, $\P((z_\bot, z_{\eps})=(z_{1},z_{2}))$ is equal to
\begin{align}
\label{eq:discrete_law_zbot_zeps}
\dfrac{1}{Z}
\indic{z_{1},z_{2}\in D_{N},
\Im(z_{2})=\Im(z_{1})+\lceil\eps\rceil_{N}}
\big(
N H_{D_{N} \cap S_{m - N^{-1},m+\lceil\eps\rceil_{N}}}
(z_1,z_2)
\big)
H_{D_{N}\cap \H_m}(z_{2},z_{1}),
\end{align}
with $m=\Im(z_{1})$.
\end{enumerate}
\end{lemma}

\begin{proof}
This is equivalent to saying that the concatenation
$\wp_{1}\wedge\wp_{2}$ under the measure
\begin{equation}
\label{Eq decomp discr 2}
\dfrac{1}{N^{2}}
\sum_{\mi(D_{N})\leq m\leq \Mi(D_{N})}
\sum_{\substack{z_{1}\in D_{N}\cap(\R + im)
\\z_{2}\in D_{N}\cap(\R + i(\mi+\lceil\eps\rceil_{N}))}}
\mu_{D_{N} \cap S_{m - N^{-1},m+\lceil\eps\rceil_{N}}}^{z_{1},z_{2}}
(d\wp_{1})
\mu_{D_{N}\cap \H_m}^{z_{2},z_{1}}(d\wp_{2})
\end{equation}
corresponds, up to rerooting of loops,
to the measure on loops $\mu_{D_{N}}^{\rm loop}$
restricted to the loops $\gamma$
with $h(\wp)\geq \varepsilon$ and
that travel only once back and forth between
$\R + i \mi(\wp)$ and
$\R + i (\mi(\wp)+\lceil\eps\rceil_{N})$.
For this, it is enough to check that the weights of the discrete
skeletons of unrooted loops under this two measures coincide.
Indeed, in both cases, the holding times conditionally on
the discrete skeletons are i.i.d.
exponential r.v.s with mean $\frac{1}{4 N^{2}}$.
Given $k\geq 2$, $k$ even,
the weight of a discrete-time nearest neighbour rooted loop
of length $k$
in $D_{N}$ under $\mu_{D_{N}}^{\rm loop}$ is
$\frac{1}{k}4^{-k}$.
So the weight of the corresponding discrete-time unrooted loop is $4^{-k}$, provided the loop is aperiodic, that is to say its smallest period is $k$.
This is simply because then the unrooted loop corresponds to $k$ different rooted loops.
Moreover, a loop that travels only once back and forth between
$\R + i \mi(\wp)$ and
$\R + i (\mi(\wp)+\lceil\eps\rceil_{N})$
is necessarily aperiodic.
Further, the weight of a possible discrete-time path with
$k_{1}$ jumps under
$\mu_{D_{N} \cap S_{m - N^{-1},m+\lceil\eps\rceil_{N}}}^{z_{1},z_{2}}$
is $N 4^{-k_{1}}$.
Similarly, the weight of a possible discrete-time path with
$k_{2}$ jumps under
$\mu_{D_{N}\cap \H_m}^{z_{2},z_{1}}$
is $N 4^{-k_{2}}$.
Thus, the weight of the couple is
$N^{2} 4^{-(k_{1}+k_{2})}$,
and $k_{1}+k_{2}$ is the length of the loop created by concatenation.
The $N^{2}$ is compensated by the $N^{-2}$ factor in
\eqref{Eq decomp discr 2}.
So the weights of the discrete skeletons coincide.
\end{proof}

We conclude this section with a result about the convergence of the quantities appearing in Lemma \ref{Lem Decomp min discr} towards the quantities appearing in Lemma \ref{lem:loops_height}. In the following result, we assume that $D$ is a bounded simply connected domain and that $(D_N)_N$ is the associated discrete approximations as in \eqref{eq:DN}.

\begin{lemma}\label{lem:number_loops_height}
1.
For all $n \geq 0$, we have
\begin{equation}
\label{eq:lem_number_loops_height}
\lim_{N \to \infty} \Prob{ \# \Lc_{D_N,\eps}^\theta = n } = \Prob{ \# \Lc_{D,\eps}^\theta = n }.
\end{equation}
2.
Let $(z_\bot^N, z_\eps^N)$ and $(z_\bot, z_\eps)$ be distributed according to the laws \eqref{eq:discrete_law_zbot_zeps} and \eqref{eq:continuum_law_zbot_zeps}, respectively. Then
\begin{equation}
\label{eq:lem_convergence_zz}
(z_\bot^N, z_\eps^N)
\xrightarrow[N \to \infty]{(d)}
(z_\bot, z_\eps).
\end{equation}
3.
Let $\wp^N$ and $\wp$ be distributed according to the laws described in Lemmas \ref{Lem Decomp min discr} and \ref{lem:loops_height}, respectively. Then
\begin{equation}
\label{eq:lem_duration_scaling}
T(\wp^N)
\xrightarrow[N \to \infty]{(d)}
T(\wp).
\end{equation}
\end{lemma}

\begin{proof}
\eqref{eq:lem_number_loops_height} and \eqref{eq:lem_convergence_zz} follow from Lemmas \ref{Lem Decomp min discr} and \ref{lem:loops_height} and from the convergence of the discrete (bulk and boundary) Poisson kernels towards their continuum analogues. Alternatively, these two claims follow from the convergence in distribution of $\Lc_{D_N,\eps}^\theta$ towards $\Lc_{D,\eps}^\theta$ for the topology induced by $d_{\mathfrak{L}}$ \eqref{eq:d_frac_L}. This latter fact is a direct consequence of the coupling of \cite{LawlerFerreras07RWLoopSoup} between random walk loop soup and Brownian loop soup. We omit the details.
To prove \eqref{eq:lem_duration_scaling}, one only needs to notice that the law of $T(\wp^N)$ is given by the law of the total duration of $\Lc_{D_N,\eps}^\theta$ conditioned on $\# \Lc_{D_N,\eps}^\theta = 1$. The same holds for the Brownian loop soup. Therefore, \eqref{eq:lem_duration_scaling} follows from the joint convergence of $\# \Lc_{D_N,\eps}^\theta$ and the total duration of $\Lc_{D_N,\eps}^\theta$ which is again a consequence of \cite{LawlerFerreras07RWLoopSoup}.
\end{proof}

\subsection{Proof of Proposition \ref{prop:discrete_vs_continuum}}

We now have all the ingredients for the proof of Proposition \ref{prop:discrete_vs_continuum}.

\begin{proof}[Proof of Proposition \ref{prop:discrete_vs_continuum}]

We will focus on the convergence of the measure $\Mc_a^{N,K}$ towards its continuum analogue $\Mc_a^K$. Indeed, since Theorem \ref{th:jegoRW} also takes care of the joint convergence of the trajectories, it is not difficult to extend our proof to the joint convergence of the measure $\Mc_a^{N,K}$ together with the killed loops $\Lc_{D_N}^\theta(K)$.

Let $\eps >0$. We first restrict $\Lc_{D_N}^\theta(K)$ to the loops with height larger than $\eps$: recall the definition \eqref{eq:RWLS_height} of $\Lc_{D_N, \eps}^\theta$ and recall that loops $\wp$ in $\Lc_{D_N, \eps}^\theta$ are naturally split into two trajectories $\wp_{\eps,1}$ and $\wp_{\eps,2}$ (see \eqref{eq:discrete_loops12}).
The first part $\wp_{\eps,1}$ becomes negligible as $\eps \to 0$. Therefore, we will not loose much by only looking at the second part and we define the following measure: for all Borel set $A$,
\[
\Mc_a^{N,K,\eps}(A) := \frac{\log N}{N^{2-a}} \sum_{z \in D_N} \indic{z \in A} \indic{ \sum_{\wp \in \Lc_{D_N,\eps}^\theta \cap \Lc_{D_N}^\theta(K)} \ell_z(\wp_{\eps,2}) \geq \frac{a}{2\pi} (\log N)^2 }.
\]
This definition is very close to the one without the restriction on the height; see \eqref{eq:def_thick_points_discrete_mass} and \eqref{eq:def_measure_discrete_mass}.
In \eqref{eq:measure_mass_height} we define an analogous measure $\Mc_a^{K,\eps}$ in the continuum. The main part of the proof is to show that for any nondecreasing bounded continuous function $g : [0,\infty) \to \R$ and any nonnegative bounded continuous function $f: D \to [0,\infty)$,
\begin{equation}
\label{eq:proof_discrete_continuum1}
\liminf_{N \to \infty} \Expect{ g \left( \scalar{ \Mc_a^{N,K,\eps}, f } \right) }
\geq
\Expect{ g \left(  c_0^a \scalar{ \Mc_a^{K,\eps}, f } \right) }.
\end{equation}
Let us assume that \eqref{eq:proof_discrete_continuum1} holds and let us explain how Proposition \ref{prop:discrete_vs_continuum} follows.
Firstly, Corollary \ref{cor:discrete_first} shows that
\[
\sup_{N \geq 1} \Expect{ \Mc_a^{N,K}(D) } < \infty
\]
implying tightness of $\left( \Mc_a^{N,K}, N \geq 1 \right)$ for the topology of weak convergence (see e.g. \cite[Lemma 1.2]{kallenberg} for an analogous statement concerning the topology of vague convergence). Let $\Mc_a^{\infty, K}$ be any subsequential limit. By first extracting a subsequence, we can assume without loss of generality that $\left( \Mc_a^{N,K}, N \geq 1 \right)$ converges in distribution towards $\Mc_a^{\infty, K}$. To conclude, we need to show that $\Mc_a^{\infty, K} \overset{(d)}{=} c_0^a \Mc_a^K$. To this end, it is enough to show that, for any nonnegative bounded continuous function $f : D \to [0,\infty)$, $ \scalar{ \Mc_a^{\infty,K}, f } $ and $c_0^a \scalar{ \Mc_a^K, f }$ have the same distribution (see e.g. \cite[Lemma 1.1]{kallenberg} for a similar statement for the topology of vague convergence).
Let $f$ be such a function, $g:[0,\infty) \to \R$ be a bounded nondecreasing function and let $\eps>0$.
By first using the convergence in distribution of $\scalar{\Mc_a^{N,K}, f}$ towards $\scalar{ \Mc_a^{\infty,K}, f}$, then by using monotonicity of $g$ and finally by exploiting \eqref{eq:proof_discrete_continuum1}, we have
\begin{align*}
\Expect{ g \left( \scalar{ \Mc_a^{\infty,K}, f } \right) }
= \lim_{N \to \infty} \Expect{ g \left( \scalar{ \Mc_a^{N,K}, f} \right) }
\geq \liminf_{N \to \infty} \Expect{ g \left( \scalar{ \Mc_a^{N,K,\eps}, f} \right) }
\geq \Expect{ g \left(  c_0^a \scalar{ \Mc_a^{K,\eps}, f } \right) }.
\end{align*}
By definition of $\Mc_a^K$ (see Definition \ref{def:measure_mass}), $\scalar{ \Mc_a^{K,\eps}, f }$ converges a.s. to $\scalar{ \Mc_a^K, f }$ as $\eps \to 0$. Hence
\[
\Expect{ g \left( \scalar{ \Mc_a^{\infty,K}, f } \right) }
\geq \Expect{ g \left(  c_0^a \scalar{ \Mc_a^K, f } \right) }.
\]
Since this is valid for all nondecreasing bounded continuous function $g$, we deduce that $\scalar{ \Mc_a^{\infty,K}, f}$ stochastically dominates $c_0^a \scalar{ \Mc_a^K, f}$. Because their expectations agree (Corollary \ref{cor:discrete_first} and Proposition \ref{prop:first_moment_killed_loops}), they must have the same distribution. This shows the expected convergence $\Mc_a^{N,K} \to c_0^a \Mc_a^K$.

Next, we move on to the proof of \eqref{eq:proof_discrete_continuum1}.
By conditioning on the number of loops in $\Lc_{D_N,\eps}^\theta$ and by Fatou's lemma, we have
\begin{align*}
\liminf_{N \to \infty} \Expect{ g \left( \scalar{ \Mc_a^{N,K,\eps}, f } \right) }
\geq \sum_{n=0}^\infty
\liminf_{N \to \infty} \Prob{ \# \Lc_{D_N,\eps}^\theta = n } \Expect{ g \left( \scalar{ \Mc_a^{N,K,\eps}, f } \right) \vert \# \Lc_{D_N,\eps}^\theta = n }.
\end{align*}
The claim \eqref{eq:lem_number_loops_height} in Lemma \ref{lem:number_loops_height} shows that for all $n \geq 0$, $\Prob{ \# \Lc_{D_N,\eps}^\theta = n }$ converges as $N \to \infty$ to its analogue in the continuum and it remains to show that
\begin{equation}
\label{eq:pf_discrete_continuum1}
\liminf_{N \to \infty} \Expect{ g \left( \scalar{ \Mc_a^{N,K,\eps}, f } \right) \vert \# \Lc_{D_N,\eps}^\theta = n }
\geq \Expect{ g \left( c_0^a \scalar{ \Mc_a^{K,\eps}, f } \right) \vert \# \Lc_{D,\eps}^\theta = n }.
\end{equation}
Fix $n \geq 1$. Let $\wp^{i,N}, i=1 \dots n$, be i.i.d. loops so that $\Lc_{D_N,\eps}^\theta$, conditioned on $\# \Lc_{D_N,\eps}^\theta = n$, has the same distribution as $\{  \wp^{1,N}, \dots, \wp^{n,N} \}$ (see Lemma \ref{Lem Decomp min discr}). We split these loops into two pieces $\wp_{\eps,1}^{i,N}$ and $\wp_{\eps,2}^{i,N}$ as in \eqref{eq:discrete_loops12}. Let $U_i, i=1 \dots n$, be i.i.d. uniform random variables on $[0,1]$ that are independent of the loops above.
By checking which loops are killed (in the next display, $I$ corresponds to the set of indices of killed loops), we can rewrite the expectation on the left hand side of \eqref{eq:pf_discrete_continuum1} as
\begin{align}
\nonumber
& \sum_{I \subset \{1, \dots, n\} } \prod_{i \notin I} \Expect{ e^{-K T (\wp^{i,N})} }
\Expect{ g \left( \scalar{ \Mc_a^{ \wp^{i,N}_{\eps,2}, i \in I} , f} \right) \indic{ \forall i \in I, U_i < 1-e^{-K T(\wp^{i,N})} }  } \\
& = \sum_{k = 0}^n \binom{n}{k} \Expect{ e^{-K T (\wp^{1,N})} }^{n-k}
\Expect{ g \left( \scalar{ \Mc_a^{ \wp^{1,N}_{\eps,2}, \dots, \wp^{k,N}_{\eps,2} } , f} \right) \indic{ \forall i =1 \dots k, U_i < 1-e^{-K T(\wp^{i,N})} }  },
\label{eq:pf_discrete_continuum2}
\end{align}
with the convention that, when $k=0$, the last expectation equals 1 and with, for all $k = 1 \dots n$,
\begin{equation}
\label{eq:pf_discrete_continuum3}
\Mc_a^{\wp^{1,N}_{\eps,2}, \dots, \wp^{k,N}_{\eps,2}}(A) := \frac{\log N}{N^{2-a}} \sum_{z \in D_N \cap A} \indic{ \sum_{i=1}^k \ell_z(\wp_{\eps,2}^{i,N}) \geq \frac{a}{2\pi} (\log N)^2 },
\quad \quad A \text{~Borel~set}.
\end{equation}
The measure above differs from the measures introduced in \eqref{eq:discrete_measure_thick_cap} since it does not require all the trajectories to visit the point $z$. By looking at the subset $I \subset \{1, \dots, k\}$ of loops that actually contribute to the thickness, we see that they are related by
\begin{equation}\label{eq:pf_discrete_continuum5}
\Mc_a^{\wp^{1,N}_{\eps,2}, \dots, \wp^{k,N}_{\eps,2}}
= \sum_{I \subset \{1, \dots, k \}} \Mc_a^{ \cap_{i \in I} \wp_{\eps,2}^{i,N}}.
\end{equation}
Let us come back to the analysis of the asymptotics of \eqref{eq:pf_discrete_continuum2}.
By \eqref{eq:lem_duration_scaling} we already have the convergence of $\Expect{ e^{-K T (\wp^{1,N})} }$ towards $\Expect{ e^{-K T (\wp)} }$ where $\wp$ is distributed according to \eqref{eq:law_loop_height}. In Lemma \ref{lem:scaling_limit} below, we show that a consequence of Theorem \ref{th:jegoRW} is that the liminf of the second expectation in \eqref{eq:pf_discrete_continuum2} is at least
\[
\Expect{ g \left( c_0^a \scalar{ \Mc_a^{ \wp^{1}_{2,\eps}, \dots, \wp^{k}_{2,\eps} } , f} \right) \indic{ \forall i =1 \dots k, U_i < 1-e^{-K T(\wp^{i})} }  }.
\]
Here $\wp^i, i =1 \dots k$, and $\Mc_a^{ \wp^{1}_{2,\eps}, \dots, \wp^{k}_{2,\eps} }$ are the continuum analogues of the notations we introduced above. More precisely, $\wp^i, i =1 \dots k$, are i.i.d. loops distributed according to \eqref{eq:law_loop_height} and
\begin{equation}
\label{eq:pf_discrete_continuum4}
\Mc_a^{\wp^1_{\eps,2}, \dots, \wp^k_{\eps,2}} := \sum_{I \subset \{1, \dots, k\} } \Mc_a^{\cap_{i \in I} \wp^i_{\eps,2}}
\end{equation}
where $\Mc_a^{\cap_{i \in I} \wp^i_{\eps,2}}$ is the Brownian chaos associated to $\wp^i_{\eps,2}, i \in I$; see Section \ref{sec:def}.
Wrapping things up, we have obtained that the liminf of the left hand side of \eqref{eq:pf_discrete_continuum1} is at least
\[
\sum_{k = 0}^n \binom{n}{k} \Expect{ e^{-K T (\wp^1)} }^{n-k}
\Expect{ g \left( c_0^a \scalar{ \Mc_a^{ \wp^{1}_{2,\eps}, \dots, \wp^{k}_{2,\eps} } , f} \right) \indic{ \forall i =1 \dots k, U_i < 1-e^{-K T(\wp^{i})} }  }.
\]
By reversing the above line of argument (which is possible thanks to Lemma \ref{lem:loops_height}), we see that this is exactly the right hand side of \eqref{eq:pf_discrete_continuum1}. It concludes the proof.
\end{proof}

We finish this section by stating and proving Lemma \ref{lem:scaling_limit}.
As in the proof of Proposition \ref{prop:discrete_vs_continuum}, we will consider two sets of i.i.d. loops $\wp^{i,N}, i=1 \dots n$, and $\wp^i, i=1 \dots n$, in the discrete and in the continuum respectively, as well as their associated measures $\Mc_a^{\wp^{1,N}_{\eps,2}, \dots, \wp^{n,N}_{\eps,2}}$ and $\Mc_a^{\wp^1_{\eps,2}, \dots, \wp^n_{\eps,2}}$ defined respectively in \eqref{eq:pf_discrete_continuum3} and \eqref{eq:pf_discrete_continuum4}. Let also $U_i, i=1 \dots n,$ be i.i.d. uniform random variables on $[0,1]$ that are independent of the loops above.

\begin{lemma}\label{lem:scaling_limit}
Let $f : D \to [0,\infty)$ be a nonnegative continuous function and $g : [0,\infty) \to \R$ be a nondecreasing bounded continuous function. Then,
\begin{align}
\label{eq:lem_scaling_limit1}
& \liminf_{N \to \infty}
\Expect{ g \left( \scalar{ \Mc_a^{ \wp^{1,N}_{2,\eps}, \dots, \wp^{n,N}_{2,\eps} } , f} \right) \indic{ \forall i =1 \dots n, U_i < 1-e^{-K T(\wp^{i,N})} }  } \\
\label{eq:lem_scaling_limit2}
& \geq \Expect{ g \left( c_0^a \scalar{ \Mc_a^{ \wp^{1}_{2,\eps}, \dots, \wp^{n}_{2,\eps} } , f} \right) \indic{ \forall i =1 \dots n, U_i < 1-e^{-K T(\wp^{i})} }  }.
\end{align}
\end{lemma}

\begin{proof}[Proof of Lemma \ref{lem:scaling_limit}]
To ease notations, we will assume that $n =1$. The general case follows from similar arguments. In particular, note that the convergence of the Brownian chaos measures in Theorem \ref{th:jegoRW} holds jointly for any number of trajectories.
In what follows, we will denote $(z_\bot^N, z_\eps^N, \wp_{\eps,1}^N, \wp_{\eps,2}^N)$, resp. $(z_\bot, z_\eps, \wp_{\eps,1}, \wp_{\eps,2})$, a random element whose law is described in Lemma \ref{Lem Decomp min discr}, resp. in Lemma \ref{lem:loops_height}.
We also consider a uniform random variable $U$ on $[0,1]$ independent of all the variables above.

The expectation in \eqref{eq:lem_scaling_limit1} is equal to
\begin{align}
\label{eq:pf_lem_sl1}
\sum_{ \tilde{z}_\bot^N , \tilde{z}_\eps^N \in D_N } \Prob{ (z_\bot^N, z_\eps^N ) = (\tilde{z}_\bot^N, \tilde{z}_\eps^N) }
\Expect{ \left. g \left( \scalar{ \Mc_a^{ \wp_{\eps,2}^N } , f} \right) \indic{ U < 1-e^{-K T(\wp^N)} } \right\vert  (z_\bot^N, z_\eps^N ) = (\tilde{z}_\bot^N, \tilde{z}_\eps^N) }.
\end{align}
Let us fix $\tilde{z}_\bot, \tilde{z}_\eps \in D$ and denote $\tilde{z}_\bot^N = N^{-1} \floor{ N \tilde{z}_\bot }$ and $\tilde{z}_\eps^N = N^{-1} \floor{ N \tilde{z}_\eps }$. Assume that the event $E_N := \{ (z_\bot^N, z_\eps^N ) = (\tilde{z}_\bot^N, \tilde{z}_\eps^N) \}$ has positive probability.
By Lemma \ref{Lem Decomp min discr}, conditioned on this event, $\wp_{1,\eps}^N$ and $\wp_{2,\eps}^N$ are independent random walk trajectories distributed according to \eqref{eq:discrete_law_wp_12}. By Theorem \ref{th:jegoRW}, the joint law of
$( \Mc_a^{\wp_{\eps,2}^N}, T(\wp_{\eps,2}^N) )$ conditioned on $E_N$ converges weakly towards the joint law of $(c_0^a \Mc_a^{\wp_{\eps,2}}, T(\wp_{\eps,2}) )$ conditioned on $E:= \{ (z_\bot, z_\eps) = ( \tilde{z}_\bot, \tilde{z}_\eps ) \}$. The topology considered is the product topology with, on the one hand, the topology of vague convergence of measures on $D(\tilde{z}_\bot) := \{ z \in D: \Im(z) > \Im(\tilde{z}_\bot) \}$ and, on the other hand, the standard Euclidean topology on $\R$.
Because of this topology, we introduce for any $\delta>0$ a bounded continuous function $f_\delta : D \to [0,\infty)$ which coincide with $f$ on $\{z \in D(\tilde{z}): \mathrm{dist}(z, \C \setminus D(\tilde{z}_\bot)) > \delta \}$ and which has a support compactly included in $D(\tilde{z}_\bot)$. We choose $f_\delta$ in such a way that $f \geq f_\delta$. Since the support of $f_\delta$ is a compact subset of $D(\tilde{z}_\bot)$, we will be able to use the convergence of the measures integrated against $f_\delta$.

By conditional independence of $\wp_{\eps,1}^N$ and $\wp_{\eps,2}^N$ (and of $\wp_{\eps,1}$ and $\wp_{\eps,2}$), we can add a third component and we have the joint convergence of $( \Mc_a^{\wp_{\eps,2}^N}, T(\wp_{\eps,2}^N), T(\wp_{\eps,1}^N) )$. We add this third component because we are interested in the total duration $T(\wp^N) = T(\wp^N_{\eps,1}) + T(\wp^N_{\eps,2})$. Overall, this shows that for all $\delta>0$,
\begin{align*}
& \liminf_{N \to \infty} \Expect{ \left. g \left( \scalar{ \Mc_a^{ \wp_{\eps,2}^N } , f} \right) \indic{ U < 1-e^{-K T(\wp^N)} } \right\vert  (z_\bot^N, z_\eps^N ) = (\tilde{z}_\bot^N, \tilde{z}_\eps^N) } \\
& \geq
\liminf_{N \to \infty} \Expect{ \left. g \left( \scalar{ \Mc_a^{ \wp_{\eps,2}^N } , f_\delta} \right) \indic{ U < 1-e^{-K T(\wp^N)} } \right\vert  (z_\bot^N, z_\eps^N ) = (\tilde{z}_\bot^N, \tilde{z}_\eps^N) } \\
& =
\Expect{ \left. g \left( c_0^a \scalar{ \Mc_a^{ \wp_{\eps,2} } , f_\delta} \right) \indic{ U < 1-e^{-K T(\wp)} } \right\vert  (z_\bot, z_\eps ) = (\tilde{z}_\bot, \tilde{z}_\eps) }.
\end{align*}
Since $\scalar{ \Mc_a^{ \wp_{\eps,2} } , f_\delta} \to \scalar{ \Mc_a^{ \wp_{\eps,2} } , f}$ as $\delta \to 0$ in $L^1$ (see Remark \ref{rk:vague} below), we have obtained
\begin{align*}
& \liminf_{N \to \infty} \Expect{ \left. g \left( \scalar{ \Mc_a^{ \wp_{\eps,2}^N } , f} \right) \indic{ U < 1-e^{-K T(\wp^N)} } \right\vert  (z_\bot^N, z_\eps^N ) = (\tilde{z}_\bot^N, \tilde{z}_\eps^N) } \\
& \geq
\Expect{ \left. g \left( c_0^a \scalar{ \Mc_a^{ \wp_{\eps,2} } , f} \right) \indic{ U < 1-e^{-K T(\wp)} } \right\vert  (z_\bot, z_\eps ) = (\tilde{z}_\bot, \tilde{z}_\eps) }.
\end{align*}
Moreover, by \eqref{eq:lem_convergence_zz}, $(z_\bot^N, z_\eps^N)$ converges in distribution towards $(z_\bot, z_\eps)$. One can then use an approach similar to the one used in \cite{jegoRW} (see especially Lemma 3.6 therein) to deduce that the liminf of \eqref{eq:pf_lem_sl1} is at least
\[
\int_{D \times D} \Prob{ (z_\bot, z_\eps) = (d\tilde{z}_\bot, d\tilde{z}_\eps) }
\Expect{ \left. g \left( c_0^a \scalar{ \Mc_a^{ \wp_{\eps,2} } , f} \right) \indic{ U < 1-e^{-K T(\wp)} } \right\vert  (z_\bot, z_\eps ) = (\tilde{z}_\bot, \tilde{z}_\eps) }.
\]
We omit the details. This concludes the proof since the last display is equal to the expectation in \eqref{eq:lem_scaling_limit2}.
\end{proof}

\begin{remark}\label{rk:vague}
In the above proof, we had to consider a function $f_\delta$ whose support was compactly included in the underlying domain. We then got rid of this function by letting $\delta \to 0$ and arguing that $\scalar{ \Mc_a^{ \wp_{\eps,2} } , f_\delta} \to \scalar{ \Mc_a^{ \wp_{\eps,2} } , f}$ in $L^1$. This is justified by the simple fact that the first moment of the measure (see (1.4) in \cite{jegoRW}), evaluated against a set located at a distance at most $\delta$ from the boundary of the domain, vanishes as $\delta \to 0$. In the discrete, because of poorer estimates on the discrete Poisson kernel, these estimates near the boundary are not as clear and this is why the convergence obtained in \cite{jegoRW} is stated for the topology of vague (instead of weak) convergence. We mention nevertheless that these difficulties might very well be overcome for a flat portion of the boundary, which is the case in the setting of the current article. But our point is that this is not needed.
\end{remark} 

\begin{appendices}

\section{Convergence of the Wick square of the discrete GFF}\label{app_Wick}

In this Section we recall and prove a folklore result (Lemma \ref{L:Wick_discrete_continuous}) about the construction of the Wick square $:\varphi^2:$ of the GFF $\varphi$ based on a discrete approximation. Recall that $:\varphi^2$ is a random generalised function that is measurable with respect to $\varphi$.
We start with a preliminary lemma.

\begin{lemma}
\label{Lem moment method}
Let $\varphi$ be the GFF on $D$ with $0$ boundary conditions.
Let $f\in C_{0}(D)$.
Assume that there is a real-valued r.v. $\xi$ coupled to the GFF
$\varphi$, such that
\begin{displaymath}
\mathbb{E}[\xi^{2}]
=\mathbb{E}[\langle :\varphi^{2}:,f\rangle^{2}],
\end{displaymath}
and for every $k\in\mathbb{N}$
and $\hat{f}_{1},\hat{f}_{2},\dots ,\hat{f}_{k}\in C_{0}(D)$,
\begin{displaymath}
\mathbb{E}\Big[
\xi
\prod_{i=1}^{k}
\langle \varphi,\hat{f}_{i}\rangle
\Big]
=
\mathbb{E}\Big[
\langle :\varphi^{2}:,f\rangle
\prod_{i=1}^{k}
\langle \varphi,\hat{f}_{i}\rangle
\Big].
\end{displaymath}
Then $\xi = \langle :\varphi^{2}:,f\rangle$ a.s.
\end{lemma}

\begin{proof}
Here we will essentially detail an argument sketched in \cite{Wolpert2}.
Denote $\tilde{\xi} = \mathbb{E}[\xi\vert \varphi]$.
By definition,
$\tilde{\xi}$ is a r.v. measurable with respect to the GFF $\varphi$.
Moreover, for every $k\in\mathbb{N}$
and $\hat{f}_{1},\hat{f}_{2},\dots ,\hat{f}_{k}\in C_{0}(D)$,
\begin{equation}
\label{Eq Hilbert product}
\mathbb{E}\Big[
\tilde{\xi}
\prod_{i=1}^{k}
\langle \varphi,\hat{f}_{i}\rangle
\Big]
=
\mathbb{E}\Big[
\xi
\prod_{i=1}^{k}
\langle \varphi,\hat{f}_{i}\rangle
\Big]
=
\mathbb{E}\Big[
\langle :\varphi^{2}:,f\rangle
\prod_{i=1}^{k}
\langle \varphi,\hat{f}_{i}\rangle
\Big].
\end{equation}
Consider the Hilbert space $L^{2}(\varphi)$ of
random variables measurable with respect to $\varphi$ and having a second moment.
The finite linear combinations of functionals of form
\begin{displaymath}
\prod_{i=1}^{k}
\langle \varphi,\hat{f}_{i}\rangle
\end{displaymath}
with all possible $k\in\mathbb{N}$
and $\hat{f}_{1},\hat{f}_{2},\dots ,\hat{f}_{n}\in C_{0}(D)$,
are dense in $L^{2}(\varphi)$.
Therefore, \eqref{Eq Hilbert product} implies that
for every $\eta\in L^{2}(\varphi)$,
\begin{displaymath}
\mathbb{E}[\tilde{\xi}\eta] =
\mathbb{E}[\langle :\varphi^{2}:,f\rangle \eta].
\end{displaymath}
Since an element of a separable Hilbert is characterized by the values of inner products against the other elements of the space, we get that
\begin{displaymath}
\langle :\varphi^{2}:,f\rangle = \tilde{\xi} =
\mathbb{E}[\xi\vert \varphi]
~~\text{a.s.}
\end{displaymath} 
Finally,
\begin{displaymath}
\mathbb{E}[\langle :\varphi^{2}:,f\rangle^{2}] =
\mathbb{E}[\xi^{2}] =
\mathbb{E}[\tilde{\xi}^{2}]
+
\mathbb{E}[(\xi - \tilde{\xi})^{2}]
=\mathbb{E}[\langle :\varphi^{2}:,f\rangle^{2}]
+\mathbb{E}[(\xi - \tilde{\xi})^{2}].
\end{displaymath}
Therefore $\mathbb{E}[(\xi - \tilde{\xi})^{2}] = 0$
and $\xi = \tilde{\xi} = \langle :\varphi^{2}:,f\rangle$ a.s.
\end{proof}

We are now ready to state the main result of this section.
Let $(D_N)_{N \geq 1}$ be the discrete approximation \eqref{eq:DN} of $D$ and let $\varphi_{N}$ be the discrete GFF on $D_{N}$ with $0$ boundary conditions. For all $z \in D_N$, let $:\varphi_N^2:(z) = \varphi_N(z)^2 - G_{D_N}(z,z)$. We view $\varphi_N$ and $:\varphi_N^2:$ as random element of $\R^{C_0(D)}$ by setting for all $f\in C_{0}(D)$,
\begin{displaymath}
\langle \varphi_N,f\rangle
=
\dfrac{1}{N^{2}}
\sum_{z\in D_{N}}\varphi_{N}(z) f(z)
\quad \text{and} \quad
\langle :\varphi_{N}^{2}:,f\rangle
=
\dfrac{1}{N^{2}}
\sum_{z\in D_{N}}:\varphi_{N}^{2}:(z) f(z).
\end{displaymath}
The topology that we will consider on the space $\R^{C_0(D)}$ will be the product topology.

\begin{lemma}\label{L:Wick_discrete_continuous}
$(\varphi_N, :\varphi_N^2:)$ converges in law to $(\varphi, :\varphi^2:)$ as $N \to \infty$.
\end{lemma}

\begin{proof}
We will use a moment method, by relying on Lemma \ref{Lem moment method}.
For the general expression of moments of Wick powers and their
representation through Feynman diagrams without self-loops,
we refer to \cite[Theorem 3.12]{Janson1997GaussHilbSpaces}.

Let $n\in\mathbb{N}$ and
$f_{1},f_{2},\dots, f_{n}\in C_{0}(D)$. We are going to show that
\[
(\varphi_{N},
\langle :\varphi_{N}^{2}:,f_{1}\rangle,\dots, 
\langle :\varphi_{N}^{2}:,f_{n}\rangle)_{N\geq 1}
\xrightarrow[N \to \infty]{\mathrm{(d)}} (\varphi,
\langle :\varphi^{2}:,f_{1}\rangle,\dots, 
\langle :\varphi^{2}:,f_{n}\rangle).
\]
We have that for every $f\in C_{0}(D)$,
\begin{displaymath}
\mathbb{E}[\langle :\varphi_{N}^{2}:,f\rangle^{2}]
=\dfrac{1}{N^{4}}\sum_{z,w\in D_{N}}
2 G_{D_{N}}(z,w)^{2} f(z) f(w),
\end{displaymath}
which converges as $N\to +\infty$ to
\begin{displaymath}
\mathbb{E}[\langle :\varphi^{2}:,f\rangle^{2}]
=\int_{D^{2}}
2 G_{D}(z,w)^{2} f(z) f(w) \, dz\, dw.
\end{displaymath}
Therefore, each of the sequences
$(\langle :\varphi_{N}^{2}:,f_{j}\rangle)_{N\geq 1}$
is tight, and thus
the sequence 
$(\varphi_{N},
\langle :\varphi_{N}^{2}:,f_{1}\rangle,\dots$, 
$\langle :\varphi_{N}^{2}:,f_{n}\rangle)_{N\geq 1}$
is tight.
Let $(\varphi,\xi_{1},\dots, \xi_{n})$
be a subsequential limit in law.
We want to check that for ever
$j\in\{1,\dots,n\}$,
$\xi_{j} = \langle :\varphi^{2}:,f_{j}\rangle$ a.s.

First, let us check that
\begin{displaymath}
\mathbb{E}[\xi_{j}^{2}] = 
\lim_{N\to +\infty}\mathbb{E}[\langle :\varphi_{N}^{2}:,f_{i}\rangle^{2}]
=\mathbb{E}[\langle :\varphi^{2}:,f_{j}\rangle^{2}].
\end{displaymath}
For this we will consider the fourth moments.
We have that
\begin{multline*}
\mathbb{E}[\langle :\varphi_{N}^{2}:,f_{j}\rangle^{4}]
=
\dfrac{1}{N^{8}}
\sum_{z_{1},z_{2},z_{3},z_{4}\in D_{N}}
\big(48 G_{D_{N}}(z_{1},z_{2})G_{D_{N}}(z_{2},z_{3})G_{D_{N}}(z_{3},z_{4})
G_{D_{N}}(z_{4},z_{1})
\\+
12 G_{D_{N}}(z_{1},z_{2})^{2}G_{D_{N}}(z_{3},z_{4})^{2}
\big)
f_{j}(z_{1})f_{j}(z_{2})f_{j}(z_{3})f_{j}(z_{4}).
\end{multline*}
The fourth moment converges as $N\to +\infty$,
and in particular it is bounded in $N$.
Therefore, 
$\mathbb{E}[\xi_{i}^{2}] = 
\lim_{N\to +\infty}\mathbb{E}[\langle :\varphi_{N}^{2}:,f_{i}\rangle^{2}]$.

Now let $k\in\mathbb{N}$
and $\hat{f}_{1},\hat{f}_{2},\dots ,\hat{f}_{k}\in C_{0}(D)$.
We are interested in the moment
\begin{displaymath}
\mathbb{E}\Big[
\langle :\varphi^{2}:,f_{j}\rangle
\prod_{i=1}^{k}
\langle \varphi,\hat{f}_{i}\rangle
\Big].
\end{displaymath}
It equals $0$ if $k=0$ or if $k$ is odd.
If $k\geq 2$ and $k$ is even, then it equals
\begin{displaymath}
\sum_{\substack{l,m\in \{1,\dots ,k\}\\l\neq m}}
\mathbb{E}\Big[
\prod_{\substack{1\leq i\leq k\\ i\neq l,m}}
\langle \varphi,\hat{f}_{i}\rangle
\Big]
\int_{D^{3}}
G_{D}(z_{1},z_{3})G_{D}(z_{2},z_{3}) \hat{f}_{l}(z_{1}) \hat{f}_{m}(z_{2}) 
f_{j}(z_{3})\, dz_{1}\, dz_{2}\,dz_{3}.
\end{displaymath}
For the moments in discrete, we have a similar expression,
with $\varphi_{N}$ instead of $\varphi$,
$G_{D_{N}}$ instead of $G_{D}$, and
\begin{displaymath}
\dfrac{1}{N^{6}} \sum_{z_{1},z_{2},z_{3}\in D_{N}}
\end{displaymath}
instead of $\int_{D^{3}}\, dz_{1}\, dz_{2}\,dz_{3}$.
Therefore, by convergence as $N\to +\infty$,
\begin{displaymath}
\mathbb{E}\Big[
\xi_{j}
\prod_{i=1}^{k}
\langle \varphi,\hat{f}_{i}\rangle
\Big]
=
\mathbb{E}\Big[
\langle :\varphi^{2}:,f_{j}\rangle
\prod_{i=1}^{k}
\langle \varphi,\hat{f}_{i}\rangle
\Big],
\end{displaymath}
and we conclude by Lemma \ref{Lem moment method}.
\end{proof}

\section{Green function}

In this section, we briefly recall the behaviour of the Green function in the discrete setting.
The Euler--Mascheroni constant $\gamma_{\rm EM}$ will appear in the asymptotics of the discrete Green function and we recall that it is defined by
\begin{equation}\label{eq:Euler-Mascheroni}
\gamma_{\rm EM} = 
\lim_{n\to +\infty}
\Big(
-\log(n) + \sum_{1\leq k\leq n} \dfrac{1}{k}
\Big).
\end{equation}

\begin{lemma}\label{lem:Green_discrete}
There exists $C>0$ such that for all $z, z' \in D_N$, 
\begin{equation}
\label{eq:lem_Green_discrete1}
G_{D_N}(z,z') \leq \frac{1}{2\pi} \log \max \left(N, \frac1{|z-z'|}\right) + C.
\end{equation}
For all set $A$ compactly included in $D$, there exists $C=C(A)>0$ such that for all $z, z' \in A \cap D_N$, 
\begin{equation}
\label{eq:lem_Green_discrete2}
G_{D_N}(z,z') \geq \frac{1}{2\pi} \log \max \left(N, \frac1{|z-z'|}\right) - C.
\end{equation}
For all $z \in D$,  if we denote $z_N$ a point in $D_N$ closest to $z$, then
\begin{equation}
\label{eq:lem_Green_discrete3}
\lim_{N \to \infty} G_{D_N}(z_N,z_N) - \frac{1}{2\pi} \log N = \frac1{2\pi} \log \CR(z,D) + \frac1{2\pi} \left( \gamma_\mathrm{EM} + \frac12 \log 8 \right).
\end{equation}
\end{lemma}

\begin{proof}
\eqref{eq:lem_Green_discrete1} and \eqref{eq:lem_Green_discrete2} are direct consequences of Theorem 4.4.4 and Proposition 4.6.2 of \cite{LawlerLimic2010RW}.  \eqref{eq:lem_Green_discrete3} can be found for instance in Theorem 1.17 of \cite{BiskupLectures}. Note that the constant $\frac1{2\pi} \left( \gamma_\mathrm{EM} + \frac12 \log 8 \right)$ is the constant order term in the expansion of the $0$-potential on $\Z^2$; see \cite[Theorem~4.4.4]{LawlerLimic2010RW}. We emphasise that in the current paper $G_{D_N}(z,z)$ blows up like $\frac{1}{2\pi} \log N$ whereas in \cite{LawlerLimic2010RW} and \cite{BiskupLectures}, $G_{D_N}(z,z)$ blows up like $\frac{2}{\pi} \log N$, hence the difference of factor 4 between our setting and theirs.
\end{proof}

\section{Special functions}\label{app:special}

In this section, we recall the definition and list a few properties of some special functions that appear in the current paper.

$\bullet$
Gamma function:
\begin{equation}
\label{eq:gamma_function}
\Gamma(x) = \int_0^\infty \frac{1}{t^{1-x}} e^{-t} \d t, \quad x >0.
\end{equation}
When $x=1/2$,
\begin{equation}
\label{eq:gamma_function_1/2}
\Gamma(1/2) = \sqrt{\pi}.
\end{equation}

$\bullet$
The Beta function is related to the Gamma function as follows:
\begin{equation}
\label{eq:beta_function}
\int_0^1 t^{x-1} (1-t)^{y-1} \d t = \frac{\Gamma(x) \Gamma(y)}{\Gamma(x+y)}, \quad x,y>0.
\end{equation}

$\bullet$
Modified Bessel function of the first kind:
\begin{equation}
\label{eq:modified_Bessel_function}
I_\alpha(x) = \sum_{n=0}^\infty \frac{1}{n! \Gamma(n+\alpha+1)} \left( \frac{x}{2} \right)^{2n+\alpha}, \quad x >0, ~\alpha > -1.
\end{equation}
Using Legendre duplication formula $\Gamma(x) \Gamma(x+1/2) = 2^{1-2x} \sqrt{\pi} \Gamma(2x)$, we see that when $\alpha = -1/2$,
\begin{equation}
\label{eq:modified_Bessel_function-1/2}
I_{-1/2}(x) = \sqrt{\frac{2}{\pi}} \frac{1}{\sqrt{x}} \cosh(x).
\end{equation}
In general, for all $\alpha > -1$,
\begin{equation}
\label{eq:modified_Bessel_function_asymptotic}
I_\alpha(x) \sim \frac{e^x}{\sqrt{2\pi x}} \quad \text{as}~ x \to \infty.
\end{equation}

$\bullet$
Kummer’s confluent hypergeometric function:
\begin{equation}
\label{eq:Kummer_function}
{}_1 F_1(\theta,1,x) = 1 + \sum_{n \geq 1} \frac{\theta (\theta+1) \dots (\theta + n -1)}{n!^2} x^n, \quad x \geq 0, ~\theta >0.
\end{equation}
For any $\theta >0$,
\begin{equation}
\label{eq:Kummer_function_asymptotic}
{}_1 F_1(\theta,1,x) \sim \frac{1}{\Gamma(\theta)} e^x x^{\theta -1} \quad \text{as}~ x \to \infty.
\end{equation}
See \cite[Section 13.5]{special}.

\end{appendices}

\bibliographystyle{alpha}
\bibliography{bibliography}

\end{document}